\documentclass[10pt,reqno]{amsart}

\usepackage[latin2]{inputenc}
\usepackage{amsmath}
\usepackage{graphicx}
\usepackage{amssymb}
\usepackage{esint}
\usepackage{xcolor}
\usepackage{amsthm}
\usepackage{dsfont}
\usepackage{enumitem}
\usepackage{mathtools}
\usepackage{esvect}
\usepackage[linktocpage=true,colorlinks=true,linkcolor=Blue,citecolor=Green]{hyperref}
\usepackage[english]{babel}

\newtheorem{theorem}{Theorem}
\newtheorem{proposition}[theorem]{Proposition}
\newtheorem{lemma}[theorem]{Lemma}

\newtheorem{remark}[theorem]{Remark}
\newtheorem*{theorem*}{Theorem}
\theoremstyle{definition}
\newtheorem{definition}[theorem]{Definition}
\newtheorem{assumption}{Assumption}
\newtheorem*{indhypothesis}{Induction hypothesis}

\def\Xint#1{\mathchoice
{\XXint\displaystyle\textstyle{#1}}%
{\XXint\textstyle\scriptstyle{#1}}%
{\XXint\scriptstyle\scriptscriptstyle{#1}}%
{\XXint\scriptscriptstyle\scriptscriptstyle{#1}}%
\!\int}
\def\XXint#1#2#3{{\setbox0=\hbox{$#1{#2#3}{\int}$ }
\vcenter{\hbox{$#2#3$ }}\kern-.6\wd0}}

\def\dashint{\Xint-}

\definecolor{Yellow}{rgb}{0.95,0.9,0.0} 
\definecolor{Red}{rgb}{0.8,0.1,0.1}
\definecolor{Green}{rgb}{0.1,0.65,0.2}
\definecolor{Blue}{rgb}{0.1,0.1,0.8}
\definecolor{Purple}{rgb}{0.7,0.1,0.7}
\definecolor{Grey}{rgb}{0.6,0.6,0.6}

\newcommand{\supp}{\operatorname{supp}}

\newcommand{\Rd}[1][d]{{\mathbb{R}^{#1}}}
\newcommand{\N}{\mathbb{N}}
\allowdisplaybreaks[2]
\newcommand{\dx}{\,\mathrm{d}x}
\newcommand{\ds}{\,\mathrm{d}s}

\newcommand{\Prob}{\mathbf{P}}

\begin{document}

\title[Higher-order linearized correctors in stochastic homogenization]
{Corrector estimates for higher-order linearizations in stochastic homogenization
of nonlinear uniformly elliptic equations}

\author{Sebastian Hensel}
\address{Institute of Science and Technology Austria (IST Austria), Am~Campus~1, 
3400 Klosterneuburg, Austria}
\email{sebastian.hensel@ist.ac.at}

\begin{abstract}
Corrector estimates constitute a key ingredient in the derivation of
optimal convergence rates via two-scale expansion techniques 
in homogenization theory of random uniformly elliptic equations.
The present work follows up---in terms of corrector estimates---on the
recent work of Fischer and Neukamm (arXiv:1908.02273) which provides a quantitative stochastic
homogenization theory of nonlinear uniformly elliptic equations under a spectral
gap assumption. We establish optimal-order estimates (with respect to the scaling in the ratio
between the microscopic and the macroscopic scale) for higher-order linearized correctors.
A rather straightforward consequence of the corrector estimates is the higher-order regularity of the
associated homogenized monotone operator. 
\end{abstract} 


\maketitle

\section{Introduction}
Consider the setting of a monotone, uniformly elliptic and bounded PDE
\begin{align}
\label{Random}
-\nabla\cdot A\Big(\frac{x}{\varepsilon},\nabla u_\varepsilon\Big) = \nabla\cdot f
\quad\text{in } \mathbb{R}^d,\quad f\in C^\infty_{\mathrm{cpt}}(\mathbb{R}^d;\mathbb{R}^d),
\end{align}
with $\varepsilon \ll 1$ denoting a microscale. We in addition assume that
the monotone nonlinearity $A$ is random (see Subsection~\ref{subsec:assumptions}
for a precise account on the assumptions of this work). The theory of nonlinear stochastic homogenization 
is then concerned with the behavior of the solutions to equation~\eqref{Random} in the
limit $\varepsilon\to 0$. 

If the monotone nonlinearity is sampled according to a stationary
and ergodic probability distribution (which we will always assume), the classical qualitative prediction
(see, e.g., \cite{Maso1986} and \cite{DalMaso1986}) consists of the convergence of 
$u_\varepsilon$ to the solution $u_{\mathrm{hom}}$ of an effective
nonlinear PDE 
\begin{align}
\label{Effective}
-\nabla\cdot A_{\mathrm{hom}}(\nabla u_{\mathrm{hom}}) = \nabla\cdot f
\quad\text{in } \mathbb{R}^d,
\end{align}
with $A_{\mathrm{hom}}$ being a monotone, uniformly elliptic and bounded operator.
In the rigorous transition from the random model~\eqref{Random} to the deterministic
effective model~\eqref{Effective}, next to the purely qualitative questions
of convergence or the derivation of a homogenization formula for the effective
operator $A_{\mathrm{hom}}$, also quantitative aspects like the 
validity of convergence rates are obviously of interest. 

For all of these questions in homogenization theory, the probably most fundamental concept 
is the notion of the homogenization corrector $\phi^\varepsilon_\xi$, which for a constant macroscopic field gradient
$\xi\in\mathbb{R}^d$ is given by the almost surely sublinearly growing solution of 
\begin{align}
\label{eq:rescaledCorrector}
-\nabla\cdot A\Big(\frac{x}{\varepsilon},\xi{+}\nabla\phi^\varepsilon_\xi\Big) = 0
\quad\text{in } \mathbb{R}^d.
\end{align}
For instance, by means of the homogenization correctors
the homogenization formula for the effective operator reads as
\begin{align}
A_{\mathrm{hom}}(\xi) = \Big\langle A\Big(\frac{x}{\varepsilon},\xi{+}\nabla\phi^\varepsilon_\xi\Big)\Big\rangle,
\end{align}
which is well-defined as a consequence of stationarity of the underlying probability distribution.
For quantitatively inclined questions like those concerned with
the derivation of convergence rates, it is useful to
introduce in addition a notion of flux correctors. For a given constant macroscopic field gradient
$\xi\in\mathbb{R}^d$, the associated flux corrector $\sigma_\xi$ is a random
field with almost surely sublinear growth at infinity,
taking values in the skew-symmetric matrices $\mathbb{R}^{d\times d}_{\mathrm{skew}}$,
and solving
\begin{align}
\label{eq:rescaledFluxCorrector}
\nabla\cdot\sigma_\xi^\varepsilon = A\Big(\frac{x}{\varepsilon},\xi{+}\nabla\phi^\varepsilon_\xi\Big)
- A_{\mathrm{hom}}(\xi)
\quad\text{in } \mathbb{R}^d.
\end{align} 
The merit of the corrector pair $(\phi^\varepsilon_\xi,\sigma^\varepsilon_\xi)$
is that it allows to represent, at least on a formal level, the error for the two-scale
expansion $w_\varepsilon := u_{\mathrm{hom}}(x) + \phi^\varepsilon_\xi(x)|_{\xi=\nabla u_{\mathrm{hom}}(x)}$
in divergence form by means of first-order linearized correctors
\begin{equation}
\label{eq:twoScaleExpansion}
\begin{aligned}
&-\nabla\cdot A\Big(\frac{x}{\varepsilon},\nabla w_\varepsilon\Big)
\\&
= \nabla\cdot f
- \nabla\cdot\Big(\big((a_\xi^\varepsilon\otimes\partial_\xi\phi^\varepsilon_\xi)|_{\xi=\nabla u_{\mathrm{hom}}(x)}
{-}\partial_\xi\sigma^\varepsilon_\xi|_{\xi=\nabla u_{\mathrm{hom}}(x)}\big)
\colon\nabla^2 u_{\mathrm{hom}}\Big),
\end{aligned}
\end{equation}
where we also introduced the linearized coefficient field 
$a_\xi^\varepsilon:=\partial_\xi A(\frac{x}{\varepsilon},\xi{+}\nabla\phi^\varepsilon_\xi)$.
It is clear from the previous display that estimates on the corrector
pair~$(\phi^\varepsilon_\xi,\sigma^\varepsilon_\xi)$ (and its first-order linearization) 
constitute a key ingredient in quantifying the convergence $u_\varepsilon \to u_{\mathrm{hom}}$.
In the present nonlinear setting, we refer to the recent work of Fischer and Neukamm~\cite{Fischer2019} where this
program was carried out in the regime of a spectral gap assumption, resulting in homogenization error estimates being optimal
in terms of scaling with respect to~$\varepsilon$.

We establish in the present work optimal-order estimates (with respect to the scaling in~$\varepsilon$) 
for higher-order linearized homogenization and flux correctors. Given a linearization order~$L\in\mathbb{N}_0$
and a family of vectors~$w_1,\ldots,w_L\in\mathbb{R}^d$, the $L$th order linearized homogenization
corrector is formally given by the directional derivative~$\phi^\varepsilon_{\xi,w_1\odot\cdots\odot w_L}=
(\partial_\xi\phi^\varepsilon_\xi)[w_1\odot\cdots\odot w_L]$. Its defining PDE may be obtained by
differentiating the nonlinear corrector problem~\eqref{eq:rescaledCorrector} in the
macroscopic variable~$\xi\in\mathbb{R}^d$. In particular, note that $\phi^\varepsilon_{\xi,w_1\odot\cdots\odot w_L}
=\varepsilon\phi_{\xi,w_1\odot\cdots\odot w_L}(\frac{\cdot}{\varepsilon})$ 
where~$\phi_{\xi,w_1\odot\cdots\odot w_L}$ formally represents the $L$th order 
directional derivative (in direction of~$w_1\odot\cdots\odot w_L$)
of the almost surely sublinearly growing solution of
\begin{align}
\label{eq:rescaledCorrectorAux}
-\nabla\cdot A(x,\xi{+}\nabla\phi_\xi) = 0
\quad\text{in } \mathbb{R}^d.
\end{align}
We then derive on the level of~$\phi_{\xi,w_1\odot\cdots\odot w_L}$, 
amongst other things (cf.\ Theorem~\ref{theo:correctorBounds}
for a more precise statement), corrector estimates of the form
\begin{align}
\label{eq:showcaseEstimate}
\bigg\langle\bigg|\,\dashint_{B_1(x_0)}
\big|\phi_{\xi,w_1\odot\cdots\odot w_L}
\big|^2\bigg|^{q}\bigg\rangle^\frac{1}{q}
\lesssim_{L,q,|\xi|} |w_1|^2\cdots|w_L|^2\mu_*^2\big(1{+}|x_0|\big)
\end{align}
with the scaling function $\mu_*\colon\mathbb{R}_{>0}\to \mathbb{R}_{>0}$ 
defined by~\eqref{eq:scalingCorrectorBounds}. This in turn implies
\begin{align}
\label{eq:showcaseEstimateRescaled}
\bigg\langle\bigg|\,\dashint_{B_\varepsilon(x_0)}
\big|\phi^\varepsilon_{\xi,w_1\odot\cdots\odot w_L}
\big|^2\bigg|^{q}\bigg\rangle^\frac{1}{q}
\lesssim_{L,q,|\xi|} \varepsilon^2\mu^2_{*}\Big(\frac{1}{\varepsilon}\Big) 
|w_1|^2\cdots|w_L|^2 \mu_*^2\big(1{+}|x_0|\big)
\end{align}
as is immediate from the scaling relation
$\phi^\varepsilon_{\xi,w_1\odot\cdots\odot w_L}
=\varepsilon\phi_{\xi,w_1\odot\cdots\odot w_L}(\frac{\cdot}{\varepsilon})$,
a change of variables as well as~\eqref{eq:scalingCorrectorBounds}.
In the case~$L=1$, this recovers the optimal-order corrector
estimates of~\cite{Fischer2019}. As properties of~$\phi_{\xi,w_1\odot\cdots\odot w_L}$
may always be translated into properties of~$\phi^\varepsilon_{\xi,w_1\odot\cdots\odot w_L}$
based on their scaling relation, from Section~\ref{subsec:assumptions} onwards we set $\varepsilon=1$
and study higher-order linearizations of~\eqref{eq:rescaledCorrectorAux}.

For a proof of corrector estimates of the form~\eqref{eq:showcaseEstimate}
in terms of higher-order linearized correctors,
we devise a suitable inductive scheme to propagate corrector estimates
from one linearization order to the next. The actual implementation of this
inductive scheme, cf.\ Subsections~\ref{subsec:indStep}--\ref{subsec:limitPassage} 
below, is in large parts directly inspired by the methods of
Gloria, Neukamm and Otto~\cite{Gloria2020}--\cite{Gloria2019}, Fischer and Neukamm~\cite{Fischer2019}
as well as Josien and Otto~\cite{Josien2020}. Similar to the latter two works, 
we also employ a small-scale regularity assumption (see Assumption~\ref{assumption:smallScaleReg} below).
%

\subsection{Applications for corrector estimates of higher-order linearizations}
The motivation for the present work derives from the expectation that estimates 
for higher-order linearized correctors constitute one of the important 
ingredients for open questions of interest in nonlinear stochastic homogenization, e.g.,
\textit{i)} an optimal quantification of the commutability of
homogenization and linearization (cf.\ \cite{Armstrong2020} and~\cite{Armstrong2020a} 
for suboptimal algebraic rates in the regime of finite range of dependence),
or \textit{ii)} the development of a nonlinear analogue of the theory of fluctuations 
as worked out for the linear case in~\cite{Duerinckx2020}, \cite{Duerinckx2020a} and \cite{Duerinckx2019a}.

The former for instance concerns the study of the homogenization of the first-order
linearized problem
\begin{align}
\label{RandomLinearized}
-\nabla\cdot \partial_\xi A\Big(\frac{x}{\varepsilon},\nabla u_\varepsilon\Big)
\nabla U_\varepsilon^{(1)} = \nabla\cdot f^{(1)}
\quad\text{in } \mathbb{R}^d,\quad f^{(1)}\in C^\infty_{\mathrm{cpt}}(\mathbb{R}^d;\mathbb{R}^d)
\end{align}
towards the linearized effective equation
\begin{align}
\label{EffectiveLinearized}
-\nabla\cdot \partial_\xi A_{\mathrm{hom}}(\nabla u_{\mathrm{hom}}) 
\nabla U_{\mathrm{hom}}^{(1)} = \nabla\cdot f^{(1)}
\quad\text{in } \mathbb{R}^d;
\end{align}
of course under appropriate regularity assumption for the nonlinearity. 
It is natural to define a two-scale
expansion of~$U_\varepsilon^{(1)}$ in terms of first-order
linearized homogenization correctors $W_\varepsilon^{(1)}:=U_{\mathrm{hom}}^{(1)} 
+ (\partial_\xi\phi^\varepsilon_{\xi})[\nabla U^{(1)}_{\mathrm{hom}}]\big|_{\xi=\nabla u_{\mathrm{hom}}}$, 
so that the difference~$\nabla U_\varepsilon^{(1)} 
{-}\nabla W_\varepsilon^{(1)}$ formally satisfies a uniformly elliptic equation with fluctuating
coefficient~$\partial_\xi A(\frac{x}{\varepsilon},\nabla u_\varepsilon)$
and a right hand side, which---amongst other terms---in particular features second-order linearized
homogenization (and flux) correctors. The estimates obtained in the
present work therefore represent a key ingredient if one aims for a 
derivation of optimal-order convergence rates of the homogenization of~\eqref{RandomLinearized}
towards~\eqref{EffectiveLinearized}.

The second topic mentioned above concerns the study of the random fluctuations
of several macroscopic observables of interest in homogenization theory, e.g.,
\begin{align}
\label{eq:macroscopicObservables}
\int_{\mathbb{R}^d} F \cdot \nabla u_\varepsilon,
\quad
\int_{\mathbb{R}^d} F \cdot \nabla \phi^\varepsilon_\xi,
\quad 
F\in C^\infty_{\mathrm{cpt}}(\mathbb{R}^d;\mathbb{R}^d).
\end{align}
In the works~\cite{Duerinckx2020} and~\cite{Duerinckx2020a}, Duerinckx, Gloria and~Otto
identified in the framework of linear stochastic homogenization~$A(\cdot,\xi)=a(\cdot)\xi$
an object, the so-called standard homogenization corrector
\begin{align}
\label{eq:standardHomCommutator}
\Xi_\xi := (a{-}a_{\mathrm{hom}})(\xi{+}\nabla\phi_\xi),
\end{align}
which relates the fluctuations of the
corrector gradients with the fluctuations of the field~$\nabla u_\varepsilon$.
That fluctuations are related in terms of a single object is by no means obvious
as substituting naively a two-scale expansion for~$\nabla u_\varepsilon$ in~$\int_{\mathbb{R}^d} F\cdot\nabla u_\varepsilon$
does not characterize the fluctuations of~$\int_{\mathbb{R}^d} F\cdot\nabla u_\varepsilon$ to leading order 
as observed in~\cite{Gu2016}. 

In a forthcoming work~\cite{Hensel2021}, we perform an intermediate
step towards understanding the fluctuations of random variables of the form~\eqref{eq:macroscopicObservables}
in nonlinear settings. To this end, we introduce a 
nonlinear counterpart of the standard homogenization commutator~\eqref{eq:standardHomCommutator}
and derive a scaling limit result in a Gaussian setting (cf.\ \cite{Duerinckx2019a}). 
As in the linear regime, this nonlinear counterpart of~\eqref{eq:standardHomCommutator} 
also dictates the fluctuations of linear functionals of the corrector gradients~$\nabla\phi_\xi$ 
(and their (higher-order) linearized descendants in terms of (higher-order) linearized homogenization commutators).
The results of the work~\cite{Hensel2021} are based, amongst other things, on estimates for higher-order
linearized homogenization and flux correctors of the dual linearized operator
$-\nabla\cdot a^{*}_\xi\nabla$ (cf.\ Section~\ref{subsec:dualCorrectors} below),
where $a^{*}_\xi$ denotes the transpose of the
linearized coefficient field $a_\xi:=\partial_\xi A(\omega,\xi{+}\nabla\phi_\xi)$.

\subsection{Stochastic homogenization of linear uniformly elliptic equations and systems}
Before we give a precise account of the underlying assumptions for the present work
in Subsection~\ref{subsec:assumptions}, let us first briefly review the by-now substantial literature on the subject. 
The classical results in qualitative stochastic homogenization are due to 
Papanicolaou and Varadhan~\cite{Papanicolaou1981} and Kozlov~\cite{Kozlov1980}, who studied heat conduction in a 
randomly heterogeneous medium under the assumption of stationarity and ergodicity 
(for the discrete setting, see~\cite{Kozlov1987} and~\cite{Kuennemann1983}). 
The first result in quantitative stochastic homogenization is due to
Yurinskii~\cite{Yurinskii1987}, who derived a suboptimal quantitative result
for linear elliptic PDEs under a uniform mixing condition. Naddaf and Spencer~\cite{Naddaf1998} 
expressed mixing for the first time in the form of a spectral gap inequality,
and as a result obtained optimal results for the fluctuations of the energy density of the corrector.
Their work is however limited to small ellipticity contrast, see also 
Conlon and Naddaf~\cite{Conlon2000} or Conlon and Fahim~\cite{Conlon2014}. 

Extensions to the non-perturbative regime in the discrete setting were established 
through a series of articles by Gloria and Otto~\cite{Gloria2011}, \cite{Gloria2012} and \cite{Gloria2017},
see also Gloria, Neukamm and Otto~\cite{Gloria2014}. These works contain optimal estimates for the approximation error
of the homogenized coefficients, the approximation error for the solutions, the corrector as well 
as the fluctuation of the energy density of the corrector under the assumption of i.i.d.\ conductivities.
In the continuum setting and under spectral gap type assumptions, we refer to the 
works~\cite{Gloria2020} and~\cite{Gloria2019} of Gloria, Neukamm and Otto for optimal-order 
estimates in linear stochastic homogenization. Armstrong, Mourrat and Kuusi~\cite{Armstrong2016}
establish these results in the finite range of dependence regime
including also optimal stochastic integrability, see to this end also
Gloria and Otto~\cite{Gloria2015}.

\subsection{Stochastic homogenization in nonlinear settings}
In the context of qualitative nonlinear stochastic homogenization, 
the first results are due to Dal~Maso and Modica~\cite{Maso1986} and~\cite{DalMaso1986}
in the setting of convex integral functionals.
Lions and Souganidis~\cite{Lions2003} studied the homogenization of Hamilton--Jacobi 
equations under the qualitative assumptions of stationarity and ergodicity.
Caffarelli, Souganidis and Wang~\cite{Caffarelli2004} obtained
stochastic homogenization in the context of nonlinear, uniformly elliptic equations in divergence form
(see also Armstrong and Smart~\cite{Armstrong2013}). A homogenization result in the same framework but 
without assuming uniform ellipticity is due to Armstrong and Smart~\cite{Armstrong2014}. For an example
of stochastic homogenization for nonlinear nonlocal equations, we refer to Schwab~\cite{Schwab2012}.

A first quantitative result in the context of nonlinear stochastic homogenization
was established by Caffarelli and Souganidis~\cite{Caffarelli2010}, who succeeded in the derivation of
a logarithmic-type convergence rate under strong mixing conditions. Substantial progress
in the nonlinear setting was later provided by the works of Armstrong and Smart~\cite{Armstrong2016a} 
on uniformly convex integral functionals, and Armstrong and Mourrat~\cite{Armstrong2015} 
on elliptic equations in divergence form with monotone coefficient fields.
In the two recent works~\cite{Armstrong2020} resp.\ \cite{Armstrong2020a},
Armstrong, Ferguson and Kuusi succeeded in proving that the processes of homogenization and
(first-order resp.\ higher-order) linearization commute. Moreover, as it is 
also the case in the previously mentioned works of Armstrong et al., they 
derive quantitative estimates in terms of a suboptimal algebraic rate of convergence with respect to the
ratio in the microscopic and macroscopic scale, assuming finite range of dependence
for the underlying probability space. The established estimates, however, are optimal
in terms of stochastic integrability. Under a spectral gap assumption,
Fischer and Neukamm~\cite{Fischer2019} recently provided quantitative
homogenization estimates for monotone uniformly elliptic coefficient fields, 
which on one side are the first being optimal in the ratio between the microscopic and macroscopic scale,
but which on the other side are non-optimal in terms of stochastic integrability. 

\subsection{Assumptions and setting}
\label{subsec:assumptions}
In this section, we give a precise account of the underlying assumptions for the present work.
They represent the natural higher-order analogues of the assumptions from~\cite{Fischer2019}. 
We start with the deterministic requirements on the family of monotone operators
(cf.\ \cite[Section 2.1]{Fischer2019}). 

\begin{assumption}[Family of monotone operators]
\label{assumption:operators}
Let $d\in\N$ be the spatial dimension, and let $0<\lambda\leq\Lambda<\infty$
be two constants (playing the role of ellipticity constants in the sequel). 
Let $n\in\N$ and $L\in\N_0$ be given. We then assume that we are equipped with
a family of operators indexed by elements of $\Rd[n]$
\begin{align*}
A\colon\Rd[n]\times\Rd \to \Rd
\end{align*} 
which is subject to the following three conditions:
\begin{itemize}[leftmargin=1.1cm]
\item[(A1)] The map $A$ gives rise to a family of monotone operators in the second
					  variable with lower bound $\lambda$. More precisely, for all $\tilde\omega\in \Rd[n]$
						we require
						\begin{align*}
						\big(A(\tilde\omega,\xi_1){-}A(\tilde\omega,\xi_2)\big)\cdot (\xi_1{-}\xi_2) 
						\geq \lambda|\xi_1{-}\xi_2|^2
						\end{align*}
						for all $\xi_1,\xi_2\in\Rd$. Furthermore, $A(\tilde\omega,0)=0$ for
						all $\tilde\omega\in \Rd[n]$.
\item[(A2)$_L$] Each operator $A(\tilde\omega,\cdot)$, $\tilde\omega\in\Rd[n]$, is $L{+}1$ times differentiable
						in the second variable. In quantitative terms, we assume that for all $\tilde\omega\in\Rd[n]$
						\begin{align*}
						\sup_{k\in\{1,\ldots,L{+}1\}}\sup_{\xi\in\Rd} \big|\partial_\xi^{k} A(\tilde\omega,\xi)\big|
						&\leq \Lambda,
						\\
						\sup_{k\in\{1,\ldots,L{+}1\}} \sup_{\xi\in\Rd} 
						\sup_{\tilde\omega_1,\tilde\omega_2\in\Rd[n],\,\tilde\omega_1\neq\tilde\omega_2}
						\frac{|\partial_\xi^{k} A(\tilde\omega_1,\xi)
						- \partial_\xi^{k} A(\tilde\omega_2,\xi)|}
						{|\tilde\omega_1 - \tilde\omega_2|} &\leq \Lambda.
						\end{align*}
						In particular, we have $|A(\tilde\omega,\xi)|\leq \Lambda|\xi|$
						for all $\tilde\omega\in\Rd[n]$ and all $\xi\in\Rd$.
\item[(A3)$_L$] For each $\xi\in\Rd$ and $k\in\{0,\ldots,L\}$, the map
						$\tilde\omega\mapsto \partial_\xi^k A(\tilde\omega,\xi)$
						is differentiable with uniformly Lipschitz continuous derivative. 
						In quantitative terms, the following bounds are required to hold true
						for all $\tilde\omega\in\Rd[n]$
						\begin{align*}
						|\partial_\omega A(\tilde\omega,\xi)|\leq \Lambda|\xi|,\quad
						\sup_{k\in\{1,\ldots,L\}}\sup_{\xi\in\Rd}
						|\partial_\omega \partial_\xi^k A(\tilde\omega,\xi)| &\leq \Lambda,\quad
						\\
						\sup_{\xi\in\Rd}
						\sup_{\tilde\omega_1,\tilde\omega_2\in\Rd[n],\,\tilde\omega_1\neq\tilde\omega_2}
						\frac{|\partial_\omega A(\tilde\omega_1,\xi)
						- \partial_\omega A(\tilde\omega_2,\xi)|}
						{|\tilde\omega_1 - \tilde\omega_2|} &\leq \Lambda|\xi|,
						\\
						\sup_{k\in\{1,\ldots,L\}}\sup_{\xi\in\Rd}
						\sup_{\tilde\omega_1,\tilde\omega_2\in\Rd[n],\,\tilde\omega_1\neq\tilde\omega_2}
						\frac{|\partial_\omega\partial_\xi^k A(\tilde\omega_1,\xi)
						- \partial_\omega\partial_\xi^k A(\tilde\omega_2,\xi)|}
						{|\tilde\omega_1 - \tilde\omega_2|} &\leq \Lambda.
						\end{align*}
\end{itemize}
For some results, we in addition require the following condition to be true.
\begin{itemize}[leftmargin=1.1cm]
\item[(A4)$_L$] For each $\tilde\omega\in\Rd[n]$, the maps
						$\xi\mapsto \partial_\xi^{L+1} A(\tilde\omega,\xi)$ and 
						$\xi\mapsto\partial_\omega\partial^L_\xi A(\tilde\omega,\xi)$
						are uniformly Lipschitz continuous. More precisely,
						for all $\tilde\omega\in\Rd[n]$ we are equipped with bounds
						\begin{align*}
						\sup_{\xi_1,\xi_2\in\Rd,\,\xi_1\neq\xi_2} 
						\frac{|\partial_\xi^{L{+}1} A(\tilde\omega,\xi_1) - \partial_\xi^{L{+}1} A(\tilde\omega,\xi_2)|}
						{|\xi_1 - \xi_2|} &\leq \Lambda.
						\\
						\sup_{\xi_1,\xi_2\in\Rd,\,\xi_1\neq\xi_2} 
						\frac{|\partial_\omega\partial_\xi^{L} A(\tilde\omega,\xi_1) 
						- \partial_\omega\partial_\xi^{L} A(\tilde\omega,\xi_2)|}
						{|\xi_1 - \xi_2|} &\leq \Lambda.
						\end{align*}
\end{itemize}
\end{assumption}

Having the deterministic requirements on the family of monotone operators in place,
we next turn to the probabilistic assumptions.

\begin{assumption}[Stationarity and quantified ergodicity for probability distribution of parameter fields]
\label{assumption:ensembleParameterFields}
Denote by $B^n_1$ the open unit ball in $\Rd[n]$. We call a measurable function 
$\omega\colon \Rd\to B_1^n$ a \emph{parameter field}, and denote by $\Omega$
the \emph{space of parameter fields} with the $L^1_{\mathrm{loc}}(\Rd;B_1^n)$ topology.
We then assume that we are equipped with a probability
measure $\Prob$ on the space of parameter fields $\Omega$ subject to the
following two conditions:
\begin{itemize}[leftmargin=1.1cm]
\item[(P1)] The probability measure $\Prob$ on $\Omega$ is $\Rd$-stationary. In other words,
						the probability distributions of $\omega(\cdot)$ and $\omega(\cdot+z)$
						coincide for all $z\in\Rd$.
\item[(P2)] The probability measure $\Prob$ on $\Omega$ satisfies a \emph{spectral gap inequality}. 
						More precisely, denoting with $\langle\cdot\rangle$ the expectation with respect to $\Prob$, 
						there exists a constant~$\rho>0$ such that for all
						random variables $X$ we have the estimate
						\begin{align}
						\label{eq:spectralGap}
						\big\langle\big|X{-}\langle X\rangle\big|^2\big\rangle
						\leq\frac{1}{\rho^2}\bigg\langle\int\bigg(\,\dashint_{B_1(x)}
						\big|\partial^{\mathrm{fct}}X\big|\,\bigg)^2\,\bigg\rangle
						\end{align}
						with the abbreviation
						\begin{align*}
						\dashint_{B_1(x)}\big|\partial^{\mathrm{fct}}X\big|
						:=\sup_{\delta\omega}\limsup_{t\to 0}
						\frac{|X(\omega{+}t\delta\omega)-X(\omega)|}{t},
						\end{align*}
						where the supremum in the previous display is taken with respect to 
						smooth parameter fields $\delta\omega\colon\Rd\to B^n_1$ 
						such that $\supp\delta\omega\subset B_1(x)$.
\end{itemize}
\end{assumption}

Before we move on with the statement of the last main assumption of this work,
we register the following standard consequence of the spectral gap inequality \eqref{eq:spectralGap}.
\begin{lemma}
\label{lemma:spectralGapHigherMoments}
Let the conditions and notation of Assumption~\ref{assumption:ensembleParameterFields} be in place,
and let $q\in [1,\infty)$. We then have for all random variables $X$ the estimate
\begin{align}
\label{eq:spectralGapHigherMoments}
\big\langle\big|X{-}\langle X\rangle\big|^{2q}\big\rangle^\frac{1}{q}
\leq C^2q^2\bigg\langle\bigg|\int\bigg(\,\dashint_{B_1(x)}
\big|\partial^{\mathrm{fct}}X\big|\,\bigg)^2\,\bigg|^q\bigg\rangle^\frac{1}{q}.
\end{align}
\end{lemma}

We finally state a small-scale regularity assumption which is essential
to obtain optimal-order estimates (i.e., with respect to the ratio of the microscopic and 
macroscopic scale) for linearized homogenization and flux correctors, as well as their
higher-order analogues.

\begin{assumption}[Annealed small-scale regularity condition]
\label{assumption:smallScaleReg}
Let the conditions and notation of 
Assumption~\ref{assumption:ensembleParameterFields} be in place.
We then in addition require that the following small-scale regularity condition is satisfied:
\begin{itemize}[leftmargin=0.9cm]
\item[(R)]  There exist an exponent $\eta\in (0,1)$ and a constant $C>0$ 
						such that for all $q \in [1,\infty)$ it holds
						\begin{align*}
						\bigg\langle\bigg|\sup_{x,y\in B_1,\,x\neq y} 
						\frac{|\omega(x)-\omega(y)|}{|x-y|^\eta}
						\bigg|^{2q}\bigg\rangle^\frac{1}{q}	\leq C^2q^{2C}.
						\end{align*}
\end{itemize}
\end{assumption}

Note that our small-scale regularity condition is slightly weaker than
the corresponding assumption in~\cite{Fischer2019}. For this reason
we provide a proof in Appendix~\ref{app:toolsEllipticRegularity} 
concerning the small-scale H\"older regularity of the (massive) corrector solving the nonlinear corrector
problem~\eqref{eq:PDEMonotoneHomCorrectorLocalized},
see Lemma~\ref{lem:annealedRegCorrectorNonlinear},
which in turn implies small-scale H\"older regularity of
the linearized coefficient field, see Lemma~\ref{lem:annealedHoelderRegLinCoefficient}.

\subsection{Example}
We give an example for a random parameter field
subject to Assumption~\ref{assumption:ensembleParameterFields}
and Assumption~\ref{assumption:smallScaleReg}. To this end,
consider a stationary and centered Gaussian random field 
$\tilde\omega\colon\mathbb{R}^d\to\mathbb{R}^n$, and assume
there exists some $\alpha>0$ such that the Fourier transform $\hat c$ of the 
covariance $c(x):=\langle\tilde\omega(x)\otimes\tilde\omega(0)\rangle$ satisfies
\begin{align*}
\hat{c}(z) \leq C (1{+}|z|)^{-(d+\alpha)}, \quad z\in\mathbb{R}^d.
\end{align*}
We also fix a $1$-Lipschitz map $\beta\colon\mathbb{R}^n\to\mathbb{R}^n$
with $\|\beta\|_{L^\infty} \leq 1$. The random field
$\omega:=\beta(\tilde\omega)$ is then subject to
the requirements of Assumption~\ref{assumption:ensembleParameterFields}
and Assumption~\ref{assumption:smallScaleReg}.
For a proof, see, e.g., \cite[Lemma 3.1, Appendix A.3.1]{Josien2020}.

\subsection{Notation}
We denote by $\N$ the set of positive integers, and define $\N_0:=\N\cup\{0\}$.
For given $d\in\N$, the space of real-valued $d{\times}d$ matrices is denoted
by~$\Rd[d{\times}d]$. The transpose of a matrix~$A\in\Rd[d{\times}d]$ is
given by~$A^*$. We write $\mathbb{R}^{d{\times}d}_{\mathrm{skew}}$ for the space of 
skew-symmetric matrices $A^*=-A$. For a given $L\in\N$, we define
$\mathrm{Par}\{1,\ldots,L\}$ to be the set of all partitions of $\{1,\ldots,L\}$. For any $x_0\in\Rd$ and $R>0$, 
we denote by $B_R(x_0)\subset\Rd$ the $d$-dimensional open ball of radius $R$ centered at $x_0$.
In case of $x_0=0$, we simply write $B_R$. In the rare occasion that the dimension of 
the ambient space is not represented by $d\in\N$ but, say, $n\in\N$, we emphasize the dimension of the
ambient space by writing $B^n_R(x_0)$ for the $n$-dimensional open ball of 
radius $R>0$ centered at $x_0\in\Rd[n]$. 

The tensor product of vectors $v_1,\ldots,v_L\in\Rd$, $L\geq 2$,
is denoted by $v_1\otimes\cdots\otimes v_L$. For the symmetric tensor product, we write
$v_1\odot\cdots\odot v_L$. The $L$-fold tensor product of a vector $v\in\Rd$ is
abbreviated as $v^{\otimes L}$; or $v^{\odot L}$ for the corresponding symmetric version.
For a differentiable map $A\colon\Rd[n]\times\Rd\to \Rd,\,(\omega,\xi) \mapsto A(\omega,\xi)$,
we make use of the usual notation $\partial_\omega A,\partial_\xi A$ for the respective partial derivatives. 
Higher-order (possibly mixed) partial derivatives of a map
$A\colon\Rd[n]\times\Rd\to \Rd$
are denoted by $\partial_\omega^l A,\partial_\xi^k A,
\partial_\omega^l\partial_\xi^k A$ for any $k,l\in\N$. 

Integrals $\int_{\Rd} f \dx$ with respect to the $d$-dimensional
Lebesgue measure are abbreviated in the course of the paper as $\int f$.
Given a Lebesgue-measurable subset $A\subset\Rd$ with finite and non-trivial Lebesgue 
measure $|A|\in (0,\infty)$, we denote by $\dashint_{A} f := 
\frac{1}{|A|}\int \mathds{1}_A f$ the average integral of $f$ over $A$.
Here, $\mathds{1}_A$ represents the characteristic function with respect to a set $A$.
For a probability measure $\Prob$ on a measure space $(\Omega,\mathcal{A})$,
we write $\langle\cdot\rangle$ for the expectation with respect to $\Prob$.

We make use of the usual notation of Lebesgue and Sobolev spaces on $\Rd$
(with respect to the Lebesgue measure), e.g., $L^p(\Rd),\,W^{1,p}(\Rd),
\,H^1(\Rd):=W^{1,2}(\Rd)$ and so on. For a probability measure $\Prob$ on a measure space $(\Omega,\mathcal{A})$,
we instead use the notation $L^p_{\langle\cdot\rangle}$.
If we want to emphasize the target space, say,
a finite-dimensional real vector space $V$, we do so by writing $L^p(\Rd;V)$.
For the Lebesgue resp.\ Sobolev spaces on $\Rd$ with only locally finite norm,
we write $L^p_{\mathrm{loc}}(\Rd),\,H^1_{\mathrm{loc}}(\Rd)$ and so on. Furthermore, in the
case of uniformly locally finite norm, i.e., 
\begin{align*}
\sup_{x_0\in\Rd} \|f\|_{L^p(B_1(x_0))} < \infty
\quad\text{resp.}\quad
\sup_{x_0\in\Rd} \|f\|_{H^1(B_1(x_0))} < \infty
\end{align*}
we reserve the notation $L^p_{\mathrm{uloc}}(\Rd)$ resp.\ $H^1_{\mathrm{uloc}}(\Rd)$.
The space of all compactly supported and smooth functions on $\Rd$ is denoted
by $C^\infty_{\mathrm{cpt}}(\Rd)$. Finally, for an exponent $q\in [1,\infty]$, 
we write $q_*\in [1,\infty]$ for its dual H\"older exponent: $\frac{1}{q} + \frac{1}{q_*}=1$.

\subsection{Structure of the paper} In the upcoming Section~\ref{sec:mainResults},
we formulate the main results of the present work and provide definitions
for the underlying key objects. Section~\ref{sec:outlineStrategy} is devoted
to a discussion of the strategy for the proof of the main results. In the course
of it, we also collect several auxiliary results representing the main steps in the proof.
Section~\ref{sec:proofs} contains the proofs of all the main and auxiliary results
as stated in the previous two sections. The paper finishes with three appendices.
In Appendix~\ref{app:toolsEllipticRegularity} we list (and partly prove) several
results from elliptic regularity theory. Most of them are classical results
from deterministic theory. In addition, we also rely on some annealed regularity
theory; however, only in a perturbative regime \`a la Meyers. Appendix~\ref{app:existenceCorrectors}
deals with existence of higher-order linearized correctors for a suitable class of
parameter fields. Finally, as the proof of the main results proceeds by an induction 
over the linearization order, we formulate and prove in Appendix~\ref{app:baseCaseInd}
the corresponding statements taking care of the base case of the induction.

\section{Main results}\label{sec:mainResults}
This section collects the statements of the main results of this work
which are twofold: \textit{i)} corrector estimates for higher-order linearizations
of the nonlinear problem, and~\textit{ii)} higher-order regularity of the
homogenized monotone operator.

\subsection{Corrector bounds for higher-order linearized correctors}
The first main result constitutes the analogue (and slight extension) of~\cite[Corollary~15]{Fischer2019}
for the higher-order linearized correctors of Definition~\ref{def:correctorsHigherOrderLinearization}.

\begin{theorem}[Corrector estimates for higher-order linearizations]
\label{theo:correctorBounds}
Let $L\in\N$ and $M>0$ be fixed.
Let the requirements and notation of (A1), (A2)$_L$ and (A3)$_L$ of
Assumption~\ref{assumption:operators}, (P1) and (P2) of
Assumption~\ref{assumption:ensembleParameterFields}, and (R) of 
Assumption~\ref{assumption:smallScaleReg} be in place.
Fix a set of vectors $w_1,\ldots,w_L\in\Rd$ and define
$B:=w_1\odot\cdots\odot w_{L}$. Let
\begin{align*}
\phi_{\xi,B}\in H^1_{\mathrm{loc}}(\Rd)
\quad\text{and}\quad
\sigma_{\xi,B}\in H^1_{\mathrm{loc}}(\Rd;\mathbb{R}^{d\times d}_{\mathrm{skew}})
\end{align*}
be the linearized homogenization and flux corrector from Definition~\ref{def:correctorsHigherOrderLinearization}.

There exists a constant $C=C(d,\lambda,\Lambda,\nu,\rho,\eta,M,L)$ such that for all $|\xi|\leq M$, 
all $q\in [1,\infty)$, all $x_0\in\Rd$, and all compactly supported and 
square-integrable $g_\phi,g_\sigma$ it holds
\begin{align}
\label{eq:linearFunctionalCorrectorGradientBound}
\bigg\langle\bigg|\bigg(\int g_\phi\cdot\nabla\phi_{\xi,B}, 
\int g_\sigma^{kl}\cdot\nabla\sigma_{\xi,B,kl}\bigg)\bigg|^{2q}\bigg\rangle^\frac{1}{q}
&\leq C^2q^{2C}|B|^2\int \big|\big(g_\phi,g_\sigma\big)\big|^2,
\\\label{eq:correctorGradientBound}
\big\langle\big\|\big(\nabla\phi_{\xi,B},\nabla\sigma_{\xi,B}
\big)\big\|^{2q}_{L^2(B_1)}\big\rangle^\frac{1}{q}
&\leq C^2q^{2C}|B|^2,
\\\label{eq:correctorGrowthBound}
\bigg\langle\bigg|\,\dashint_{B_1(x_0)}
\big|\big(\phi_{\xi,B},\sigma_{\xi,B}\big)
\big|^2\bigg|^{q}\bigg\rangle^\frac{1}{q}
&\leq C^2q^{2C}|B|^2\mu_*^2\big(1{+}|x_0|\big),
\end{align}
with the scaling function $\mu_*\colon\mathbb{R}_{>0}\to \mathbb{R}_{>0}$ defined by
\begin{align}
\label{eq:scalingCorrectorBounds}
\mu_*(\ell) := \begin{cases}
							 \ell^\frac{1}{2}, & d=1, \\
							 \log^\frac{1}{2}(1{+}\ell), & d=2, \\
							 1, & d\geq 3.
							 \end{cases}
\end{align}

Let $\xi_0\in\Rd$ and $K\in\N_0$ be fixed, and assume in addition to the previous requirements
that~(A2)$_{L{+}K}$ and~(A3)$_{L{+}K}$ from Assumption~\ref{assumption:operators} hold true.
We may then define $\Prob$-almost surely $K$th-order Taylor expansions
for the linearized homogenization and flux correctors with base point~$\xi_0$ by means of
\begin{align}
\label{eq:TaylorLinearizedHomCorrector}
\Phi^K_{\xi_0,B}(\xi) &:= \phi_{\xi,B} - 
\sum_{k=0}^K \frac{1}{k!}\partial_\xi\phi_{\xi_0,B}\big[(\xi-\xi_0)^{\odot k}\big]
\in H^1_{\mathrm{loc}}(\Rd),
\\
\label{eq:TaylorLinearizedFluxCorrector}
\Sigma^K_{\xi_0,B}(\xi) &:= \sigma_{\xi,B}
- \sum_{k=0}^K \frac{1}{k!}\partial_\xi\sigma_{\xi_0,B}\big[(\xi-\xi_0)^{\odot k}\big]
\in H^1_{\mathrm{loc}}(\Rd;\mathbb{R}^{d\times d}_{\mathrm{skew}}).
\end{align}
Under the stronger assumptions of~(A2)$_{L{+}K{+}1}$ and~(A3)$_{L{+}K{+}1}$ from Assumption~\ref{assumption:operators},
there exists a constant $C=C(d,\lambda,\Lambda,\nu,\rho,\eta,M,L,K)$ such that for all $|(\xi_0,\xi)|\leq M$, 
all $q\in [1,\infty)$, and all $x_0\in\Rd$ it holds
\begin{align}
\label{eq:correctorGradientBoundDifferences}
\big\langle\big\|\big(\nabla\Phi^K_{\xi_0,B}(\xi),\nabla\Sigma^K_{\xi_0,B}(\xi)
\big)\big\|^{2q}_{L^2(B_1)}\big\rangle^\frac{1}{q}
&\leq C^2q^{2C}|B|^2|\xi{-}\xi_0|^{2(K{+}1)},
\\\label{eq:correctorGrowthBoundDifferences}
\bigg\langle\bigg|\,\dashint_{B_1(x_0)}\big(
\Phi^K_{\xi_0,B}(\xi),\Sigma^K_{\xi_0,B}(\xi)
\big)\bigg|^{2q}\bigg\rangle^\frac{1}{q}
&\leq C^2q^{2C}|B|^2|\xi{-}\xi_0|^{2(K{+}1)}\mu_*^2\big(1{+}|x_0|\big).
\end{align}
\end{theorem}


\subsection{Differentiability of the homogenized operator}
A rather straightforward consequence of the estimates
for higher-order linearized correctors is the higher-order regularity of the
associated homogenized monotone operator. 

\begin{theorem}[Higher-order regularity of the homogenized operator]
\label{theo:diffHomOperator}
Let $L\in\N$ and $M>0$ be fixed.
Let the requirements and notation of (A1), (A2)$_{L}$ and (A3)$_{L}$ 
of Assumption~\ref{assumption:operators}, (P1) and (P2) of
Assumption~\ref{assumption:ensembleParameterFields}, and (R) of 
Assumption~\ref{assumption:smallScaleReg} be in place.
Fix next a set of vectors $w_1,\ldots,w_L\in\Rd$ and define
$B:=w_1\odot\cdots\odot w_{L}$. Let finally
\begin{align}
\label{def:homogenizedOperator}
\Rd\ni\xi \mapsto \bar A(\xi) := \langle q_{\xi} \rangle
\end{align}
be the homogenized operator, with the flux $q_\xi$ being defined in~\eqref{eq:PDEMonotoneFlux}.

The homogenized operator is $L$ times differentiable as a map $\xi\mapsto \bar A(\xi)$. 
There exists a constant $C=C(d,\lambda,\Lambda,\nu,\rho,\eta,M,L)$ such that for all $|\xi|\leq M$
its $L$th G\^ateaux derivative in direction $B$ admits the bound
\begin{align}
\label{eq:boundDerivHomOperator}
\big|\partial_\xi^{L}\bar A(\xi)[B]\big| \leq C|B|.
\end{align}
Finally, we have the following representation for the $L$th order G\^ateaux derivative in direction $B$
\begin{align}
\label{eq:repDerivHomOperator}
\partial_\xi^{L}\bar A(\xi)[B] = \langle q_{\xi,B} \rangle.
\end{align}
Here, $q_{\xi,B}$ denotes the linearized flux from~\eqref{eq:LinearizedFlux}.
\end{theorem}

\subsection{Basic definitions}
We introduce the precise definitions of the 
homogenization and flux correctors and their (higher-order) linearized analogues.
We start by recalling these notions on the level of the nonlinear problem.

Given $\xi\in\Rd$, the equation for the homogenization corrector is given by
\begin{subequations}
\begin{align}
\label{eq:PDEMonotoneHomCorrector}
&-\nabla\cdot A(\omega,\xi{+}\nabla\phi_\xi)
= 0.
\end{align}
Abbreviating the flux by means of
\begin{align}
\label{eq:PDEMonotoneFlux}
q_{\xi} &:= A(\omega,\xi{+}\nabla\phi_\xi),
\end{align}
the equation for the corresponding flux corrector is given by
\begin{align}
\label{eq:PDEMonotoneFluxCorrector}
- \Delta\sigma_{\xi,kl}
&= (e_l\otimes e_k - e_k\otimes e_l) : \nabla q_{\xi}.
\end{align}
Sublinear growth of the flux corrector gives rise to
\begin{align}
\label{eq:PDEMonotoneHelmholtzDecomp}
q_\xi - \langle q_\xi \rangle
= \nabla\cdot\sigma_\xi.
\end{align}
\end{subequations}

\begin{definition}[Homogenization correctors and flux correctors of the nonlinear problem]
\label{def:correctorsMonotonePDE}
Let the requirements and notation of (A1), (A2)$_0$ and (A3)$_0$ of
Assumption~\ref{assumption:operators}, as well as (P1) and (P2) of
Assumption~\ref{assumption:ensembleParameterFields} be in place.
Let $\xi\in\Rd$ be given. The corresponding 
\emph{homogenization corrector}~$\phi_{\xi}$
and \emph{flux corrector} $\sigma_{\xi}$
are two random fields
\begin{align*}
(\phi_{\xi},\sigma_{\xi})\colon\Omega\times\Rd
\to \mathbb{R}\times\mathbb{R}^{d\times d}_{\mathrm{skew}}
\end{align*}
subject to the following list of requirements:
\begin{itemize}[leftmargin=0.7cm]
\item[(i)]  It holds $\Prob$-almost surely that $\phi_{\xi}\in H^1_{\mathrm{loc}}(\Rd)$,
            $\sigma_{\xi}\in H^1_{\mathrm{loc}}(\Rd;\mathbb{R}^{d\times d}_{\mathrm{skew}})$
						as well as $\dashint_{B_1}(\phi_{\xi},\sigma_{\xi})\dx = 0$.
						In addition, the associated PDEs \eqref{eq:PDEMonotoneHomCorrector} resp.\
						\eqref{eq:PDEMonotoneFluxCorrector}~and~\eqref{eq:PDEMonotoneHelmholtzDecomp} are satisfied
						in the distributional sense $\Prob$-almost surely.
						
\item[(ii)] The gradients $\nabla\phi_{\xi}$ and $\nabla\sigma_{\xi}$
						are stationary random fields. Moreover, it holds
						\begin{align*}
						\big\langle(\nabla\phi_{\xi},\nabla\sigma_{\xi})\big\rangle=0, \quad
						\big\langle|\nabla\phi_{\xi}|^2\big\rangle
						+\big\langle|\nabla\sigma_{\xi}|^2\big\rangle < \infty.
						\end{align*}
\item[(iii)] The two random fields $\phi_{\xi}$ and $\sigma_{\xi}$ 
						 feature $\Prob$-almost surely sublinear growth at infinity
						 \begin{align*}
						 \lim_{R\to\infty}\frac{1}{R^2}\dashint_{B_R} 
						 \big|(\phi_{\xi},\sigma_{\xi})\big|^2 \dx = 0.
						 \end{align*}
\end{itemize}
\end{definition}

We next introduce the (higher-order) linearized analogues of
the corrector equations~\eqref{eq:PDEMonotoneHomCorrector}--\eqref{eq:PDEMonotoneHelmholtzDecomp}
by formally differentiating in the macroscopic variable.
To this end, let a linearization order $L\in\N$ be fixed.
We also fix vectors $w_1,\ldots,w_L\in\Rd$ and let
$B:=w_1\odot\cdots\odot w_{L}$. Finally, fix $\xi\in\Rd$ and
denote by $a_\xi$ the coefficient field $\partial_\xi A(\omega,\xi+\nabla\phi_\xi)$. 
Due to (A1) and (A2)$_{0}$ from Assumption~\ref{assumption:operators}, this coefficient field
is uniformly elliptic and bounded with respect to the same constants $(\lambda,\Lambda)$
from Assumption~\ref{assumption:operators}. In terms of statistical properties,
it is stationary and ergodic. 

As suggested by the Fa\`a~di~Bruno formula, the equation for the \emph{$L$th-order 
linearized homogenization corrector in direction $B$} shall be given by
\begin{subequations}
\begin{equation}
\label{eq:PDEhigherOrderLinearizedCorrector}
\begin{aligned}
&- \nabla\cdot a_\xi(\mathds{1}_{L=1}B + \nabla \phi_{\xi,B})
\\
&= \nabla\cdot\sum_{\substack{\Pi\in\mathrm{Par}\{1,\ldots,L\} \\ \Pi\neq\{\{1,\ldots,L\}\}}}
\partial_\xi^{|\Pi|} A(\omega,\xi{+}\nabla \phi_\xi)
\Big[\bigodot_{\pi\in\Pi}(\mathds{1}_{|\pi|=1}B'_\pi + \nabla \phi_{\xi,B'_\pi})\Big],
\end{aligned}
\end{equation}
where we also introduced the notational convention
\begin{align*}
B'_\pi := \bigodot_{m\in\pi} v_{m},\quad\forall\pi\in\Pi,\,\Pi\in \mathrm{Par}\{1,\ldots,L\}.
\end{align*}
Note that the right hand side only features linearized correctors of order $\leq L{-}1$, if any.
Motivated by this observation, existence of solutions to the 
linearized corrector problem~\eqref{eq:PDEhigherOrderLinearizedCorrector}
with stationary gradient and (almost sure) sublinear growth at infinity will be given \textit{inductively} through 
approximation with an additional massive term, see~\eqref{eq:PDEhigherOrderLinearizedCorrectorLocalized}
for the associated corrector problem. For the latter, solutions may be constructed---again 
inductively---on purely deterministic grounds (under suitable assumptions which are in particular
modeled on the small-scale regularity condition~(R) from Assumption~\ref{assumption:smallScaleReg}). 
For more details, we refer the reader to the discussion in Section~\ref{subsec:indHypo} below.

To state the equation for the linearized flux corrector, we first define the 
linearized flux by means of
\begin{equation}
\label{eq:LinearizedFlux}
\begin{aligned}
q_{\xi,B} &:= a_\xi(\mathds{1}_{L=1}B + \nabla \phi_{\xi,B})
\\&~~~~
+ \sum_{\substack{\Pi\in\mathrm{Par}\{1,\ldots,L\} \\ \Pi\neq\{\{1,\ldots,L\}\}}}
\partial_\xi^{|\Pi|} A(\omega,\xi{+}\nabla \phi_\xi)
\Big[\bigodot_{\pi\in\Pi}(\mathds{1}_{|\pi|=1}B'_\pi + \nabla \phi_{\xi,B'_\pi})\Big].
\end{aligned}
\end{equation}
The associated flux corrector shall then be a solution of
\begin{align}
\label{eq:PDEhigherOrderLinearizedFluxCorrector}
- \Delta\sigma_{\xi,B,kl} 
&= (e_l\otimes e_k - e_k\otimes e_l) : \nabla q_{\xi,B}.
\end{align}
Due to the sublinear growth of the correctors, the previous relations entail that
\begin{align}
\label{eq:PDEhigherOrderLinearizedHelmholtzDecomp}
q_{\xi,B} - \langle q_{\xi,B} \rangle
= \nabla\cdot\sigma_{\xi,B}.
\end{align}
\end{subequations}

\begin{definition}[Higher-order linearized homogenization correctors and flux correctors]
\label{def:correctorsHigherOrderLinearization}
Let $L\in\N$ and $\xi\in\Rd$ be fixed.
Let the requirements and notation of (A1), (A2)$_L$ and (A3)$_L$ of
Assumption~\ref{assumption:operators}, (P1) and (P2) of
Assumption~\ref{assumption:ensembleParameterFields}, and (R) of 
Assumption~\ref{assumption:smallScaleReg} be in place.
We also fix a set of vectors $w_1,\ldots,w_L\in\Rd$ and define
$B:=w_1\odot\cdots\odot w_{L}$. The corresponding
\emph{linearized homogenization corrector}~$\phi_{\xi,B}$
and \emph{flux corrector} $\sigma_{\xi,B}$
are two random fields
\begin{align*}
(\phi_{\xi,B},\sigma_{\xi,B})\colon\Omega\times\Rd
\to \mathbb{R}\times\mathbb{R}^{d\times d}_{\mathrm{skew}}
\end{align*}
subject to the following list of requirements:
\begin{itemize}[leftmargin=0.7cm]
\item[(i)]  It holds $\Prob$-almost surely that $\phi_{\xi,B}\in H^1_{\mathrm{loc}}(\Rd)$,
            $\sigma_{\xi,B}\in H^1_{\mathrm{loc}}(\Rd;\mathbb{R}^{d\times d}_{\mathrm{skew}})$
						as well as $\dashint_{B_1}(\phi_{\xi,B},\sigma_{\xi,B})\dx = 0$.
						In addition, the associated PDEs \eqref{eq:PDEhigherOrderLinearizedCorrector} resp.\
						\eqref{eq:PDEhigherOrderLinearizedFluxCorrector}~and~\eqref{eq:PDEhigherOrderLinearizedHelmholtzDecomp} 
						are satisfied in the distributional sense $\Prob$-almost surely.
\item[(ii)] The gradients $\nabla\phi_{\xi,B}$ and $\nabla\sigma_{\xi,B}$
						are stationary random fields. Moreover, it holds
						\begin{align*}
						\big\langle(\nabla\phi_{\xi,B},\nabla\sigma_{\xi,B})\big\rangle=0, \quad
						\big\langle|\nabla\phi_{\xi,B}|^2\big\rangle
						+\big\langle|\nabla\sigma_{\xi,B}|^2\big\rangle < \infty.
						\end{align*}
\item[(iii)] The two random fields $\phi_{\xi,B}$ and $\sigma_{\xi,B}$ 
						 feature $\Prob$-almost surely sublinear growth at infinity
						 \begin{align*}
						 \lim_{R\to\infty}\frac{1}{R^2}\dashint_{B_R} 
						 \big|(\phi_{\xi,B},\sigma_{\xi,B})\big|^2 \dx = 0.
						 \end{align*}
\end{itemize}
\end{definition}

\subsection{Linearized correctors for the dual linearized operator}
\label{subsec:dualCorrectors}
It is an immediate consequence of the proofs that 
analogous results hold true for the (higher-order) correctors
of the dual linearized operator $-\nabla\cdot a^{*}_\xi\nabla$,
where $a^{*}_\xi$ denotes the transpose of the
linearized coefficient field $a_\xi:=\partial_\xi A(\omega,\xi{+}\nabla\phi_\xi)$.
We state these corrector results for the dual linearized operator for ease of reference
for future works. 

Let $L\in\N$ and $M>0$ be fixed.
Let the requirements and notation of (A1), (A2)$_L$ and (A3)$_L$ of
Assumption~\ref{assumption:operators}, (P1) and (P2) of
Assumption~\ref{assumption:ensembleParameterFields}, and (R) of 
Assumption~\ref{assumption:smallScaleReg} be in place.
Fix moreover a set of vectors $w_1,w_2,\ldots,w_L\in\Rd$ and define
$B:=w_1\otimes(w_2\odot\cdots\odot w_{L})$. For a partition 
$\Pi\in\mathrm{Par}\{1,\ldots,L\}$ with $\Pi\neq\{\{1,\ldots,L\}\}$,
denote by $\pi_1^*$ the unique element $\pi\in\Pi$ such that $1\in\pi$.
The equation for the \emph{$L$th-order linearized homogenization corrector 
in direction $B$ of the dual linearized operator} is then given by
\begin{subequations}
\begin{align}
\label{eq:PDEhigherOrderLinearizedCorrectorDual}
&- \nabla\cdot a^*_\xi(\mathds{1}_{L=1}B + \nabla \phi^*_{\xi,B})
\\ \nonumber
&= \nabla\cdot\sum_{\substack{\Pi\in\mathrm{Par}\{1,\ldots,L\} \\ \Pi\neq\{\{1,\ldots,L\}\}}}
\partial_\xi^{|\Pi|{-}1} \left.\big(a_\xi^*v\big)
\right|_{v=\mathds{1}_{|\pi_1^*|=1}B'_{\pi_1^*}{+}\nabla\phi^*_{\xi,B'_{\pi_1^*}}}
\bigg[\bigodot_{\substack{\pi\in\Pi \\ 1\notin \pi}}
\big(\mathds{1}_{|\pi|=1}B'_\pi {+} \nabla \phi_{\xi,B'_\pi}\big)\bigg],
\end{align}
where we also relied, for each $\Pi\in\mathrm{Par}\{1,\ldots,L\}$, on the notational convention
\begin{align*}
B'_\pi := 
\begin{cases}
v_1 \otimes (v_{l_2}\odot \cdots \odot v_{l_{|\pi|}}), 
& \text{ if } \pi = \pi_1^* = \{1,l_2,\ldots,l_{|\pi|}\}
\text{ with } l_2 < \cdots < l_{|\pi|}, \\
\bigodot_{m\in\pi} v_{m}, & \text{ else}.
\end{cases}
\end{align*}
With the dual linearized flux given by 
\begin{align}
\label{eq:LinearizedFluxDual}
q^*_{\xi,B} &:= a^*_\xi(\mathds{1}_{L=1}B + \nabla \phi^*_{\xi,B})
\\&~\nonumber
+\sum_{\substack{\Pi\in\mathrm{Par}\{1,\ldots,L\} \\ \Pi\neq\{\{1,\ldots,L\}\}}}
\partial_\xi^{|\Pi|{-}1} \left.\big(a_\xi^*v\big)
\right|_{v=\mathds{1}_{|\pi_1^*|=1}B'_{\pi_1^*}{+}\nabla\phi^*_{\xi,B'_{\pi_1^*}}}
\bigg[\bigodot_{\substack{\pi\in\Pi \\ 1\notin \pi}}
\big(\mathds{1}_{|\pi|=1}B'_\pi {+} \nabla \phi_{\xi,B'_\pi}\big)\bigg],
\end{align}
the \emph{$L$th order linearized flux corrector in direction $B$ of the dual linearized operator} 
is in turn a solution of
\begin{align}
\label{eq:PDEhigherOrderLinearizedFluxCorrectorDual}
- \Delta\sigma^*_{\xi,B,kl} 
&= (e_l\otimes e_k - e_k\otimes e_l) : \nabla q^*_{\xi,B},
\end{align}
as well as
\begin{align}
\label{eq:PDEhigherOrderLinearizedHelmholtzDecompDual}
q^*_{\xi,B} - \langle q^*_{\xi,B} \rangle
= \nabla\cdot\sigma^*_{\xi,B}.
\end{align}
\end{subequations}
(More precisely, the notion of linearized homogenization and flux correctors of the dual linearized problem are
understood in the precise sense of Definition~\ref{def:correctorsHigherOrderLinearization}, with the 
equations~\eqref{eq:PDEhigherOrderLinearizedCorrector}--\eqref{eq:PDEhigherOrderLinearizedHelmholtzDecomp} replaced by 
the equations~\eqref{eq:PDEhigherOrderLinearizedCorrectorDual}--\eqref{eq:PDEhigherOrderLinearizedHelmholtzDecompDual}).  

Under the above assumptions, the following analogous results to Theorem~\ref{theo:correctorBounds} then hold true.
First, there exists a constant $C=C(d,\lambda,\Lambda,\nu,\rho,\eta,M,L)$ such that for all $|\xi|\leq M$, 
all $q\in [1,\infty)$, all $x_0\in\Rd$, and all compactly supported and 
square-integrable $g_\phi,g_\sigma$ it holds
\begin{align}
\label{eq:linearFunctionalCorrectorGradientBoundDual}
\bigg\langle\bigg|\bigg(\int g_\phi\cdot\nabla\phi^*_{\xi,B}, 
\int g_\sigma^{kl}\cdot\nabla\sigma^*_{\xi,B,kl}\bigg)\bigg|^{2q}\bigg\rangle^\frac{1}{q}
&\leq C^2q^{2C}|B|^2\int \big|\big(g_\phi,g_\sigma\big)\big|^2,
\\\label{eq:correctorGradientBoundDual}
\big\langle\big\|\big(\nabla\phi^*_{\xi,B},\nabla\sigma^*_{\xi,B}
\big)\big\|^{2q}_{L^2(B_1)}\big\rangle^\frac{1}{q}
&\leq C^2q^{2C}|B|^2,
\\\label{eq:correctorGrowthBoundDual}
\bigg\langle\bigg|\,\dashint_{B_1(x_0)}
\big|\big(\phi^*_{\xi,B},\sigma^*_{\xi,B}\big)
\big|^2\bigg|^{q}\bigg\rangle^\frac{1}{q}
&\leq C^2q^{2C}|B|^2\mu_*^2\big(1{+}|x_0|\big),
\end{align}
with the scaling function $\mu_*\colon\mathbb{R}_{>0}\to \mathbb{R}_{>0}$ 
defined in~\eqref{eq:scalingCorrectorBounds}. 

Fix $K\in\N$, and assume that~(A2)$_{L{+}K}$ and~(A3)$_{L{+}K}$ from Assumption~\ref{assumption:operators}
hold true on top of the previous assumptions of this section. Then, both the maps $\xi\mapsto\nabla\phi^*_{\xi,B}$ 
and $\xi\mapsto\nabla\sigma^*_{\xi,B}$ are $\Prob$-almost surely
$K$ times G\^ateaux differentiable with values in the 
Fr\'echet space $L^2_{\langle\cdot\rangle}L^2_{\mathrm{loc}}(\Rd)$. Moreover,
for any collection of vectors $w_{L+1},\ldots,w_{L+K}\in\Rd$ and any $k\in\{1,\ldots,K\}$ 
we have the following representations of the $k$th order G\^ateaux derivatives in direction 
$\hat B_k := w_{L{+}1}\odot\cdots\odot w_{L{+}k}$:
\begin{align}
\partial_\xi^k\phi^*_{\xi,B}[\hat B_k] = 
\phi^*_{\xi,w_1\otimes(w_2\odot\cdots\odot w_{L+k})},
\quad \partial_\xi^k\sigma^*_{\xi,B}[\hat B_k] = 
\sigma^*_{\xi,Bw_1\otimes(w_2\odot\cdots\odot w_{L+k})}. 
\end{align} 

Denote by $\overline{a_\xi^*}\in \Rd[d\times d]$ the 
\emph{homogenized coefficient of the dual linearized operator}
characterized by
\begin{align*}
\overline{a_\xi^*}w = \langle q_{\xi,w}^* \rangle, \quad w\in\Rd.
\end{align*}
Then the following version of Theorem~\ref{theo:diffHomOperator} holds true for $L\geq 1$. 
The map $\xi\mapsto \overline{a_\xi^*}$ is $L{-}1$ times differentiable.
There also exists a constant $C=C(d,\lambda,\Lambda,\nu,\rho,\eta,M,L)$ such that for all $|\xi|\leq M$
its $(L{-}1)$th G\^ateaux derivative in the direction of $B'=w_2\odot\cdots\odot w_L$ admits the bound
\begin{align}
\label{eq:boundDerivHomOperatorDual}
\big|\partial_\xi^{L{-}1}\overline{a_\xi^*}[B']\big| \leq C|B'|.
\end{align}
We finally have for all $w\in\Rd$
the following representation for the $(L{-}1)$th order 
G\^ateaux derivative in direction $B'$
\begin{align}
\label{eq:repDerivHomOperatorDual}
\partial_\xi^{L-1}\big(\,\overline{a_\xi^*}w\big)[B'] = \langle q_{\xi,w\otimes B'}^* \rangle.
\end{align}

\section{Outline of strategy}\label{sec:outlineStrategy}

The proof of the corrector bounds from Theorem~\ref{theo:correctorBounds} is based
on the massive approximation of the operator 
$-\nabla\cdot\partial_\xi A(\omega,\xi+\nabla\phi_\xi)$. For the problem with an additional massive term,
we will argue by an induction with respect to the order of the linearization. This will entail the
following analogue of Theorem~\ref{theo:correctorBounds} in terms of the massive approximation.

\begin{theorem}[Estimates for massive correctors]
\label{theo:correctorBoundsMassiveApprox}
Let $L\in\N$, $M>0$ as well as $T\in [1,\infty)$ be fixed.
Let the requirements and notation of (A1), (A2)$_L$ and (A3)$_L$ of
Assumption~\ref{assumption:operators}, (P1) and (P2) of
Assumption~\ref{assumption:ensembleParameterFields}, and (R) of 
Assumption~\ref{assumption:smallScaleReg} be in place.
Fix a set of unit vectors $v_1,\ldots,v_L\in\Rd$ and define
$B:=v_1\odot\cdots\odot v_{L}$. Let
\begin{align*}
\phi^T_{\xi,B}\in H^1_{\mathrm{uloc}}(\Rd),
\quad \sigma^T_{\xi,B}\in H^1_{\mathrm{uloc}}(\Rd;\mathbb{R}^{d\times d}_{\mathrm{skew}})
\quad\text{and}\quad
\psi^T_{\xi,B}\in H^1_{\mathrm{uloc}}(\Rd;\mathbb{R}^{d})
\end{align*}
denote the unique solutions of the linearized corrector problems
\emph{\eqref{eq:PDEhigherOrderLinearizedCorrectorLocalized}--\eqref{eq:HigherOrderLinearizedFlux}},
which $\Prob$-almost surely exist by means of Lemma~\ref{eq:lemmaExistenceLinearizedCorrectorsShort} below.

There exists a constant $C=C(d,\lambda,\Lambda,\nu,\rho,\eta,M,L)$ such that for all $|\xi|\leq M$, 
all $q\in [1,\infty)$, and all compactly supported and square-integrable deterministic fields $g_\phi,g_\sigma,g_\psi$
resp.\ $f_\phi,f_\sigma,f_\psi$ it holds
\begin{equation}
\label{eq:linearFunctionalCorrectorGradientBoundMassiveApprox}
\begin{aligned}
&\bigg\langle\bigg|\bigg(\int g_\phi\cdot\nabla\phi^T_{\xi,B}, 
\int g_\sigma^{kl}\cdot\nabla\sigma^T_{\xi,B,kl},
\int g_\psi^k\cdot\frac{\nabla\psi^T_{\xi,B,k}}{\sqrt{T}}\bigg)\bigg|^{2q}\bigg\rangle^\frac{1}{q}
\\
&\leq C^2q^{2C}\int \big|\big(g_\phi,g_\sigma,g_\psi\big)\big|^2,
\end{aligned}
\end{equation}
and
\begin{equation}
\label{eq:linearFunctionalCorrectorGradientBoundMassiveApprox2}
\begin{aligned}
&\bigg\langle\bigg|\bigg(\int \frac{1}{T}f_\phi\phi^T_{\xi,B}, 
\int \frac{1}{T}f_\sigma^{kl}\sigma^T_{\xi,B,kl},
\int \frac{1}{T}f_\psi^k\frac{\psi^T_{\xi,B,k}}{\sqrt{T}}\bigg)\bigg|^{2q}\bigg\rangle^\frac{1}{q}
\\
&\leq C^2q^{2C}\int \frac{1}{T}\big|\big(f_\phi,f_\sigma,f_\psi\big)\big|^2,
\end{aligned}
\end{equation}
as well as
\begin{align}
\label{eq:correctorGradientBoundMassiveApprox}
\Big\langle\Big\|\Big(\nabla\phi^T_{\xi,B},
\nabla\sigma^T_{\xi,B},\frac{\nabla\psi^T_{\xi,B}}{\sqrt{T}}
\Big)\Big\|^{2q}_{L^2(B_1)}\Big\rangle^\frac{1}{q}
&\leq C^2q^{2C},
\\\label{eq:correctorSchauderMassiveApprox}
\Big\langle\Big\|\Big(\nabla\phi^T_{\xi,B},
\nabla\sigma^T_{\xi,B},\frac{\nabla\psi^T_{\xi,B}}{\sqrt{T}}
\Big)\Big\|^{2q}_{C^\alpha(B_1)}\Big\rangle^\frac{1}{q}
&\leq C^2q^{2C},
\\\label{eq:correctorGrowthoundMassiveApprox}
\bigg\langle\bigg|\,\dashint_{B_1}\Big(\phi^T_{\xi,B},
\sigma^T_{\xi,B},\frac{\psi^T_{\xi,B}}{\sqrt{T}}
\Big)\bigg|^{2q}\bigg\rangle^\frac{1}{q}
&\leq C^2q^{2C}\mu_*^2(\sqrt{T}),
\end{align}
with the scaling function $\mu_*\colon\mathbb{R}_{>0}\to \mathbb{R}_{>0}$
given in~\eqref{eq:scalingCorrectorBounds}. Moreover, the 
relation~\eqref{eq:PDEhigherOrderLinearizedHelmholtzDecompLocalized} holds true.
\end{theorem}

As an input for the base case of the induction we will take the localized corrector of the nonlinear problem.
So let us start by quickly reviewing the corresponding results from~\cite{Fischer2019}.

\subsection{Corrector estimates for the nonlinear PDE: A brief review}
Given $\xi\in\Rd$ and $T\in [1,\infty)$,
the equation for the localized homogenization corrector is given by
\begin{subequations}
\begin{align}
\label{eq:PDEMonotoneHomCorrectorLocalized}
&\frac{1}{T}\phi^T_\xi-\nabla\cdot A(\omega,\xi{+}\nabla\phi_\xi^T)
= 0.
\end{align}
Abbreviating the flux by means of
\begin{align}
\label{eq:fluxNonlinear}
q_{\xi}^T &:= A(\omega,\xi{+}\nabla\phi_\xi^T),
\end{align}
the equation for the corresponding localized flux corrector is given by
\begin{align}
\label{eq:PDEMonotoneFluxCorrectorLocalized}
\frac{1}{T}\sigma_{\xi,kl}^T - \Delta\sigma_{\xi,kl}^T 
&= (e_l\otimes e_k - e_k\otimes e_l) : \nabla q_{\xi}^T.
\end{align}
Moreover, we introduce an auxiliary localized corrector by means of
\begin{align}
\label{eq:PDEMonotoneHelmholtzCorrectorLocalized}
\frac{1}{T}\psi^T_\xi - \Delta\psi^T_\xi = q^T_\xi - \langle q^T_\xi\rangle - \nabla\phi^T_\xi.
\end{align}
The motivation behind the introduction of the auxiliary corrector $\psi^T_\xi$ is
to mimic equation~\eqref{eq:PDEMonotoneHelmholtzDecomp} for the flux correction 
at the level of the massive approximation:
\begin{align}
\label{eq:PDEMonotoneHelmholtzDecompLocalized}
q^T_\xi - \langle q^T_\xi \rangle
= \nabla\cdot\sigma^T_\xi + \frac{1}{T}\psi^T_\xi.
\end{align}
\end{subequations}
We then have the following result, which was essentially proven 
by Fischer and Neukamm~\cite{Fischer2019}. For a proof of
those facts which are not explicitly spelled out in~\cite{Fischer2019},
we refer to the beginning of Appendix~\ref{app:baseCaseInd}.

\begin{proposition}[Estimates for localized homogenization correctors of the nonlinear problem]
\label{prop:estimates1stOrderCorrector}
Let the requirements and notation of (A1), (A2)$_0$ and (A3)$_0$ of
Assumption~\ref{assumption:operators}, as well as (P1) and (P2) of
Assumption~\ref{assumption:ensembleParameterFields} be in place.
Let $T\in [1,\infty)$ be fixed, and for any $\xi\in\Rd$ let 
\begin{align*}
\phi^T_{\xi}\in H^1_{\mathrm{uloc}}(\Rd)
\end{align*}
denote the unique solution of the localized corrector problem~\eqref{eq:PDEMonotoneHomCorrectorLocalized}.
The localized homogenization corrector $\phi^T_\xi$ then admits the following
list of estimates:
\begin{itemize}[leftmargin=0.6cm]
\item There exists a constant $C=C(d,\lambda,\Lambda,\nu,\rho)$ such that for
all $q\in [1,\infty)$, and all compactly supported and square-integrable $f,g$ we have \emph{corrector estimates}
\begin{equation}
\label{eq:BaseCaseCorrectorBounds}
\begin{aligned}
\bigg\langle\bigg|\bigg(\int g\cdot\nabla\phi^T_{\xi}, 
\int \frac{1}{T}f \phi^T_{\xi} \bigg)\bigg|^{2q}\bigg\rangle^\frac{1}{q}
&\leq C^2q^{2C}|\xi|^2\int \Big|\Big(g,\frac{f}{\sqrt{T}}\Big)\Big|^2,
\\ 
\Big\langle\Big\|\Big(\frac{\phi^T_{\xi}}{\sqrt{T}},
\nabla\phi^T_{\xi}\Big)\Big\|^{2q}_{L^2(B_1)}\Big\rangle^\frac{1}{q}
&\leq C^2q^{2C}|\xi|^2.
\end{aligned}
\end{equation}
\begin{subequations}
\item Fix $p\in (2,\infty)$, and let $g$ be a compactly supported and $p$-integrable random field. 
						Then there exists a random field
						$G_\xi^T\in L^1_{\mathrm{uloc}}(\Rd;\Rd[n])$ being related to $g$ via $\phi^T_{\xi}$ 
						in the sense that, $\Prob$-almost surely, it holds for all compactly supported and smooth perturbations 
						$\delta\omega\colon\Rd\to \Rd[n]$ with $\|\delta\omega\|_{L^\infty}\leq 1$
						\begin{align}
						\label{eq:BaseCaseRepresentationMalliavinDerivative}
						\int g\cdot\nabla\delta\phi^T_{\xi} = \int G^T_{\xi}\cdot\delta\omega.
						\end{align}
						
						For any $\kappa\in (0,1]$ there moreover exists a constant $C=C(d,\lambda,\Lambda,\nu,\rho,\kappa)$
						such that for all $q\in [1,\infty)$ the random field
						$G^T_{\xi}$ gives rise to a \emph{sensitivity estimate}
						\begin{align}
						\label{eq:BaseCaseSensitivityBound}
						\bigg\langle\bigg|\,\int\bigg(
						\,\dashint_{B_1(x)}|G^T_{\xi}|\,\bigg)^2\bigg|^{q}\bigg\rangle^\frac{1}{q}
						&\leq C^2q^{2C}|\xi|^2\sup_{\langle F^{2q_*}\rangle=1}
						\int\big\langle|Fg|^{2(q/\kappa)_*}\big\rangle^\frac{1}{(q/\kappa)_*}.
						\end{align}
						
						If $(g_r)_{r\geq 1}$ is a sequence of compactly supported and $p$-integrable random fields, denote by
						$G_{\xi}^{T,r}\in L^1_{\mathrm{uloc}}(\Rd;\Rd[n])$, $r\geq 1$, the random field
						associated to $g_r$, $r\geq 1$, in the sense of~\eqref{eq:BaseCaseRepresentationMalliavinDerivative}.
						Let $g$ be an $L^p_{\mathrm{uloc}}(\Rd;\Rd)$-valued random field, and assume that
						$\Prob$-almost surely it holds $g_r\to g$ in $L^p_{\mathrm{uloc}}(\Rd;\Rd)$.
						Then there exists a random field $G^T_{\xi}$ such that $\Prob$-almost surely						
						\begin{align}
						\label{eq:BaseCaseSensitivityApprox}
						G_{\xi}^{T,r} \to G_{\xi}^{T} \text{ as } r\to\infty 
						\text{ in } L^1_{\mathrm{uloc}}(\Rd;\Rd[n]).
						\end{align}
						In the special case of $g_r = \mathds{1}_{B_r}g$, $r\geq 1$, the limit random
						field is in addition subject to	the sensitivity estimate~\eqref{eq:BaseCaseSensitivityBound}.
\end{subequations}
\item Let in addition to the above requirements the condition~(R) of 
Assumption~\ref{assumption:smallScaleReg} be in place, and let $M>0$ be fixed.
There exist $\alpha=\alpha(d,\lambda,\Lambda)\in (0,\eta)$ and $C=C(d,\lambda,\Lambda,\nu,\rho,\eta,M)$
such that for all $q\in [1,\infty)$ and all $|\xi|\leq M$ we have a 
\emph{small-scale annealed Schauder estimate} in form of
\begin{align}
\label{eq:BaseCaseAnnealedSchauder}
\big\langle\big\|\nabla\phi^T_{\xi}\big\|^{2q}_{C^\alpha(B_1)}\big\rangle^\frac{1}{q}
\leq C^2q^{2C}.
\end{align}
\end{itemize}
\end{proposition}

\subsection{Corrector bounds for higher-order linearizations: the induction hypotheses}
\label{subsec:indHypo}
Let $L\in\N$ and $T\in [1,\infty)$ be fixed. If not otherwise explicitly stated,
let the requirements and notation of (A1), (A2)$_L$ and (A3)$_L$ of
Assumption~\ref{assumption:operators}, (P1) and (P2) of
Assumption~\ref{assumption:ensembleParameterFields}, and (R) of 
Assumption~\ref{assumption:smallScaleReg} be in place.
Before we can formulate the induction hypothesis, we first have to introduce the
analogues for the higher-order linearized homogenization correctors
on the level of the massive approximation. To this end, we fix a set of 
unit vectors $v_1,\ldots,v_L\in\Rd$ and define
$B:=v_1\odot\cdots\odot v_{L}$. Next, fix $\xi\in\Rd$ and
denote by $a_\xi^T$ the coefficient field $\partial_\xi A(\omega,\xi+\nabla\phi^T_\xi)$. 
Due to Assumption~\ref{assumption:operators}, the coefficient field
is uniformly elliptic and bounded with respect to the constants $(\lambda,\Lambda)$
from Assumption~\ref{assumption:operators}.

In anticipation of the higher-order
differentiability of the localized corrector for the nonlinear PDE, we introduce
the equation for the \emph{localized $L$th-order linearized homogenization corrector
in direction $B$} by means of the Fa\`a~di~Bruno formula in form of
\begin{subequations}
\begin{equation}
\label{eq:PDEhigherOrderLinearizedCorrectorLocalized}
\begin{aligned}
&\frac{1}{T}\phi^T_{\xi,B} 
- \nabla\cdot a_\xi^T\big(\mathds{1}_{L=1}B + \nabla \phi^T_{\xi,B}\big)
\\&
= \nabla\cdot\sum_{\substack{\Pi\in\mathrm{Par}\{1,\ldots,L\} \\ \Pi\neq\{\{1,\ldots,L\}\}}}
\partial_\xi^{|\Pi|} A(\omega,\xi{+}\nabla \phi^T_\xi)
\Big[\bigodot_{\pi\in\Pi}(\mathds{1}_{|\pi|=1}B'_\pi + \nabla \phi^T_{\xi,B'_\pi})\Big],
\end{aligned}
\end{equation}
where we also introduced the notational convention
\begin{align*}
B'_\pi := \bigodot_{m\in\pi} v_{m},\quad\forall\pi\in\Pi,\,\Pi\in \mathrm{Par}\{1,\ldots,L\}.
\end{align*}

Note that the right hand side of~\eqref{eq:PDEhigherOrderLinearizedCorrectorLocalized}
only features linearized homogenization correctors up to order $L{-}1$, if any.
Hence, it turns out that we may argue inductively using standard (and, in particular, only deterministic) arguments,
that the corrector problem~\eqref{eq:PDEhigherOrderLinearizedCorrectorLocalized}
admits for every random parameter field $\omega\in\Omega$ a unique solution 
\begin{align*}
\phi^T_{\xi,B}=\phi^T_{\xi,B}(\cdot,\omega)\in H^1_{\mathrm{uloc}}(\Rd).
\end{align*}
In particular, the uniqueness part of this statement entails stationarity of the
linearized corrector $\phi^T_{\xi,B}$
in the sense that for each $z\in\Rd$ and each random $\omega\in\Omega$ it holds
\begin{align*}
\phi^T_{\xi,B}(\cdot + z,\omega) = \phi^T_{\xi,B}(\cdot,\omega(\cdot + z))
\quad\text{almost everywhere in } \Rd.
\end{align*}
An analogous statement holds true for the linearized flux correctors 
\begin{align*}
\sigma^T_{\xi,B}\in H^1_{\mathrm{uloc}}(\Rd;\mathbb{R}^{d\times d}_{\mathrm{skew}})
\quad\text{and}\quad
\psi^T_{\xi,B}\in H^1_{\mathrm{uloc}}(\Rd;\mathbb{R}^{d}).
\end{align*}
These are more precisely the unique solutions of
\begin{align}
\label{eq:PDEhigherOrderLinearizedFluxCorrectorLocalized}
\frac{1}{T}\sigma_{\xi,B,kl}^T - \Delta\sigma_{\xi,B,kl}^T 
&= (e_l\otimes e_k - e_k\otimes e_l) : \nabla q_{\xi,B}^T,
\end{align}
respectively
\begin{align}
\label{eq:PDEhigherOrderLinearizedHelmholtzCorrectorLocalized}
\frac{1}{T}\psi^T_{\xi,B} - \Delta\psi^T_{\xi,B} 
= q^T_{\xi,B} - \langle q^T_{\xi,B}\rangle - \nabla\phi^T_{\xi,B},
\end{align}
with the linearized flux being defined by
\begin{equation}
\label{eq:HigherOrderLinearizedFlux}
\begin{aligned}
q_{\xi,B}^T &:= a_\xi^T(\mathds{1}_{L=1}B + \nabla \phi^T_{\xi,B})
\\&~~~~
+ \sum_{\substack{\Pi\in\mathrm{Par}\{1,\ldots,L\} \\ \Pi\neq\{\{1,\ldots,L\}\}}}
\partial_\xi^{|\Pi|} A(\omega,\xi{+}\nabla \phi^T_\xi)
\Big[\bigodot_{\pi\in\Pi}(\mathds{1}_{|\pi|=1}B'_\pi + \nabla \phi^T_{\xi,B'_\pi})\Big].
\end{aligned}
\end{equation}
As in the case of the corrector for the nonlinear PDE with an additional massive term,
the relations~\eqref{eq:PDEhigherOrderLinearizedCorrectorLocalized}--\eqref{eq:HigherOrderLinearizedFlux} 
will give rise to the equation
\begin{align}
\label{eq:PDEhigherOrderLinearizedHelmholtzDecompLocalized}
q^T_{\xi,B} - \langle q^T_{\xi,B} \rangle
= \nabla\cdot\sigma^T_{\xi,B} + \frac{1}{T}\psi^T_{\xi,B}.
\end{align}
\end{subequations}

With all of this notation in place, we can state the following
result on existence of (higher-order) linearized correctors.
For a proof, we refer the reader to Appendix~\ref{app:existenceCorrectors}
where we also formulate and proof a corresponding result
on the differentiability of (higher-order) linearized correctors
with respect to the parameter field.

\begin{lemma}[Existence of localized correctors]
\label{eq:lemmaExistenceLinearizedCorrectorsShort}
Let $L\in\N$ and $T\in [1,\infty)$ be fixed.
Let the requirements and notation of (A1), (A2)$_{L{-}1}$ and (A3)$_{L{-}1}$ of
Assumption~\ref{assumption:operators} be in place. Fix $\eta\in (0,1)$,
and consider a parameter field $\tilde\omega\colon\Rd\to B^n_1$
such that for all $p\geq 2$ it holds
\begin{align}
\label{eq:ergodicityParameterField}
\sup_{x_0\in\Rd} \limsup_{R\to\infty}\dashint_{B_R(x_0)}
\bigg|\sup_{y,z\in B_1(x),\,y\neq z} \frac{|\tilde\omega(y)-\tilde\omega(z)|}{|y-z|^\eta}\,\bigg|^p < \infty.
\end{align} 

Under these assumptions, one obtains inductively that for all $l\in\{1,\ldots,L\}$ and all $B:=v_1\odot\cdots\odot v_l$
formed by unit vectors $v_1,\ldots,v_l\in\Rd$, there exists a unique solution
\begin{align*}
\phi^T_{\xi,B}=\phi^T_{\xi,B}(\cdot,\tilde\omega)
\in H^1_{\mathrm{uloc}}(\Rd)
\end{align*}
of the linearized corrector problem~\eqref{eq:PDEhigherOrderLinearizedCorrectorLocalized}
with $\omega$ replaced by $\tilde\omega$. The linearized corrector $\phi^T_{\xi,B}(\cdot,\tilde\omega)$
moreover satisfies for all $p\in [2,\infty)$
\begin{align}
\label{eq:ergodicityLinearizedCorrector}
\sup_{x_0\in\Rd} \limsup_{R\to\infty}\dashint_{B_R(x_0)}
\Big|\Big(\frac{\phi^T_{\xi,B}(\cdot,\tilde\omega)}{\sqrt{T}},
\nabla\phi^T_{\xi,B}(\cdot,\tilde\omega)\Big)\Big|^p < \infty.
\end{align} 
There also exist unique solutions
\begin{align*}
\sigma^T_{\xi,B}=\sigma^T_{\xi,B}(\cdot,\tilde\omega)
&\in H^1_{\mathrm{uloc}}(\Rd;\mathbb{R}^{d\times d}_{\mathrm{skew}}),
\\
\psi^T_{\xi,B}=\psi^T_{\xi,B}(\cdot,\tilde\omega)
&\in H^1_{\mathrm{uloc}}(\Rd;\mathbb{R}^{d})
\end{align*}
of the linearized flux corrector problems 
\eqref{eq:PDEhigherOrderLinearizedFluxCorrectorLocalized}
resp.\ \eqref{eq:PDEhigherOrderLinearizedHelmholtzCorrectorLocalized}
with $\omega$ replaced by $\tilde\omega$. The analogue
of~\eqref{eq:ergodicityLinearizedCorrector} holds true for these flux correctors.

In particular, under the requirements of (A1), (A2)$_{L{-}1}$ and (A3)$_{L{-}1}$ of
Assumption~\ref{assumption:operators}, (P1) and (P2) of
Assumption~\ref{assumption:ensembleParameterFields}, and (R) of 
Assumption~\ref{assumption:smallScaleReg}, there exists
a set $\Omega'\subset\Omega$ of full $\Prob$-measure on which the existence of
(higher-order) linearized correctors is guaranteed
in the above sense for all random parameter fields $\omega\in\Omega'$.
\end{lemma}

We have by now everything in place to proceed with the statement of the

\begin{indhypothesis}
Let $L\in\N$, $M>0$ and $T\in [1,\infty)$ be fixed. 
Let the requirements and notation of (A1), (A2)$_L$ and (A3)$_L$ of
Assumption~\ref{assumption:operators}, (P1) and~(P2) of
Assumption~\ref{assumption:ensembleParameterFields}, and (R) of 
Assumption~\ref{assumption:smallScaleReg} be in place.
For any $l\leq L{-}1$ and any 
collection of unit vectors $v_1',\ldots,v_l'\in\Rd$
we assume that under the above conditions 
the associated localized $l$th-order linearized homogenization corrector $\phi^T_{\xi,B'}$ 
in direction $B':=v_1'\odot\cdots\odot v_l'$ satisfies the following list of conditions
(if $l=0$---and thus $B'$ being an empty symmetric tensor product---$\phi^T_{\xi,B'}$ is understood to denote the localized homogenization corrector of the nonlinear PDE):
\begin{itemize}[leftmargin=0.6cm]
\item There exists a constant $C=C(d,\lambda,\Lambda,\nu,\rho,\eta,M,L)$ such that for all $|\xi|\leq M$, 
all $q\in [1,\infty)$, and all compactly supported and square-integrable $f,g$ we have \emph{corrector estimates}
\begin{equation}
\label{eq:indHypoCorrectorBounds}
\tag{H1}
\begin{aligned}
\nonumber
\bigg\langle\bigg|\bigg(\int g\cdot\nabla\phi^T_{\xi,B'}, 
\int \frac{1}{T}f \phi^T_{\xi,B'} \bigg)\bigg|^{2q}\bigg\rangle^\frac{1}{q}
&\leq C^2q^{2C}\int \Big|\Big(g,\frac{f}{\sqrt{T}}\Big)\Big|^2,
\\ 
\Big\langle\Big\|\Big(\frac{\phi^T_{\xi,B'}}{\sqrt{T}},
\nabla\phi^T_{\xi,B'}\Big)\Big\|^{2q}_{L^2(B_1)}\Big\rangle^\frac{1}{q}
&\leq C^2q^{2C},
\end{aligned}
\end{equation}
with the scaling function $\mu_*\colon\mathbb{R}_{>0}\to \mathbb{R}_{>0}$ 
from~\eqref{eq:scalingCorrectorBounds}.
\item Fix $p>2$, and let $g$ be a compactly supported and $p$-integrable random field. Then there exists a random field
						$G_{\xi,B'}^T\in L^1_{\mathrm{uloc}}(\Rd;\Rd[n])$ being related to $g$ via $\phi^T_{\xi,B'}$ 
						in the sense that, $\Prob$-almost surely, it holds for all compactly supported and smooth perturbations 
						$\delta\omega\colon\Rd\to \Rd[n]$ with $\|\delta\omega\|_{L^\infty}\leq 1$
						\begin{align}
						\label{eq:representationMalliavinDerivative}
						\tag{H2a}
						\int g\cdot\nabla\delta\phi^T_{\xi,B'} = \int G^T_{\xi,B'}\cdot\delta\omega.
						\end{align}
						Here, $(\frac{\delta\phi^T_{\xi,B'}}{\sqrt{T}},\nabla\delta\phi^T_{\xi,B'})
						\in L^2_{\mathrm{uloc}}(\Rd;\Rd[]{\times}\Rd)$ denotes the G\^ateaux derivative
						of the linearized corrector $\phi^T_{\xi,B'}$ and its gradient in direction $\delta\omega$,
						cf.\ Lemma~\ref{eq:lemmaExistenceLinearizedCorrectorsExtended}.
						
						For any $\kappa\in (0,1]$ there moreover exists $C=C(d,\lambda,\Lambda,\nu,\rho,\eta,M,L,\kappa)$
						such that for all $|\xi|\leq M$ and all $q\in [1,\infty)$ the random field
						$G^T_{\xi,B'}$ gives rise to a \emph{sensitivity estimate} of the form
						\begin{align}
						\label{eq:indHypoSensitivityBound}
						\tag{H2b}
						\bigg\langle\bigg|\,\int\bigg(
						\,\dashint_{B_1(x)}|G^T_{\xi,B'}|\,\bigg)^2\bigg|^{q}\bigg\rangle^\frac{1}{q}
						&\leq C^2q^{2C}\sup_{\langle F^{2q_*}\rangle=1}
						\int\big\langle|Fg|^{2(q/\kappa)_*}\big\rangle^\frac{1}{(q/\kappa)_*}.
						\end{align}
						
						If $(g_r)_{r\geq 1}$ is a sequence of compactly supported and $p$-integrable random fields, denote by
						$G_{\xi,B'}^{T,r}\in L^1_{\mathrm{uloc}}(\Rd;\Rd[n])$, $r\geq 1$, the random field
						associated to $g_r$, $r\geq 1$, in the sense of~\eqref{eq:representationMalliavinDerivative}.
						Let $g$ be an $L^p_{\mathrm{uloc}}(\Rd;\Rd)$-valued random field, and assume that
						$\Prob$-almost surely it holds $g_r\to g$ in $L^p_{\mathrm{uloc}}(\Rd;\Rd)$.
						Then there exists a random field $G^T_{\xi,B'}$ such that $\Prob$-almost surely
						\begin{align}
						\label{eq:indHypoSensitivityApprox}
						\tag{H2c}
						G_{\xi,B'}^{T,r} \to G_{\xi,B'}^{T} \text{ as } r\to\infty 
						\text{ in } L^1_{\mathrm{uloc}}(\Rd;\Rd[n]).
						\end{align}
						In the special case of $g_r = \mathds{1}_{B_r}g$, $r\geq 1$, the limit random
						field is in addition subject to	the sensitivity estimate~\eqref{eq:indHypoSensitivityBound}.
\item There exist $\alpha=\alpha(d,\lambda,\Lambda)\in (0,\eta)$ and $C=C(d,\lambda,\Lambda,\nu,\rho,\eta,M,L)$
such that for all $|\xi|\leq M$ and all $q\in [1,\infty)$ we have a 
\emph{small-scale annealed Schauder estimate} of the form
\begin{align}
\label{eq:indHypoAnnealedSchauder}
\tag{H3}
\big\langle\big\|\nabla\phi^T_{\xi,B'}\big\|^{2q}_{C^\alpha(B_1)}\big\rangle^\frac{1}{q}
\leq C^2q^{2C}.
\end{align}
\end{itemize}
\end{indhypothesis}

\subsection{Corrector bounds for higher-order linearizations: the base case}
The first step in the proof of Theorem~\ref{theo:correctorBounds}---on the level
of the massive approximation in form of Theorem~\ref{theo:correctorBoundsMassiveApprox}---is of 
course to verify the induction hypotheses
\eqref{eq:indHypoCorrectorBounds}--\eqref{eq:indHypoAnnealedSchauder} for the
corrector of the nonlinear problem~\eqref{eq:PDEMonotoneHomCorrectorLocalized}.
This is covered by Proposition~\ref{prop:estimates1stOrderCorrector}
which constitutes one of the main results of~\cite{Fischer2019}. We briefly summarize
at the beginning of Appendix~\ref{app:baseCaseInd} how to obtain the assertions
of Proposition~\ref{prop:estimates1stOrderCorrector}.

\subsection{Corrector bounds for higher-order linearizations: the induction step}
\label{subsec:indStep}
The main step in the proof of Theorem~\ref{theo:correctorBoundsMassiveApprox}
consists of lifting the induction hypotheses
\eqref{eq:indHypoCorrectorBounds}--\eqref{eq:indHypoAnnealedSchauder}
to the $L$th-order linearized homogenization corrector $\phi^T_{\xi,B}$
satisfying~\eqref{eq:PDEhigherOrderLinearizedCorrectorLocalized}.
This task is performed by means of several auxiliary results
where we are guided by the well-established literature on
quantitative stochastic homogenization, cf.\ for instance 
\cite{Gloria2020}, \cite{Gloria2019}, \cite{Fischer2019} and~\cite{Josien2020}.
We start with the concept of a minimal radius for the (higher-order)
linearized corrector equation~\eqref{eq:PDEhigherOrderLinearizedCorrectorLocalized}.

\begin{definition}[Minimal radius for linearized corrector problem]
\label{def:minRadiusHigherOrderLinearizations}
Let the assumptions and notation of Section~\ref{subsec:indHypo} be in place;
in particular, the induction hypotheses \eqref{eq:indHypoCorrectorBounds}--\eqref{eq:indHypoAnnealedSchauder}.
For a given constant $\underline{\gamma}>0$ we then define a random variable
\begin{align*}
r_{*,T,\xi,B}
&:= \inf\bigg\{2^k\colon k\in\N_0,\,
\text{and for all } R=2^l,\,l\geq k,\,\text{it holds:}
\\&~~~~~~~~~~~~~~~~~~~~~~~~~~~~~~~~~~~~~~~~~
\inf_{b\in\mathbb{R}}\bigg\{\frac{1}{R^2}\dashint_{B_R}
\big|\phi^T_{\xi,B}{-}b\big|^2
{+}\frac{1}{T}|b|^2\bigg\}
\leq 1,
\\&~~~~~~~~~~~~~~~~~~~~
\sup_{\substack{\Pi\in\mathrm{Par}\{1,\ldots,L\} \\ \Pi\neq\{\{1,\ldots,L\}\}}}
\sup_{\pi\in\Pi}\, \dashint_{B_R} \big|\mathds{1}_{|\pi|=1}B'_\pi + \nabla \phi^T_{\xi,B'_\pi}\big|^{4|\Pi|}
\leq R^{4\underline{\gamma}}\bigg\}.
\end{align*}
The stationary extension $r_{*,T,\xi,B}(x,\omega):=
r_{*,T,\xi,B}\big(\omega(\cdot{+}x)\big)$ is called the \emph{minimal radius
for the linearized corrector problem~\eqref{eq:PDEhigherOrderLinearizedCorrectorLocalized}}.
\end{definition}

Stochastic moments of the linearized homogenization correctors 
are related to stochastic moments of the minimal radius in the following way.

\begin{lemma}[Annealed small-scale energy estimate]
\label{lem:annealedSmallScaleEnergyEstimate}
Let the assumptions and notation of Section~\ref{subsec:indHypo} be in place;
in particular, the induction hypotheses \emph{\eqref{eq:indHypoCorrectorBounds}--\eqref{eq:indHypoAnnealedSchauder}}.
Let $r_{*,T,\xi,B}$ denote the minimal radius for the linearized corrector 
problem~\eqref{eq:PDEhigherOrderLinearizedCorrectorLocalized}
from Definition~\eqref{def:minRadiusHigherOrderLinearizations}.
Then, there exists a constant $C=C(d,\lambda,\Lambda,L)$
and an exponent $\delta=\delta(d,\lambda,\Lambda)$ such that for all $\xi\in\Rd$ and all $q\in [1,\infty)$ it holds
\begin{align}
\label{eq:annealedSmallScaleEnergyEstimateMinRadius}
\Big\langle\Big\|\Big(\frac{\phi^T_{\xi,B}}{\sqrt{T}},
\nabla\phi^T_{\xi,B}\Big)\Big\|^{2q}_{L^2(B_1)}\Big\rangle^\frac{1}{q}
\leq C^2\big\langle r_{*,T,\xi,B}^{(d-\delta+2\underline{\gamma})q}\big\rangle^\frac{1}{q}.
\end{align}
Moreover, we have the suboptimal estimate
\begin{align}
\label{eq:annealedSmallScaleEnergyEstimateSuboptimal}
\Big\langle\Big\|\Big(\frac{\phi^T_{\xi,B}}{\sqrt{T}},
\nabla\phi^T_{\xi,B}\Big)\Big\|^{2q}_{L^2(B_1)}\Big\rangle^\frac{1}{q}
\lesssim_q \sqrt{T}^d.
\end{align}
\end{lemma}

As it is already the case for the proof of corrector estimates with respect to 
first-order linearizations, cf.\ \cite{Fischer2019}, the argument
for the realization of the induction step relies on a small-scale 
regularity estimate for the linearized correctors too.

\begin{lemma}[Annealed small-scale Schauder estimate]
\label{lem:annealedSmallScaleSchauderIndStep}
Let again the assumptions and notation of Section~\ref{subsec:indHypo} be in place;
in particular, the induction hypotheses \emph{\eqref{eq:indHypoCorrectorBounds}--\eqref{eq:indHypoAnnealedSchauder}}.
There exists $\alpha=\alpha(d,\lambda,\Lambda)\in (0,\eta)$,
and for every $\tau\in (0,1)$ a constant $C=C(d,\lambda,\Lambda,\nu,\rho,\eta,M,L,\tau)$,
such that for all $|\xi|\leq M$ and all $q\in [1,\infty)$ it holds
\begin{align}
\label{eq:annealedSmallScaleSchauderIndStep}
\Big\langle\Big\|\Big(\frac{\phi^T_{\xi,B}}{\sqrt{T}},
\nabla\phi^T_{\xi,B}\Big)\Big\|^{2q}_{C^\alpha(B_1)}\Big\rangle^\frac{1}{q}
\leq C^2q^{2C}\bigg\{1{+}\Big\langle\Big\|\Big(\frac{\phi^T_{\xi,B}}{\sqrt{T}},
\nabla\phi^T_{\xi,B}\Big)\Big\|^\frac{2q}{1-\tau}_{L^2(B_1)}\Big\rangle^\frac{1-\tau}{q}\bigg\}.
\end{align}
\end{lemma}

Lemma~\ref{lem:annealedSmallScaleEnergyEstimate} shifts the task
of establishing stretched exponential moment bounds for the linearized corrector
to the task of proving stretched exponential moments for the associated minimal radius.
For the latter, a key input are stochastic moment bounds for linear functionals
of the linearized corrector gradient. This in turn is the content of the following result.

\begin{lemma}[Annealed estimates for linear functionals of the homogenization corrector and its gradient]
\label{lem:annealedLinFunctionalsIndStep}
Let the assumptions and notation of Section~\ref{subsec:indHypo} be in place;
in particular, the induction hypotheses \emph{\eqref{eq:indHypoCorrectorBounds}--\eqref{eq:indHypoAnnealedSchauder}}.
Let $g,f$ be two square-integrable and compactly supported deterministic fields. For every $\tau\in (0,1)$
there exists a constant $C=C(d,\lambda,\Lambda,\nu,\rho,\eta,M,L,\tau)$
such that for all $|\xi|\leq M$ and all $q\in [C,\infty)$ it holds
\begin{equation}
\label{eq:annealedLinFunctionalsIndStep}
\begin{aligned}
&\bigg\langle\bigg|\bigg(\int g\cdot\nabla\phi^T_{\xi,B}, 
\int \frac{1}{T}f \phi^T_{\xi,B} \bigg)\bigg|^{2q}\bigg\rangle^\frac{1}{q}
\\&
\leq C^2q^{2C} \bigg\{1+\Big\langle\Big\|\Big(\frac{\phi^T_{\xi,B}}{\sqrt{T}},
\nabla\phi^T_{\xi,B}\Big)\Big\|^\frac{2q}{(1-\tau)^2}_{L^2(B_1)}\Big\rangle^\frac{(1-\tau)^2}{q}\bigg\}
\int \Big|\Big(g,\frac{f}{\sqrt{T}}\Big)\Big|^2.
\end{aligned}
\end{equation}
\end{lemma}

We have everything in place to prove a stretched exponential moment bound
for the minimal radius $r_{*,T,\xi,B}$. The key ingredient of the proof
is a buckling argument based on the annealed 
estimates~\eqref{eq:annealedSmallScaleEnergyEstimateMinRadius}--\eqref{eq:annealedLinFunctionalsIndStep}.

\begin{lemma}[Stretched exponential moment bound for minimal radius]
\label{lem:momentBoundMinRadius}
Let the assumptions and notation of Section~\ref{subsec:indHypo} be in place;
in particular, the induction hypotheses \emph{\eqref{eq:indHypoCorrectorBounds}--\eqref{eq:indHypoAnnealedSchauder}}.
Let $r_{*,T,\xi,B}$ denote the minimal radius for the linearized corrector 
problem~\eqref{eq:PDEhigherOrderLinearizedCorrectorLocalized}
from Definition~\eqref{def:minRadiusHigherOrderLinearizations}.
Then, there exists $\underline{\gamma}=\underline{\gamma}(d,\lambda,\Lambda)$, 
a constant $C=C(d,\lambda,\Lambda,\nu,\rho,\eta,M,L)$
and an exponent $\bar\nu=\bar\nu(d,\lambda,\Lambda,\nu,\rho,\eta,M,L)$ 
such that for all $|\xi|\leq M$ it holds
\begin{align}
\label{eq:momentBoundMinRadius}
\bigg\langle\exp\bigg(\bigg(\frac{r_{*,T,\xi,B}}{C}\bigg)^{\bar\nu}\bigg)\bigg\rangle \leq 2.
\end{align}
\end{lemma}

A rather straightforward post-processing of Lemma~\ref{lem:annealedSmallScaleEnergyEstimate},
Lemma~\ref{lem:annealedSmallScaleSchauderIndStep} and Lemma~\ref{lem:annealedLinFunctionalsIndStep}
based on the stretched exponential moment bounds for the minimal radius from 
Lemma~\ref{lem:momentBoundMinRadius} now allows to conclude the induction step.

\begin{lemma}
\label{lem:closingIndStep}
Let the assumptions and notation of Section~\ref{subsec:indHypo} be in place;
in particular, the induction hypotheses \emph{\eqref{eq:indHypoCorrectorBounds}--\eqref{eq:indHypoAnnealedSchauder}}.
Then the $L$th-order linearized homogenization corrector $\phi^T_{\xi,B}$ 
also satisfies \emph{\eqref{eq:indHypoCorrectorBounds}--\eqref{eq:indHypoAnnealedSchauder}}.
\end{lemma}

\subsection{Estimates for higher-order linearized flux correctors}
In view of the defining equations~\eqref{eq:PDEhigherOrderLinearizedFluxCorrectorLocalized}
and~\eqref{eq:PDEhigherOrderLinearizedHelmholtzCorrectorLocalized}
for the linearized flux correctors, it is natural to establish first the analogues of the
estimates~\eqref{eq:linearFunctionalCorrectorGradientBoundMassiveApprox}--\eqref{eq:correctorGrowthoundMassiveApprox}
for the linearized flux $q^T_{\xi,B}$. Actually, it suffices to establish
the pendant of induction hypothesis~\eqref{eq:indHypoSensitivityBound}. This is captured in the following result.

\begin{lemma}[Sensitivity estimate for the linearized flux]
\label{lem:correctorEstimatesLinearizedFlux}
Let the assumptions and notation of Section~\ref{subsec:indHypo} be in place.
In particular, let $q^T_{\xi,B}$ be the linearized flux as defined by~\eqref{eq:HigherOrderLinearizedFlux}.
Fix $p\in (2,\infty)$, and consider a compactly supported and $p$-integrable field $g$. 
Then there exists a random field
$Q^T_{\xi,B}\in L^1_{\mathrm{uloc}}(\Rd;\Rd[n])$ being related to $g$ via $q^T_{\xi,B}$
in the sense that for all compactly supported and smooth perturbations 
$\delta\omega\colon\Rd\to \Rd[n]$ with $\|\delta\omega\|_{L^\infty}\leq 1$ it holds
\begin{align}
\label{eq:representationMalliavinLinearizedFlux}
\int g\cdot\delta q^T_{\xi,B} = \int Q^T_{\xi,B}\cdot\delta\omega.
\end{align}
In addition, there exists $C=C(d,\lambda,\Lambda,\nu,\rho,\eta,M,L)$
such that for all $|\xi|\leq M$ and all $q\in [1,\infty)$ the random field
$Q^T_{\xi,B}$ gives rise to a \emph{sensitivity estimate} of the form
\begin{align}
\label{eq:sensitivityBoundLinearizedFlux}
\bigg\langle\bigg|\,\int\bigg(
\,\dashint_{B_1(x)}|Q^T_{\xi,B}|\,\bigg)^2\bigg|^{q}\bigg\rangle^\frac{1}{q}
&\leq C^2q^{2C}\sup_{\langle F^{2q_*}\rangle=1}
\int\big\langle|Fg|^{2q_*}\big\rangle^\frac{1}{q_*}.
\end{align}

If $(g_r)_{r\geq 1}$ is a sequence of compactly supported and $p$-integrable random fields, denote by
						$Q_{\xi,B}^{T,r}\in L^1_{\mathrm{uloc}}(\Rd;\Rd[n])$, $r\geq 1$, the random field
						associated to $g_r$, $r\geq 1$, in the sense of~\eqref{eq:representationMalliavinLinearizedFlux}.
						Let $g$ be an $L^p_{\mathrm{uloc}}(\Rd;\Rd)$-valued random field, and assume that
						$\Prob$-almost surely it holds $g_r\to g$ in $L^p_{\mathrm{uloc}}(\Rd;\Rd)$.
						Then there exists a random field $Q^T_{\xi,B}$ such that $\Prob$-almost surely
						\begin{align}
						\label{eq:indHypoSensitivityApproxFlux}
						Q_{\xi,B}^{T,r} \to Q_{\xi,B}^{T} \text{ as } r\to\infty 
						\text{ in } L^1_{\mathrm{uloc}}(\Rd;\Rd[n]),
						\end{align}
						and the sensitivity estimate~\eqref{eq:sensitivityBoundLinearizedFlux} holds true.
\end{lemma}

Once this result is established, the asserted estimates in Theorem~\ref{theo:correctorBoundsMassiveApprox} 
for the massive linearized flux correctors~$(\sigma^T_{\xi,B},\psi^T_{\xi,B})$ follow readily.

\subsection{Differentiability of the massive correctors and the massive approximation
of the homogenized operator}
As a preparation for the proof of the 
estimates~\eqref{eq:correctorGradientBoundDifferences}--\eqref{eq:correctorGrowthBoundDifferences},
which in particular contain estimates for differences of linearized correctors,
and the higher-order differentiability of the homogenized operator
in form of Theorem~\ref{theo:diffHomOperator}, we establish the desired
differentiability properties on the level of the massive approximation.
A first step in this direction are the following estimates for differences
of linearized correctors.

\begin{lemma}[Estimates for differences of linearized correctors]
\label{lem:differencesLinearizedCorrectors}
Let $L\in\N$, $M>0$ as well as $T\in [1,\infty)$ be fixed.
Let the requirements and notation of (A1), (A2)$_L$, (A3)$_L$ and (A4)$_L$ of
Assumption~\ref{assumption:operators}, (P1) and (P2) of
Assumption~\ref{assumption:ensembleParameterFields}, and (R) of 
Assumption~\ref{assumption:smallScaleReg} be in place.
We also fix a set of unit vectors $v_1,\ldots,v_L\in\Rd$ and define
$B:=v_1\odot\cdots\odot v_{L}$. 

For every $\beta\in (0,1)$, there exists a constant 
$C=C(d,\lambda,\Lambda,\nu,\rho,\eta,M,L,\beta)$ such that for all $|\xi|\leq M$, 
all $q\in [1,\infty)$, all unit vectors $e\in\Rd$ and all $|h|\leq 1$ it holds
\begin{align}
\label{eq:correctorBoundForDifferencesMassiveApprox}
\big\langle\big\|\big(
\nabla\phi^T_{\xi+he,B}{-}\nabla\phi^T_{\xi,B},
\nabla\sigma^T_{\xi+he,B}{-}\nabla\sigma^T_{\xi,B}
\big)\big\|^{2q}_{L^2(B_1)}\big\rangle^\frac{1}{q}
&\leq C^2q^{2C}|h|^{2(1{-}\beta)},
\\\label{eq:SchauderEstimatesDifferencesMassiveApprox}
\big\langle\big\|\big(
\nabla\phi^T_{\xi+he,B}{-}\nabla\phi^T_{\xi,B},
\nabla\sigma^T_{\xi+he,B}{-}\nabla\sigma^T_{\xi,B}
\big)\big\|^{2q}_{C^\alpha(B_1)}\big\rangle^\frac{1}{q}
&\leq C^2q^{2C}|h|^{2(1{-}\beta)},
\end{align}
as well as for all compactly supported and square-integrable $g_\phi,g_\sigma$
\begin{equation}
\label{eq:LinearFunctionalEstimatesDifferencesMassiveApprox}
\begin{aligned}
&\bigg\langle\bigg|\bigg(\int g_\phi\cdot(\nabla\phi^T_{\xi+he,B}{-}\nabla\phi^T_{\xi,B}), 
\int g_\sigma^{kl}\cdot(\nabla\sigma^T_{\xi+he,B,kl}{-}\nabla\sigma^T_{\xi,B,kl})
\bigg)\bigg|^{2q}\bigg\rangle^\frac{1}{q}
\\
&\leq C^2q^{2C}|h|^{2(1-\beta)}\int \big|\big(g_\phi,g_\sigma\big)\big|^2.
\end{aligned}
\end{equation}
\end{lemma}

The already mentioned differentiability result for the massive linearized correctors
and the massive approximation of the homogenized operator now reads as follows.

\begin{lemma}[Differentiability of massive correctors and the massive version of the homogenized operator]
\label{lem:diffMassiveApprox}
Let $L\in\N$, $M>0$ and $T\in [1,\infty)$ be fixed.
Let the requirements and notation of (A1), (A2)$_L$, (A3)$_L$ and (A4)$_L$ of
Assumption~\ref{assumption:operators}, (P1) and (P2) of
Assumption~\ref{assumption:ensembleParameterFields}, and (R) of 
Assumption~\ref{assumption:smallScaleReg} be in place.
We also fix a set of unit vectors $v_1,\ldots,v_L\in\Rd$ and define
$B:=v_1\odot\cdots\odot v_{L}$. 
Then, both the maps $\xi\mapsto\nabla\phi^T_{\xi,B}$ and $\xi\mapsto\nabla\sigma^T_{\xi,B}$
are Fr\'echet differentiable with values in the Fr\'echet space $L^2_{\langle\cdot\rangle}L^2_{\mathrm{loc}}(\Rd)$.

Given a vector $\xi\in\Rd$, a unit vector $e\in\Rd$ and some $|h|\leq 1$ we define
\begin{align*}
\bar{A}^T_{B,e,h}(\xi) &:= \big\langle q^T_{\xi+he,B} \big\rangle
													 - \big\langle q^T_{\xi,B} \big\rangle
													 - \big\langle q^T_{\xi,B\odot e} \big\rangle h.
\end{align*}
For every $\beta\in (0,1)$, there then exists a constant 
$C=C(d,\lambda,\Lambda,\nu,\rho,\eta,M,L,\beta)$ such that for all $|\xi|\leq M$, 
all unit vectors $e\in\Rd$, all $q\in [1,\infty)$, and all $|h|\leq 1$ it holds
\begin{align}
\label{eq:regMassiveVersionHomOperator}
\big|\bar{A}^T_{B,e,h}(\xi)\big|^{2}
&\leq C^2h^{4(1-\beta)}.
\end{align}

Assume in addition to the above conditions that the stronger forms of~(A2)$_{L{+}1}$,
(A3)$_{L{+}1}$ and (A4)$_{L{+}1}$ from Assumption~\ref{assumption:operators} hold true.
We then have the following quantitative estimates on first-order Taylor expansions
of the linearized correctors $\phi^T_{\xi,B}$ and $\sigma^T_{\xi,B}$.
Given a vector $\xi\in\Rd$, a unit vector $e\in\Rd$ and some $|h|\leq 1$ we define
\begin{align*}
\phi^T_{\xi,B,e,h} &:= \phi^T_{\xi+he,B} - \phi^T_{\xi,B} - \phi^T_{\xi,B\odot e}h,
\\
\sigma^T_{\xi,B,e,h} &:= \sigma^T_{\xi+he,B} - \sigma^T_{\xi,B} - \sigma^T_{\xi,B\odot e}h.
\end{align*}
For every $\beta\in (0,1)$, there then exists a constant 
$C=C(d,\lambda,\Lambda,\nu,\rho,\eta,M,L,\beta)$ such that for all $|\xi|\leq M$, 
all unit vectors $e\in\Rd$, all $q\in [1,\infty)$, all $|h|\leq 1$
and all compactly supported and square-integrable $g_\phi,g_\sigma$ it holds
\begin{align}
\label{eq:regGradient}
\big\langle\big\|\big(\nabla\phi^T_{\xi,B,e,h},
\nabla\sigma^T_{\xi,B,e,h}\big)\big\|^{2q}_{L^2(B_1)}
\big\rangle^\frac{1}{q} &\leq C^2q^{2C}h^{4(1-\beta)},
\\
\label{eq:regLinearFunctionals}
\bigg\langle\bigg|\bigg(\int g_\phi\cdot\nabla\phi^T_{\xi,B,e,h}, 
\int g_\sigma^{kl}\cdot\nabla\sigma^T_{\xi,B,e,h}
\bigg)\bigg|^{2q}\bigg\rangle^\frac{1}{q}
&\leq C^2q^{2C}|h|^{4(1-\beta)}\int \big|\big(g_\phi,g_\sigma\big)\big|^2.
\end{align}
\end{lemma}

\begin{remark}
In case of $q=1$, the estimates~\eqref{eq:regGradient} and~\eqref{eq:regLinearFunctionals}
actually hold true requiring only~(A2)$_{L}$, (A3)$_{L}$ and (A4)$_L$ from Assumption~\ref{assumption:operators}.
This in turn represents exactly the form of~Assumption~\ref{assumption:operators} for which qualitative
differentiability of the linearized corrector gradients is established in Lemma~\ref{lem:diffMassiveApprox}.
A proof of this claim is contained in the proof of Lemma~\ref{lem:diffMassiveApprox}.
\end{remark}

\subsection{The limit passage in the massive approximation}
\label{subsec:limitPassage}
The last main ingredient in the proof of Theorem~\ref{theo:correctorBounds}
and Theorem~\ref{theo:diffHomOperator} consists of studying the limit $T\to\infty$
in the massive approximation. More precisely, we establish the following result.

\begin{lemma}[Limit passage in the massive approximation]
\label{lem:limitMassiveApprox}
Let $L\in\N$ and $M>0$ be fixed.
Let the requirements and notation of (A1), (A2)$_L$ and (A3)$_L$ of
Assumption~\ref{assumption:operators}, (P1) and (P2) of
Assumption~\ref{assumption:ensembleParameterFields}, and (R) of 
Assumption~\ref{assumption:smallScaleReg} be in place.
We also fix a set of unit vectors $v_1,\ldots,v_L\in\Rd$ and define
$B:=v_1\odot\cdots\odot v_{L}$. 

Then the sequence
\begin{align*}
\Big(\nabla\phi^T_{\xi,B},\,\nabla\sigma^T_{\xi,B}\Big)_{T\in [1,\infty)}
\end{align*}
is Cauchy in $L^2_{\langle\cdot\rangle}L^2_{\mathrm{loc}}(\Rd)$ (with respect
to the strong topology). Moreover, 
there exists $C=C(d,\lambda,\Lambda,\nu,\rho,\eta,M,L)$ 
such that for all $|\xi|\leq M$ and all $T\in [1,\infty)$ we have the estimates
\begin{align}
\label{eq:convLinearizedCorrectors}
\big\langle\big\|\big(\nabla\phi^{2T}_{\xi,B} - \nabla\phi^T_{\xi,B},
\nabla\sigma^{2T}_{\xi,B} - \nabla\sigma^T_{\xi,B}\big)\big\|^{2}_{L^2(B_1)}\big\rangle
&\leq C^2\frac{\mu_*^2(\sqrt{T})}{T},
\\\label{eq:convHomOperator}
\big|\big\langle q^{2T}_{\xi,B}\big\rangle - \big\langle q^T_{\xi,B}\big\rangle\big|^2
&\leq C^2\frac{\mu_*^2(\sqrt{T})}{T}.
\end{align}
The corresponding limits give rise to higher-order linearized homogenization correctors and flux correctors
in the sense of Definition~\ref{def:correctorsHigherOrderLinearization}. Moreover,
\begin{align}
\label{eq:convFluxes}
\big\langle\big|q^{T}_{\xi,B}-q_{\xi,B}\big|^2\big\rangle \to 0,
\end{align}
with the limiting linearized flux $q_{\xi,B}$ defined in~\eqref{eq:LinearizedFlux}.
\end{lemma}

\section{Proofs}\label{sec:proofs}

\subsection{Proof of Lemma~\ref{lem:annealedSmallScaleEnergyEstimate} 
{\normalfont (Annealed small-scale energy estimate)}}
Applying the hole filling estimate~\eqref{eq:holeFillingEstimate} 
to equation~\eqref{eq:PDEhigherOrderLinearizedCorrectorLocalized}
for the linearized homogenization corrector (putting the term 
$-\nabla\cdot a_\xi^T\mathds{1}_{L=1}B$ on the right hand side) 
yields in combination with~(A2)$_{L}$ from Assumption~\ref{assumption:operators}
\begin{align*}
&\Big\|\Big(\frac{\phi^T_{\xi,B}}{\sqrt{T}},
\nabla\phi^T_{\xi,B}\Big)\Big\|^2_{L^2(B_1)}
\\
&\lesssim_{d,\lambda,\Lambda,L} r^{d-\delta}_{*,T,\xi,B}\,
\dashint_{B_{r_{*,T,\xi,B}}}\Big|\Big(\frac{\phi^T_{\xi,B}}{\sqrt{T}},
\nabla\phi^T_{\xi,B}\Big)\Big|^2 + r^{d-\delta}_{*,T,\xi,B}\mathds{1}_{L=1}
\\&~~~
+ r^{d}_{*,T,\xi,B} \sum_{\substack{\Pi\in\mathrm{Par}\{1,\ldots,L\} 
\\ \Pi\neq\{\{1,\ldots,L\}\}}}
\,\dashint_{B_{r_{*,T,\xi,B}}}\frac{1}{|x|^\delta}\prod_{\pi\in\Pi}
\big|\mathds{1}_{|\pi|=1}B'_\pi + \nabla \phi^T_{\xi,B'_\pi}\big|^2.
\end{align*}
For the first right hand side term, we proceed by making use of
Caccioppoli's inequality~\eqref{eq:Caccioppoli} with respect to 
equation~\eqref{eq:PDEhigherOrderLinearizedCorrectorLocalized};
and in the course of this we again rely on~(A2)$_{L}$ from Assumption~\ref{assumption:operators}
in order to bound the right hand side term appearing in
equation~\eqref{eq:PDEhigherOrderLinearizedCorrectorLocalized}.
For the second right hand side term, we simply argue by H\"older's inequality.
In total, we obtain the estimate
\begin{align*}
&\Big\|\Big(\frac{\phi^T_{\xi,B}}{\sqrt{T}},
\nabla\phi^T_{\xi,B}\Big)\Big\|^2_{L^2(B_1)}
\\
&\lesssim_{d,\lambda,\Lambda,L} r^{d-\delta}_{*,T,\xi,B}
\inf_{b\in\mathbb{R}}\,\bigg\{\frac{1}{(2r_{*,T,\xi,B})^2}\,\dashint_{B_{2r_{*,T,\xi,B}}}
\big|\phi^T_{\xi,B}{-}b\big|^2 + \frac{1}{T}|b|^2\bigg\}
\\&~~~
+ r^{d-\delta}_{*,T,\xi,B} \sum_{\substack{\Pi\in\mathrm{Par}\{1,\ldots,L\} 
\\ \Pi\neq\{\{1,\ldots,L\}\}}}
\,\dashint_{B_{2r_{*,T,\xi,B}}}\prod_{\pi\in\Pi}
\big|\mathds{1}_{|\pi|=1}B'_\pi + \nabla \phi^T_{\xi,B'_\pi}\big|^2
\\&~~~
+ r^{d-\delta}_{*,T,\xi,B}\mathds{1}_{L=1}
\\&~~~
+r^{d-\delta}_{*,T,\xi,B} \sum_{\substack{\Pi\in\mathrm{Par}\{1,\ldots,L\} 
\\ \Pi\neq\{\{1,\ldots,L\}\}}}
\prod_{\pi\in\Pi}\bigg(\,\dashint_{B_{r_{*,T,\xi,B}}}
\big|\mathds{1}_{|\pi|=1}B'_\pi + \nabla \phi^T_{\xi,B'_\pi}\big|^{4|\Pi|}\bigg)^\frac{1}{2|\Pi|}.
\end{align*}
Taking into account Definition~\ref{def:minRadiusHigherOrderLinearizations} of the minimal radius,
the fact that $r_{*,T,\xi,B}\geq 1$, as well as H\"older's and Jensen's inequality
(to deal with the second right hand side term in the previous display) we deduce from this
\begin{align*}
&\Big\|\Big(\frac{\phi^T_{\xi,B}}{\sqrt{T}},
\nabla\phi^T_{\xi,B}\Big)\Big\|^2_{L^2(B_1)}
\lesssim_{d,\lambda,\Lambda,L} r^{d-\delta+2\gamma}_{*,T,\xi,B}.
\end{align*}
Taking stochastic moments thus entails the asserted estimate~\eqref{eq:annealedSmallScaleEnergyEstimateMinRadius}.

For a proof of~\eqref{eq:annealedSmallScaleEnergyEstimateSuboptimal}, we rely on 
the weighted energy estimate~\eqref{eq:expLocalization}. In order to apply it,
we first have to check the polynomial growth at infinity (in the precise
sense of the statement of Lemma~\ref{lem:expLocalization}) of the constituents in
the linearized corrector problem~\eqref{eq:PDEhigherOrderLinearizedCorrectorLocalized}. 
For the solution $(\phi^T_{\xi,B},\nabla\phi^T_{\xi,B})$ itself, this is a consequence
of the fact that $\phi^T_{\xi,B}\in H^1_{\mathrm{uloc}}(\Rd)$. For the right hand side
term in equation~\eqref{eq:PDEhigherOrderLinearizedCorrectorLocalized} we argue as follows.
First, thanks to the ergodic theorem we may choose almost surely a radius $R_0>0$ such that
\begin{equation}
\label{eq:ergodicTheorem}
\begin{aligned}
&\dashint_{B_R}\prod_{\pi\in\Pi}
\big|\mathds{1}_{|\pi|=1}B'_\pi + \nabla \phi^T_{\xi,B'_\pi}\big|^2
\\&
\leq 1 + \Big\langle\prod_{\pi\in\Pi}
\big|\mathds{1}_{|\pi|=1}B'_\pi + \nabla \phi^T_{\xi,B'_\pi}\big|^2\Big\rangle
\quad\text{for all } R\geq R_0,
\end{aligned} 
\end{equation}
uniformly over all partitions $\Pi\in\mathrm{Par}\{1,\ldots,L\}$
with $\Pi\neq\{\{1,\ldots,L\}\}$. Thanks to the induction
hypothesis~\eqref{eq:indHypoCorrectorBounds} and~\eqref{eq:indHypoAnnealedSchauder},
we may then smuggle in spatial averages over the unit ball $B_1$
followed by an application of H\"older's inequality to deduce that
\begin{align*}
&\Big\langle\prod_{\pi\in\Pi}
\big|\mathds{1}_{|\pi|=1}B'_\pi + \nabla \phi^T_{\xi,B'_\pi}\big|^2\Big\rangle
\\
&\lesssim \Big\langle\prod_{\pi\in\Pi}
\big\|\mathds{1}_{|\pi|=1}B'_\pi + \nabla \phi^T_{\xi,B'_\pi}\big\|_{C^\alpha(B_1)}^2\Big\rangle
+ \Big\langle\prod_{\pi\in\Pi}\int_{B_1}
\big|\mathds{1}_{|\pi|=1}B'_\pi + \nabla \phi^T_{\xi,B'_\pi}\big|^2\Big\rangle
\\
&\lesssim \prod_{\pi\in\Pi}\Big\langle
\big\|\mathds{1}_{|\pi|=1}B'_\pi + 
\nabla \phi^T_{\xi,B'_\pi}\big\|_{C^\alpha(B_1)}^{2|\Pi|}\Big\rangle^\frac{1}{|\Pi|}
+ \prod_{\pi\in\Pi}\Big\langle
\big\|\mathds{1}_{|\pi|=1}B'_\pi + 
\nabla \phi^T_{\xi,B'_\pi}\big\|_{L^2(B_1)}^{2|\Pi|}\Big\rangle^\frac{1}{|\Pi|}
\\&
\lesssim 1,
\end{align*}
uniformly over all partitions $\Pi\in\mathrm{Par}\{1,\ldots,L\}$
with $\Pi\neq\{\{1,\ldots,L\}\}$. Inserting this back into~\eqref{eq:ergodicTheorem}
shows that also the right hand side term of equation~\eqref{eq:PDEhigherOrderLinearizedCorrectorLocalized} 
features at most polynomial growth at infinity.


Hence, we may apply the weighted energy estimate~\eqref{eq:expLocalization} to
equation~\eqref{eq:PDEhigherOrderLinearizedCorrectorLocalized} which 
entails in combination with Jensen's and H\"older's inequality
\begin{align*}
&\Big\langle\Big\|\Big(\frac{\phi^T_{\xi,B}}{\sqrt{T}},
\nabla\phi^T_{\xi,B}\Big)\Big\|^{2q}_{L^2(B_1)}\Big\rangle^\frac{1}{q}
\\&
\lesssim \sqrt{T}^d\mathds{1}_{L=1} 
+ \sqrt{T}^d\sum_{\substack{\Pi\in\mathrm{Par}\{1,\ldots,L\} \\ \Pi\neq\{\{1,\ldots,L\}\}}}
\int \ell_{\gamma,T}\Big\langle\prod_{\pi\in\Pi}
\big|\mathds{1}_{|\pi|=1}B'_\pi + \nabla \phi^T_{\xi,B'_\pi}\big|^{2q}
\Big\rangle^\frac{1}{q}
\\&
\lesssim \sqrt{T}^d\mathds{1}_{L=1} + \sqrt{T}^d\sum_{\substack{\Pi\in\mathrm{Par}\{1,\ldots,L\} 
\\ \Pi\neq\{\{1,\ldots,L\}\}}} \int \ell_{\gamma,T} \prod_{\pi\in\Pi}
\Big\langle\big|\mathds{1}_{|\pi|=1}B'_\pi 
+ \nabla \phi^T_{\xi,B'_\pi}\big|^{2q|\Pi|}\Big\rangle^\frac{1}{q|\Pi|}.
\end{align*}
Taking into account the induction hypothesis~\eqref{eq:indHypoCorrectorBounds} 
and~\eqref{eq:indHypoAnnealedSchauder}---the latter in particular allowing us
to smuggle in a spatial average over unit balls $B_1(x)$---
and making use of stationarity of the linearized homogenization correctors, we thus infer
from the previous display the asserted estimate~\eqref{eq:annealedSmallScaleEnergyEstimateSuboptimal}. \qed

\subsection{Proof of Lemma~\ref{lem:annealedSmallScaleSchauderIndStep} 
{\normalfont (Annealed small-scale Schauder estimate)}}
We aim to apply the local Schauder estimate~\eqref{eq:localSchauder}
to equation~\eqref{eq:PDEhigherOrderLinearizedCorrectorLocalized}
for the linearized homogenization corrector (putting to this end the term 
$-\nabla\cdot a_\xi^T\mathds{1}_{L=1}B$ on the right hand side). 
This is facilitated by the annealed H\"older regularity of the linearized
coefficient field $a_\xi^T=\partial_\xi A(\omega,\xi+\nabla\phi^T_\xi)$,
cf.\ Lemma~\ref{lem:annealedHoelderRegLinCoefficient} in Appendix~\ref{app:toolsEllipticRegularity}.
Hence, in view of the local Schauder estimate~\eqref{eq:localSchauder}
we may estimate by an application of H\"older's inequality
with respect to the exponents $(\frac{1}{\tau},\frac{1}{1-\tau})$, $\tau\in (0,1)$,
\begin{align*}
&\Big\langle\Big\|\Big(\frac{\phi^T_{\xi,B}}{\sqrt{T}},
\nabla\phi^T_{\xi,B}\Big)\Big\|^{2q}_{C^\alpha(B_1)}
\Big\rangle^\frac{1}{q}
\\& 
\leq C^2\big\langle\big\|a_\xi^T\big\|^{\frac{2q}{\tau}\frac{d}{\alpha}(\frac{1}{2}{+}\frac{1}{d})}_{C^\alpha(B_2)}
\big\rangle^\frac{\tau}{q}\Big\langle\Big\|\Big(\frac{\phi^T_{\xi,B}}{\sqrt{T}},
\nabla\phi^T_{\xi,B}\Big)\Big\|^\frac{2q}{1-\tau}_{L^2(B_2)}\Big\rangle^\frac{1-\tau}{q}
\\&~~~
+  C^2\big\langle\big\|a_\xi^T\big\|^{4q(\frac{1}{\alpha}-1)}_{C^\alpha(B_2)}
\big\rangle^\frac{1}{2q}\sum_{\substack{\Pi\in\mathrm{Par}\{1,\ldots,L\} 
\\ \Pi\neq\{\{1,\ldots,L\}\}}}\Big\langle\Big\|\prod_{\pi\in\Pi}
\big|\mathds{1}_{|\pi|=1}B'_\pi + \nabla \phi^T_{\xi,B'_\pi}\big|\Big\|^{4q}_{C^\alpha(B_2)}\Big\rangle^\frac{1}{2q}
\\&~~~
+  C^2\big\langle\big\|a_\xi^T\big\|^{\frac{2q}{\alpha}}_{C^\alpha(B_2)}
\big\rangle^\frac{1}{q}.
\end{align*}
A combination of the annealed estimate~\eqref{eq:annealedHoelderRegLinCoefficient}
for the H\"older norm of the linearized coefficient, the small-scale annealed 
Schauder estimate from induction hypothesis~\eqref{eq:indHypoAnnealedSchauder},
the stationarity of the linearized coefficient field and of the linearized homogenization correctors,
and H\"older's inequality updates the previous display to
\begin{align*}
&\Big\langle\Big\|\Big(\frac{\phi^T_{\xi,B}}{\sqrt{T}},
\nabla\phi^T_{\xi,B}\Big)\Big\|^{2q}_{C^\alpha(B_1)}
\Big\rangle^\frac{1}{q}
\\&
\leq C^2q^{2C}\bigg\{1+\Big\langle\Big\|\Big(\frac{\phi^T_{\xi,B}}{\sqrt{T}},
\nabla\phi^T_{\xi,B}\Big)\Big\|^\frac{2q}{1-\tau}_{L^2(B_1)}\Big\rangle^\frac{1-\tau}{q}\bigg\}
\\&~~~
+ C^2q^{2C}\sum_{\substack{\Pi\in\mathrm{Par}\{1,\ldots,L\} 
\\ \Pi\neq\{\{1,\ldots,L\}\}}}\prod_{\pi\in\Pi}\Big\langle\Big\|
\big|\mathds{1}_{|\pi|=1}B'_\pi + \nabla \phi^T_{\xi,B'_\pi}\big|\Big\|^{4q|\Pi|}_{C^\alpha(B_1)}
\Big\rangle^\frac{1}{2q|\Pi|}
\\&
\leq C^2q^{2C}\bigg\{1+\Big\langle\Big\|\Big(\frac{\phi^T_{\xi,B}}{\sqrt{T}},
\nabla\phi^T_{\xi,B}\Big)\Big\|^\frac{2q}{1-\tau}_{L^2(B_1)}\Big\rangle^\frac{1-\tau}{q}\bigg\}.
\end{align*}
This concludes the proof of Lemma~\ref{lem:annealedSmallScaleSchauderIndStep}. \qed

\subsection{Proof of Lemma~\ref{lem:annealedLinFunctionalsIndStep} 
{\normalfont (Annealed estimates for linear functionals of the homogenization corrector and its gradient)}}
The proof proceeds in three steps. In the course of it, we will make use
of the abbreviation $\sum_{\Pi}:=\sum_{\Pi\in\mathrm{Par}\{1,\ldots,L\},\,\Pi\neq\{\{1,\ldots,L\}\}}$.
Since $(\frac{\phi^T_{\xi,B}}{\sqrt{T}},\nabla\phi^T_{\xi,B})\in
L^2_{\mathrm{uloc}}(\Rd;\Rd[]{\times}\Rd)$, we may assume
for the proof of~\eqref{eq:annealedLinFunctionalsIndStep} without loss 
of generality through an approximation argument that
\begin{align}
\label{eq:assumpRegTestVectors}
(g,f)\in C^\infty_{\mathrm{cpt}}(\Rd;\Rd{\times}\Rd[]).
\end{align} 

\textit{Step 1 (Computation of functional derivative):} 
We start by computing the functional derivative of
\begin{align}
\label{def:LinFunctionalLinearizedHomCorrector}
F_\phi := \int g\cdot\nabla\phi^T_{\xi,B} - \int \frac{1}{T}f\phi^T_{\xi,B}.
\end{align}
To this end, let $\delta\omega\colon\Rd\to \Rd[n]$ be compactly supported and smooth such that
$\|\delta\omega\|_{L^\infty}\leq 1$. Based on Lemma~\ref{eq:lemmaExistenceLinearizedCorrectorsExtended},
we may $\Prob$-almost surely differentiate the defining 
equation~\eqref{eq:PDEhigherOrderLinearizedCorrectorLocalized} for the linearized homogenization
corrector with respect to the parameter field
in the direction of $\delta\omega$. This yields $\Prob$-almost surely the following PDE for the variation
$(\frac{\delta\phi^T_{\xi,B}}{\sqrt{T}},\nabla\delta\phi^T_{\xi,B})
\in L^2_{\mathrm{uloc}}(\Rd;\Rd[]{\times}\Rd)$ of the linearized corrector $\phi^T_{\xi,B}$
with massive term:
\begin{align}
\nonumber
&\frac{1}{T}\delta\phi^T_{\xi,B} - \nabla\cdot a_\xi^T\nabla\delta\phi^T_{\xi,B}
\\&\nonumber
= \nabla\cdot\partial_\omega\partial_\xi A(\omega,\xi+\nabla\phi^T_\xi)
\big[\delta\omega\odot\nabla \phi^T_{\xi,B}\big]
\\&~~~\nonumber
+ \nabla\cdot\partial_\xi^2 A(\omega,\xi+\nabla\phi^T_\xi)
\big[\nabla\delta\phi^T_\xi\odot\nabla \phi^T_{\xi,B}\big]
\\&~~~\nonumber
+\nabla\cdot \sum_{\Pi} \partial_\omega\partial_\xi^{|\Pi|} A(\omega,\xi{+}\nabla \phi^T_\xi)
\Big[\delta\omega\odot\bigodot_{\pi\in\Pi}(\mathds{1}_{|\pi|=1}B'_\pi {+} \nabla \phi^T_{\xi,B'_\pi})\Big]
\\&~~~\nonumber
+\nabla\cdot \sum_{\Pi} \partial_\xi^{1+|\Pi|} A(\omega,\xi{+}\nabla \phi^T_\xi)
\Big[\nabla\delta\phi_\xi^T\odot\bigodot_{\pi\in\Pi}(\mathds{1}_{|\pi|=1}B'_\pi {+} \nabla \phi^T_{\xi,B'_\pi})\Big]
\\&~~~\nonumber
+\nabla\cdot \sum_{\Pi} \partial_\xi^{|\Pi|} A(\omega,\xi{+}\nabla \phi^T_\xi)
\Big[\sum_{\pi\in\Pi}\nabla\delta\phi^T_{\xi,B'_\pi}\odot
\bigodot_{\substack{{\pi'\in\Pi} \\ \pi'\neq \pi}}
(\mathds{1}_{|\pi'|=1}B'_{\pi'} {+} \nabla \phi^T_{\xi,B'_{\pi'}})\Big]
\\& \label{eq:PDEperturbedLinearizedHomCorrector}
=: \nabla\cdot R^{T,(1)}_{\xi,B}\delta\omega
+ \nabla\cdot R^{T,(2)}_{\xi,B}\nabla\delta\phi^T_\xi
+ \nabla\cdot R^{T,(3)}_{\xi,B}\delta\omega
+ \nabla\cdot R^{T,(4)}_{\xi,B}\nabla\delta\phi^T_\xi
\\&~~~~\nonumber
+  \sum_{\Pi}\sum_{\pi\in\Pi}\nabla\cdot 
R^{T,(5),\pi}_{\xi,B}\nabla\delta\phi^T_{\xi,B'_\pi}.
\end{align}
Observe that as a consequence of~\eqref{eq:ergodicityLinearizedCorrector}
and~(A2)$_L$ resp.\ (A3)$_L$ of Assumption~\ref{assumption:operators}
we have $\Prob$-almost surely
\begin{align}
\label{eq:regAuxSensitivity}
R^{T,(1)}_{\xi,B},R^{T,(3)}_{\xi,B} \in L^p_{\mathrm{uloc}}(\Rd;\Rd[d{\times}n])
,\,R^{T,(2)}_{\xi,B},R^{T,(4)}_{\xi,B}, R^{T,(5),\pi}_{\xi,B}
\in L^p_{\mathrm{uloc}}(\Rd;\Rd[d{\times}d])
\end{align}
for all $p\geq 2$. For $r\geq 1$, let $(\frac{\delta\phi^{T,r}_{\xi,B}}{\sqrt{T}},\nabla\delta\phi^{T,r}_{\xi,B})
\in L^2(\Rd;\Rd[]{\times}\Rd)$ be the unique Lax--Milgram solution of
\begin{align}
\nonumber
&\frac{1}{T}\delta\phi^{T,r}_{\xi,B} - \nabla\cdot a_\xi^T\nabla\delta\phi^{T,r}_{\xi,B}
\\& \label{eq:PDEperturbedLinearizedHomCorrectorApprox}
=: \nabla\cdot \mathds{1}_{B_r} R^{T,(1)}_{\xi,B}\delta\omega
+ \nabla\cdot \mathds{1}_{B_r} R^{T,(2)}_{\xi,B}\nabla\delta\phi^T_\xi
+ \nabla\cdot \mathds{1}_{B_r} R^{T,(3)}_{\xi,B}\delta\omega
\\&~~~~\nonumber
+ \nabla\cdot \mathds{1}_{B_r} R^{T,(4)}_{\xi,B}\nabla\delta\phi^T_\xi
+  \sum_{\Pi}\sum_{\pi\in\Pi}\nabla\cdot 
\mathds{1}_{B_r} R^{T,(5),\pi}_{\xi,B}\nabla\delta\phi^T_{\xi,B'_\pi}.
\end{align}
Note that $\Prob$-almost surely
\begin{align}
\label{eq:convAuxSensitivity}
\Big(\frac{\delta\phi^{T,r}_{\xi,B}}{\sqrt{T}},\nabla\delta\phi^{T,r}_{\xi,B}\Big)
\to \Big(\frac{\delta\phi^{T}_{\xi,B}}{\sqrt{T}},\nabla\delta\phi^{T}_{\xi,B}\Big) 
\text{ as } r\to\infty \text{ in } L^2_{\mathrm{uloc}}(\Rd;\Rd[d{\times}d])
\end{align}
by means of applying the weighted energy estimate~\eqref{eq:expLocalization} to the 
difference of the equations~\eqref{eq:PDEperturbedLinearizedHomCorrector}
and~\eqref{eq:PDEperturbedLinearizedHomCorrectorApprox}; recall to this
end also~\eqref{eq:ergodicityLinearizedCorrectorGateaux}.

Denoting the transpose of $a^T_\xi$ by $a^{T,*}_\xi$, we may now compute by means of the
dual operator $(\frac{1}{T}-\nabla\cdot a^{T,*}_\xi\nabla)$, \eqref{eq:regAuxSensitivity},
\eqref{eq:PDEperturbedLinearizedHomCorrectorApprox}, \eqref{eq:convAuxSensitivity} 
and~\eqref{eq:ergodicityLinearizedCorrectorGateaux}
\begin{align}
\nonumber
\delta F_\phi &=
\int g\cdot\nabla\delta\phi^T_{\xi,B} - \int \frac{1}{T}f\delta\phi^T_{\xi,B}
\\& \nonumber
= \lim_{r\to\infty} \bigg\{\int g\cdot\nabla\delta\phi^{T,r}_{\xi,B} 
- \int \frac{1}{T}f\delta\phi^{T,r}_{\xi,B}\bigg\}
\\& \nonumber
= - \lim_{r\to\infty} \bigg\{\int \nabla\delta\phi^{T,r}_{\xi,B}\cdot a^{T,*}_\xi
\nabla\Big(\frac{1}{T}{-}\nabla\cdot a^{T,*}_\xi\nabla\Big)^{-1}
\Big(\frac{1}{T}f{+}\nabla\cdot g\Big)
\\&~~~~~~~~~~~~~~~~ \nonumber
+ \int \delta\phi^{T,r}_{\xi,B}\,
\frac{1}{T}\Big(\frac{1}{T}{-}\nabla\cdot a^{T,*}_\xi\nabla\Big)^{-1}
\Big(\frac{1}{T}f{+}\nabla\cdot g\Big) \bigg\}
\\& \label{eq:funcDerivLinFunctionalsLinearizedHomCorrector}
= \lim_{r\to\infty} \bigg\{\int \mathds{1}_{B_r} R^{T,(1)}_{\xi,B}\delta\omega\cdot 
\nabla\Big(\frac{1}{T}{-}\nabla\cdot a^{T,*}_\xi\nabla\Big)^{-1}
\Big(\frac{1}{T}f{+}\nabla\cdot g\Big)
\\&~~~~~~~ \nonumber
+ \int \mathds{1}_{B_r} R^{T,(2)}_{\xi,B}\nabla\delta\phi^T_\xi\cdot 
\nabla\Big(\frac{1}{T}{-}\nabla\cdot a^{T,*}_\xi\nabla\Big)^{-1}
\Big(\frac{1}{T}f{+}\nabla\cdot g\Big)
\\&~~~~~~~ \nonumber
+ \int \mathds{1}_{B_r} R^{T,(3)}_{\xi,B}\delta\omega\cdot 
\nabla\Big(\frac{1}{T}{-}\nabla\cdot a^{T,*}_\xi\nabla\Big)^{-1}
\Big(\frac{1}{T}f{+}\nabla\cdot g\Big)
\\&~~~~~~~ \nonumber
+ \int \mathds{1}_{B_r} R^{T,(4)}_{\xi,B}\nabla\delta\phi^T_\xi\cdot 
\nabla\Big(\frac{1}{T}{-}\nabla\cdot a^{T,*}_\xi\nabla\Big)^{-1}
\Big(\frac{1}{T}f{+}\nabla\cdot g\Big)
\\&~~~~~~~ \nonumber
+ \sum_{\Pi}\sum_{\pi\in\Pi} \int \mathds{1}_{B_r} R^{T,(5),\pi}_{\xi,B}
\nabla\delta\phi^T_{\xi,B'_\pi}\cdot \nabla\Big(\frac{1}{T}{-}\nabla\cdot a^{T,*}_\xi\nabla\Big)^{-1}
\Big(\frac{1}{T}f{+}\nabla\cdot g\Big)\bigg\}.
\end{align}
Note that thanks to the approximation argument, we may indeed use $\delta\phi^{T,r}_{\xi,B}$
as a test function in the equation of $\big(\frac{1}{T}{-}\nabla\cdot a^{T,*}_\xi\nabla\big)^{-1}
\big(\frac{1}{T}f{+}\nabla\cdot g\big)$, and vice versa. Moreover, by~\eqref{eq:regAuxSensitivity},
\eqref{eq:assumpRegTestVectors} and the (local) Meyers estimate for the operator 
$(\frac{1}{T}-\nabla\cdot a^{T,*}_\xi\nabla)$, we obtain that $\Prob$-almost surely
\begin{equation}
\label{eq:auxReg100}
\begin{aligned}
\big(R^{T,(i)}_{\xi,B}\big)^*\nabla\Big(\frac{1}{T}{-}\nabla\cdot a^{T,*}_\xi\nabla\Big)^{-1}
\Big(\frac{1}{T}f{+}\nabla\cdot g\Big) &\in L^{p'}_{\mathrm{uloc}}(\Rd;\Rd[n]),\,i\in\{1,3\},
\\
\big(R^{T,(i)}_{\xi,B}\big)^*\nabla\Big(\frac{1}{T}{-}\nabla\cdot a^{T,*}_\xi\nabla\Big)^{-1}
\Big(\frac{1}{T}f{+}\nabla\cdot g\Big) &\in L^{p'}_{\mathrm{uloc}}(\Rd;\Rd),\,i\in\{2,4\},
\\
\big(R^{T,(5),\pi}_{\xi,B}\big)^*\nabla\Big(\frac{1}{T}{-}\nabla\cdot a^{T,*}_\xi\nabla\Big)^{-1}
\Big(\frac{1}{T}f{+}\nabla\cdot g\Big) &\in L^{p'}_{\mathrm{uloc}}(\Rd;\Rd)
\end{aligned}
\end{equation}
for some suitable Meyers exponents $p' > 2$.
We have everything in place to proceed with the next step of the proof.

\textit{Step 2 (Application of the spectral gap inequality):}
Note first that $\langle F_\phi \rangle = 0$ for the functional $F_\phi$ from~\eqref{def:LinFunctionalLinearizedHomCorrector}.
Indeed, $\langle(\phi^T_{\xi,B},\nabla\phi^T_{\xi,B})\rangle=0$ is a direct consequence of stationarity and testing the
linearized corrector problem~\eqref{eq:PDEhigherOrderLinearizedCorrectorLocalized}.
Hence, we may apply the spectral gap inequality in form of~\eqref{eq:spectralGapHigherMoments}
and thus obtain in view of~(A2)$_L$ and~(A3)$_L$ of Assumption~\ref{assumption:operators},
\eqref{eq:funcDerivLinFunctionalsLinearizedHomCorrector}, 
\eqref{eq:auxReg100}, induction hypothesis~\eqref{eq:indHypoSensitivityApprox},
and the sensitivity estimates from induction hypothesis~\eqref{eq:indHypoSensitivityBound} the estimate
(with $\kappa_1,\kappa_2\in (0,1]$ yet to be determined)
\begin{align}
\nonumber
&\big\langle\big|F_\phi\big|^{2q}\big\rangle^\frac{1}{q}
\\&\nonumber
\leq C^2q^{2}\bigg\langle\bigg|
\int\bigg(\,\dashint_{B_1(x)}\big|\nabla\phi^T_{\xi,B}\big|^2\bigg)
\bigg(\,\dashint_{B_1(x)}\Big|\Big(\frac{1}{T}{-}\nabla\cdot a^{T,*}_\xi\nabla\Big)^{-1}
\Big(\frac{1}{T}f{+}\nabla\cdot g\Big)\Big|^2\bigg)
\bigg|^q\bigg\rangle^{\frac{1}{q}}
\\&~~~\nonumber
+ C^2q^{2}\sum_{\Pi} \bigg\langle\bigg|
\int\bigg(\,\dashint_{B_1(x)}\prod_{\pi\in\Pi}
\big|\mathds{1}_{|\pi|=1}B'_\pi {+} \nabla \phi^T_{\xi,B'_\pi}\big|^2\bigg)
\\&~~~~~~~~~~~~~~~~~~~~~~~~~~~~~\nonumber\times
\bigg(\,\dashint_{B_1(x)}\Big|\nabla\Big(\frac{1}{T}{-}\nabla\cdot a^{T,*}_\xi\nabla\Big)^{-1}
\Big(\frac{1}{T}f{+}\nabla\cdot g\Big)\Big|^2\bigg)
\bigg|^q\bigg\rangle^{\frac{1}{q}}
\\&~~~\nonumber
+ C^2q^{2C} \sup_{\langle F^{2q_*}\rangle=1} \int
\Big\langle\Big|\nabla\phi^T_{\xi,B}\Big|^{2(\frac{q}{\kappa_1})_*}
\Big|\nabla\Big(\frac{1}{T}{-}\nabla\cdot a^{T,*}_\xi\nabla\Big)^{-1}
\Big(\frac{1}{T}Ff{+}\nabla\cdot Fg\Big)\Big|^{2(\frac{q}{\kappa_1})_*}
\Big\rangle^\frac{1}{(\frac{q}{\kappa_1})_*}
\\&~~~\nonumber
+ C^2q^{2C} \sum_{\Pi} \sup_{\langle F^{2q_*}\rangle=1}
\int\Big\langle\Big|\nabla\Big(\frac{1}{T}{-}\nabla\cdot a^{T,*}_\xi\nabla\Big)^{-1}
\Big(\frac{1}{T}Ff{+}\nabla\cdot Fg\Big)\Big|^{2(\frac{q}{\kappa_2})_*}
\\&~~~~~~~~~~~~~~~~~~~~~~~~~~~~~~~~~~~~~~~~~~~~~~~~~~~~\nonumber\times
\prod_{\pi\in\Pi} \big|\mathds{1}_{|\pi|=1}B'_\pi {+} \nabla \phi^T_{\xi,B'_\pi}\big|^{2(\frac{q}{\kappa_2})_*}
\Big\rangle^\frac{1}{(\frac{q}{\kappa_2})_*}
\\&~~~\nonumber
+ C^2q^{2C} \sum_{\Pi}\sum_{\pi\in\Pi} \sup_{\langle F^{2q_*}\rangle=1}
\int\Big\langle\Big|\nabla\Big(\frac{1}{T}{-}\nabla\cdot a^{T,*}_\xi\nabla\Big)^{-1}
\Big(\frac{1}{T}Ff{+}\nabla\cdot Fg\Big)\Big|^{2(\frac{q}{\kappa_2})_*}
\\&~~~~~~~~~~~~~~~~~~~~~~~~~~~~~~~~~~~~~~~~~~~~~~~~~~~~\nonumber\times
\prod_{\substack{\pi'\in\Pi \\ \pi'\neq\pi}} 
\big|\mathds{1}_{|\pi'|=1}B'_{\pi'} {+} \nabla \phi^T_{\xi,B'_{\pi'}}\big|^{2(\frac{q}{\kappa_2})_*}
\Big\rangle^\frac{1}{(\frac{q}{\kappa_2})_*}
\\&\label{eq:sensitivityEstimateLinearizedHomCorrector}
=: I_1 + I_2 + I_3 + I_4 + I_5. 
\end{align}
For the last three right hand side terms in the previous display, we also exploited the fact that $F$
is purely random so that one can simply multiply with $F$ the equation satisfied 
by $(\frac{1}{T}-\nabla\cdot a^{T,*}_\xi\nabla)^{-1}(\frac{1}{T}f + \nabla\cdot g)$.

\textit{Step 3 (Post-processing the right hand side of~\eqref{eq:sensitivityEstimateLinearizedHomCorrector}):}
We estimate each of the right hand side terms of~\eqref{eq:sensitivityEstimateLinearizedHomCorrector}
separately. By duality in $L^q_{\langle\cdot\rangle}$, stationarity of the
linearized homogenization corrector $\phi^T_{\xi,B}$, and H\"older's inequality
with respect to the exponents $(\frac{q}{1-\tau},(\frac{q}{1-\tau})_*)$, $\tau\in (0,1)$,
we estimate the contribution from~$I_1$ by
\begin{align*}
|I_1| &\leq C^2q^2 \Big\langle\Big\|\Big(\frac{\phi^T_{\xi,B}}{\sqrt{T}},
\nabla\phi^T_{\xi,B}\Big)\Big\|^\frac{2q}{1-\tau}_{L^2(B_1)}\Big\rangle^\frac{1-\tau}{q}
\\&~~~~~\times
\sup_{\langle F^{2q_*}\rangle=1} \int\dashint_{B_1(x)} 
\Big\langle\Big|\nabla\Big(\frac{1}{T}{-}\nabla\cdot a^{T,*}_\xi\nabla\Big)^{-1}
\Big(\frac{1}{T}Ff{+}\nabla\cdot Fg\Big)\Big|^{2(\frac{q}{1-\tau})_*}
\Big\rangle^\frac{1}{(\frac{q}{1-\tau})_*}.
\end{align*}
As $(\frac{q}{1-\tau})_*=\frac{q}{q-(1-\tau)} < \frac{q}{q-1} = q_*$,
it follows from Jensen's inequality, the fact that $\int\dashint_{B_1(x)}h = \int h$
for all non-negative $h$, and the annealed Calder\'on--Zygmund
estimate~\eqref{eq:annealedCZMeyers} that
\begin{equation}
\label{eq:I1}
\begin{aligned}
|I_1| &\leq C^2q^2\Big\langle\Big\|\Big(\frac{\phi^T_{\xi,B}}{\sqrt{T}},
\nabla\phi^T_{\xi,B}\Big)\Big\|^\frac{2q}{1-\tau}_{L^2(B_1)}\Big\rangle^\frac{1-\tau}{q}
\sup_{\langle F^{2q_*}\rangle=1} \int \Big\langle
\Big|\Big(\frac{Ff}{\sqrt{T}},Fg\Big)\Big|^{2q_*}\Big\rangle^\frac{1}{q_*}
\\&
\leq C^2q^2\Big\langle\Big\|\Big(\frac{\phi^T_{\xi,B}}{\sqrt{T}},
\nabla\phi^T_{\xi,B}\Big)\Big\|^\frac{2q}{1-\tau}_{L^2(B_1)}\Big\rangle^\frac{1-\tau}{q}
\int \Big|\Big(\frac{f}{\sqrt{T}},g\Big)\Big|^2
\end{aligned}
\end{equation}
provided $|q_*-1|$ is sufficiently small. In other words, we obtain a bound
of required type for all sufficiently large $q\in [1,\infty)$. 

For the contribution from $I_2$, we may estimate for all sufficiently large $q\in [1,\infty)$ 
based on the same ingredients as in the estimate of $I_1$ (we could actually take $\tau=0$
but prefer to keep the general form for later reference)
\begin{align}
\nonumber
|I_2| &\leq C^2q^2\sum_{\Pi}\prod_{\pi\in\Pi} \bigg\langle\,\dashint_{B_1}
\big|\mathds{1}_{|\pi|=1}B'_\pi {+} \nabla \phi^T_{\xi,B'_\pi}
\big|^{\frac{2q|\Pi|}{1-\tau}}\bigg\rangle^\frac{1-\tau}{q|\Pi|}
\\&~~~~~\nonumber\times
\sup_{\langle F^{2q_*}\rangle=1} \int\dashint_{B_1(x)} 
\Big\langle\Big|\nabla\Big(\frac{1}{T}{-}\nabla\cdot a^{T,*}_\xi\nabla\Big)^{-1}
\Big(\frac{1}{T}Ff{+}\nabla\cdot Fg\Big)\Big|^{2(\frac{q}{1-\tau})_*}\Big\rangle^{\frac{1}{(\frac{q}{1-\tau})_*}}
\\&\label{eq:I2}
\leq C^2q^{2C}\int \Big|\Big(\frac{f}{\sqrt{T}},g\Big)\Big|^2,
\end{align}
where in the second step we in addition made use of the small-scale annealed 
Schauder estimate from induction hypothesis~\eqref{eq:indHypoAnnealedSchauder}
(by smuggling in a spatial average over the unit ball)
and the corrector estimates from induction hypothesis~\eqref{eq:indHypoCorrectorBounds}.

We next estimate the contribution from the term $I_3$. To this end,
we choose $\kappa_1=\frac{\tau}{2}$ and estimate via stationarity
of the linearized homogenization corrector $\phi^T_{\xi,B}$,
the fact that $(\frac{q}{\kappa_1})_*=\frac{q}{q-\kappa_1}$,
an application of H\"older's inequality with respect
to the exponents $(\frac{q-\kappa_1}{1-\tau},(\frac{q-\kappa_1}{1-\tau})_*)$,
and an application of Jensen's inequality based on
$(\frac{q}{\kappa_1})_*(\frac{q-\kappa_1}{1-\tau})_*
=\frac{q}{q-\kappa_1-(1-\tau)} < \frac{q}{q-1}=q_*$
\begin{align*}
 |I_3| &\leq C^2q^{2C} \big\langle
\big\|\nabla\phi^T_{\xi,B}\big\|_{C^\alpha(B_1)}^\frac{2q}{1-\tau}
\big\rangle^\frac{q}{1-\tau}
\\&~~~~~\times
\sup_{\langle F^{2q_*}\rangle=1} \int 
\Big\langle\Big|\nabla\Big(\frac{1}{T}{-}\nabla\cdot a^{T,*}_\xi\nabla\Big)^{-1}
\Big(\frac{1}{T}Ff{+}\nabla\cdot Fg\Big)\Big|^{2q_*}\Big\rangle^\frac{1}{q_*}.
\end{align*}
Hence, by means of the annealed Calder\'on--Zygmund
estimate~\eqref{eq:annealedCZMeyers} and the annealed small-scale 
Schauder estimate~\eqref{eq:annealedSmallScaleSchauderIndStep}
(with $q$ replaced by $\frac{q}{1-\tau}$) we obtain
\begin{align}
\label{eq:I3}
|I_3| \leq C^2q^{2C}
\Big\langle\Big\|\Big(\frac{\phi^T_{\xi,B}}{\sqrt{T}},
\nabla\phi^T_{\xi,B}\Big)\Big\|^\frac{2q}{(1-\tau)^2}_{L^2(B_1)}\Big\rangle^\frac{(1-\tau)^2}{q}
\int \Big|\Big(\frac{f}{\sqrt{T}},g\Big)\Big|^2,
\end{align}
at least for sufficiently large $q\in [1,\infty)$. This is again a bound of required type.

We next deal with the contribution from $I_4$. To this end, we simply choose $\kappa_2=\frac{1}{2}$
and argue based on stationarity of the linearized homogenization correctors,  
$(\frac{q}{\kappa_2})_*=\frac{q}{q-\kappa_2}$, an application of H\"older's inequality with respect
to the exponents $(\frac{q-\kappa_2}{q-1},(\frac{q-\kappa_2}{q-1})_*)$, and the fact that 
$(\frac{q}{\kappa_2})_*(\frac{q-\kappa_2}{q-1})_* = \frac{q}{1-\kappa_2} = 2q$
\begin{align*}
|I_4| &\leq  C^2q^{2C}\sum_{\Pi}\prod_{\pi\in\Pi} \Big\langle
\big|\mathds{1}_{|\pi|=1}B'_\pi {+} \nabla \phi^T_{\xi,B'_\pi}
\big|^{2q|\Pi|}\Big\rangle^\frac{1}{q|\Pi|}
\\&~~~~~\times
\sup_{\langle F^{2q_*}\rangle=1} \int 
\Big\langle\Big|\nabla\Big(\frac{1}{T}{-}\nabla\cdot a^{T,*}_\xi\nabla\Big)^{-1}
\Big(\frac{1}{T}Ff{+}\nabla\cdot Fg\Big)\Big|^{2q_*}\Big\rangle^\frac{1}{q_*}.
\end{align*}
As it is by now routine, the second factor in the right hand side term of the previous display
is dealt with by appealing to the annealed Calder\'on--Zygmund estimate~\eqref{eq:annealedCZMeyers}
for which we only have to choose $q\in [1,\infty)$ sufficiently large. For the first factor,
we may smuggle in a spatial average over the ball $B_1$ and then estimate by
means of the induction hypotheses~\eqref{eq:indHypoCorrectorBounds}
and~\eqref{eq:indHypoAnnealedSchauder}. In total, we obtain
\begin{align}
\label{eq:I4}
|I_4| &\leq C^2q^{2C} \int \Big|\Big(\frac{f}{\sqrt{T}},g\Big)\Big|^2
\end{align}
provided $q\in [1,\infty)$ is sufficiently large. As the contribution
from $I_5$ can be treated analogously, the combination of
the estimates~\eqref{eq:I1}--\eqref{eq:I4} establishes the asserted bound~\eqref{eq:annealedLinFunctionalsIndStep}.
This concludes the proof of Lemma~\ref{lem:annealedLinFunctionalsIndStep}. \qed

\subsection{Proof of Lemma~\ref{lem:momentBoundMinRadius}
{\normalfont (Stretched exponential moment bound for minimal radius
of the linearized corrector problem)}}
We start with the estimate
\begin{align}
\label{eq:momentsViaProb}
\big\langle r_{*,T,\xi,B}^{(d-\frac{\delta}{2})2q}\big\rangle
\leq 1 + \sum_{k=1}^\infty 2^{k(d-\frac{\delta}{2})2q}
\Prob\big[\big\{r_{*,T,\xi,B}=2^k\big\}\big].
\end{align}
Fix an integer $k\geq 1$, and let $R:=2^k$. By Definition~\ref{def:minRadiusHigherOrderLinearizations}
of the minimal radius $r_{*,T,\xi,B}$, in the event of~$\{r_{*,T,\xi,B}=R\}$ we either have
\begin{equation}
\label{eq:case1}
\begin{aligned}
&\exists\Pi\in\mathrm{Par}\{1,\ldots,L\},\,\Pi\neq\{\{1,\ldots,L\}\}\colon
\exists\pi\in\Pi\text{ such that}
\\&
\dashint_{B_\frac{R}{2}} \big|\mathds{1}_{|\pi|=1}B'_\pi + \nabla \phi^T_{\xi,B'_\pi}\big|^{4|\Pi|}
> \Big(\frac{R}{2}\Big)^{4\underline{\gamma}},
\end{aligned}
\end{equation}
or that
\begin{align}
\label{eq:case2}
\inf_{b\in\mathbb{R}}\bigg\{\frac{1}{(R/2)^2}\dashint_{B_\frac{R}{2}}
\big|\phi^T_{\xi,B}{-}b\big|^2{+}\frac{1}{T}|b|^2\bigg\} > 1.
\end{align}
We distinguish in the following between these two events, and provide estimates
on their probability separately.

\textit{Case 1: (Estimate in the event of~\eqref{eq:case1})}
Fix a partition $\Pi\in\mathrm{Par}\{1,\ldots,L\}$, $\Pi\neq\{\{1,\ldots,L\}\}$,
and some $\pi\in\Pi$ such that the conclusion of~\eqref{eq:case1} holds true.
Covering $B_\frac{R}{2}$ by a family of $\sim R^d$ many open unit balls,
using stationarity of the linearized homogenization correctors, Jensen's inequality, 
the small-scale annealed Schauder estimate from induction hypothesis~\eqref{eq:indHypoAnnealedSchauder}
which in particular allows to smuggle in a spatial average over the unit ball,
and finally the corrector bounds from induction hypotheses~\eqref{eq:indHypoCorrectorBounds}, we infer that
\begin{align*}
\bigg\langle\bigg|\,\dashint_{B_\frac{R}{2}} \big|\mathds{1}_{|\pi|=1}B'_\pi 
+ \nabla \phi^T_{\xi,B'_\pi}\big|^{4|\Pi|}\bigg|^{2q}\bigg\rangle^\frac{1}{q}
\leq C^2q^{2C}
\end{align*}
for all $q\in [1,\infty)$. 
It thus follows from Markov's inequality
and the previous display (with $q$ replaced by $\frac{dq}{\underline{\gamma}}$) that
\begin{align}
\label{eq:boundProbCase1}
\Prob\big[\big\{r_{*,T,\xi,B}=2^k\big\}
\cap\{\eqref{eq:case1}\text{ holds true}\}\big]
\leq C^{2q}q^{2Cq}2^{-k8dq}
\end{align}
with a constant $C$ only depending on the
admissible data $(d,\lambda,\Lambda,\nu,\rho,M,L)$.

\textit{Case 2: (Suboptimal estimate in the event of~\eqref{eq:case2} but not~\eqref{eq:case1}})
Let $0 < R'\leq R$, and abbreviate by $(h)_{R'}(x):=\dashint_{B_{R'}(x)} h$
mollification on scale~$R'$ for any locally integrable~$h$. In the event
of $\{\eqref{eq:case2}\text{ holds true}\}\cap\{\eqref{eq:case1}\text{ holds not true}\}$
we deduce from the triangle inequality that
\begin{equation}
\label{eq:case2Aux10}
\begin{aligned}
1 &\lesssim \frac{1}{R^2}\dashint_{B_R} \big|\phi^T_{\xi,B}{-}(\phi^T_{\xi,B})_{R'}\big|^2
\\&~~~
+ \frac{1}{R^2}\dashint_{B_R} \bigg|(\phi^T_{\xi,B})_{R'}
{-}\,\dashint_{B_{R}}(\phi^T_{\xi,B})_{R'}\bigg|^2
+ \frac{1}{T}\bigg|\,\dashint_{B_{R}}(\phi^T_{\xi,B})_{R'}\bigg|^2.
\end{aligned}
\end{equation}
The first term on the right hand side of the previous display is estimated by
\begin{align*}
\frac{1}{R^2} \dashint_{B_R} \big|\phi^T_{\xi,B}{-}(\phi^T_{\xi,B})_{R'}\big|^2
\lesssim \Big(\frac{R'}{R}\Big)^2 \dashint_{B_{2R}} 
\big|\nabla\phi^T_{\xi,B}\big|^2,
\end{align*}
the second right hand side term based on Poincar\'e's inequality and the
definition of $(\cdot)_{R'}$ by
\begin{align*}
\frac{1}{R^2}\dashint_{B_R} \bigg|(\phi^T_{\xi,B})_{R'}
{-}\,\dashint_{B_{R}}(\phi^T_{\xi,B})_{R'}\bigg|^2
\lesssim \dashint_{B_{2R}} \bigg|\,\dashint_{B_{R'}(x)}\nabla\phi^T_{\xi,B}\bigg|^2,
\end{align*}
and the third right hand side term simply by plugging in the definition of $(\cdot)_{R'}$
and Jensen's inequality
\begin{align*}
\frac{1}{T}\bigg|\,\dashint_{B_{R}}(\phi^T_{\xi,B})_{R'}\bigg|^2
\lesssim \dashint_{B_{2R}} \bigg|\,\dashint_{B_{R'}(x)}
\frac{1}{\sqrt{T}}\phi^T_{\xi,B}\bigg|^2.
\end{align*}
In the event of $\{\eqref{eq:case2}\text{ holds true}\}\cap\{\eqref{eq:case1}\text{ holds not true}\}$,
the combination of the last four displays therefore entails
\begin{align}
\label{eq:case2Aux1}
1 \lesssim \Big(\frac{R'}{R}\Big)^2 \dashint_{B_{2R}} 
\big|\nabla\phi^T_{\xi,B}\big|^2
+ \dashint_{B_{2R}} \bigg|\,\dashint_{B_{R'}(x)}\nabla\phi^T_{\xi,B}\bigg|^2
+ \dashint_{B_{2R}} \bigg|\,\dashint_{B_{R'}(x)}
\frac{1}{\sqrt{T}}\phi^T_{\xi,B}\bigg|^2.
\end{align}
Applying next the Caccioppoli inequality~\eqref{eq:Caccioppoli}
to equation~\eqref{eq:PDEhigherOrderLinearizedCorrectorLocalized}
for the linearized homogenization corrector (putting the term 
$-\nabla\cdot a_\xi^T\mathds{1}_{L=1}B$ on the right hand side),
and making use of Definition~\eqref{def:minRadiusHigherOrderLinearizations} of the minimal radius $r_{*,T,\xi,B}$
we obtain in the event of $\{\eqref{eq:case2}\text{ holds true}\}\cap\{\eqref{eq:case1}\text{ holds not true}\}$
\begin{align}
\label{eq:case2Absorption}
\Big(\frac{R'}{R}\Big)^2 \dashint_{B_{2R}} 
\big|\nabla\phi^T_{\xi,B}\big|^2
&\leq C(d,\lambda,\Lambda)\Big(\frac{R'}{R}\Big)^2\big(1 + R^{2\bar\gamma}\big).
\end{align}
Hence, restricting $\underline{\gamma}=\underline{\gamma}(d,\lambda,\Lambda)\in (0,\frac{1}{2})$ 
(but otherwise yet to be determined),
choosing $R'=\theta R^{1-\underline{\gamma}}$ with $\theta=\theta(d,\lambda,\Lambda)$ such that 
$2\theta^2C(d,\lambda,\Lambda)=\frac{1}{2}$, we infer from~\eqref{eq:case2Aux1}
and~\eqref{eq:case2Absorption} in the event 
of $\{\eqref{eq:case2}\text{ holds true}\}\cap\{\eqref{eq:case1}\text{ holds not true}\}$
\begin{align}
\label{eq:case2Aux2}
1 \lesssim_{d,\lambda,\Lambda} \dashint_{B_{2R}} 
\bigg|\,\dashint_{B_{\theta R^{1-\underline{\gamma}}}(x)}\nabla\phi^T_{\xi,B}\bigg|^2
+ \dashint_{B_{2R}} \bigg|\,\dashint_{B_{\theta R^{1-\underline{\gamma}}}(x)}
\frac{1}{\sqrt{T}}\phi^T_{\xi,B}\bigg|^2.
\end{align}
It thus follows from an application of Markov's inequality in combination
with stationarity of the linearized homogenization correctors and Jensen's inequality that
\begin{equation}
\label{eq:case2SubotimalBoundProbAux}
\begin{aligned}
&\Prob\big[\big\{r_{*,T,\xi,B}=2^k\big\}
\cap\{\eqref{eq:case2}\text{ holds true}\}
\cap\{\eqref{eq:case1}\text{ holds not true}\}\big]
\\&
\leq C(d,\lambda,\Lambda)^{2q}
\bigg\langle\bigg|
\,\dashint_{B_{\theta 2^{k(1-\underline{\gamma})}}}
\nabla \phi^T_{\xi,B},\,
\,\dashint_{B_{\theta 2^{k(1-\underline{\gamma})}}}
\frac{1}{\sqrt{T}}\phi^T_{\xi,B}
\bigg|^{4q}\bigg\rangle.
\end{aligned}
\end{equation}
Thanks to the moment bounds~\eqref{eq:annealedLinFunctionalsIndStep} 
for linear functionals of the linearized homogenization corrector and its gradient
(applied with, say, $\tau=\frac{1}{2}$) 
in combination with the suboptimal small-scale energy estimate~\eqref{eq:annealedSmallScaleEnergyEstimateSuboptimal},
we get the following update of the previous display
\begin{equation}
\label{eq:case2SubotimalBoundProb}
\begin{aligned}
&\Prob\big[\big\{r_{*,T,\xi,B}=2^k\big\}
\cap\{\eqref{eq:case2}\text{ holds true}\}
\cap\{\eqref{eq:case1}\text{ holds not true}\}\big]
\\&
\lesssim_{d,\lambda,\Lambda,q} \sqrt{T}^{2dq}2^{-k(1-\underline{\gamma})2dq},
\end{aligned}
\end{equation}
which is---with respect to the scaling in the stochastic integrability $q$ and
the massive approximation $T$---a highly supobtimal estimate for this probability.

\textit{Intermediate summary: (Suboptimal estimate on stochastic moments of the minimal radius)}
We collect the information provided by the estimates~\eqref{eq:boundProbCase1}
and~\eqref{eq:case2SubotimalBoundProb}, and combine it with~\eqref{eq:momentsViaProb}
resulting in
\begin{equation}
\label{eq:suboptimalMomentBoundMinRadiusAux}
\begin{aligned}
\big\langle r_{*,T,\xi,B}^{(d-\frac{\delta}{2})2q}\big\rangle
&\leq 1 + C^{2q}q^{2Cq}\sum_{k=1}^\infty 2^{k(d-\frac{\delta}{2})2q}2^{-k8dq}
\\&~~~
+ \bar C_q\sqrt{T}^{2dq}\sum_{k=1}^\infty 2^{k(d-\frac{\delta}{2})2q}
2^{-k(1-\underline{\gamma})2dq}
\end{aligned}
\end{equation}
with a constant $C$ only depending on the
admissible data $(d,\lambda,\Lambda,\nu,\rho,M,L)$,
and a constant $\bar C_q$ which in addition depends
in a possibly highly suboptimal way on $q\in[1,\infty)$.
Choosing $\underline{\gamma}:=\frac{1}{2}\wedge\frac{\delta}{4d}$
thus entails the suboptimal estimate
\begin{align}
\label{eq:suboptimalMomentBoundMinRadius}
\big\langle r_{*,T,\xi,B}^{(d-\frac{\delta}{2})2q}\big\rangle \leq
\bar C_q\sqrt{T}^{2dq}.
\end{align}

\textit{Conclusion: (From suboptimal moment bounds to stretched exponential moments)}
The merit of~\eqref{eq:suboptimalMomentBoundMinRadius} is that
it at least provides finiteness of arbitrarily high stochastic moments
of the minimal radius~$r_{*,T,\xi,B}$. We now leverage on that information
in a buckling argument. We observe from the previous argument that
suboptimality was only a result of using~\eqref{eq:annealedLinFunctionalsIndStep}
with a non-optimized $\tau\in (0,1)$ and using the suboptimal
small-scale energy estimate~\eqref{eq:annealedSmallScaleEnergyEstimateSuboptimal}
in order to transition from~\eqref{eq:case2SubotimalBoundProbAux}
to~\eqref{eq:case2SubotimalBoundProb}. 

If we instead apply~\eqref{eq:annealedLinFunctionalsIndStep} with $\tau\in(0,1)$ yet to be optimized,
and then feed in the annealed small-scale energy estimate from Lemma~\ref{lem:annealedSmallScaleEnergyEstimate} 
in form of~\eqref{eq:annealedSmallScaleEnergyEstimateMinRadius}
(with $q$ replaced by $\frac{2q}{(1-\tau)^2}$), we may update~\eqref{eq:case2SubotimalBoundProb} to
\begin{equation}
\label{eq:case2OptimalBoundProb}
\begin{aligned}
&\Prob\big[\big\{r_{*,T,\xi,B}=2^k\big\}
\cap\{\eqref{eq:case2}\text{ holds true}\}
\cap\{\eqref{eq:case1}\text{ holds not true}\}\big]
\\&
\leq C^{2q}q^{2Cq} 2^{-k(1-\underline{\gamma})2dq} 
\Big\langle r_{*,T,\xi,B}^{\frac{d-\delta+2\underline{\gamma}}{(1-\tau)^2}2q}\Big\rangle^{(1-\tau)^2},
\end{aligned}
\end{equation}
with a constant $C$ only depending on the
admissible data $(d,\lambda,\Lambda,\nu,\rho,M,L,\tau)$.
This in turn provides the following improvement of~\eqref{eq:suboptimalMomentBoundMinRadiusAux}
\begin{align*}
\big\langle r_{*,T,\xi,B}^{(d-\frac{\delta}{2})2q}\big\rangle
\leq C^{2q}q^{2Cq} + C^{2q}q^{2Cq}\sum_{k=1}^\infty 2^{k(d-\frac{\delta}{2})2q}
2^{-k(1-\underline{\gamma})2dq} 
\Big\langle r_{*,T,\xi,B}^{\frac{d-\delta+2\underline{\gamma}}{(1-\tau)^2}2q}\Big\rangle^{(1-\tau)^2}.
\end{align*}
Choosing $\underline{\gamma}:=\frac{1}{2}\wedge\frac{\delta}{8d}$
and then $\tau\in (0,1)$ such that 
$\frac{d-\delta+2\underline{\gamma}}{(1-\tau)^2}=d-\frac{\delta}{2}$
we obtain
\begin{align*}
\big\langle r_{*,T,\xi,B}^{(d-\frac{\delta}{2})2q}\big\rangle
\leq C^{2q}q^{2Cq}\Big(1 + \big\langle r_{*,T,\xi,B}^{(d-\frac{\delta}{2})2q}\big\rangle^{(1-\tau)^2}\Big).
\end{align*}
Because of~\eqref{eq:suboptimalMomentBoundMinRadius}
and $r_{*,T,\xi,B}\geq 1$, this concludes the proof of 
Lemma~\ref{lem:momentBoundMinRadius}. \qed

\subsection{Proof of Lemma~\ref{lem:closingIndStep}
{\normalfont (Conclusion of the induction step)}}
Validity of the corrector bounds from~\eqref{eq:indHypoCorrectorBounds}
with $\phi^T_{\xi,B'}$ replaced by $\phi^T_{\xi,B}$
is an immediate consequence of~\eqref{eq:annealedSmallScaleEnergyEstimateMinRadius},
\eqref{eq:annealedLinFunctionalsIndStep}, and~\eqref{eq:momentBoundMinRadius}.
The small-scale annealed Schauder estimate from
induction hypothesis~\eqref{eq:indHypoAnnealedSchauder}
with $\phi^T_{\xi,B'}$ replaced by $\phi^T_{\xi,B}$
follows from combining~\eqref{eq:annealedSmallScaleEnergyEstimateMinRadius},
\eqref{eq:annealedSmallScaleSchauderIndStep}, and~\eqref{eq:momentBoundMinRadius}.

It remains to establish induction 
hypotheses~\eqref{eq:representationMalliavinDerivative}--\eqref{eq:indHypoSensitivityApprox}
with $\phi^T_{\xi,B'}$ replaced by~$\phi^T_{\xi,B}$. Starting point is now 
the following adaption of~\eqref{eq:funcDerivLinFunctionalsLinearizedHomCorrector}
\begin{align}
\nonumber
\int g\cdot\nabla\delta\phi^T_{\xi,B}
&= - \lim_{r\to\infty} \bigg\{\int \nabla\delta\phi^{T,r}_{\xi,B}\cdot a^{T,*}_\xi
\nabla\Big(\frac{1}{T}{-}\nabla\cdot a^{T,*}_\xi\nabla\Big)^{-1}
(\nabla\cdot g)
\\&~~~~~~~~~~~~~~~~ \nonumber
+ \int \delta\phi^{T,r}_{\xi,B}\,
\frac{1}{T}\Big(\frac{1}{T}{-}\nabla\cdot a^{T,*}_\xi\nabla\Big)^{-1}
(\nabla\cdot g) \bigg\}
\\& \label{eq:funcDerivLinFunctionalsLinearizedHomCorrectorIndStep}
= \lim_{r\to\infty} \bigg\{\int \mathds{1}_{B_r} R^{T,(1)}_{\xi,B}\delta\omega\cdot 
\nabla\Big(\frac{1}{T}{-}\nabla\cdot a^{T,*}_\xi\nabla\Big)^{-1}
(\nabla\cdot g)
\\&~~~~~~~ \nonumber
+ \int \mathds{1}_{B_r} R^{T,(2)}_{\xi,B}\nabla\delta\phi^T_\xi\cdot 
\nabla\Big(\frac{1}{T}{-}\nabla\cdot a^{T,*}_\xi\nabla\Big)^{-1}
(\nabla\cdot g)
\\&~~~~~~~ \nonumber
+ \int \mathds{1}_{B_r} R^{T,(3)}_{\xi,B}\delta\omega\cdot 
\nabla\Big(\frac{1}{T}{-}\nabla\cdot a^{T,*}_\xi\nabla\Big)^{-1}
(\nabla\cdot g)
\\&~~~~~~~ \nonumber
+ \int \mathds{1}_{B_r} R^{T,(4)}_{\xi,B}\nabla\delta\phi^T_\xi\cdot 
\nabla\Big(\frac{1}{T}{-}\nabla\cdot a^{T,*}_\xi\nabla\Big)^{-1}
(\nabla\cdot g)
\\&~~~~~~~ \nonumber
+ \sum_{\Pi}\sum_{\pi\in\Pi} \int \mathds{1}_{B_r} R^{T,(5),\pi}_{\xi,B}
\nabla\delta\phi^T_{\xi,B'_\pi}\cdot \nabla\Big(\frac{1}{T}{-}\nabla\cdot a^{T,*}_\xi\nabla\Big)^{-1}
(\nabla\cdot g)\bigg\},
\end{align}
with the five remainder terms being defined in~\eqref{eq:PDEperturbedLinearizedHomCorrector}.
Since the five right hand side terms of~\eqref{eq:funcDerivLinFunctionalsLinearizedHomCorrectorIndStep} 
only contain variations of linearized homogenization correctors
up to order $L{-}1$, the representation~\eqref{eq:representationMalliavinDerivative}
with $\phi^T_{\xi,B'}$ replaced by $\phi^T_{\xi,B}$ follows from
the induction hypotheses~\eqref{eq:representationMalliavinDerivative}
and~\eqref{eq:indHypoSensitivityApprox} thanks to~\eqref{eq:auxReg100}.
In addition, we deduce the following representation of the random
field $G^T_{\xi,B}=G^T_{\xi,B}\big[g\big]$ in form of
\begin{equation}
\label{eq:indStepMalliavinDerivative}
\begin{aligned}
G^T_{\xi,B}\big[g\big] &=
\Big(R^{T,(1)}_{\xi,B}{+}R^{T,(3)}_{\xi,B}\Big)^*
\nabla\Big(\frac{1}{T}{-}\nabla\cdot a^{T,*}_\xi\nabla\Big)^{-1}\big(\nabla\cdot g\big)
\\&~~~
+ G^T_{\xi}\bigg[\Big(R^{T,(2)}_{\xi,B}{+}R^{T,(4)}_{\xi,B}\Big)^*
\nabla\Big(\frac{1}{T}{-}\nabla\cdot a^{T,*}_\xi\nabla\Big)^{-1}\big(\nabla\cdot g\big)\bigg]
\\&~~~
+ \sum_\Pi\sum_{\pi\in\Pi} G^T_{\xi,B'_\pi}\bigg[\Big(R^{T,(5),\pi}_{\xi,B}\Big)^*
\nabla\Big(\frac{1}{T}{-}\nabla\cdot a^{T,*}_\xi\nabla\Big)^{-1}\big(\nabla\cdot g\big)\bigg].
\end{aligned}
\end{equation}
Now, the same argument leading to \eqref{eq:sensitivityEstimateLinearizedHomCorrector} shows
\begin{align}
\label{eq:estimateMalliavinIndStep}
\bigg\langle\bigg|\,\int\bigg(
\,\dashint_{B_1(x)}|G^T_{\xi,B}|\,\bigg)^2\bigg|^{q}\bigg\rangle^\frac{1}{q}
\leq I_1 + I_2 + I_3 + I_4 + I_5,
\end{align}
with the right hand side terms being identical to those of~\eqref{eq:sensitivityEstimateLinearizedHomCorrector};
except for the slight notational simplification as $f\equiv 0$. In order to post-process
the right hand side of~\eqref{eq:estimateMalliavinIndStep} we may in fact follow very closely
the arguments from \textit{Step 3} in the proof of Lemma~\ref{lem:annealedLinFunctionalsIndStep}.
So let us only mention the minor differences. 

For the contribution from~$I_1$, we simply choose $\tau:= 1-\kappa$ and avoid
the use of Jensen's inequality in the corresponding argument which yields
\begin{align*}
|I_1| \leq C^2q^{2C} \sup_{\langle F^{2q_*}\rangle=1}
\int\big\langle|Fg|^{2(q/\kappa)_*}\big\rangle^\frac{1}{(q/\kappa)_*}.
\end{align*}
Note that for an application of the annealed Calder\'on--Zygmund estimate~\eqref{eq:annealedCZMeyers}
we have to ensure in this argument that $|(q/\kappa)_*-1|$ is sufficiently small.  Or equivalently,
that $q$ is sufficiently large which this time also depends on the fixed $\kappa\in (0,1)$. Finally,
note that we can get rid of the energy term appearing on the right hand side of~\eqref{eq:I1}
since we already have in place the corrector bounds from induction hypothesis~\eqref{eq:indHypoCorrectorBounds}
with $\phi^T_{\xi,B'}$ replaced by $\phi^T_{\xi,B}$. As the last two remarks also apply
to all of the remaining right hand side terms in~\eqref{eq:estimateMalliavinIndStep}, we will
not mention them anymore from now on.

For the contribution from the second term $I_2$, the only change concerns taking 
$\tau:=1-\kappa$ in the argument for~\eqref{eq:I2} in order to deduce
\begin{align*}
|I_2| \leq C^2q^{2C} \sup_{\langle F^{2q_*}\rangle=1}
\int\big\langle|Fg|^{2(q/\kappa)_*}\big\rangle^\frac{1}{(q/\kappa)_*}.
\end{align*}
With respect to the term $I_3$, the corresponding argument is the one leading to~\eqref{eq:I3}.
To adapt it to our needs here, we choose $\kappa_1:=\frac{\kappa}{2}$ and $1-\tau = \frac{\kappa}{2}$
which then entails the estimate
\begin{align*}
|I_3| \leq C^2q^{2C} \sup_{\langle F^{2q_*}\rangle=1}
\int\big\langle|Fg|^{2(q/\kappa)_*}\big\rangle^\frac{1}{(q/\kappa)_*}.
\end{align*}
Last but not least---as $I_5$ can again be treated analogously---the adaption of the argument for $I_4$
leading to~\eqref{eq:I4} consists of taking $\kappa_2=\frac{\kappa}{2}$ and
applying H\"older's inequality with respect to the exponents 
$(\frac{q-\kappa_2}{q-\kappa},(\frac{q-\kappa_2}{q-\kappa})_*)$. Based on these modifications, we obtain
\begin{align*}
|I_4| \leq C^2q^{2C} \sup_{\langle F^{2q_*}\rangle=1}
\int\big\langle|Fg|^{2(q/\kappa)_*}\big\rangle^\frac{1}{(q/\kappa)_*}.
\end{align*}
In summary, the preliminary estimate~\eqref{eq:estimateMalliavinIndStep}
updates to the desired bound~\eqref{eq:indHypoSensitivityBound} with $\phi^T_{\xi,B'}$ replaced by $\phi^T_{\xi,B}$.

Finally, the validity of~\eqref{eq:indHypoSensitivityApprox} with 
$\phi^T_{\xi,B'}$ replaced by $\phi^T_{\xi,B}$ is a consequence of
the induction hypothesis~\eqref{eq:indHypoSensitivityApprox}
and the identity~\eqref{eq:indStepMalliavinDerivative} in the following way. 
First, we observe that for some $p'\in (2,p)$ it holds by means of
the (local) Meyers estimate for the operator $(\frac{1}{T}-\nabla\cdot a^{T,*}_\xi\nabla)$ that
\begin{align*}
\nabla\Big(\frac{1}{T}{-}\nabla\cdot a^{T,*}_\xi\nabla\Big)^{-1}(\nabla\cdot g_r)
\to \nabla\Big(\frac{1}{T}{-}\nabla\cdot a^{T,*}_\xi\nabla\Big)^{-1}(\nabla\cdot g) \in L^{p'}_{\mathrm{uloc}}(\Rd;\Rd).
\end{align*}
Together with the regularity estimates~\eqref{eq:regAuxSensitivity}
and the induction hypothesis~\eqref{eq:indHypoSensitivityApprox},
we deduce that the right hand side of~\eqref{eq:indStepMalliavinDerivative}
(with $g$ replaced by $g_r$) converges $\Prob$-almost surely in $L^1_{\mathrm{uloc}}(\Rd;\Rd[n])$
to some random field $G^T_{\xi,B}$ as $r\to\infty$. The validity of~\eqref{eq:indHypoSensitivityBound}
for $(G^T_{\xi,B},g)$ then follows based on the following two ingredients:
\textit{i)} we already have~\eqref{eq:indHypoSensitivityBound}
at our disposal with respect to the pair $(G^{T,r}_{\xi,B},g_r)$, and \textit{ii)}
we may apply Fatou's lemma. This, however, concludes the proof of 
Lemma~\ref{lem:closingIndStep}. \qed

\subsection{Proof of Lemma~\ref{lem:correctorEstimatesLinearizedFlux}
{\normalfont (Sensitivity estimate for the linearized flux)}} 
Consider some $\delta\omega\colon\Rd\to \Rd[n]$ which is compactly supported, smooth, 
and in addition satisfies $\|\delta\omega\|_{L^\infty}\leq 1$. It is immediate
from the definition~\eqref{eq:HigherOrderLinearizedFlux} of the linearized flux
and the computation~\eqref{eq:PDEperturbedLinearizedHomCorrector} concerning 
the variation $\delta\phi^T_{\xi,B}$ for the linearized homogenization corrector that
$\Prob$-almost surely
\begin{align}
 \nonumber
\delta q^T_{\xi,B}
&=  \partial_\omega\partial_\xi A(\omega,\xi+\nabla\phi^T_\xi)
\big[\delta\omega\odot\nabla \phi^T_{\xi,B}\big]
\\&~~~\nonumber
+  \partial_\xi^2 A(\omega,\xi+\nabla\phi^T_\xi)
\big[\nabla\delta\phi^T_\xi\odot\nabla \phi^T_{\xi,B}\big]
\\&~~~\nonumber
+  \sum_{\Pi} \partial_\omega\partial_\xi^{|\Pi|} A(\omega,\xi{+}\nabla \phi^T_\xi)
\Big[\delta\omega\odot\bigodot_{\pi\in\Pi}(\mathds{1}_{|\pi|=1}B'_\pi {+} \nabla \phi^T_{\xi,B'_\pi})\Big]
\\&~~~\nonumber
+  \sum_{\Pi} \partial_\xi^{1+|\Pi|} A(\omega,\xi{+}\nabla \phi^T_\xi)
\Big[\nabla\delta\phi_\xi\odot\bigodot_{\pi\in\Pi}(\mathds{1}_{|\pi|=1}B'_\pi {+} \nabla \phi^T_{\xi,B'_\pi})\Big]
\\&~~~\nonumber
+  \sum_{\Pi} \partial_\xi^{|\Pi|} A(\omega,\xi{+}\nabla \phi^T_\xi)
\Big[\sum_{\pi\in\Pi}\nabla\delta\phi^T_{\xi,B'_\pi}\odot
\bigodot_{\substack{{\pi'\in\Pi} \\ \pi'\neq \pi}}
(\mathds{1}_{|\pi'|=1}B'_{\pi'} {+} \nabla \phi^T_{\xi,B'_{\pi'}})\Big]
\\&~~~\nonumber
+  \partial_\xi A(\omega,\xi+\nabla\phi^T_\xi)\nabla\delta\phi^T_{\xi,B}
\\& \label{eq:PerturbedLinearizedFlux}
=:  R^{T,(1)}_{\xi,B}\delta\omega
+  R^{T,(2)}_{\xi,B}\nabla\delta\phi^T_\xi
+  R^{T,(3)}_{\xi,B}\delta\omega
+  R^{T,(4)}_{\xi,B}\nabla\delta\phi^T_\xi
\\&~~~~\nonumber
+  \sum_{\Pi}\sum_{\pi\in\Pi}
R^{T,(5),\pi}_{\xi,B}\nabla\delta\phi^T_{\xi,B'_\pi}
+  \partial_\xi A(\omega,\xi+\nabla\phi^T_\xi)\nabla\delta\phi^T_{\xi,B}.
\end{align}
In particular, denoting again by $a_\xi^{T,*}$ the transpose
of the uniformly elliptic and bounded coefficient field 
$a_\xi^T:=\partial_\xi A(\omega,\xi+\nabla\phi^T_\xi)$ we get
\begin{align}
\label{eq:auxMalliavinLinFlux}
\int g\cdot\delta q^T_{\xi,B} 
&= \int a_{\xi}^{T,*}g\cdot\nabla\delta\phi^T_{\xi,B}
+ \int \big(R^{T,(1)}_{\xi,B}{+}R^{T,(3)}_{\xi,B}\big)^*g\cdot\delta\omega
\\&~~~\nonumber
+ \int \big(R^{T,(2)}_{\xi,B}{+}R^{T,(4)}_{\xi,B}\big)^*g\cdot\nabla\delta\phi^T_\xi
+ \sum_{\Pi}\sum_{\pi\in\Pi} \int \big(R^{T,(5),\pi}_{\xi,B}\big)^*g\cdot\nabla\delta\phi^T_{\xi,B'_\pi}.
\end{align}
Hence, we obtain a representation of the asserted form~\eqref{eq:representationMalliavinLinearizedFlux}
because of~\eqref{eq:representationMalliavinDerivative}, which as a result of Lemma~\ref{lem:closingIndStep}
is even available for linearized homogenization correctors up to order $L$. By the same argument,
we may then derive~\eqref{eq:sensitivityBoundLinearizedFlux} from~\eqref{eq:indHypoSensitivityBound}
(applied with $\kappa=1$). The convergence assertion~\eqref{eq:indHypoSensitivityApproxFlux}
is in light of~\eqref{eq:auxMalliavinLinFlux} an immediate consequence of~\eqref{eq:indHypoSensitivityApprox}
(which again is already available up to linearization order $L$ thanks to Lemma~\ref{lem:closingIndStep}),
whereas the validity of~ \eqref{eq:sensitivityBoundLinearizedFlux} for the limit pair $(Q^T_{\xi,B},g)$
follows from~\eqref{eq:sensitivityBoundLinearizedFlux} applied to the already admissible pair $(Q^{T,r}_{\xi,B},g_r)$
and Fatou's lemma. This in turn concludes the proof of Lemma~\ref{lem:correctorEstimatesLinearizedFlux}. \qed

\subsection{Proof of Theorem~\ref{theo:correctorBoundsMassiveApprox}
{\normalfont (Estimates for massive correctors)}} 
We split the proof into four parts.

\textit{Step 1: (Proof of the estimates~\eqref{eq:linearFunctionalCorrectorGradientBoundMassiveApprox}
and~\eqref{eq:linearFunctionalCorrectorGradientBoundMassiveApprox2})}
In case of the linearized homogenization corrector $\phi^T_{\xi,B}$ this already follows from
Lemma~\ref{lem:closingIndStep} in form of~\eqref{eq:indHypoCorrectorBounds}.
For the linearized flux correctors $(\sigma^T_{\xi,B},\psi^T_{\xi,B})$, we start 
by computing the functional derivatives of
\begin{align}
\label{def:LinFunctionalLinearizedFluxCorrector}
F_\sigma := \int g_\sigma^{kl}\cdot\nabla\sigma^T_{\xi,B,kl},
\quad F_\psi := \int g_\psi^k\cdot\frac{\nabla\psi^T_{\xi,B,k}}{\sqrt{T}}.
\end{align}
To this end, let $\delta\omega\colon\Rd\to \Rd[n]$ be compactly supported and smooth such that
$\|\delta\omega\|_{L^\infty}\leq 1$. Since $(\nabla\sigma^T_{\xi,B,kl},\nabla\psi^T_{\xi,B,k})\in
L^2_{\mathrm{uloc}}(\Rd;\Rd{\times}\Rd)$, we may assume
for the proof of~~\eqref{eq:linearFunctionalCorrectorGradientBoundMassiveApprox}
and~\eqref{eq:linearFunctionalCorrectorGradientBoundMassiveApprox2} without loss 
of generality through an approximation argument that
\begin{align}
\label{eq:assumpRegTestVectors2}
(g_\sigma^{kl},g_\psi^k)\in C^\infty_{\mathrm{cpt}}(\Rd;\Rd{\times}\Rd).
\end{align}
Differentiating the defining 
equation~\eqref{eq:PDEhigherOrderLinearizedFluxCorrectorLocalized} for the linearized flux
corrector $\sigma^T_{\xi,B,kl}$ with respect to the parameter field
in the direction of $\delta\omega$ yields $\Prob$-almost surely
\begin{equation}
\label{eq:PDEperturbedLinearizedFluxCorrector}
\begin{aligned}
\frac{1}{T}\delta\sigma_{\xi,B,kl}^T - \Delta\delta\sigma_{\xi,B,kl}^T 
&= (e_l\otimes e_k - e_k\otimes e_l) : \nabla \delta q_{\xi,B}^T
\\&
= - \nabla\cdot \big((e_l\otimes e_k - e_k\otimes e_l)\delta q^T_{\xi,B}\big).
\end{aligned}
\end{equation}
Moreover, differentiating~\eqref{eq:PDEhigherOrderLinearizedHelmholtzCorrectorLocalized}
for the linearized flux corrector~$\psi^T_{\xi,B,k}$ entails
\begin{align}
\label{eq:PDEperturbedLinearizedHelmholtzorrector}
\frac{1}{T}\delta\psi^T_{\xi,B,k} - \Delta\delta\psi^T_{\xi,B,k} 
= e_k\cdot\delta q^T_{\xi,B} - e_k\cdot\nabla\delta\phi^T_{\xi,B}.
\end{align}
For $r\geq 1$, denote by $(\frac{\delta\sigma^{T,r}_{\xi,B,kl}}{\sqrt{T}},\nabla\delta\sigma^{T,r}_{\xi,B,kl})
\in L^2(\Rd;\Rd[]{\times}\Rd)$ the unique Lax--Milgram solution of
\begin{equation}
\label{eq:PDEperturbedLinearizedFluxCorrectorApprox}
\begin{aligned}
\frac{1}{T}\delta\sigma_{\xi,B,kl}^{T,r} - \Delta\delta\sigma_{\xi,B,kl}^{T,r} 
&= (e_l\otimes e_k - e_k\otimes e_l) : \nabla \mathds{1}_{B_r}\delta q_{\xi,B}^T
\\&
= - \nabla\cdot \mathds{1}_{B_r}\big((e_l\otimes e_k - e_k\otimes e_l)\delta q^T_{\xi,B}\big),
\end{aligned}
\end{equation}
as well as by $(\frac{\delta\psi^{T,r}_{\xi,B,kl}}{\sqrt{T}},
\nabla\delta\psi^{T,r}_{\xi,B,kl}) \in L^2(\Rd;\Rd[]{\times}\Rd)$
the unique Lax--Milgram solution of
\begin{align}
\label{eq:PDEperturbedLinearizedHelmholtzorrectorApprox}
\frac{1}{T}\delta\psi^{T,r}_{\xi,B,k} - \Delta\delta\psi^{T,r}_{\xi,B,k} 
= \mathds{1}_{B_r}e_k\cdot\delta q^T_{\xi,B} - \mathds{1}_{B_r}e_k\cdot\nabla\delta\phi^T_{\xi,B}.
\end{align}
Note that $\Prob$-almost surely
\begin{align}
\label{eq:convAuxSensitivityFlux}\hspace*{-1ex}
\Big(\frac{\delta\sigma^{T,r}_{\xi,B,kl}}{\sqrt{T}},\nabla\delta\sigma^{T,r}_{\xi,B,kl}\Big)
\to \Big(\frac{\delta\sigma^{T}_{\xi,B,kl}}{\sqrt{T}},\nabla\delta\sigma^{T}_{\xi,B,kl}\Big) 
\text{ as } r\to\infty \text{ in } L^2_{\mathrm{uloc}}(\Rd;\Rd[]{\times}\Rd)
\end{align}
as well as $\Prob$-almost surely
\begin{align}
\label{eq:convAuxSensitivityFlux2}
\Big(\frac{\delta\psi^{T,r}_{\xi,B,k}}{\sqrt{T}},\nabla\delta\psi^{T,r}_{\xi,B,k}\Big)
\to \Big(\frac{\delta\psi^{T}_{\xi,B,k}}{\sqrt{T}},\nabla\delta\psi^{T}_{\xi,B,k}\Big) 
\text{ as } r\to\infty \text{ in } L^2_{\mathrm{uloc}}(\Rd;\Rd[]{\times}\Rd)
\end{align}
as a consequence of applying the weighted energy estimate~\eqref{eq:expLocalization} to the 
difference of the equations~\eqref{eq:PDEperturbedLinearizedFluxCorrector}
and~\eqref{eq:PDEperturbedLinearizedFluxCorrectorApprox}, 
respectively~\eqref{eq:PDEperturbedLinearizedHelmholtzorrector}
and~\eqref{eq:PDEperturbedLinearizedHelmholtzorrectorApprox}. 
We thus deduce from~\eqref{eq:PDEperturbedLinearizedFluxCorrectorApprox},
\eqref{eq:PDEperturbedLinearizedHelmholtzorrectorApprox}, \eqref{eq:convAuxSensitivityFlux}
and~\eqref{eq:convAuxSensitivityFlux2} that  
\begin{equation}
\label{eq:funcDerivLinFunctionalsLinearizedFluxCorrector}
\begin{aligned}
\delta F_\sigma 
&= \lim_{r\to\infty} \int g_\sigma^{kl}\cdot\nabla\delta\sigma^{T,r}_{\xi,B,kl}
\\&
= - \lim_{r\to\infty}\bigg\{\int \nabla\delta\sigma^T_{\xi,B,kl}\cdot
\nabla\Big(\frac{1}{T}{-}\Delta\Big)^{-1}\big(\nabla\cdot g_\sigma^{kl}\big)
\\&~~~~~~~~~~~~~~~~
+ \int \delta\sigma^T_{\xi,B,kl}
\frac{1}{T}\Big(\frac{1}{T}{-}\Delta\Big)^{-1}\big(\nabla\cdot g_\sigma^{kl}\big)\bigg\}
\\&
= \lim_{r\to\infty} \int \mathds{1}_{B_r}(e_l\otimes e_k - e_k\otimes e_l)
\nabla\Big(\frac{1}{T}{-}\Delta\Big)^{-1}\big(\nabla\cdot g_\sigma^{kl}\big)
\cdot \delta q^T_{\xi,B}, 
\end{aligned}
\end{equation}
as well as
\begin{equation}
\label{eq:funcDerivLinFunctionalsLinearizedHelmholtzCorrector}
\begin{aligned}
\delta F_\psi &= \lim_{r\to\infty}\int g_\psi^{k}\cdot\frac{\nabla\delta\psi^{T,r}_{\xi,B,k}}{\sqrt{T}}
\\
&= - \lim_{r\to\infty}\bigg\{\int \mathds{1}_{B_r}\Big(\frac{1}{T}{-}\Delta\Big)^{-1}
\Big(\nabla\cdot\frac{g_\psi^{k}}{\sqrt{T}}\Big)e_k\cdot \delta q^T_{\xi,B}
\\&~~~~~~~~~~~~~~~~
- \int \mathds{1}_{B_r}\Big(\frac{1}{T}{-}\Delta\Big)^{-1}\Big(\nabla\cdot\frac{g_\psi^{k}}{\sqrt{T}}\Big)
e_k\cdot \nabla\delta \phi^T_{\xi,B}\bigg\}.
\end{aligned}
\end{equation}
Note that thanks to the approximation argument, we may indeed use $\delta\sigma^{T,r}_{\xi,B,kl}$
resp.\ $\delta\psi^{T,r}_{\xi,B,k}$ as test functions in the weak formulation of the equations satisfied by
$\big(\frac{1}{T}{-}\Delta\big)^{-1}\big(\nabla\cdot g^{kl}_\sigma\big)$
resp.\ $\big(\frac{1}{T}{-}\Delta\big)^{-1}\big(\nabla\cdot g^{k}_\psi\big)$, and vice versa. 
Moreover, by the Meyers estimate for the operator 
$(\frac{1}{T}-\Delta)$ in combination with the assumption~\eqref{eq:assumpRegTestVectors2}, 
we obtain that $\Prob$-almost surely
\begin{equation}
\label{eq:auxReg200}
\begin{aligned}
(e_l\otimes e_k - e_k\otimes e_l)
\nabla\Big(\frac{1}{T}{-}\Delta\Big)^{-1}\big(\nabla\cdot g_\sigma^{kl}\big)
 &\in L^{p'}(\Rd;\Rd),
\\
\Big(\frac{1}{T}{-}\Delta\Big)^{-1}
\Big(\nabla\cdot\frac{g_\psi^{k}}{\sqrt{T}}\Big)e_k
 &\in L^{p'}(\Rd;\Rd)
\end{aligned}
\end{equation}
for some suitable Meyers exponents $p' > 2$.
Hence, applying the spectral gap inequality in form of~\eqref{eq:spectralGapHigherMoments}
with respect to the centered random variable $F_\sigma$ yields because 
of~\eqref{eq:funcDerivLinFunctionalsLinearizedFluxCorrector}, \eqref{eq:auxReg200},
\eqref{eq:sensitivityBoundLinearizedFlux}, \eqref{eq:indHypoSensitivityApproxFlux}, and a simple energy estimate
\begin{align*}
\big\langle\big|F_\sigma\big|^{2q}\big\rangle^\frac{1}{q}
&\leq C^2q^{2C}\sup_{\langle F^{2q_*}\rangle=1}
\int\Big\langle\Big|F(e_l\otimes e_k - e_k\otimes e_l)
\nabla\Big(\frac{1}{T}{-}\Delta\Big)^{-1}\big(\nabla\cdot g_\sigma^{kl}\big)
\Big|^{2q_*}\Big\rangle^\frac{1}{q_*}
\\&
\leq C^2q^{2C}\sup_{kl}\int \Big|\nabla\Big(\frac{1}{T}{-}\Delta\Big)^{-1}
\big(\nabla\cdot g_\sigma^{kl}\big)\Big|^2
\\&
\leq  C^2q^{2C}\int \big|g_\sigma\big|^2.
\end{align*}
Moreover, applying the spectral gap inequality~\eqref{eq:spectralGapHigherMoments}
with respect to the centered random variable $F_\psi$ entails the estimate
\begin{align*}
\big\langle\big|F_\psi\big|^{2q}\big\rangle^\frac{1}{q}
&\leq C^2q^{2C}\sup_{\langle F^{2q_*}\rangle=1}
\int\Big\langle\Big|F\Big(\frac{1}{T}{-}\Delta\Big)^{-1}
\Big(\nabla\cdot \frac{g_\psi^{k}}{\sqrt{T}}\Big)e_k
\Big|^{2q_*}\Big\rangle^\frac{1}{q_*}
\\&
\leq C^2q^{2C}\sup_{k} \int \frac{1}{T}\Big|\Big(\frac{1}{T}{-}\Delta\Big)^{-1}
\big(\nabla\cdot g_\psi^{k}\big)\Big|^2
\\&
\leq  C^2q^{2C}\int \big|g_\psi\big|^2.
\end{align*}
For the previous display, we relied on a combination of~\eqref{eq:funcDerivLinFunctionalsLinearizedHelmholtzCorrector}, 
\eqref{eq:auxReg200}, \eqref{eq:sensitivityBoundLinearizedFlux}, 
\eqref{eq:indHypoSensitivityApproxFlux}, \eqref{eq:indHypoSensitivityBound}
applied to $\phi^T_{\xi,B}$ with $\kappa=1$ (which is admissible thanks to Lemma~\ref{lem:closingIndStep}) 
and again a simple energy estimate. This concludes the 
proof of~\eqref{eq:linearFunctionalCorrectorGradientBoundMassiveApprox}.
The proof of~\eqref{eq:linearFunctionalCorrectorGradientBoundMassiveApprox2}
for the linearized flux correctors $(\sigma^T_{\xi,B},\psi^T_{\xi,B})$ follows along similar lines.

\textit{Step 2: (Proof of the estimate~\eqref{eq:correctorGradientBoundMassiveApprox})}
In case of the linearized homogenization corrector $\phi^T_{\xi,B}$ this again already follows from
Lemma~\ref{lem:closingIndStep} in form of~\eqref{eq:indHypoCorrectorBounds}.
Hence, we only have to discuss the case of the linearized flux correctors.

Applying the Caccioppoli estimate~\eqref{eq:Caccioppoli}
to equation~\eqref{eq:PDEhigherOrderLinearizedFluxCorrectorLocalized} for the linearized flux
corrector~$\sigma^T_{\xi,B,kl}$
entails the estimate
\begin{align*}
&\Big\|\Big(\frac{\sigma^T_{\xi,B,kl}}{\sqrt{T}},
\nabla \sigma^T_{\xi,B,kl}\Big)\Big\|_{L^2(B_1)}^2
\\& 
\lesssim_{d,\lambda,\Lambda}  
\inf_{b\in\Rd[]} \Big\{\frac{1}{R^2}\big\|\sigma^T_{\xi,B,kl}-b\big\|_{L^2(B_{2})}^2 
+ \frac{1}{T}|b|^2 \Big\} + \big\|q^T_{\xi,B}\|_{L^2(B_2)}^2.
\end{align*}
By the same argument which starts from the right hand side of~\eqref{eq:case2Aux10}
and produces the right hand side of~\eqref{eq:case2Aux1}, we obtain
(with $R$ replaced by $2$, $R'=\theta 2$, $\theta=\theta(d,\lambda,\Lambda)$
yet to be determined)  
\begin{align*}
\Big\|\Big(\frac{\sigma^T_{\xi,B,kl}}{\sqrt{T}},
\nabla \sigma^T_{\xi,B,kl}\Big)\Big\|_{L^2(B_1)}^2
& \lesssim_{d,\lambda,\Lambda} \theta^2 \dashint_{B_{4}} 
\big|\nabla\sigma^T_{\xi,B,kl}\big|^2  + \big\|q^T_{\xi,B}\|_{L^2(B_2)}^2
\\&~~~
+ \dashint_{B_{4}} \bigg|\,\dashint_{B_{\theta 2}(x)}\nabla\sigma^T_{\xi,B,kl}\bigg|^2
+ \dashint_{B_{4}} \bigg|\,\dashint_{B_{\theta 2}(x)}\frac{1}{\sqrt{T}}\sigma^T_{\xi,B,kl}\bigg|^2.
\end{align*}
Hence, taking stochastic moments and exploiting stationarity yields
\begin{align}
\nonumber
&\Big\langle\Big\|\Big(\frac{\sigma^T_{\xi,B,kl}}{\sqrt{T}},
\nabla \sigma^T_{\xi,B,kl}\Big)\Big\|_{L^2(B_1)}^{2q}\Big\rangle^\frac{1}{q}
\\&\label{eq:aux1ProofTheoremMassiveCorr}
\lesssim_{d,\lambda,\Lambda} \theta^2\Big\langle\Big\|\Big(\frac{\sigma^T_{\xi,B,kl}}{\sqrt{T}},
\nabla \sigma^T_{\xi,B,kl}\Big)\Big\|_{L^2(B_1)}^{2q}\Big\rangle^\frac{1}{q}
+ \big\langle\big\|q^T_{\xi,B}\|_{L^2(B_1)}^{2q}\big\rangle^\frac{1}{q}
\\&~~~\nonumber
+ \bigg\langle\bigg|\,\dashint_{B_{\theta 2}}\nabla\sigma^T_{\xi,B,kl}\bigg|^{2q}\bigg\rangle^\frac{1}{q}
+ \bigg\langle\bigg|\,\dashint_{B_{\theta 2}}\frac{1}{\sqrt{T}}\sigma^T_{\xi,B,kl}\bigg|^{2q}\bigg\rangle^\frac{1}{q}.
\end{align}
In principle, we would like to absorb now the first right hand side term of the previous
display into the left hand side by choosing $\theta$ appropriately. However, we first have
to verify finiteness of $\langle \| (\frac{\sigma^T_{\xi,B,kl}}{\sqrt{T}},
\nabla \sigma^T_{\xi,B,kl} ) \|_{L^2(B_1)}^{2q} \rangle^\frac{1}{q}$. This is done
by appealing to the weighted energy estimate~\eqref{eq:expLocalization}, which
with respect to equation~\eqref{eq:PDEhigherOrderLinearizedFluxCorrectorLocalized} for the linearized flux
corrector~$\sigma^T_{\xi,B,kl}$ entails
\begin{align*}
&\Big\langle\Big\|\Big(\frac{\sigma^T_{\xi,B,kl}}{\sqrt{T}},
\nabla\sigma^T_{\xi,B,kl}\Big)\Big\|^{2q}_{L^2(B_1)}\Big\rangle^\frac{1}{q}
\lesssim \sqrt{T}^d \int \ell_{\gamma,T}\big\langle
\big|q^T_{\xi,B}\big|^{2q}\big\rangle^\frac{1}{q}.
\end{align*}
Plugging in the definition~\eqref{eq:HigherOrderLinearizedFlux} for the linearized flux
and making use of~(A2)$_{L}$ in Assumption~\ref{assumption:operators} 
then gives the following update of the previous display
\begin{align*}
&\Big\langle\Big\|\Big(\frac{\sigma^T_{\xi,B,kl}}{\sqrt{T}},
\nabla\sigma^T_{\xi,B,kl}\Big)\Big\|^{2q}_{L^2(B_1)}\Big\rangle^\frac{1}{q}
\\&
\lesssim \sqrt{T}^d \int \ell_{\gamma,T}\big\langle
\big|\mathds{1}_{L=1}B + \nabla \phi^T_{\xi,B}\big|^{2q}\big\rangle^\frac{1}{q}
\\&~~~
+ \sqrt{T}^d \sum_{\substack{\Pi\in\mathrm{Par}\{1,\ldots,L\} \\ \Pi\neq\{\{1,\ldots,L\}\}}}
\int \ell_{\gamma,T}\prod_{\pi\in\Pi}\Big\langle\Big|\mathds{1}_{|\pi|=1}B'_\pi 
+ \nabla \phi^T_{\xi,B'_\pi}\Big|^{2q|\Pi|}\Big\rangle^\frac{1}{q|\Pi|}.
\end{align*}
Hence, it now follows from stationarity of the linearized homogenization correctors,
smuggling in spatial averages over the unit ball based on~\eqref{eq:indHypoAnnealedSchauder}
(which is available also for $\phi^T_{\xi,B}$ thanks to Lemma~\ref{lem:closingIndStep}),
and finally the corrector bounds from~\eqref{eq:linearFunctionalCorrectorGradientBoundMassiveApprox} 
that $$\Big\langle\Big\|\Big(\frac{\sigma^T_{\xi,B,kl}}{\sqrt{T}},
\nabla \sigma^T_{\xi,B,kl}\Big)\Big\|_{L^2(B_1)}^{2q}\Big\rangle^\frac{1}{q} \lesssim \sqrt{T}^d.$$

We may now run an absorption argument for the first right hand side term in~\eqref{eq:aux1ProofTheoremMassiveCorr},
and then combine this with a bound for $\big\langle\big\|q^T_{\xi,B}\|_{L^2(B_1)}^{2q}\big\rangle^\frac{1}{q}$
(by plugging in~\eqref{eq:HigherOrderLinearizedFlux} and 
using again the corrector bounds from~\eqref{eq:linearFunctionalCorrectorGradientBoundMassiveApprox} 
similar to the preceding discussion) as well as the already established
estimates~\eqref{eq:linearFunctionalCorrectorGradientBoundMassiveApprox}
and~\eqref{eq:linearFunctionalCorrectorGradientBoundMassiveApprox2} to infer
\begin{align*}
\Big\langle\Big\|\Big(\frac{\sigma^T_{\xi,B,kl}}{\sqrt{T}},
\nabla \sigma^T_{\xi,B,kl}\Big)\Big\|_{L^2(B_1)}^{2q}\Big\rangle^\frac{1}{q}
\leq C^2q^{2C}.
\end{align*}
As the argument for $\psi^T_{\xi,B}$ proceeds along the same lines, 
this time of course based on the defining 
equation~\eqref{eq:PDEhigherOrderLinearizedHelmholtzCorrectorLocalized}
in form of
\begin{align*}
\frac{1}{T}\frac{\psi^T_{\xi,B,k}}{\sqrt{T}} 
- \Delta\frac{\psi^T_{\xi,B,k}}{\sqrt{T}} 
= \frac{1}{T} e_k\cdot \sqrt{T}q^T_{\xi,B}
- \frac{1}{T} e_k\cdot \sqrt{T}\nabla \phi^T_{\xi,B},
\end{align*}
we move on to the next step of the proof.

\textit{Step 3: (Proof of the estimate~\eqref{eq:correctorSchauderMassiveApprox})}
This is a direct consequence of the corrector bounds~\eqref{eq:correctorGradientBoundMassiveApprox},
the definition~\eqref{eq:HigherOrderLinearizedFlux} of the linearized flux,
the annealed H\"older regularity of the linearized coefficient~\eqref{eq:annealedHoelderRegLinCoefficient},
and the local Schauder estimate~\eqref{eq:localSchauder} applied to
the localized corrector equations~\eqref{eq:PDEhigherOrderLinearizedCorrectorLocalized},
\eqref{eq:PDEhigherOrderLinearizedFluxCorrectorLocalized}, 
and~\eqref{eq:PDEhigherOrderLinearizedHelmholtzCorrectorLocalized}.

\textit{Step 4: (Proof of the estimate~\eqref{eq:correctorGrowthoundMassiveApprox})}
We may compute
\begin{align*}
\dashint_{B_1} \phi^T_{\xi,B} &= 
\int\phi^T_{\xi,B}\frac{1}{T}\Big(\frac{1}{T}-\Delta\Big)^{-1}
\Big(\frac{\mathds{1}_{B_1}}{|B_1|}\Big)
+\int\nabla\phi^T_{\xi,B}\cdot\nabla
\Big(\frac{1}{T}-\Delta\Big)^{-1}
\Big(\frac{\mathds{1}_{B_1}}{|B_1|}\Big).
\end{align*}
As a consequence of~\eqref{eq:linearFunctionalCorrectorGradientBoundMassiveApprox} we obtain
\begin{align*}
\bigg\langle\bigg|\,\dashint_{B_1} \phi^T_{\xi,B}\bigg|^{2q}\bigg\rangle^\frac{1}{q}
&\leq C^2q^{2C} \int \Big|\Big(\frac{1}{T}-\Delta\Big)^{-1}
\Big(\frac{\mathds{1}_{B_1}}{|B_1|}\Big)\Big|^2
\leq C^2q^{2C}\mu_*^2(\sqrt{T}).
\end{align*}
The bound for the linearized flux correctors
$(\sigma^T_{\xi,B},\psi^T_{\xi,B})$ follows analogously. 

\textit{Step 5: (Proof of equation~\eqref{eq:PDEhigherOrderLinearizedHelmholtzDecompLocalized})}
By the sublinear growth of the flux corrector, it suffices to verify
\begin{align}
\label{eq:indStepFluxCorrectorEquation}
\Big(\frac{1}{T}-\Delta\Big)\Big(\nabla\cdot\sigma^T_{\xi,B}\Big)
= \Big(\frac{1}{T}-\Delta\Big)\Big(q^T_{\xi,B} - \langle q^T_{\xi,B} \rangle
+ \frac{1}{T}\psi^T_{\xi,B}\Big).
\end{align}
We obtain from~\eqref{eq:PDEhigherOrderLinearizedCorrectorLocalized},
\eqref{eq:PDEhigherOrderLinearizedFluxCorrectorLocalized}
and~\eqref{eq:HigherOrderLinearizedFlux} that
\begin{align*}
\Big(\frac{1}{T}-\Delta\Big)\Big(\nabla\cdot\sigma^T_{\xi,B}\Big)
= \partial_l\partial_k \big(q^T_{\xi,B}\big)_k
- \partial_k\partial_k \big(q^T_{\xi,B}\big)_l
= \frac{1}{T}\partial_l\phi^T_{\xi,B} - \Delta q_l.
\end{align*}
Hence, plugging in~\eqref{eq:PDEhigherOrderLinearizedHelmholtzCorrectorLocalized}
yields the asserted identity~\eqref{eq:indStepFluxCorrectorEquation}.
This in turn concludes the proof of Theorem~\ref{theo:correctorBoundsMassiveApprox}. \qed

\subsection{Proof of Lemma~\ref{lem:differencesLinearizedCorrectors}
{\normalfont (Estimates for differences of linearized correctors)}}
The proof 
of~\eqref{eq:correctorBoundForDifferencesMassiveApprox}--\eqref{eq:LinearFunctionalEstimatesDifferencesMassiveApprox} 
proceeds via an induction over the linearization order.

\textit{Step 1: (Induction hypotheses)} Let $L\in\N$, $T\in [1,\infty)$ and $M>0$ be fixed. 
Let the requirements and notation of (A1), (A2)$_L$, (A3)$_L$ and (A4)$_L$ of
Assumption~\ref{assumption:operators}, (P1) and (P2) of
Assumption~\ref{assumption:ensembleParameterFields}, and (R) of 
Assumption~\ref{assumption:smallScaleReg} be in place.

For any $0\leq l\leq L{-}1$, any $|\xi|\leq M$, any $|h|\leq 1$, and any
collection of unit vectors $v_1',\ldots,v_l'\in\Rd$ the difference
$\phi^T_{\xi+he,B'}{-}\phi^T_{\xi,B'}$ of linearized homogenization commutators
in direction $B':=v_1'\odot\cdots\odot v_l'$ is assumed to satisfy---under the above conditions---the 
following list of estimates (if $l=0$---and thus $B'$ being an empty symmetric tensor product---$\phi^T_{\xi,B'}$ is 
understood to denote the localized homogenization corrector $\phi^T_\xi$ of the nonlinear problem
with a massive term):
\begin{itemize}[leftmargin=0.4cm]
\item For any $\beta\in (0,1)$, 
there exists a constant $C=C(d,\lambda,\Lambda,\nu,\rho,\eta,M,L,\beta)$ such that for all $q\in [1,\infty)$, 
and all compactly supported and square-integrable $f,g$ we have \emph{corrector estimates for differences}
\begin{equation}
\label{eq:indHypoCorrectorBoundsDiff}
\tag{Hdiff1}
\begin{aligned}
\nonumber
\bigg\langle\bigg|\int g\cdot\big(\nabla\phi^T_{\xi+he,B'}{-}\nabla\phi^T_{\xi,B'}\big)
\bigg|^{2q}\bigg\rangle^\frac{1}{q}
&\leq C^2q^{2C}|h|^{2(1-\beta)}\int \big|g\big|^2,
\\ 
\nonumber
\bigg\langle\bigg|\int \frac{1}{T}f \big(\phi^T_{\xi+he,B'}{-}\phi^T_{\xi,B'}\big)\bigg|^{2q}\bigg\rangle^\frac{1}{q}
&\leq C^2q^{2C}|h|^{2(1-\beta)}\int \Big|\frac{f}{\sqrt{T}}\Big|^2,
\\
\Big\langle\Big\|\Big(\frac{\phi^T_{\xi+he,B'}{-}\phi^T_{\xi,B'}}{\sqrt{T}},
\nabla\phi^T_{\xi+he,B'}{-}\nabla\phi^T_{\xi,B'}\Big)\Big\|^{2q}_{L^2(B_1)}\Big\rangle^\frac{1}{q}
&\leq C^2q^{2C}|h|^{2(1-\beta)}.
\end{aligned}
\end{equation}
\item Fix $p\in (2,\infty)$, and let $g$ be a compactly supported and $p$-integrable random field. 
						Then there exists a random field
						$G_{\xi,B',h,e}^T\in L^1_{\mathrm{uloc}}(\Rd;\Rd[n])$ being related to $g$ via $\phi^T_{\xi+he,B'}{-}\phi^T_{\xi,B'}$ 
						in the sense that, $\Prob$-almost surely, it holds for all compactly supported and smooth perturbations 
						$\delta\omega\colon\Rd\to \Rd[n]$ with $\|\delta\omega\|_{L^\infty}\leq 1$
						\begin{align}
						\label{eq:representationMalliavinDerivativeDiff}
						\tag{Hdiff2a}
						\int g\cdot\nabla\big(\delta\phi^T_{\xi+he,B'}{-}\delta\phi^T_{\xi,B'}\big) 
						= \int G^T_{\xi,B',h,e}\cdot\delta\omega;
						\end{align}
						see also Lemma~\ref{eq:lemmaExistenceLinearizedCorrectorsExtended} for the G\^ateaux derivative
						of the linearized corrector and its gradient in direction $\delta\omega$.						
						
						For any $\kappa\in (0,1]$ and any $\beta\in (0,1)$, 
						there moreover exists some constant $C=C(d,\lambda,\Lambda,\nu,\rho,\eta,M,L,\kappa,\beta)$
						such that for all $q\in [1,\infty)$ and all $|\xi|\leq M$ the random field
						$G^T_{\xi,B',h,e}$ gives rise to a \emph{sensitivity estimate for differences} of the form
						\begin{align}
						\label{eq:indHypoSensitivityBoundDiff}
						\tag{Hdiff2b}
						\bigg\langle\bigg|\,\int\bigg(
						\,\dashint_{B_1(x)}|G^T_{\xi,B',h,e}|\,\bigg)^2\bigg|^{q}\bigg\rangle^\frac{1}{q}
						&\leq C^2q^{2C}|h|^{2(1-\beta)}\sup_{\langle F^{2q_*}\rangle=1}
						\int\big\langle|Fg|^{2(\frac{q}{\kappa})_*}\big\rangle^\frac{1}{(\frac{q}{\kappa})_*}.
						\end{align}
						
						If $(g_r)_{r\geq 1}$ is a sequence of compactly supported and $p$-integrable random fields, denote by
						$G_{\xi,B',h,e}^{T,r}\in L^1_{\mathrm{uloc}}(\Rd;\Rd[n])$, $r\geq 1$, the random field
						associated to $g_r$, $r\geq 1$, in the sense of~\eqref{eq:representationMalliavinDerivativeDiff}.
						Let $g$ be an $L^p_{\mathrm{uloc}}(\Rd;\Rd)$-valued random field, and assume that
						$\Prob$-almost surely it holds $g_r\to g$ in $L^p_{\mathrm{uloc}}(\Rd;\Rd)$.
						Then there exists a random field $G^T_{\xi,B',h,e}$ such that $\Prob$-almost surely
						\begin{align}
						\label{eq:indHypoSensitivityApproxDiff}
						\tag{Hdiff2c}
						G_{\xi,B',h,e}^{T,r} \to G_{\xi,B',h,e}^{T} \text{ as } r\to\infty 
						\text{ in } L^1_{\mathrm{uloc}}(\Rd;\Rd[n]).
						\end{align}
						In the special case of $g_r = \mathds{1}_{B_r}g$, $r\geq 1$, the limit random
						field is in addition subject to	the sensitivity estimate~\eqref{eq:indHypoSensitivityBoundDiff}.
\item There exists $\alpha=\alpha(d,\lambda,\Lambda)\in (0,\eta)$ such that for all $\beta\in (0,1)$
there exists a constant $C=C(d,\lambda,\Lambda,\nu,\rho,\eta,M,L,\beta)$
such that for all $q\in [1,\infty)$ and all $|\xi|\leq M$ we have a 
\emph{small-scale annealed Schauder estimate for differences} 
\begin{align}
\label{eq:indHypoAnnealedSchauderdiff}
\tag{Hdiff3}
\big\langle\big\|\nabla\phi^T_{\xi+he,B'}{-}\nabla\phi^T_{\xi,B'}
\big\|^{2q}_{C^\alpha(B_1)}\big\rangle^\frac{1}{q}
\leq C^2q^{2C}|h|^{2(1-\beta)}.
\end{align}
\end{itemize}

\textit{Step 2: (Base case of the induction)} The base case concerns
the correctors of the nonlinear problem with an additional massive term.
A proof of the corresponding assertions from the induction hypotheses 
is given in Appendix~\ref{app:baseCaseInd} by means of Lemma~\ref{prop:estimatesDiffHomCorrectorNonlinear}.

\textit{Step 3: (Induction step---Reduction to linear functionals)}
Subtracting the defining equations~\eqref{eq:PDEhigherOrderLinearizedCorrectorLocalized}
for $\phi^T_{\xi+he,B}$ resp.\ $\phi^T_{\xi,B}$, as well as adding zero yields
the following equation for the difference of linearized correctors
\begin{align}
\nonumber
&\frac{1}{T}\big(\phi^T_{\xi+he,B}{-}\phi^T_{\xi,B}\big)
-\nabla\cdot a_\xi^T\big(\nabla\phi^T_{\xi+he,B}{-}\nabla\phi^T_{\xi,B}\big)
\\&\label{eq:PDEdiffLinearizedHomogenizationCorrectors}
=  \nabla\cdot\big(a_{\xi+he}^T-a^T_\xi\big)
\big(\mathds{1}_{L=1}B + \nabla \phi^T_{\xi+he,B}\big)
\\&~~~\nonumber
+ \nabla\cdot\sum_{\substack{\Pi\in\mathrm{Par}\{1,\ldots,L\} \\ \Pi\neq\{\{1,\ldots,L\}\}}}
\partial_\xi^{|\Pi|} A(\omega,\xi{+}he{+}\nabla \phi^T_{\xi+he})
\Big[\bigodot_{\pi\in\Pi}(\mathds{1}_{|\pi|=1}B'_\pi + \nabla \phi^T_{\xi+he,B'_\pi})\Big]
\\&~~~\nonumber
-\nabla\cdot\sum_{\substack{\Pi\in\mathrm{Par}\{1,\ldots,L\} \\ \Pi\neq\{\{1,\ldots,L\}\}}}
\partial_\xi^{|\Pi|} A(\omega,\xi{+}\nabla \phi^T_\xi)
\Big[\bigodot_{\pi\in\Pi}(\mathds{1}_{|\pi|=1}B'_\pi + \nabla \phi^T_{\xi,B'_\pi})\Big]
\\&\nonumber
=: \nabla\cdot R^T_{\xi,B,h,e}.
\end{align}
Recall that we denote by $a^T_\xi$
the uniformly elliptic and bounded coefficient field
$\partial_\xi A(\omega,\xi+\nabla\phi^T_\xi)$ with respect
to the constants $(\lambda,\Lambda)$ from~Assumption~\ref{assumption:operators}. 
Let a radius $R\in [1,\infty)$ be fixed. Applying first the hole filling 
estimate~\eqref{eq:holeFillingEstimate} and then
Caccioppoli's estimate~\eqref{eq:Caccioppoli} to 
equation~\eqref{eq:PDEdiffLinearizedHomogenizationCorrectors} yields
\begin{align*}
&\Big\langle\Big\|\Big(\frac{\phi^T_{\xi+he,B}{-}\phi^T_{\xi,B}}{\sqrt{T}},
\nabla\phi^T_{\xi+he,B}{-}\nabla\phi^T_{\xi,B}\Big)\Big\|^{2q}_{L^2(B_1)}\Big\rangle^\frac{1}{q}
\\&
\lesssim_{d,\lambda,\Lambda}
R^{d-\delta}\bigg\langle\bigg|\frac{1}{(2R)^2}\inf_{b\in\Rd[]}\,\dashint_{B_{2R}} 
\big|\phi^T_{\xi+he,B}{-}\phi^T_{\xi,B}{-}b\big|^2
+ \frac{1}{T}|b|^2\bigg|^{q}\bigg\rangle^\frac{1}{q}
\\&~~~
+ R^{d-\delta}\bigg\langle\bigg|\,\dashint_{B_{2R}} \big|R^T_{\xi,B,h,e}\big|^2
\bigg|^{q}\bigg\rangle^\frac{1}{q}
+ R^d\bigg\langle\bigg|\,\dashint_{B_{R}} \frac{1}{|x|^\delta}\big|R^T_{\xi,B,h,e}\big|^2
\bigg|^{q}\bigg\rangle^\frac{1}{q}.
\end{align*}
By the same argument which starts from the right hand side of~\eqref{eq:case2Aux10}
and produces the right hand side of~\eqref{eq:case2Aux1} (with $R'=\theta R$, 
$\theta=\theta(d,\lambda,\Lambda)$ yet to be determined, and with $\phi^T_{\xi,B}$ 
replaced by $\phi^T_{\xi+he,B}{-}\phi^T_{\xi,B}$), and the stationarity
of linearized homogenization correctors we obtain  
\begin{align*}
&\bigg\langle\bigg|\frac{1}{(2R)^2}\inf_{b\in\Rd[]}\,\dashint_{B_{2R}} 
\big|\phi^T_{\xi+he,B}{-}\phi^T_{\xi,B}{-}b\big|^2
+ \frac{1}{T}|b|^2\bigg|^{q}\bigg\rangle^\frac{1}{q}
\\&
\lesssim_{d,\lambda,\Lambda} \theta^2
\bigg\langle\bigg|\,\dashint_{B_{R}}
\big|\nabla\phi^T_{\xi+he,B}{-}\nabla\phi^T_{\xi,B}\big|^2
\bigg|^{q}\bigg\rangle^\frac{1}{q}
\\&~~~
+ \bigg\langle\bigg|\,\dashint_{B_{\theta R}}
\nabla\phi^T_{\xi+he,B}{-}\nabla\phi^T_{\xi,B} 
\bigg|^{2q}\bigg\rangle^\frac{1}{q}
+ \bigg\langle\bigg|\,\dashint_{B_{\theta R}}
\frac{1}{\sqrt{T}}\phi^T_{\xi+he,B}{-}\frac{1}{\sqrt{T}}\phi^T_{\xi,B} 
\bigg|^{2q}\bigg\rangle^\frac{1}{q}
\\&
\lesssim_{d,\lambda,\Lambda} \theta^2\bigg\langle\bigg|\frac{1}{(2R)^2}\inf_{b\in\Rd[]}\,\dashint_{B_{2R}} 
\big|\phi^T_{\xi+he,B}{-}\phi^T_{\xi,B}{-}b\big|^2
+ \frac{1}{T}|b|^2\bigg|^{q}\bigg\rangle^\frac{1}{q}
\\&~~~
+ \theta^2\bigg\langle\bigg|\,\dashint_{B_{2R}} \big|R^T_{\xi,B,h,e}\big|^2
\bigg|^{q}\bigg\rangle^\frac{1}{q}
\\&~~~
+ \bigg\langle\bigg|\,\dashint_{B_{\theta R}}
\nabla\phi^T_{\xi+he,B}{-}\nabla\phi^T_{\xi,B} 
\bigg|^{2q}\bigg\rangle^\frac{1}{q}
+ \bigg\langle\bigg|\,\dashint_{B_{\theta R}}
\frac{1}{\sqrt{T}}\phi^T_{\xi+he,B}{-}\frac{1}{\sqrt{T}}\phi^T_{\xi,B} 
\bigg|^{2q}\bigg\rangle^\frac{1}{q}.
\end{align*}
In the second step, we again used Caccioppoli's inequality
with respect to equation~\eqref{eq:PDEdiffLinearizedHomogenizationCorrectors}.
Choosing $\theta=\theta(d,\lambda,\Lambda)$ sufficiently small,
and making use of stationarity of the right hand side
term~$R^T_{\xi,B,h,e}$ of equation~\eqref{eq:PDEdiffLinearizedHomogenizationCorrectors},
then entails in light of the previous two displays by an
absorption argument that
\begin{align}
\nonumber
&\Big\langle\Big\|\Big(\frac{\phi^T_{\xi+he,B}{-}\phi^T_{\xi,B}}{\sqrt{T}},
\nabla\phi^T_{\xi+he,B}{-}\nabla\phi^T_{\xi,B}\Big)\Big\|^{2q}_{L^2(B_1)}\Big\rangle^\frac{1}{q}
\\& \label{eq:intermediateDiff}
\lesssim_{d,\lambda,\Lambda} R^{d-\delta}
\bigg\langle\bigg|\,\dashint_{B_{\theta R}}
\nabla\phi^T_{\xi+he,B}{-}\nabla\phi^T_{\xi,B} 
\bigg|^{2q}\bigg\rangle^\frac{1}{q}
\\&~~~\nonumber
+ R^{d-\delta}\bigg\langle\bigg|\,\dashint_{B_{\theta R}}
\frac{1}{\sqrt{T}}\phi^T_{\xi+he,B}{-}\frac{1}{\sqrt{T}}\phi^T_{\xi,B} 
\bigg|^{2q}\bigg\rangle^\frac{1}{q}
\\&~~~\nonumber
+ R^{d-\delta}\bigg\langle\bigg|\,\dashint_{B_1} \big|R^T_{\xi,B,h,e}\big|^4
\bigg|^\frac{q}{2}\bigg\rangle^\frac{1}{q}.
\end{align}

For the remaining parts of the proof, let us make use
of the abbreviation $\sum_{\Pi}:=\sum_{\Pi\in\mathrm{Par}\{1,\ldots,L\},\, 
\Pi\neq\{\{1,\ldots,L\}\}}$. By adding zero, we may then express the right hand side
term~$R^T_{\xi,B,h,e}$ in equation~\eqref{eq:PDEdiffLinearizedHomogenizationCorrectors} in
the following equivalent form:
\begin{align}
\label{eq:remainderDiff}
&R^T_{\xi,B,h,e} 
\\&\nonumber
= \big(\partial_\xi A(\omega,\xi{+}he{+}\nabla \phi^T_{\xi+he})
- \partial_\xi A(\omega,\xi{+}\nabla \phi^T_\xi)\big)
\big(\mathds{1}_{L=1}B {+} \nabla \phi^T_{\xi+he,B}\big)
\\&~~~\nonumber
+ \sum_{\Pi}\big(\partial_\xi^{|\Pi|} A(\omega,\xi{+}he{+}\nabla \phi^T_{\xi+he})
- \partial_\xi^{|\Pi|} A(\omega,\xi{+}\nabla \phi^T_\xi)\big)
\Big[\bigodot_{\pi\in\Pi}(\mathds{1}_{|\pi|=1}B'_\pi {+} \nabla \phi^T_{\xi+he,B'_\pi})\Big]
\\&~~~\nonumber
- \sum_{\Pi}\partial_\xi^{|\Pi|} A(\omega,\xi{+}\nabla \phi^T_\xi)
\Big[\bigodot_{\pi\in\Pi}(\mathds{1}_{|\pi|=1}B'_\pi {+} \nabla \phi^T_{\xi,B'_\pi})
- \bigodot_{\pi\in\Pi}(\mathds{1}_{|\pi|=1}B'_\pi {+} \nabla \phi^T_{\xi+he,B'_\pi})\Big].
\end{align}
As a consequence of an application of H\"older's inequality,
stationarity of the linearized homogenization correctors,
and~(A2)$_{L}$ from Assumption~\ref{assumption:operators}, we deduce from the previous display that
\begin{align*}
&\bigg\langle\bigg|\,\dashint_{B_1} \big|R^T_{\xi,B,h,e}\big|^2
\bigg|^{q}\bigg\rangle^\frac{1}{q}
\\&
\lesssim_{\Lambda}
\big\langle\big\|\mathds{1}_{L=1}B {+} \nabla \phi^T_{\xi+he,B}
\big\|^{4q}_{C^\alpha(B_1)}\big\rangle^\frac{1}{2q}
\bigg\langle\bigg|\,\dashint_{B_1}
\big| he{+}\nabla\phi^T_{\xi+he}{-}\nabla\phi^T_{\xi}\big|^2
\bigg|^{2q}\bigg\rangle^\frac{1}{2q}
\\&~~~
+ \sum_{\Pi}\prod_{\pi\in\Pi}\big\langle\big\|
\mathds{1}_{|\pi|=1}B'_\pi {+} \nabla \phi^T_{\xi+he,B'_\pi}
\big\|^{4q|\Pi|}_{C^\alpha(B_1)}\big\rangle^\frac{1}{2q|\Pi|}
\bigg\langle\bigg|\,\dashint_{B_1}
\big| he{+}\nabla\phi^T_{\xi+he}{-}\nabla\phi^T_{\xi}\big|^2
\bigg|^{2q}\bigg\rangle^\frac{1}{2q}
\\&~~~
+ \sum_{\Pi}\sup_{\pi\in\Pi}
\bigg\langle\bigg|\,\dashint_{B_1}
\big|\nabla\phi^T_{\xi+he,B'_{\pi}}{-}\nabla\phi^T_{\xi,B'_{\pi}}\big|^2
\bigg|^{q|\Pi|}\bigg\rangle^\frac{1}{q|\Pi|}
\\&~~~~~~~~~~~~~~~~~\times
\sup_{\substack{\pi'\in\Pi \\ \pi'\neq\pi}}
\Big\{1{+}\big\langle\big\|\nabla \phi^T_{\xi+he,B'_\pi}
\big\|^{2q|\Pi|}_{C^\alpha(B_1)}\big\rangle^\frac{1}{q|\Pi|}
+\big\langle\big\|\nabla \phi^T_{\xi,B'_\pi}
\big\|^{2q|\Pi|}_{C^\alpha(B_1)}\big\rangle^\frac{1}{q|\Pi|}\Big\}.
\end{align*}
It thus follows from the induction hypothesis~\eqref{eq:indHypoCorrectorBoundsDiff},
the small-scale annealed Schauder estimates~\eqref{eq:indHypoAnnealedSchauder}
(which are available to any linearization order $\leq L$) and the previous
three displays that
\begin{align}
\label{eq:prelimRemainderEstimate}
&\bigg\langle\bigg|\,\dashint_{B_1} \big|R^T_{\xi,B,h,e}\big|^2
\bigg|^{q}\bigg\rangle^\frac{1}{q} \leq C^2q^{2C}|h|^{2(1-\frac{\beta}{2})}.
\end{align}
Observe also that based on induction hypothesis~\eqref{eq:indHypoAnnealedSchauderdiff}
and the above argument for the right hand side
term~$R^T_{\xi,B,h,e}$ of equation~\eqref{eq:PDEdiffLinearizedHomogenizationCorrectors},
we also get the estimate
\begin{align}
\label{eq:annealedSchauderRemainderDiff}
&\big\langle\big\| R^T_{\xi,B,h,e} \big\|^{2q}_{C^\alpha(B_1)}\big\rangle^\frac{1}{q}
\leq C^2q^{2C}|h|^{2(1-\frac{\beta}{2})}.
\end{align}
Smuggling in a spatial average over the unit ball, we deduce from the 
previous two displays that
\begin{align*}
&\bigg\langle\bigg|\,\dashint_{B_1} \big|R^T_{\xi,B,h,e}\big|^4
\bigg|^\frac{q}{2}\bigg\rangle^\frac{1}{q}
\\&
\lesssim \big\langle\big\| R^T_{\xi,B,h,e} \big\|^{2q}_{C^\alpha(B_1)}\big\rangle^\frac{1}{q}
+ \bigg\langle\bigg|\,\dashint_{B_1} \big|R^T_{\xi,B,h,e}\big|^2
\bigg|^{q}\bigg\rangle^\frac{1}{q} \leq C^2q^{2C}|h|^{2(1-\frac{\beta}{2})}.
\end{align*}
This in turn entails the following update of~\eqref{eq:intermediateDiff}:
\begin{align}
\nonumber
&\Big\langle\Big\|\Big(\frac{\phi^T_{\xi+he,B}{-}\phi^T_{\xi,B}}{\sqrt{T}},
\nabla\phi^T_{\xi+he,B}{-}\nabla\phi^T_{\xi,B}\Big)\Big\|^{2q}_{L^2(B_1)}\Big\rangle^\frac{1}{q}
\\&\label{eq:intermediateDiff2}
\lesssim_{d,\lambda,\Lambda} R^{d-\delta}
\bigg\langle\bigg|\,\dashint_{B_{\theta R}}
\nabla\phi^T_{\xi+he,B}{-}\nabla\phi^T_{\xi,B} 
\bigg|^{2q}\bigg\rangle^\frac{1}{q}
\\&~~~\nonumber
+ R^{d-\delta}\bigg\langle\bigg|\,\dashint_{B_{\theta R}}
\frac{1}{\sqrt{T}}\phi^T_{\xi+he,B}{-}\frac{1}{\sqrt{T}}\phi^T_{\xi,B} 
\bigg|^{2q}\bigg\rangle^\frac{1}{q}
\\&~~~\nonumber
+ R^{d-\delta}C^2q^{2C}|h|^{2(1-\frac{\beta}{2})}.
\end{align}
The upshot of the argument is now the following. We will establish
in the next step of the proof that for any $\tau=\tau(d,\lambda,\Lambda,\beta)\in (0,1)$
\begin{equation}
\label{eq:estimateLinFunctionalsDiffAbsorption}
\begin{aligned}
&\bigg\langle\bigg|\int g\cdot\big(\nabla\phi^T_{\xi+he,B}{-}\nabla\phi^T_{\xi,B}\big)
\bigg|^{2q}\bigg\rangle^\frac{1}{q}
\\&
\leq C^2q^{2C} \Big\langle\Big\|\nabla\phi^T_{\xi+he,B}{-}\nabla\phi^T_{\xi,B}
\Big\|^{2\frac{q}{(1-\tau)^2}}_{L^2(B_1)}\Big\rangle^\frac{(1-\tau)^2}{q}
\int \big|g\big|^2
\\&~~~
+ C^2q^{2C}|h|^{2(1-\frac{\beta}{2})}\int \big|g\big|^2,
\end{aligned}
\end{equation}
as well as
\begin{equation}
\label{eq:estimateLinFunctionalsDiffAbsorption2}
\begin{aligned}
&\bigg\langle\bigg|\int \frac{1}{T}f 
\big(\phi^T_{\xi+he,B'}{-}\phi^T_{\xi,B'}\big)\bigg|^{2q}\bigg\rangle^\frac{1}{q}
\\&
\leq C^2q^{2C} \Big\langle\Big\|\nabla\phi^T_{\xi+he,B}{-}\nabla\phi^T_{\xi,B}
\Big\|^{2\frac{q}{(1-\tau)^2}}_{L^2(B_1)}\Big\rangle^\frac{(1-\tau)^2}{q}
\int \Big|\frac{f}{\sqrt{T}}\Big|^2
\\&~~~
+ C^2q^{2C}|h|^{2(1-\frac{\beta}{2})}\int \Big|\frac{f}{\sqrt{T}}\Big|^2.
\end{aligned}
\end{equation}
The latter two estimates in turn update~\eqref{eq:intermediateDiff2} to
\begin{align}
\nonumber
&\Big\langle\Big\|\Big(\frac{\phi^T_{\xi+he,B}{-}\phi^T_{\xi,B}}{\sqrt{T}},
\nabla\phi^T_{\xi+he,B}{-}\nabla\phi^T_{\xi,B}\Big)\Big\|^{2q}_{L^2(B_1)}\Big\rangle^\frac{1}{q}
\\&\label{eq:intermediateDiff3}
\lesssim_{d,\lambda,\Lambda}
R^{-\delta}C^2q^{2C}\Big\langle\Big\|\nabla\phi^T_{\xi+he,B}{-}\nabla\phi^T_{\xi,B}
\Big\|^{2\frac{q}{(1-\tau)^2}}_{L^2(B_1)}\Big\rangle^\frac{(1-\tau)^2}{q}
\\&~~~\nonumber
+ R^{d-\delta}C^2q^{2C}|h|^{2(1-\frac{\beta}{2})}.
\end{align}
We are one step away from choosing a suitable radius $R\in [1,\infty)$.
Before we do so, we first want to exploit that we already have---to any linearization 
order $\leq L$---the corrector estimates~\eqref{eq:correctorGradientBoundMassiveApprox} at our disposal.
We leverage on that in form of decomposing $\frac{q}{(1-\tau)^2}=q(1{-}\tau) + q(\frac{1}{(1-\tau)^2}-(1{-}\tau))$
and applying H\"older's inequality with respect to the exponents $(\frac{1}{1-\tau},\frac{1}{\tau})$
\begin{align*}
&\Big\langle\Big\|\nabla\phi^T_{\xi+he,B}{-}\nabla\phi^T_{\xi,B}
\Big\|^{2\frac{q}{(1-\tau)^2}}_{L^2(B_1)}\Big\rangle^\frac{(1-\tau)^2}{q}
\leq C^2q^{2C}\Big\langle\Big\|\nabla\phi^T_{\xi+he,B}{-}\nabla\phi^T_{\xi,B}
\Big\|^{2q}_{L^2(B_1)}\Big\rangle^\frac{(1-\tau)^3}{q}
\end{align*}
with the constant $C$ independent of $|h|\leq 1$. This provides an
upgrade of~\eqref{eq:intermediateDiff3} in form of
\begin{align}
\nonumber
&\Big\langle\Big\|\Big(\frac{\phi^T_{\xi+he,B}{-}\phi^T_{\xi,B}}{\sqrt{T}},
\nabla\phi^T_{\xi+he,B}{-}\nabla\phi^T_{\xi,B}\Big)\Big\|^{2q}_{L^2(B_1)}\Big\rangle^\frac{1}{q}
\\&\label{eq:intermediateDiff4}
\lesssim_{d,\lambda,\Lambda}
R^{-\delta}C^2q^{2C}\Big\langle\Big\|\Big(\frac{\phi^T_{\xi+he,B}{-}\phi^T_{\xi,B}}{\sqrt{T}},
\nabla\phi^T_{\xi+he,B}{-}\nabla\phi^T_{\xi,B}\Big)\Big\|^{2q}_{L^2(B_1)}\Big\rangle^\frac{(1-\tau)^3}{q}
\\&~~~\nonumber
+ R^{d-\delta}C^2q^{2C}|h|^{2(1-\frac{\beta}{2})}.
\end{align}
We may assume without loss of generality that the target term (i.e, the left hand
side term of the previous display) is $\geq |h|^{2}$ (otherwise,
there is nothing to prove in the first place). We then choose the radius
$R\in [1,\infty)$ and the parameter $\tau=\tau(d,\lambda,\Lambda,\beta)$ in form of
\begin{align*}
(1-\tau)^3 &:= 1 - \frac{\beta}{2}\frac{\delta}{d-\delta},
\\
R^\delta &:= 
\frac{\Theta(d,\lambda,\Lambda)C^2q^{2C}}
{1\wedge  \langle \| (\frac{\phi^T_{\xi+he,B}{-}\phi^T_{\xi,B}}{\sqrt{T}},
\nabla\phi^T_{\xi+he,B}{-}\nabla\phi^T_{\xi,B} ) \|^{2q}_{L^2(B_1)} \rangle^\frac{1-(1-\tau)^3}{q}},  
\end{align*}
with $\Theta(d,\lambda,\Lambda)$ sufficiently large in order to allow
for an absorption argument in~\eqref{eq:intermediateDiff4}. In summary, we obtain
\begin{align*}
\Big\langle\Big\|\Big(\frac{\phi^T_{\xi+he,B}{-}\phi^T_{\xi,B}}{\sqrt{T}},
\nabla\phi^T_{\xi+he,B}{-}\nabla\phi^T_{\xi,B}\Big)\Big\|^{2q}_{L^2(B_1)}\Big\rangle^\frac{1}{q}
\leq C^2q^{2C}|h|^{2(1-\beta)}.
\end{align*}
In other words, we proved that the last estimate in~\eqref{eq:indHypoCorrectorBoundsDiff} with $B'$ replaced by $B$
is satisfied. Plugging this information back into~\eqref{eq:estimateLinFunctionalsDiffAbsorption}
and~\eqref{eq:estimateLinFunctionalsDiffAbsorption2} in turn establishes the first two estimates
from~\eqref{eq:indHypoCorrectorBoundsDiff} with $B'$ replaced by~$B$.

Last but not least, applying the local Schauder estimate~\eqref{eq:localSchauder} to 
equation~\eqref{eq:PDEdiffLinearizedHomogenizationCorrectors} 
(which is admissible based on
the annealed H\"older regularity of the linearized coefficient 
field, see~Lemma~\ref{lem:annealedHoelderRegLinCoefficient}) and making
use of H\"older's inequality in combination with the 
estimates~\eqref{eq:annealedHoelderRegLinCoefficient}
and~\eqref{eq:annealedSchauderRemainderDiff} moreover yields
\begin{align}
\label{eq:indStepSchauderPrelim}
&\Big\langle\Big\|\Big(\frac{\phi^T_{\xi+he,B}{-}\phi^T_{\xi,B}}{\sqrt{T}},
\nabla\phi^T_{\xi+he,B}{-}\nabla\phi^T_{\xi,B}\Big)\Big\|^{2q}_{C^\alpha(B_1)}\Big\rangle^\frac{1}{q}
\\&\nonumber
\leq C^2q^{2C}|h|^{2(1-\frac{\beta}{2})} + C^2q^{2C}
\Big\langle\Big\|\Big(\frac{\phi^T_{\xi+he,B}{-}\phi^T_{\xi,B}}{\sqrt{T}},
\nabla\phi^T_{\xi+he,B}{-}\nabla\phi^T_{\xi,B}\Big)\Big\|^{2\frac{q}{1-\tau}}_{L^2(B_1)}\Big\rangle^\frac{1-\tau}{q}.
\end{align}
In particular, the small-scale annealed Schauder estimate~\eqref{eq:indHypoAnnealedSchauderdiff}
with $B'$ replaced by~$B$ also holds true by the above reasoning once the validity of the
estimates~\eqref{eq:estimateLinFunctionalsDiffAbsorption}
and~\eqref{eq:estimateLinFunctionalsDiffAbsorption2} is established.

\textit{Step 4: (Induction step---Estimates~\eqref{eq:estimateLinFunctionalsDiffAbsorption}
and~\eqref{eq:estimateLinFunctionalsDiffAbsorption2} for linear functionals)}
For notational convenience, we only discuss in detail the derivation
of the estimate~\eqref{eq:estimateLinFunctionalsDiffAbsorption}. The second one
follows along the same lines. 

We start with the computation of the functional derivative of the
centered random variable $F_{\phi}^{\mathrm{diff}}:=\int g\cdot(\nabla\phi^T_{\xi+he,B}-\nabla\phi^T_{\xi,B})$.
To this end, let $\delta\omega\colon\Rd\to \Rd[n]$ be compactly supported and smooth with
$\|\delta\omega\|_{L^\infty}\leq 1$. 
Based on Lemma~\ref{eq:lemmaExistenceLinearizedCorrectorsExtended},
we may $\Prob$-almost surely differentiate the 
equation~\eqref{eq:PDEdiffLinearizedHomogenizationCorrectors} for differences
of linearized homogenization correctors with respect to the parameter field
in the direction of $\delta\omega$ (taking already into account the representation~\eqref{eq:remainderDiff} 
of the right hand side term). This yields $\Prob$-almost surely the following PDE for the variation
$$\Big(\frac{(\delta\phi^T_{\xi+he,B}-\delta\phi^T_{\xi,B})}{\sqrt{T}},
\nabla(\delta\phi^T_{\xi+he,B}-\delta\phi^T_{\xi,B})\Big)
\in L^2_{\mathrm{uloc}}(\Rd;\Rd[]{\times}\Rd)$$ of differences
of linearized homogenization correctors:
\begin{align}
\nonumber
&\frac{1}{T}(\delta\phi^T_{\xi+he,B}-\delta\phi^T_{\xi,B})
- \nabla\cdot a_\xi^T\nabla(\delta\phi^T_{\xi+he,B}-\delta\phi^T_{\xi,B}) 
\\&\nonumber
= \nabla\cdot\partial_\omega\partial_\xi A(\omega,\xi+\nabla\phi^T_\xi)
\big[\delta\omega\odot(\nabla \phi^T_{\xi+he,B}{-}\nabla\phi^T_{\xi,B})\big]
\\&~~~\nonumber
+ \nabla\cdot\partial_\xi^2 A(\omega,\xi+\nabla\phi^T_\xi)
\big[\nabla\delta\phi^T_\xi\odot(\nabla \phi^T_{\xi+he,B}{-}\nabla\phi^T_{\xi,B})\big]
\\&~~~\nonumber
- \nabla\cdot\big(\partial_\omega\partial_\xi A(\omega,\xi{+}he{+}\nabla \phi^T_{\xi+he})
- \partial_\omega\partial_\xi A(\omega,\xi{+}\nabla \phi^T_\xi)\big)
\big[\delta\omega\odot(\mathds{1}_{L=1}B {+} \nabla \phi^T_{\xi+he,B})\big]
\\&~~~\nonumber
-\nabla\cdot\big(\partial_\xi^2 A(\omega,\xi{+}he{+}\nabla \phi^T_{\xi+he})
- \partial_\xi^2 A(\omega,\xi{+}\nabla \phi^T_\xi)\big)
\big[\nabla\delta\phi^T_{\xi+he}\odot(\mathds{1}_{L=1}B {+} \nabla \phi^T_{\xi+he,B})\big]
\\&~~~\nonumber
-\nabla\cdot\partial_\xi^2 A(\omega,\xi{+}\nabla \phi^T_\xi)
\big[\nabla\delta(\phi^T_{\xi+he}{-}\phi^T_{\xi})\odot
(\mathds{1}_{L=1}B {+} \nabla \phi^T_{\xi+he,B})\big]
\\&~~~\nonumber
- \nabla\cdot\big(\partial_\xi A(\omega,\xi{+}he{+}\nabla \phi^T_{\xi+he})
- \partial_\xi A(\omega,\xi{+}\nabla \phi^T_\xi)\big)\nabla \delta\phi^T_{\xi+he,B}
\\&~~~\nonumber
+\nabla\cdot\sum_{\Pi}\big(\partial_\omega\partial_\xi^{|\Pi|} A(\omega,\xi{+}he{+}\nabla \phi^T_{\xi+he})
- \partial_\omega\partial_\xi^{|\Pi|} A(\omega,\xi{+}\nabla \phi^T_\xi)\big)
\Big[\delta\omega\odot\bigodot_{\pi\in\Pi}(\mathds{1}_{|\pi|=1}B'_\pi {+} \nabla \phi^T_{\xi+he,B'_\pi})\Big]
\\&~~~\nonumber
+\nabla\cdot\sum_{\Pi}\big(\partial_\xi^{|\Pi|+1} A(\omega,\xi{+}he{+}\nabla \phi^T_{\xi+he})
- \partial_\xi^{|\Pi|+1} A(\omega,\xi{+}\nabla \phi^T_\xi)\big)
\Big[\nabla\delta\phi^T_{\xi+he}\odot\bigodot_{\pi\in\Pi}(\mathds{1}_{|\pi|=1}B'_\pi {+} \nabla \phi^T_{\xi+he,B'_\pi})\Big]
\\&~~~\nonumber
+\nabla\cdot\sum_{\Pi}\partial_\xi^{|\Pi|+1} A(\omega,\xi{+}\nabla \phi^T_\xi)
\Big[\nabla\delta(\phi^T_{\xi+he}{-}\phi^T_{\xi})\odot\bigodot_{\pi\in\Pi}
(\mathds{1}_{|\pi|=1}B'_\pi {+} \nabla \phi^T_{\xi+he,B'_\pi})\Big]
\\&~~~\nonumber
+ \nabla\cdot\sum_{\Pi}\big(\partial_\xi^{|\Pi|} A(\omega,\xi{+}he{+}\nabla \phi^T_{\xi+he})
- \partial_\xi^{|\Pi|} A(\omega,\xi{+}\nabla \phi^T_\xi)\big)
\Big[\sum_{\pi\in\Pi}\nabla\delta\phi^T_{\xi+he,B'_\pi}\odot
\bigodot_{\substack{\pi'\in\Pi \\ \pi'\neq \pi}}
(\mathds{1}_{|\pi'|=1}B'_{\pi'}{+}\nabla \phi^T_{\xi+he,B'_{\pi'}})\Big]
\\&~~~\nonumber
-\nabla\cdot\sum_{\Pi}\partial_\omega\partial_\xi^{|\Pi|} A(\omega,\xi{+}\nabla \phi^T_\xi)
\Big[\delta\omega\odot\bigodot_{\pi\in\Pi}(\mathds{1}_{|\pi|=1}B'_\pi {+} \nabla \phi^T_{\xi,B'_\pi})
- \delta\omega\odot\bigodot_{\pi\in\Pi}(\mathds{1}_{|\pi|=1}B'_\pi {+} \nabla \phi^T_{\xi+he,B'_\pi})\Big]
\\&~~~\nonumber
- \nabla\cdot\sum_{\Pi}\partial_\xi^{|\Pi|} A(\omega,\xi{+}\nabla \phi^T_\xi)
\Big[\sum_{\pi\in\Pi}\nabla\delta(\phi^T_{\xi,B'_\pi}{-}\phi^T_{\xi+he,B'_{\pi}})\odot
\bigodot_{\substack{\pi'\in\Pi \\ \pi'\neq \pi}}
(\mathds{1}_{|\pi'|=1}B'_{\pi'}{+}\nabla \phi^T_{\xi,B'_{\pi'}})\Big]
\\&~~~\nonumber
+ \nabla\cdot\sum_{\Pi}\partial_\xi^{|\Pi|} A(\omega,\xi{+}\nabla \phi^T_\xi)
\Big[\sum_{\pi\in\Pi}\nabla\delta\phi^T_{\xi+he,B'_\pi}\odot
\bigodot_{\substack{\pi'\in\Pi \\ \pi'\neq \pi}}
(\mathds{1}_{|\pi'|=1}B'_{\pi'}{+}\nabla \phi^T_{\xi+he,B'_{\pi'}})\Big]
\\&~~~\nonumber
- \nabla\cdot\sum_{\Pi}\partial_\xi^{|\Pi|} A(\omega,\xi{+}\nabla \phi^T_\xi)
\Big[\sum_{\pi\in\Pi}\nabla\delta\phi^T_{\xi+he,B'_\pi}\odot
\bigodot_{\substack{\pi'\in\Pi \\ \pi'\neq \pi}}
(\mathds{1}_{|\pi'|=1}B'_{\pi'}{+}\nabla \phi^T_{\xi,B'_{\pi'}})\Big]
\\&\label{eq:representationMalliavinDiffLinearizedHomCorrectors}
=: \nabla\cdot R^{T,(1)}_{\xi,B,h,e}\delta\omega 
+ \nabla\cdot R^{T,(2)}_{\xi,B,h,e}\nabla\delta\phi^T_\xi
+ \nabla\cdot R^{T,(3)}_{\xi,B,h,e}\delta\omega 
\\&~~~~\nonumber
+ \nabla\cdot R^{T,(4)}_{\xi,B,h,e}\nabla\delta\phi^T_{\xi+he}
+ \nabla\cdot R^{T,(5)}_{\xi,B,h,e}\nabla(\delta\phi^T_{\xi+he}{-}\delta\phi^T_{\xi}) 
+ \nabla\cdot R^{T,(6)}_{\xi,B,h,e}\nabla\delta\phi^T_{\xi+he,B}
\\&~~~~\nonumber
+ \nabla\cdot R^{T,(7)}_{\xi,B,h,e}\delta\omega  
+ \nabla\cdot R^{T,(8)}_{\xi,B,h,e}\nabla\delta\phi^T_{\xi+he}
+ \nabla\cdot R^{T,(9)}_{\xi,B,h,e}\nabla(\delta\phi^T_{\xi+he}{-}\delta\phi^T_{\xi})
\\&~~~~\nonumber
+ \sum_\Pi\sum_{\pi\in\Pi} \nabla\cdot R^{T,(10),\pi}_{\xi,B,h,e}\nabla\delta\phi^T_{\xi+he,B'_\pi} 
+ \nabla\cdot R^{T,(11)}_{\xi,B,h,e}\delta\omega 
\\&~~~~\nonumber
+ \sum_\Pi\sum_{\pi\in\Pi} \nabla\cdot R^{T,(12),\pi}_{\xi,B,h,e}
\nabla(\delta\phi^T_{\xi+he,B'_\pi}{-}\delta\phi^T_{\xi,B'_\pi})
\\&~~~~\nonumber
+ \sum_\Pi\sum_{\pi\in\Pi} \nabla\cdot R^{T,(13),\pi}_{\xi,B,h,e}
\nabla\delta\phi^T_{\xi+he,B'_\pi}.
\end{align}
In principle, one needs to resort to an approximation argument
in order to proceed from here. As this can be done along
the same lines as in \textit{Step 1} of the proof 
of Lemma~\ref{lem:annealedLinFunctionalsIndStep}, we gloss
over this technical issue and continue directly for the sake of brevity.
More precisely, by a duality argument based on the dual
operator $(\frac{1}{T}-\nabla\cdot a^{T,*}_\xi\nabla)$
we may deduce from~\eqref{eq:representationMalliavinDiffLinearizedHomCorrectors} that  
\begin{align*}
\delta F_\phi^{\mathrm{diff}} &=
\int g\cdot\nabla(\delta\phi^T_{\xi+he,B}{-}\delta\phi^T_{\xi,B})
\\&
= - \int \nabla(\delta\phi^T_{\xi+he,B}{-}\delta\phi^T_{\xi,B})\cdot a^{T,*}_\xi
\nabla\Big(\frac{1}{T}{-}\nabla\cdot a^{T,*}_\xi\nabla\Big)^{-1}
\big(\nabla\cdot g\big)
\\&~~~
- \int \delta(\phi^T_{\xi+he,B}{-}\phi^T_{\xi,B})\,
\frac{1}{T}\Big(\frac{1}{T}{-}\nabla\cdot a^{T,*}_\xi\nabla\Big)^{-1}
\big(\nabla\cdot g\big)
\\&
= \int \big(R^{T,(1)}_{\xi,B,h,e}\delta\omega +
R^{T,(2)}_{\xi,B,h,e}\nabla\delta\phi^T_\xi \big)\cdot 
\nabla\Big(\frac{1}{T}{-}\nabla\cdot a^{T,*}_\xi\nabla\Big)^{-1}
\big(\nabla\cdot g\big)
\\&~~~
+ \int \big(R^{T,(3)}_{\xi,B,h,e} + R^{T,(7)}_{\xi,B,h,e} + 
R^{T,(11)}_{\xi,B,h,e} \big)\delta\omega\cdot 
\nabla\Big(\frac{1}{T}{-}\nabla\cdot a^{T,*}_\xi\nabla\Big)^{-1}
\big(\nabla\cdot g\big)
\\&~~~
+ \int \big(R^{T,(4)}_{\xi,B,h,e} + 
R^{T,(8)}_{\xi,B,h,e} \big)\nabla\delta\phi^T_{\xi+he}\cdot 
\nabla\Big(\frac{1}{T}{-}\nabla\cdot a^{T,*}_\xi\nabla\Big)^{-1}
\big(\nabla\cdot g\big)
\\&~~~
+ \int R^{T,(6)}_{\xi,B,h,e}\nabla\delta\phi^T_{\xi+he,B}\cdot 
\nabla\Big(\frac{1}{T}{-}\nabla\cdot a^{T,*}_\xi\nabla\Big)^{-1}
\big(\nabla\cdot g\big)
\\&~~~
+ \sum_\Pi\sum_{\pi\in\Pi} \int 
\big(R^{T,(10),\pi}_{\xi,B,h,e} {+} R^{T,(13),\pi}_{\xi,B,h,e}\big)
\nabla\delta\phi^T_{\xi+he,B'_\pi}\cdot 
\nabla\Big(\frac{1}{T}{-}\nabla\cdot a^{T,*}_\xi\nabla\Big)^{-1}
\big(\nabla\cdot g\big)
\\&~~~
+ \int \big(R^{T,(5)}_{\xi,B,h,e} {+} R^{T,(9)}_{\xi,B,h,e}\big)
\nabla(\delta\phi^T_{\xi+he}{-}\delta\phi^T_{\xi})\cdot 
\nabla\Big(\frac{1}{T}{-}\nabla\cdot a^{T,*}_\xi\nabla\Big)^{-1}
\big(\nabla\cdot g\big)
\\&~~~
+ \sum_\Pi\sum_{\pi\in\Pi} \int 
R^{T,(12),\pi}_{\xi,B,h,e}
\nabla(\delta\phi^T_{\xi+he,B'_\pi}{-}\delta\phi^T_{\xi,B'_\pi})\cdot 
\nabla\Big(\frac{1}{T}{-}\nabla\cdot a^{T,*}_\xi\nabla\Big)^{-1}
\big(\nabla\cdot g\big).
\end{align*}
Thanks to the induction hypothesis~\eqref{eq:representationMalliavinDerivativeDiff}
and~\eqref{eq:representationMalliavinDerivative} (the latter being already available
to any linearization order $\leq L$) we obtain a representation of the form
\begin{align}
\label{eq:sensDiffAux10}
\delta F_\phi^{\mathrm{diff}} &= \sum_{i=1}^{13} \int G^{T,(i)}_{\xi,B,h,e}\cdot\delta\omega.
\end{align}
(For the rigorous argument based on an approximation procedure in the spirit
of \textit{Step 1} of the proof of Lemma~\ref{lem:annealedLinFunctionalsIndStep},
one in addition relies on~\eqref{eq:indHypoSensitivityApproxDiff} and~\eqref{eq:indHypoSensitivityApprox};
the latter again up to linearization order $\leq L$ which is admissible thanks to
Lemma~\ref{lem:closingIndStep}.) We can feed~\eqref{eq:sensDiffAux10} into the spectral 
gap inequality in form of~\eqref{eq:spectralGapHigherMoments} which entails
\begin{equation}
\label{eq:spectralGapDiff}
\begin{aligned}
\big\langle\big|F_\phi^{\mathrm{diff}}\big|^{2q}\big\rangle^\frac{1}{q}
&\leq C^2q^2\sum_{i=1}^{13}\bigg\langle\bigg|\int\bigg(
\,\dashint_{B_1(x)}\big|G^{T,(i)}_{\xi,B,h,e}\big|\bigg)^2\bigg|^q\bigg\rangle^\frac{1}{q}.
\end{aligned}
\end{equation}
It remains to estimate the terms on the right hand side of the previous display. 

We start with the first two terms on the right hand side of~\eqref{eq:spectralGapDiff},
which are precisely those being responsible for the first right hand side term 
in~\eqref{eq:estimateLinFunctionalsDiffAbsorption}. By duality in $L^q_{\langle\cdot\rangle}$,
H\"older's inequality, stationarity of the linearized homogenization correctors,
and~(A3)$_{L}$ from Assumption~\ref{assumption:operators} we get
\begin{align*}
&\bigg\langle\bigg|\int\bigg(
\,\dashint_{B_1(x)}\big|G^{T,(1)}_{\xi,B,h,e}\big|\bigg)^2\bigg|^q\bigg\rangle^\frac{1}{q}
\\&
\leq C^2\big\langle\big\|\nabla\phi^T_{\xi+he,B}{-}\nabla\phi^T_{\xi,B}\big\|^{2q}_{L^2(B_1)}\big\rangle^\frac{1}{q}
\sup_{\langle F^{2q_*}\rangle=1}\int\Big\langle\Big|\Big(\frac{1}{T}{-}\nabla\cdot a^{T,*}_\xi\nabla\Big)^{-1}
\big(\nabla\cdot Fg\big)\Big|^{2q_*}\Big\rangle^\frac{1}{q_*}. 
\end{align*}
Hence, at least for sufficiently large $q\in [1,\infty)$ such that $|q_*-1|$
is small enough in order to be in the perturbative regime of the annealed Calder\'on--Zygmund
estimate in form of~\eqref{lem:annealedCZMeyers}, we deduce from the previous display that
\begin{align}
\label{eq:rhs1}
&\bigg\langle\bigg|\int\bigg(
\,\dashint_{B_1(x)}\big|G^{T,(1)}_{\xi,B,h,e}\big|\bigg)^2\bigg|^q\bigg\rangle^\frac{1}{q}
\leq C^2\big\langle\big\|\nabla\phi^T_{\xi+he,B}{-}\nabla\phi^T_{\xi,B}\big\|^{2q}_{L^2(B_1)}
\big\rangle^\frac{1}{q} \int |g|^2.
\end{align}
For the second term, we instead rely on~\eqref{eq:indHypoSensitivityBound} (applied to $\nabla\phi^T_\xi$
with the choice $\kappa=\tau$) and~(A2)$_{L}$ from Assumption~\ref{assumption:operators} to infer
\begin{align*}
&\bigg\langle\bigg|\int\bigg(
\,\dashint_{B_1(x)}\big|G^{T,(2)}_{\xi,B,h,e}\big|\bigg)^2\bigg|^q\bigg\rangle^\frac{1}{q}
\\&
\leq C^2q^{2C}\sup_{\langle F^{2q_*}\rangle=1}
\int\Big\langle\Big|\nabla\phi^T_{\xi+he,B}{-}\nabla\phi^T_{\xi,B}\Big|^{2(\frac{q}{\tau})_*}
\Big|\Big(\frac{1}{T}{-}\nabla\cdot a^{T,*}_\xi\nabla\Big)^{-1}
\big(\nabla\cdot Fg\big)\Big|^{2(\frac{q}{\tau})_*}\Big\rangle^\frac{1}{(\frac{q}{\tau})_*}.
\end{align*}
Applying H\"older's inequality with exponents $(\frac{q-\tau}{1-\tau},
(\frac{q-\tau}{1-\tau})_*=\frac{q-\tau}{q-1})$ it then follows from
$(\frac{q}{\tau})_*=\frac{q}{q-\tau}$ that
\begin{align*}
&\bigg\langle\bigg|\int\bigg(
\,\dashint_{B_1(x)}\big|G^{T,(2)}_{\xi,B,h,e}\big|\bigg)^2\bigg|^q\bigg\rangle^\frac{1}{q}
\leq  C^2\big\langle\big\|\nabla\phi^T_{\xi+he,B}{-}\nabla\phi^T_{\xi,B}\big\|^{2\frac{q}{1-\tau}}_{C^\alpha(B_1)}
\big\rangle^\frac{1-\tau}{q} \int |g|^2,
\end{align*}
again at least for sufficiently large $q\in [1,\infty)$. Combining this 
with~\eqref{eq:indStepSchauderPrelim} updates the previous display to
\begin{equation}
\label{eq:rhs2}
\begin{aligned}
&\bigg\langle\bigg|\int\bigg(
\,\dashint_{B_1(x)}\big|G^{T,(2)}_{\xi,B,h,e}\big|\bigg)^2\bigg|^q\bigg\rangle^\frac{1}{q}
\\&
\leq C^2q^{2C}|h|^{2(1-\frac{\beta}{2})} \int |g|^2 
+ C^2\big\langle\big\|\nabla\phi^T_{\xi+he,B}{-}\nabla\phi^T_{\xi,B}\big\|^{2\frac{q}{(1-\tau)^2}}_{L^2(B_1)}
\big\rangle^\frac{(1-\tau)^2}{q} \int |g|^2,
\end{aligned}
\end{equation}
at least for sufficiently large $q\in [1,\infty)$.

For the remaining terms on the right hand side of~\eqref{eq:spectralGapDiff}, note that all of them
incorporate a difference of \textit{lower-order} linearized homogenization correctors. In view
of induction hypotheses~\eqref{eq:indHypoCorrectorBoundsDiff}, \eqref{eq:indHypoSensitivityBoundDiff} 
and~\eqref{eq:indHypoAnnealedSchauderdiff}, one thus expects them to contribute only to the second right hand side
term of~\eqref{eq:estimateLinFunctionalsDiffAbsorption}. We verify this by grouping them
into three categories. 

First, we estimate by duality in $L^q_{\langle\cdot\rangle}$, H\"older's inequality,
stationarity of the linearized homogenization correctors, and~(A3)$_{L}$ as
well as~(A4)$_L$ from Assumption~\ref{assumption:operators}
\begin{align*}
&\sum_{i\in\{3,7,11\}}\bigg\langle\bigg|\int\bigg(
\,\dashint_{B_1(x)}\big|G^{T,(i)}_{\xi,B,h,e}\big|\bigg)^2\bigg|^q\bigg\rangle^\frac{1}{q}
\\&
\leq C^2\big\langle\big\|he{+}\nabla\phi^T_{\xi+he}{-}\nabla\phi^T_{\xi}
\big\|^{4q}_{L^2(B_1)}\big\rangle^\frac{1}{2q}
\big\langle\big\|\mathds{1}_{L=1}B{+}\nabla\phi^T_{\xi+he,B}
\big\|^{4q}_{C^\alpha(B_1)}\big\rangle^\frac{1}{2q}
\\&~~~~~~~~~~~~~~~~~~~~~~~~~~~~~~~~~\times
\sup_{\langle F^{2q_*}\rangle=1}\int\Big\langle\Big|\Big(\frac{1}{T}{-}\nabla\cdot a^{T,*}_\xi\nabla\Big)^{-1}
\big(\nabla\cdot Fg\big)\Big|^{2q_*}\Big\rangle^\frac{1}{q_*}
\\&~~~
+ C^2\big\langle\big\|he{+}\nabla\phi^T_{\xi+he}{-}\nabla\phi^T_{\xi}
\big\|^{4q}_{L^2(B_1)}\big\rangle^\frac{1}{2q}
\sum_{\Pi}\prod_{\pi\in\Pi} \big\langle\big\|\mathds{1}_{|\pi|=1}B'_{\pi}{+}\nabla\phi^T_{\xi+he,B'}
\big\|^{4q|\Pi|}_{C^\alpha(B_1)}\big\rangle^\frac{1}{2q|\Pi|}
\\&~~~~~~~~~~~~~~~~~~~~~~~~~~~~~~~~~\times
\sup_{\langle F^{2q_*}\rangle=1}\int\Big\langle\Big|\Big(\frac{1}{T}{-}\nabla\cdot a^{T,*}_\xi\nabla\Big)^{-1}
\big(\nabla\cdot Fg\big)\Big|^{2q_*}\Big\rangle^\frac{1}{q_*}
\\&~~~
+C^2\sum_{\Pi}\sup_{\pi\in\Pi}
\big\langle\big\|
\nabla\phi^T_{\xi+he,B'_{\pi}}{-}\nabla\phi^T_{\xi,B'_{\pi}}
\big\|^{2q|\Pi|}_{L^2(B_1)}\big\rangle^\frac{1}{q|\Pi|}
\\&~~~~~~~~~~~~~~~~~\times
\sup_{\substack{\pi'\in\Pi \\ \pi'\neq\pi}}
\Big\{1{+}\big\langle\big\|\nabla \phi^T_{\xi+he,B'_\pi}
\big\|^{2q|\Pi|}_{C^\alpha(B_1)}\big\rangle^\frac{1}{q|\Pi|}
+\big\langle\big\|\nabla \phi^T_{\xi,B'_\pi}
\big\|^{2q|\Pi|}_{C^\alpha(B_1)}\big\rangle^\frac{1}{q|\Pi|}\Big\}
\\&~~~~~~~~~~~~~~~~~\times
\sup_{\langle F^{2q_*}\rangle=1}\int\Big\langle\Big|\Big(\frac{1}{T}{-}\nabla\cdot a^{T,*}_\xi\nabla\Big)^{-1}
\big(\nabla\cdot Fg\big)\Big|^{2q_*}\Big\rangle^\frac{1}{q_*}.
\end{align*}
Hence, a combination of the induction hypothesis~\eqref{eq:indHypoCorrectorBoundsDiff}
with the annealed small-scale Schauder estimate~\eqref{eq:indHypoAnnealedSchauder}
(which is available to any linearization order $\leq L$) and the
perturbative annealed Calder\'on--Zygmund estimate~\eqref{eq:annealedCZMeyers}
entails for sufficiently large $q\in [1,\infty)$
\begin{align}
\label{eq:rhs3}
\sum_{i\in\{3,7,11\}}\bigg\langle\bigg|\int\bigg(
\,\dashint_{B_1(x)}\big|G^{T,(i)}_{\xi,B,h,e}\big|\bigg)^2\bigg|^q\bigg\rangle^\frac{1}{q}
\leq C^2q^{2C}|h|^{2(1-\frac{\beta}{2})} \int |g|^2.
\end{align}

Second, we estimate by means of~\eqref{eq:indHypoSensitivityBound} with $\kappa=\frac{1}{2}$ (which is 
already available up to any linearization order $\leq L$), stationarity of the linearized
homogenization correctors, H\"older's inequality with respect to the exponents
$\big(\frac{q-\frac{1}{2}}{1-\frac{1}{2}},
(\frac{q-\frac{1}{2}}{1-\frac{1}{2}})_*=\frac{q-\frac{1}{2}}{q-1}\big)$, the fact that
$(\frac{q}{\frac{1}{2}})_*=\frac{q}{q-\frac{1}{2}}$, and~(A2)$_{L}$ as
well as~(A4)$_L$ from Assumption~\ref{assumption:operators}
\begin{align*}
&\sum_{i\in\{4,6,8,10,13\}}\bigg\langle\bigg|\int\bigg(
\,\dashint_{B_1(x)}\big|G^{T,(i)}_{\xi,B,h,e}\big|\bigg)^2\bigg|^q\bigg\rangle^\frac{1}{q}
\\&
\leq C^2\big\langle\big\|he{+}\nabla\phi^T_{\xi+he}{-}\nabla\phi^T_{\xi}
\big\|^{8q}_{C^\alpha(B_1)}\big\rangle^\frac{1}{4q}
\big\langle\big\|\mathds{1}_{L=1}B{+}\nabla\phi^T_{\xi+he,B}
\big\|^{8q}_{C^\alpha(B_1)}\big\rangle^\frac{1}{4q}
\\&~~~~~~~~~~~~~~~~~~~~~~~~~~~~~~~~~\times
\sup_{\langle F^{2q_*}\rangle=1}\int\Big\langle\Big|\Big(\frac{1}{T}{-}\nabla\cdot a^{T,*}_\xi\nabla\Big)^{-1}
\big(\nabla\cdot Fg\big)\Big|^{2q_*}\Big\rangle^\frac{1}{q_*}
\\&~~~
+ C^2\big\langle\big\|he{+}\nabla\phi^T_{\xi+he}{-}\nabla\phi^T_{\xi}
\big\|^{4q}_{C^\alpha(B_1)}\big\rangle^\frac{1}{2q}
\\&~~~~~~~~~~~~~~~~~~~~~~~~~~~~~~~~~\times
\sup_{\langle F^{2q_*}\rangle=1}\int\Big\langle\Big|\Big(\frac{1}{T}{-}\nabla\cdot a^{T,*}_\xi\nabla\Big)^{-1}
\big(\nabla\cdot Fg\big)\Big|^{2q_*}\Big\rangle^\frac{1}{q_*}
\\&~~~
+ C^2\big\langle\big\|he{+}\nabla\phi^T_{\xi+he}{-}\nabla\phi^T_{\xi}
\big\|^{8q}_{C^\alpha(B_1)}\big\rangle^\frac{1}{4q}
\sum_{\Pi}\prod_{\pi\in\Pi} \big\langle\big\|\mathds{1}_{|\pi|=1}B'_{\pi}{+}\nabla\phi^T_{\xi+he,B'}
\big\|^{8q|\Pi|}_{C^\alpha(B_1)}\big\rangle^\frac{1}{4q|\Pi|}
\\&~~~~~~~~~~~~~~~~~~~~~~~~~~~~~~~~~\times
\sup_{\langle F^{2q_*}\rangle=1}\int\Big\langle\Big|\Big(\frac{1}{T}{-}\nabla\cdot a^{T,*}_\xi\nabla\Big)^{-1}
\big(\nabla\cdot Fg\big)\Big|^{2q_*}\Big\rangle^\frac{1}{q_*}
\\&~~~
+C^2\sum_{\Pi}\sup_{\pi\in\Pi}
\big\langle\big\|
\nabla\phi^T_{\xi+he,B'_{\pi}}{-}\nabla\phi^T_{\xi,B'_{\pi}}
\big\|^{4q|\Pi|}_{C^\alpha(B_1)}\big\rangle^\frac{1}{2q|\Pi|}
\\&~~~~~~~~~~~~~~~~~\times
\sup_{\substack{\pi'\in\Pi \\ \pi'\neq\pi}}
\Big\{1{+}\big\langle\big\|\nabla \phi^T_{\xi+he,B'_\pi}
\big\|^{4q|\Pi|}_{C^\alpha(B_1)}\big\rangle^\frac{1}{2q|\Pi|}
+\big\langle\big\|\nabla \phi^T_{\xi,B'_\pi}
\big\|^{4q|\Pi|}_{C^\alpha(B_1)}\big\rangle^\frac{1}{2q|\Pi|}\Big\}
\\&~~~~~~~~~~~~~~~~~\times
\sup_{\langle F^{2q_*}\rangle=1}\int\Big\langle\Big|\Big(\frac{1}{T}{-}\nabla\cdot a^{T,*}_\xi\nabla\Big)^{-1}
\big(\nabla\cdot Fg\big)\Big|^{2q_*}\Big\rangle^\frac{1}{q_*}.
\end{align*}
This time, it thus follows from induction hypothesis~\eqref{eq:indHypoAnnealedSchauderdiff}
in combination with the annealed small-scale Schauder estimate~\eqref{eq:indHypoAnnealedSchauder}
(which is already available to any linearization order $\leq L$) and the
perturbative annealed Calder\'on--Zygmund estimate~\eqref{eq:annealedCZMeyers}
\begin{align}
&\sum_{i\in\{4,6,8,10,13\}}\bigg\langle\bigg|\int\bigg(
\,\dashint_{B_1(x)}\big|G^{T,(i)}_{\xi,B,h,e}\big|\bigg)^2\bigg|^q\bigg\rangle^\frac{1}{q}
\label{eq:rhs4}
\leq C^2q^{2C}|h|^{2(1-\frac{\beta}{2})} \int |g|^2,
\end{align}
at least for sufficiently large $q\in [1,\infty)$.

Third, and last, we estimate based on induction hypothesis~\eqref{eq:indHypoSensitivityBoundDiff}
with $\kappa=\frac{1}{2}$, stationarity of the linearized
homogenization correctors, H\"older's inequality with respect to the exponents
$\big(\frac{q-\frac{1}{2}}{1-\frac{1}{2}},
(\frac{q-\frac{1}{2}}{1-\frac{1}{2}})_*=\frac{q-\frac{1}{2}}{q-1}\big)$, the fact that
$(\frac{q}{\frac{1}{2}})_*=\frac{q}{q-\frac{1}{2}}$, and finally~(A2)$_{L}$ from Assumption~\ref{assumption:operators}
\begin{align*}
&\sum_{i\in\{5,9,12\}}\bigg\langle\bigg|\int\bigg(
\,\dashint_{B_1(x)}\big|G^{T,(i)}_{\xi,B,h,e}\big|\bigg)^2\bigg|^q\bigg\rangle^\frac{1}{q}
\\&
\leq C^2q^{2C}|h|^{2(1-\frac{\beta}{2})}\big\langle\big\|\mathds{1}_{L=1}B{+}\nabla\phi^T_{\xi+he,B}
\big\|^{4q}_{C^\alpha(B_1)}\big\rangle^\frac{1}{2q}
\\&~~~~~~~~~~~~~~~~~~~~~~~~~~~~~~~~~\times
\sup_{\langle F^{2q_*}\rangle=1}\int\Big\langle\Big|\Big(\frac{1}{T}{-}\nabla\cdot a^{T,*}_\xi\nabla\Big)^{-1}
\big(\nabla\cdot Fg\big)\Big|^{2q_*}\Big\rangle^\frac{1}{q_*}
\\&~~~
+ C^2q^{2C}|h|^{2(1-\frac{\beta}{2})}\sum_{\Pi}\prod_{\pi\in\Pi} 
\big\langle\big\|\mathds{1}_{|\pi|=1}B'_{\pi}{+}\nabla\phi^T_{\xi+he,B'}
\big\|^{4q|\Pi|}_{C^\alpha(B_1)}\big\rangle^\frac{1}{2q|\Pi|}
\\&~~~~~~~~~~~~~~~~~~~~~~~~~~~~~~~~~\times
\sup_{\langle F^{2q_*}\rangle=1}\int\Big\langle\Big|\Big(\frac{1}{T}{-}\nabla\cdot a^{T,*}_\xi\nabla\Big)^{-1}
\big(\nabla\cdot Fg\big)\Big|^{2q_*}\Big\rangle^\frac{1}{q_*}.
\end{align*}
We then obtain for sufficiently large $q\in [1,\infty)$ the estimate
\begin{align}
&\sum_{i\in\{5,9,12\}}\bigg\langle\bigg|\int\bigg(
\,\dashint_{B_1(x)}\big|G^{T,(i)}_{\xi,B,h,e}\big|\bigg)^2\bigg|^q\bigg\rangle^\frac{1}{q}
\label{eq:rhs5}
\leq C^2q^{2C}|h|^{2(1-\frac{\beta}{2})} \int |g|^2
\end{align}
by means of the same ingredients as for~\eqref{eq:rhs4}.

Collecting the estimates~\eqref{eq:rhs1}--\eqref{eq:rhs5} and
feeding them back into~\eqref{eq:spectralGapDiff} eventually
entails the asserted estimate~\eqref{eq:estimateLinFunctionalsDiffAbsorption}.

\textit{Step 5: (Induction step---Conclusion)} As we
already argued at the end of \textit{Step~3} of this proof, 
the estimates from~\eqref{eq:indHypoCorrectorBoundsDiff}
and~\eqref{eq:indHypoAnnealedSchauderdiff} now also hold true with
$B'$ replaced by $B$. In order to deduce the validity of~\eqref{eq:representationMalliavinDerivativeDiff}
and~\eqref{eq:indHypoSensitivityBoundDiff}, one may in fact follow the principles
of the proof of Lemma~\ref{lem:closingIndStep} and adapt them to the arguments
from the previous step. This concludes the proof of~\eqref{eq:correctorBoundForDifferencesMassiveApprox}
at least in case of the linearized homogenization correctors.

\textit{Step 6: (Proof of~\eqref{eq:correctorBoundForDifferencesMassiveApprox} 
for linearized flux correctors)}
The difference of two linearized flux correctors satisfies the equation
\begin{equation}
\label{eq:PDEdiffLinearizedFLuxCorrectors}
\begin{aligned}
&\frac{1}{T}(\sigma_{\xi+he,B,kl}^T{-}\sigma_{\xi,B,kl}^T) 
- \Delta(\sigma_{\xi+he,B,kl}^T{-}\sigma_{\xi,B,kl}^T) 
\\&
= - \nabla\cdot \big((e_l\otimes e_k - e_k\otimes e_l) 
(q^T_{\xi+he,B}{-}q^T_{\xi,B})\big).
\end{aligned}
\end{equation}
Following the argument in \textit{Step 2} of the proof of Theorem~\ref{theo:correctorBoundsMassiveApprox}
(see, e.g., \eqref{eq:aux1ProofTheoremMassiveCorr}),
the desired estimate on the difference $\sigma_{\xi+he,B,kl}^T{-}\sigma_{\xi,B,kl}^T$ 
of linearized flux correctors boils down to an estimate of linear functionals for
the difference $\sigma_{\xi+he,B,kl}^T{-}\sigma_{\xi,B,kl}^T$ and an estimate
on the difference of linearized fluxes $q^T_{\xi+he,B}{-}q^T_{\xi,B}$. 

With respect to the latter, we derive from~\eqref{eq:HigherOrderLinearizedFlux}, 
\eqref{eq:PDEdiffLinearizedHomogenizationCorrectors}
and~\eqref{eq:remainderDiff} that
\begin{align}
\nonumber
&q^T_{\xi+he,B}{-}q^T_{\xi,B} 
\\&\label{eq:linearizedFluxDiff}
= \partial_\xi A(\omega,\xi{+}\nabla \phi^T_\xi)\big(\nabla\phi^T_{\xi+he,B}{-}\nabla\phi^T_{\xi,B}\big)
\\&~~~\nonumber
- \big(\partial_\xi A(\omega,\xi{+}he{+}\nabla \phi^T_{\xi+he})
- \partial_\xi A(\omega,\xi{+}\nabla \phi^T_\xi)\big)
\big(\mathds{1}_{L=1}B {+} \nabla \phi^T_{\xi+he,B}\big)
\\&~~~\nonumber
+ \sum_{\Pi}\big(\partial_\xi^{|\Pi|} A(\omega,\xi{+}he{+}\nabla \phi^T_{\xi+he})
- \partial_\xi^{|\Pi|} A(\omega,\xi{+}\nabla \phi^T_\xi)\big)
\Big[\bigodot_{\pi\in\Pi}(\mathds{1}_{|\pi|=1}B'_\pi {+} \nabla \phi^T_{\xi+he,B'_\pi})\Big]
\\&~~~\nonumber
- \sum_{\Pi}\partial_\xi^{|\Pi|} A(\omega,\xi{+}\nabla \phi^T_\xi)
\Big[\bigodot_{\pi\in\Pi}(\mathds{1}_{|\pi|=1}B'_\pi {+} \nabla \phi^T_{\xi,B'_\pi})
- \bigodot_{\pi\in\Pi}(\mathds{1}_{|\pi|=1}B'_\pi {+} \nabla \phi^T_{\xi+he,B'_\pi})\Big].
\end{align}
Since we already have established the last estimate from~\eqref{eq:indHypoCorrectorBoundsDiff}
with $B'$ replaced by $B$, we may conclude together with the argument leading to~\eqref{eq:prelimRemainderEstimate} that
\begin{align}
\label{eq:estimateDiffLinearizedFLux}
\big\langle\big\|q^T_{\xi+he,B}{-}q^T_{\xi,B}\big\|_{L^2(B_1)}^{2q}\big\rangle^\frac{1}{q}
\leq C^2q^{2C}|h|^{2(1-\beta)}.
\end{align}

For an estimate on linear functionals of the difference $\sigma_{\xi+he,B,kl}^T{-}\sigma_{\xi,B,kl}^T$,
the argument from \textit{Step 1} of the proof of Theorem~\ref{theo:correctorBoundsMassiveApprox}
applied to equation~\eqref{eq:PDEdiffLinearizedFLuxCorrectors} shows that it suffices
to have a corresponding estimate on linear functionals of the difference of linearized fluxes
$q^T_{\xi+he,B}{-}q^T_{\xi,B}$ (or more precisely, the analogue of~\eqref{eq:sensitivityBoundLinearizedFlux}
for differences with an additional rate $|h|^{2(1-\beta)}$).
However, this in turn is an immediate consequence of the
argument in \textit{Step 4} of this proof. Indeed, comparing with the right hand side
of~\eqref{eq:representationMalliavinDiffLinearizedHomCorrectors} the only additional term
which has to be dealt with in a sensitivity estimate for~\eqref{eq:linearizedFluxDiff} is given by
\begin{align*}
&\partial_\omega\partial_\xi A(\omega,\xi{+}\nabla \phi^T_\xi)\big[\delta\omega\odot
\big(\nabla\phi^T_{\xi+he,B}{-}\nabla\phi^T_{\xi,B}\big)\big]
\\&
+ \partial_\xi A(\omega,\xi{+}\nabla \phi^T_\xi)
\nabla\big(\delta\phi^T_{\xi+he,B}{-}\delta\phi^T_{\xi,B}\big).
\end{align*}
However, as we already lifted the estimates~\eqref{eq:indHypoCorrectorBoundsDiff}
and~\eqref{eq:indHypoSensitivityBoundDiff} from $B'$ to $B$, we immediately
obtain the desired sensitivity estimate on differences of linearized fluxes. 
This in turn implies the estimates
\begin{equation}
\label{eq:linFuncDiffFluxCor}
\begin{aligned}
\bigg\langle\bigg|\int g\cdot\big(\nabla\sigma^T_{\xi+he,B}{-}\nabla\sigma^T_{\xi,B}\big)
\bigg|^{2q}\bigg\rangle^\frac{1}{q}
&\leq C^2q^{2C}|h|^{2(1-\beta)}\int \big|g\big|^2,
\\
\bigg\langle\bigg|\int \frac{1}{T}f \big(\sigma^T_{\xi+he,B}{-}\sigma^T_{\xi,B}\big)\bigg|^{2q}\bigg\rangle^\frac{1}{q}
&\leq C^2q^{2C}|h|^{2(1-\beta)}\int \Big|\frac{f}{\sqrt{T}}\Big|^2.
\end{aligned}
\end{equation}
Feeding back~\eqref{eq:estimateDiffLinearizedFLux} and~\eqref{eq:linFuncDiffFluxCor}
into the analogue of~\eqref{eq:aux1ProofTheoremMassiveCorr} 
with respect to equation~\eqref{eq:PDEdiffLinearizedFLuxCorrectors}
then yields the asserted estimate~\eqref{eq:correctorBoundForDifferencesMassiveApprox}
for differences of linearized flux correctors.

This also eventually concludes the proof of Lemma~\ref{lem:differencesLinearizedCorrectors}. \qed

\subsection{Proof of Lemma~\ref{lem:diffMassiveApprox}
{\normalfont (Differentiability of massive correctors and the massive version of the homogenized operator)}}
We first consider the case of $q=1$, and argue in favor of~\eqref{eq:regGradient} 
by induction over the linearization order. The base case consisting of
the correctors of the nonlinear problem is treated in Appendix~\ref{app:baseCaseInd}
by means of Lemma~\ref{prop:estimatesRegHomCorrectorNonlinear}.
We then establish the estimates~\eqref{eq:regGradient} and~\eqref{eq:regLinearFunctionals}
for general $q\in[1,\infty)$---first for the linearized homogenization correctors and
then for linearized flux correctors---, and finally conclude with a proof of~\eqref{eq:regMassiveVersionHomOperator}.

\textit{Step 1: (Induction hypothesis)} 
Let $L\in\N$, $T\in [1,\infty)$ and $M>0$ be fixed. 
Let the requirements and notation of (A1), (A2)$_L$, (A3)$_L$ and (A4)$_L$ of
Assumption~\ref{assumption:operators}, (P1) and (P2) of
Assumption~\ref{assumption:ensembleParameterFields}, and (R) of 
Assumption~\ref{assumption:smallScaleReg} be in place.

For any $0\leq l\leq L{-}1$, any $|\xi|\leq M$, any $|h|\leq 1$, and any
collection of unit vectors $v_1',\ldots,v_l',e\in\Rd$ the first-order Taylor expansion
$$\phi^T_{\xi,B',e,h}:=\phi^T_{\xi+he,B'}{-}\phi^T_{\xi,B'}{-}\phi^T_{\xi,B'\odot e}h$$ of linearized homogenization correctors
in direction $B':=v_1'\odot\cdots\odot v_l'$ is assumed to satisfy---under the above conditions---the following estimate
(if $l=0$---and thus $B'$ being an empty symmetric tensor product---$\phi^T_{\xi,B'}$ is 
understood to denote the localized homogenization corrector $\phi^T_\xi$ of the nonlinear problem
with an additional massive term):
\begin{align}
\label{eq:indHypoTaylor1}
\tag{Hreg}
\big\langle\big\|\nabla\phi^T_{\xi,B',e,h}\big\|^2_{L^2(B_1)}
\big\rangle &\leq C^2h^{4(1-\beta)}.
\end{align}

\textit{Step 2: (Induction step)} We start by writing the equation for
the first-order Taylor expansion $\phi^T_{\xi,B,e,h}=\phi^T_{\xi+he,B}{-}\phi^T_{\xi,B}{-}\phi^T_{\xi,B\odot e}h$
in a suitable form. To this end, we first derive a suitable representation
of the first-order Taylor expansion for the linearized fluxes
$q^T_{\xi,B,e,h}=q^T_{\xi+he,B}{-}q^T_{\xi,B}{-}q^T_{\xi,B\odot e}h$. 
Note that these expressions---most importantly $\phi^T_{\xi,B\odot e}$
resp.\ $q^T_{\xi,B\odot e}$---are indeed well-defined $\Prob$-almost surely 
under the assumptions of Lemma~\ref{lem:diffMassiveApprox} thanks to 
Lemma~\ref{eq:lemmaExistenceLinearizedCorrectorsShort}. Furthermore,
for a proof of~\eqref{eq:regGradient} in case of $q=1$ we will not rely on corrector
estimates for $\phi^T_{\xi,B\odot e}$ but only on estimates for (differences
of) correctors up to linearization order $\leq L$. This is the reason why
we can stick for the moment with Assumption~\ref{assumption:operators} realized to linearization order $L$
as claimed in Lemma~\ref{lem:diffMassiveApprox}.

In view of the definition~\eqref{eq:HigherOrderLinearizedFlux} of the linearized fluxes,
we split this task into two substeps. By adding zero and abbreviating as always
$a_\xi^T:=\partial_\xi A(\omega,\xi{+}\nabla\phi^T_\xi)$, we may rewrite the contribution
from the first term on the right hand side of~\eqref{eq:HigherOrderLinearizedFlux} 
as follows
\begin{align}
\nonumber
& a_{\xi+he}^T(\mathds{1}_{L=1}B{+}\nabla\phi^T_{\xi+he,B})
- a_{\xi}^T(\mathds{1}_{L=1}B{+}\nabla\phi^T_{\xi,B}) 
- a_{\xi}^T\nabla\phi^T_{\xi,B\odot e}h
\\&\label{eq:repTaylorLinearizedFluxes1}
= a_{\xi}^T(\nabla\phi^T_{\xi+he,B}{-}\nabla\phi^T_{\xi,B}{-}\nabla\phi^T_{\xi,B\odot e}h)
+ (a_{\xi+he}^T{-}a_{\xi}^T)(\nabla\phi^T_{\xi+he,B}{-}\nabla\phi^T_{\xi,B})
\\&~~~\nonumber
+ (a_{\xi+he}^T{-}a_{\xi}^T)(\mathds{1}_{L=1}B{+}\nabla\phi^T_{\xi,B}).
\end{align}
For the contribution from the second right hand side term of~\eqref{eq:HigherOrderLinearizedFlux},
it is useful to split the sum in case of $q^T_{B\odot e}$ in the following way:
\begin{align}
\nonumber
&\sum_{\substack{\Pi\in\mathrm{Par}\{1,\ldots,L,L{+}1\} \\ \Pi\neq\{\{1,\ldots,L,L{+}1\}\}}}
\partial_\xi^{|\Pi|} A(\omega,\xi{+}\nabla \phi^T_\xi)
\Big[\bigodot_{\pi\in\Pi}(\mathds{1}_{|\pi|=1}(B\odot e)'_\pi {+} \nabla \phi^T_{\xi,(B\odot e)'_\pi})\Big]
\\&\label{eq:repTaylorLinearizedFluxes2}
= \partial_\xi a_\xi^T\big[(eh{+}\nabla\phi^T_{\xi,e}h)\big]\big(\mathds{1}_{L=1}B{+}\nabla\phi^T_{\xi,B}\big)
\\&~~~\nonumber
+ \sum_{\substack{\Pi\in\mathrm{Par}\{1,\ldots,L\} \\ \Pi\neq\{\{1,\ldots,L\}\}}}
\partial_\xi^{|\Pi|+1} A(\omega,\xi{+}\nabla \phi^T_\xi)
\Big[(eh{+}\nabla\phi^T_{\xi,e}h)
\odot\bigodot_{\pi\in\Pi}(\mathds{1}_{|\pi|=1}B'_\pi {+} \nabla \phi^T_{\xi,B'_\pi})\Big]
\\&~~~\nonumber
+ \sum_{\substack{\Pi\in\mathrm{Par}\{1,\ldots,L\} \\ \Pi\neq\{\{1,\ldots,L\}\}}}
\partial_\xi^{|\Pi|} A(\omega,\xi{+}\nabla \phi^T_\xi)
\Big[\sum_{\pi\in\Pi}\nabla\phi^T_{\xi,B'_\pi\odot e}h\odot
\bigodot_{\substack{\pi'\in\Pi \\ \pi'\neq \pi}}(\mathds{1}_{|\pi'|=1}B'_{\pi'} {+} \nabla \phi^T_{\xi,B'_{\pi'}})\Big].
\end{align}
Adding several times zero and combining terms then yields based on~\eqref{eq:repTaylorLinearizedFluxes1}
and~\eqref{eq:repTaylorLinearizedFluxes2} (where we from now on again abbreviate 
$\sum_{\Pi}:=\sum_{\Pi\in\mathrm{Par}\{1,\ldots,L\},\, \Pi\neq\{\{1,\ldots,L\}\}}$)
\begin{align}
\label{eq:repTaylorLinearizedFluxes3}
&q^T_{\xi+he,B}{-}q^T_{\xi,B}{-}q^T_{\xi,B\odot e}h = R_0 + R_1 + R_2 + R_3 + R_4,
\end{align}
with the right hand side terms being given by
\begin{align*}
R_0 &:= a_{\xi}^T(\nabla\phi^T_{\xi+he,B}{-}\nabla\phi^T_{\xi,B}{-}\nabla\phi^T_{\xi,B\odot e}h),
\\
R_1 &:= (a_{\xi+he}^T{-}a_{\xi}^T)(\nabla\phi^T_{\xi+he,B}{-}\nabla\phi^T_{\xi,B})
\\&~~~~
+ \big(a_{\xi+he}^T{-}a_{\xi}^T{-}\partial_\xi a_\xi^T\big[(eh{+}\nabla\phi^T_{\xi,e}h)\big]\big)
\big(\mathds{1}_{L=1}B{+}\nabla\phi^T_{\xi,B}\big),
\\
R_2 &:= \sum_\Pi \Big\{\partial_\xi^{|\Pi|} A(\omega,\xi{+}he{+}\nabla\phi^T_{\xi})
{-}\partial_\xi^{|\Pi|} A(\omega,\xi{+}\nabla\phi^T_{\xi})
{-}\partial_\xi^{|\Pi|+1} A(\omega,\xi{+}\nabla \phi^T_\xi)[eh]\Big\}
\\&~~~~~~~~~~~~~~~~~~~~~~~~~~~~~~~~~~~~~~~~~~~~~~~~~~~~~~~~~~~~~~\times
\Big[\bigodot_{\pi\in\Pi}(\mathds{1}_{|\pi|=1}B'_\pi {+} \nabla \phi^T_{\xi,B'_\pi})\Big],
\\
R_3 &:= \sum_\Pi \Big\{\partial_\xi^{|\Pi|} A(\omega,\xi{+}he{+}\nabla\phi^T_{\xi+he})
{-}\partial_\xi^{|\Pi|} A(\omega,\xi{+}he{+}\nabla\phi^T_{\xi})\Big\}
\\&~~~~~~~~~~~~~~~~~~~~~~~~~~~~~~~~~~~~~~~\times
\Big[\bigodot_{\pi\in\Pi}(\mathds{1}_{|\pi|=1}B'_\pi {+} \nabla \phi^T_{\xi,B'_\pi})\Big]
\\&~~~~
- \sum_\Pi \partial_\xi^{|\Pi|+1} A(\omega,\xi{+}\nabla\phi^T_{\xi}) \Big[\nabla\phi^T_{\xi,e}h
\odot\bigodot_{\pi\in\Pi}(\mathds{1}_{|\pi|=1}B'_\pi {+} \nabla \phi^T_{\xi,B'_\pi})\Big],
\end{align*}
as well as
\begin{align*}
R_4 &:= \sum_\Pi \partial_\xi^{|\Pi|} A(\omega,\xi{+}he{+}\nabla\phi^T_{\xi+he})
\\&~~~~~~~~~\times
\Big[\bigodot_{\pi\in\Pi}(\mathds{1}_{|\pi|=1}B'_\pi {+} \nabla \phi^T_{\xi+he,B'_\pi})
{-}\bigodot_{\pi\in\Pi}(\mathds{1}_{|\pi|=1}B'_\pi {+} \nabla \phi^T_{\xi,B'_\pi})\Big]
\\&~~~~
-  \sum_{\Pi} \partial_\xi^{|\Pi|} A(\omega,\xi{+}\nabla \phi^T_\xi)
\Big[\sum_{\pi\in\Pi}\nabla\phi^T_{\xi,B'_\pi\odot e}h\odot
\bigodot_{\substack{\pi'\in\Pi \\ \pi'\neq \pi}}(\mathds{1}_{|\pi'|=1}B'_{\pi'} 
{+} \nabla \phi^T_{\xi,B'_{\pi'}})\Big].
\end{align*}
In particular, we obtain the following equation for the first-order Taylor expansion 
$\phi^T_{\xi,B,e,h}=\phi^T_{\xi+he,B}{-}\phi^T_{\xi,B}{-}\phi^T_{\xi,B\odot e}h$
of linearized homogenization correctors
\begin{align*}
&\frac{1}{T}(\phi^T_{\xi+he,B}{-}\phi^T_{\xi,B}{-}\phi^T_{\xi,B\odot e}h)
- \nabla\cdot a_\xi^T(\nabla\phi^T_{\xi+he,B}{-}\nabla\phi^T_{\xi,B}{-}\nabla\phi^T_{\xi,B\odot e}h)
= \nabla\cdot \sum_{i=1}^4 R_i.
\end{align*}
Applying the weighted energy estimate~\eqref{eq:expLocalization} to the equation
from the previous display then yields the bound
\begin{align*}
&\int \ell_{\gamma,\sqrt{T}} \Big|\Big(\frac{\phi^T_{\xi+he,B}{-}\phi^T_{\xi,B}{-}\phi^T_{\xi,B\odot e}h}{\sqrt{T}},
\nabla\phi^T_{\xi+he,B}{-}\nabla\phi^T_{\xi,B}{-}\nabla\phi^T_{\xi,B\odot e}h\Big)\Big|^2
\\&
\leq C^2\sup_{i=1,\ldots,4} \int \ell_{\gamma,\sqrt{T}} |R_i|^2.
\end{align*}
By stationarity of the linearized homogenization correctors, we may take the expected value
in the latter estimate and infer
\begin{equation}
\label{eq:regIndStepAux0}
\begin{aligned}
&\Big\langle\Big\|\Big(\frac{\phi^T_{\xi+he,B}{-}\phi^T_{\xi,B}{-}\phi^T_{\xi,B\odot e}h}{\sqrt{T}},
\nabla\phi^T_{\xi+he,B}{-}\nabla\phi^T_{\xi,B}{-}\nabla\phi^T_{\xi,B\odot e}h\Big)\Big\|_{L^2(B_1)}^2\Big\rangle
\\&
\leq C^2\sup_{i=1,\ldots,4}\big\langle\big\|R_i\big\|_{L^2(B_1)}^2\big\rangle.
\end{aligned}
\end{equation}
It remains to post-process the four right hand side terms of the previous display.

\textit{Estimate for $R_1$:} We first rewrite
\begin{equation}
\label{eq:regIndStepAux1}
\begin{aligned}
a_{\xi+he}^T - a_{\xi}^T 
&= \partial_\xi A(\omega,\xi{+}he{+}\nabla\phi^T_{\xi+he})
- \partial_\xi A(\omega,\xi{+}he{+}\nabla\phi^T_{\xi})
\\&~~~
+ \partial_\xi A(\omega,\xi{+}he{+}\nabla\phi^T_{\xi})
- \partial_\xi A(\omega,\xi{+}\nabla\phi^T_{\xi}).
\end{aligned}
\end{equation}
By means of~(A2)$_{L}$ from Assumption~\ref{assumption:operators}, the previous display in particular entails
\begin{align}
\nonumber
&a_{\xi+he}^T{-}a_{\xi}^T{-}\partial_\xi a_\xi^T\big[(eh{+}\nabla\phi^T_{\xi,e}h)\big]
\\&\label{eq:regIndStepAux2}
= \int_0^1 \Big\{\partial_\xi^2 A(\omega,\xi{+}\nabla\phi^T_{\xi}{+}she)
{-}\partial_\xi^{2} A(\omega,\xi{+}\nabla \phi^T_\xi)\big\}[eh] \ds 
\\&~~~\nonumber
+\int_0^1 \partial_\xi^2 A\big(\omega,\xi{+}he{+}s\nabla\phi^T_{\xi+he}{+}(1{-}s)\nabla\phi^T_{\xi}\big)
\big[\nabla\phi^T_{\xi+he}{-}\nabla\phi^T_{\xi}{-}\nabla\phi^T_{\xi,e}h\big] \ds 
\\&~~~\nonumber
+\int_0^1 \Big\{\partial_\xi^2 A\big(\omega,\xi{+}he{+}s\nabla\phi^T_{\xi+he}{+}(1{-}s)\nabla\phi^T_{\xi}\big)
{-}\partial_\xi^{2} A(\omega,\xi{+}\nabla \phi^T_\xi)\Big\} \big[\nabla\phi^T_{\xi,e}h\big] \ds.
\end{align}
Hence, it follows from~\eqref{eq:regIndStepAux1} and~\eqref{eq:regIndStepAux2}, 
the induction hypothesis~\eqref{eq:indHypoTaylor1}, the 
corrector estimates~\eqref{eq:correctorGradientBoundMassiveApprox} and~\eqref{eq:correctorSchauderMassiveApprox},
the corrector estimates for differences~\eqref{eq:correctorBoundForDifferencesMassiveApprox} 
and~\eqref{eq:SchauderEstimatesDifferencesMassiveApprox}, and~(A2)$_{L}$ as
well as~(A4)$_L$ from Assumption~\ref{assumption:operators} that
\begin{align}
\label{eq:estimateR1}
\big\langle\big\|R_1\big\|_{L^2(B_1)}^2\big\rangle
\leq C^2h^{4(1-\beta)}.
\end{align}

\textit{Estimate for $R_2$:} It is a simple consequence of~(A2)$_{L}$ from Assumption~\ref{assumption:operators} that
\begin{equation}
\label{eq:regIndStepAux3}
\begin{aligned}
&\partial_\xi^{|\Pi|} A(\omega,\xi{+}he{+}\nabla\phi^T_{\xi})
{-}\partial_\xi^{|\Pi|} A(\omega,\xi{+}\nabla\phi^T_{\xi})
{-}\partial_\xi^{|\Pi|+1} A(\omega,\xi{+}\nabla \phi^T_\xi)[eh]
\\&
= \int_0^1 \Big\{\partial_\xi^{|\Pi|+1} A(\omega,\xi{+}\nabla \phi^T_\xi{+}she)
{-}\partial_\xi^{|\Pi|+1} A(\omega,\xi{+}\nabla \phi^T_\xi)\Big\}[eh] \ds.
\end{aligned}
\end{equation}
It thus follows from~\eqref{eq:regIndStepAux3}, H\"older's inequality, the 
corrector estimates~\eqref{eq:correctorGradientBoundMassiveApprox} and~\eqref{eq:correctorSchauderMassiveApprox}, 
and~(A2)$_{L}$ as well as~(A4)$_L$ from Assumption~\ref{assumption:operators} that
\begin{align}
\label{eq:estimateR2}
\big\langle\big\|R_2\big\|_{L^2(B_1)}^2\big\rangle
\leq C^2h^{4}.
\end{align}

\textit{Estimate for $R_3$:} We first express $R_3$ in equivalent form as follows:
\begin{align*}
R_3 &= - \sum_\Pi \Big\{ \partial_\xi^{|\Pi|+1} A(\omega,\xi{+}\nabla\phi^T_{\xi})
{-} \partial_\xi^{|\Pi|+1} A(\omega,\xi{+}\nabla\phi^T_{\xi}{+}he) \Big\} 
\\&~~~~~~~~~~~~~~~~~~~~~~~~~~\times
\Big[\nabla\phi^T_{\xi,e}h
\odot\bigodot_{\pi\in\Pi}(\mathds{1}_{|\pi|=1}B'_\pi {+} \nabla \phi^T_{\xi,B'_\pi})\Big] 
\\&~~~
+ \sum_\Pi \Big\{\partial_\xi^{|\Pi|} A(\omega,\xi{+}he{+}\nabla\phi^T_{\xi+he})
{-}\partial_\xi^{|\Pi|} A(\omega,\xi{+}he{+}\nabla\phi^T_{\xi})
\\&~~~~~~~~~~~~~~
{-}\partial_\xi^{|\Pi|+1} A(\omega,\xi{+}he{+}\nabla\phi^T_{\xi})\big[\nabla\phi^T_{\xi,e}h\big]\Big\}
\Big[\bigodot_{\pi\in\Pi}(\mathds{1}_{|\pi|=1}B'_\pi {+} \nabla \phi^T_{\xi,B'_\pi})\Big].
\end{align*}
We also have thanks to~(A2)$_{L}$ from Assumption~\ref{assumption:operators}
\begin{align}
\nonumber
&\partial_\xi^{|\Pi|} A(\omega,\xi{+}he{+}\nabla\phi^T_{\xi+he})
{-}\partial_\xi^{|\Pi|} A(\omega,\xi{+}he{+}\nabla\phi^T_{\xi})
{-}\partial_\xi^{|\Pi|+1} A(\omega,\xi{+}he{+}\nabla\phi^T_{\xi})
\big[\nabla\phi^T_{\xi,e}h\big]
\\&\label{eq:regIndStepAux4}
= \int_0^1 \partial_\xi^{|\Pi|+1} A(\omega,\xi{+}he{+}s\nabla\phi^T_{\xi+he}{+}(1{-}s)\nabla\phi^T_\xi)
\big[\nabla\phi^T_{\xi+he}{-}\nabla\phi^T_{\xi}{-}\nabla\phi^T_{\xi,e}h\big] \ds
\\&~~~\nonumber
+ \int_0^1 \Big\{\partial_\xi^{|\Pi|+1} A(\omega,\xi{+}he{+}s\nabla\phi^T_{\xi+he}{+}(1{-}s)\nabla\phi^T_\xi)
{-}\partial_\xi^{|\Pi|+1} A(\omega,\xi{+}he{+}\nabla\phi^T_{\xi})\Big\}
\big[\nabla\phi^T_{\xi,e}h\big] \ds.
\end{align}
Hence, it follows from~\eqref{eq:regIndStepAux4}, 
the induction hypothesis~\eqref{eq:indHypoTaylor1}, an application of H\"older's inequality, the 
corrector estimates~\eqref{eq:correctorGradientBoundMassiveApprox} and~\eqref{eq:correctorSchauderMassiveApprox},
the corrector estimates for differences~\eqref{eq:correctorBoundForDifferencesMassiveApprox} 
and~\eqref{eq:SchauderEstimatesDifferencesMassiveApprox}, and~(A2)$_{L}$ as
well as~(A4)$_L$ from Assumption~\ref{assumption:operators} that
\begin{align}
\label{eq:estimateR3}
\big\langle\big\|R_3\big\|_{L^2(B_1)}^2\big\rangle
\leq C^2h^{4(1-\beta)}.
\end{align}

\textit{Estimate for $R_4$:} By adding zero, we may decompose
$R_4 = R_4' + R_4''$ with
\begin{align*}
R_4' &:=  \sum_\Pi \partial_\xi^{|\Pi|} A(\omega,\xi{+}\nabla\phi^T_{\xi})
\Big[\bigodot_{\pi\in\Pi}(\mathds{1}_{|\pi|=1}B'_\pi {+} \nabla \phi^T_{\xi+he,B'_\pi})
{-}\bigodot_{\pi\in\Pi}(\mathds{1}_{|\pi|=1}B'_\pi {+} \nabla \phi^T_{\xi,B'_\pi})\Big]
\\&~~~~
-  \sum_{\Pi} \partial_\xi^{|\Pi|} A(\omega,\xi{+}\nabla \phi^T_\xi)
\Big[\sum_{\pi\in\Pi}\nabla\phi^T_{\xi,B'_\pi\odot e}h\odot
\bigodot_{\substack{\pi'\in\Pi \\ \pi'\neq \pi}}(\mathds{1}_{|\pi'|=1}B'_{\pi'} 
{+} \nabla \phi^T_{\xi,B'_{\pi'}})\Big],
\end{align*}
and where $R_4''$ can be treated by the arguments from the previous items.
Thanks to~(A2)$_{L}$ from Assumption~\ref{assumption:operators} and the Leibniz rule for differences
we have the bound
\begin{equation}
\label{eq:regIndStepAux5}
\begin{aligned}
|R_4'|^2 &\leq C^2\sup_{\Pi}\sup_{\substack{\pi,\pi'\in\Pi \\ \pi\neq\pi'}}
\big|\nabla \phi^T_{\xi+he,B'_\pi}{-}\nabla \phi^T_{\xi,B'_\pi}\big|^2
\big|\nabla \phi^T_{\xi+he,B'_{\pi'}}{-}\nabla \phi^T_{\xi,B'_{\pi'}}\big|^2
\\&~~~~~~~~~~~~\times
\bigg\{1 {+} \prod_{\substack{\pi''\in\Pi \\ \pi''\notin\{\pi,\pi'\}}}
\Big(\big|\nabla \phi^T_{\xi,B'_{\pi''}}\big|^2{+}\big|\nabla \phi^T_{\xi+he,B'_{\pi''}}\big|^2\Big)\bigg\}
\\&~~~
+ C^2\sup_{\Pi}\sup_{\pi\in\Pi}
\big|\nabla \phi^T_{\xi+he,B'_\pi}{-}\nabla \phi^T_{\xi,B'_\pi}{-}\nabla\phi^T_{\xi,B'_\pi\odot e}h\big|^2
\\&~~~~~~~~~~~~\times
\bigg\{1 {+} \prod_{\substack{\pi'\in\Pi \\ \pi'\neq\pi}}
\Big(\big|\nabla \phi^T_{\xi,B'_{\pi'}}\big|^2{+}\big|\nabla \phi^T_{\xi+he,B'_{\pi'}}\big|^2\Big)\bigg\}.
\end{aligned}
\end{equation}
Hence, it follows as a combination of~\eqref{eq:regIndStepAux5}, 
the induction hypothesis~\eqref{eq:indHypoTaylor1}, H\"older's inequality, the 
corrector estimates~\eqref{eq:correctorGradientBoundMassiveApprox} and~\eqref{eq:correctorSchauderMassiveApprox},
as well as the corrector estimates for differences~\eqref{eq:correctorBoundForDifferencesMassiveApprox} 
and~\eqref{eq:SchauderEstimatesDifferencesMassiveApprox} that
\begin{align}
\label{eq:estimateR4}
\big\langle\big\|R_4\big\|_{L^2(B_1)}^2\big\rangle
\leq C^2h^{4(1-\beta)}.
\end{align}

Inserting the estimates~\eqref{eq:estimateR1}, \eqref{eq:estimateR2},
\eqref{eq:estimateR3}, and \eqref{eq:estimateR4} back into~\eqref{eq:regIndStepAux0}
then finally entails the bound
\begin{align*}
\big\langle\big\|\nabla\phi^T_{\xi+he,B}{-}\nabla\phi^T_{\xi,B}
{-}\nabla\phi^T_{\xi,B\odot e}h\big\|_{L^2(B_1)}^2\big\rangle
\leq C^2h^{4(1-\beta)}.
\end{align*}
This is the asserted estimate~\eqref{eq:regGradient} on the level
of the linearized homogenization corrector in the case of $q=1$.
In particular, the map $\xi\mapsto\nabla\phi^T_{\xi,B}$ is 
Fr\'echet differentiable with values in the Fr\'echet space $L^2_{\langle\cdot\rangle}L^2_{\mathrm{loc}}(\Rd)$.
Note that as a consequence of~\eqref{eq:repTaylorLinearizedFluxes3} we then also get the estimate
\begin{align}
\label{eq:qualDiffLinearizedFlux}
\big\langle\big\|q^T_{\xi+he,B}{-}q^T_{\xi,B}
{-}q^T_{\xi,B\odot e}h\big\|_{L^2(B_1)}^2\big\rangle
\leq C^2h^{4(1-\beta)}.
\end{align}
In particular, the map  $\xi\mapsto q^T_{\xi,B}$ is also
Fr\'echet differentiable with values in the Fr\'echet space $L^2_{\langle\cdot\rangle}L^2_{\mathrm{loc}}(\Rd)$.

\textit{Step 3: (Proof of the estimates~\eqref{eq:regGradient} and~\eqref{eq:regLinearFunctionals} 
for general $q\in[1,\infty)$)} As we already established qualitative differentiability
of the map $\xi\mapsto\nabla\phi^T_{\xi,B}$ in the Fr\'echet space $L^2_{\langle\cdot\rangle}L^2_{\mathrm{loc}}(\Rd)$,
we may estimate based on the corrector estimate for differences~\eqref{eq:correctorBoundForDifferencesMassiveApprox}
\begin{align*}
&\big\langle\big\|\nabla\phi^T_{\xi+he,B}{-}\nabla\phi^T_{\xi,B}
{-}\nabla\phi^T_{\xi,B\odot e}h\big\|_{L^2(B_1)}^{2q}\big\rangle^\frac{1}{2q}
\\&
\leq \int_0^1 \big\langle\big\|\nabla\phi^T_{\xi+she,B\odot e}h
{-}\nabla\phi^T_{\xi,B\odot e}h\big\|_{L^2(B_1)}^{2q}\big\rangle^\frac{1}{2q} \ds
\leq Ch^{2(1-\beta)},
\end{align*}
which is precisely the asserted bound~\eqref{eq:regGradient}. Based on the corrector estimate for 
differences~\eqref{eq:LinearFunctionalEstimatesDifferencesMassiveApprox}, the 
estimate~\eqref{eq:regLinearFunctionals} is derived analogously. 

\textit{Step 4: (Proof of the estimates~\eqref{eq:regGradient} and~\eqref{eq:regLinearFunctionals} 
for linearized flux correctors)} The equation for the first-order Taylor expansion 
$\sigma^T_{\xi+he,B}{-}\sigma^T_{\xi,B}{-}\sigma^T_{\xi,B\odot e}h$
of linearized flux correctors is simply given by
\begin{align*}
&\frac{1}{T}(\sigma^T_{\xi+he,B}{-}\sigma^T_{\xi,B}{-}\sigma^T_{\xi,B\odot e}h)
- \Delta(\sigma^T_{\xi+he,B}{-}\sigma^T_{\xi,B}{-}\sigma^T_{\xi,B\odot e}h)
\\&
= - \nabla\cdot \big((e_l\otimes e_k - e_k\otimes e_l)
(q^T_{\xi+he,B}{-}q^T_{\xi,B}{-}q^T_{\xi,B\odot e}h)\big).
\end{align*}
By an application of the weighted energy estimate~\eqref{eq:expLocalization},
the stationarity of linearized flux correctors and linearized fluxes, 
and the estimate~\eqref{eq:qualDiffLinearizedFlux} we obtain
\begin{align*}
\big\langle\big\|\nabla\sigma^T_{\xi+he,B}{-}\nabla\sigma^T_{\xi,B}
{-}\nabla\sigma^T_{\xi,B\odot e}h\big\|_{L^2(B_1)}^2\big\rangle
\leq C^2h^{4(1-\beta)}.
\end{align*}
In particular, the map $\xi\mapsto\nabla\sigma^T_{\xi,B}$ is 
Fr\'echet differentiable with values in the Fr\'echet space $L^2_{\langle\cdot\rangle}L^2_{\mathrm{loc}}(\Rd)$.
Based on the corrector estimates for differences~\eqref{eq:correctorBoundForDifferencesMassiveApprox}
and~\eqref{eq:LinearFunctionalEstimatesDifferencesMassiveApprox}, we then
infer along the same lines as in \textit{Step 3} of this proof that the asserted
estimates~\eqref{eq:regGradient} and~\eqref{eq:regLinearFunctionals} indeed hold true
for the linearized flux correctors.

\textit{Step 5: (Proof of the estimate~\eqref{eq:regMassiveVersionHomOperator})}
As this is an immediate consequence of the estimate~\eqref{eq:qualDiffLinearizedFlux}
in combination with the stationarity of the linearized fluxes, we may now conclude
the proof of Lemma~\ref{lem:diffMassiveApprox}. \qed

\subsection{Proof of Lemma~\ref{lem:limitMassiveApprox}
{\normalfont (Limit passage in the massive approximation)}}
We again argue by induction over the linearization order. For the base case consisting of
the correctors of the nonlinear problem we refer to Lemma~\ref{prop:estimatesLimitHomCorrectorNonlinear}
in Appendix~\ref{app:baseCaseInd}.

\textit{Step 1: (Induction hypothesis)} 
Let $L\in\N$, $T\in [1,\infty)$ and $M>0$ be fixed. 
Let the requirements and notation of (A1), (A2)$_L$ and (A3)$_L$ of
Assumption~\ref{assumption:operators}, (P1) and (P2) of
Assumption~\ref{assumption:ensembleParameterFields}, and (R) of 
Assumption~\ref{assumption:smallScaleReg} be in place.

For any $0\leq l\leq L{-}1$, any $|\xi|\leq M$, any $|h|\leq 1$,
any $T\in[1,\infty)$ and any collection of unit vectors $v_1',\ldots,v_l'\in\Rd$ the difference
$\phi^{2T}_{\xi,B'}{-}\phi^T_{\xi,B'}$ of linearized homogenization correctors
in direction $B':=v_1'\odot\cdots\odot v_l'$ is assumed to satisfy---under the above conditions---the following estimate
(if $l=0$---and thus $B'$ being an empty symmetric tensor product---$\phi^T_{\xi,B'}$ is 
understood to denote the localized homogenization corrector $\phi^T_\xi$ of the nonlinear problem
with a massive term):
\begin{align}
\label{eq:indHypoConv1}
\tag{Hconv}
\big\langle\big\|\nabla\phi^{2T}_{\xi,B'} - \nabla\phi^T_{\xi,B'}
\big\|^{2}_{L^2(B_1)}\big\rangle
&\leq C^2\frac{\mu_*^2(\sqrt{T})}{T}.
\end{align}

\textit{Step 2: (Induction step)} The difference $\phi^{2T}_{\xi,B}{-}\phi^T_{\xi,B}$ 
of linearized homogenization correctors is subject to the equation
\begin{align}
\nonumber
&\frac{1}{2T}(\phi^{2T}_{\xi,B}{-}\phi^T_{\xi,B})
-\nabla\cdot a_\xi^{2T}(\nabla\phi^{2T}_{\xi,B}{-}\nabla\phi^T_{\xi,B})
\\&\label{eq:PDEdiffConvergence}
= \frac{1}{2T}\phi^T_{\xi,B} - \nabla\cdot
\big(\partial_\xi A(\omega,\xi{+}\nabla \phi^{T}_\xi){-}
\partial_\xi A(\omega,\xi{+}\nabla \phi^{2T}_\xi)\big)
(\mathds{1}_{L=1}B + \nabla\phi^T_{\xi,B})
\\&~~~\nonumber
+ \nabla\cdot\sum_{\Pi}
\big(\partial_\xi^{|\Pi|} A(\omega,\xi{+}\nabla \phi^{2T}_\xi){-}
\partial_\xi^{|\Pi|} A(\omega,\xi{+}\nabla \phi^{T}_\xi)\big)
\Big[\bigodot_{\pi\in\Pi}(\mathds{1}_{|\pi|=1}B'_\pi {+} \nabla \phi^{2T}_{\xi,B'_\pi})\Big]
\\&~~~\nonumber
+ \nabla\cdot\sum_{\Pi}
\partial_\xi^{|\Pi|} A(\omega,\xi{+}\nabla \phi^{T}_\xi)
\Big[\bigodot_{\pi\in\Pi}(\mathds{1}_{|\pi|=1}B'_\pi {+} \nabla \phi^{2T}_{\xi,B'_\pi})
{-}\bigodot_{\pi\in\Pi}(\mathds{1}_{|\pi|=1}B'_\pi {+} \nabla \phi^{T}_{\xi,B'_\pi})\Big].
\end{align}
Here, we made use of the abbreviation $\sum_{\Pi}:=\sum_{\Pi\in\mathrm{Par}\{1,\ldots,L\},\, 
\Pi\neq\{\{1,\ldots,L\}\}}$. Applying the weighted energy estimate~\eqref{eq:expLocalization}
to equation~\eqref{eq:PDEdiffConvergence} entails by means of~(A2)$_{L}$ from Assumption~\ref{assumption:operators}
\begin{align*}
&\int \ell_{\gamma,\sqrt{T}} \Big|\Big(\frac{\phi^{2T}_{\xi,B}{-}\phi^T_{\xi,B}}{\sqrt{2T}},
\nabla\phi^{2T}_{\xi,B}{-}\nabla\phi^T_{\xi,B}\Big)\Big|^2
\\&
\leq C^2\int \ell_{\gamma,\sqrt{T}} \frac{1}{2T}\big|\phi^T_{\xi,B}\big|^2
+ C^2\int \ell_{\gamma,\sqrt{T}} \big|\nabla\phi^{2T}_\xi{-}\nabla\phi^T_\xi\big|^2
\big|\mathds{1}_{L=1}B + \nabla\phi^T_{\xi,B}\big|^2
\\&~~~
+ C^2 \sum_\Pi\int \ell_{\gamma,\sqrt{T}} \Big|\nabla\phi^{2T}_\xi{-}\nabla\phi^T_\xi\Big|^2
\Big|\bigodot_{\pi\in\Pi}(\mathds{1}_{|\pi|=1}B'_\pi {+} \nabla \phi^{2T}_{\xi,B'_\pi})\Big|^2
\\&~~~
+ C^2 \sum_\Pi\int \ell_{\gamma,\sqrt{T}}
\Big|\bigodot_{\pi\in\Pi}(\mathds{1}_{|\pi|=1}B'_\pi {+} \nabla \phi^{2T}_{\xi,B'_\pi})
{-}\bigodot_{\pi\in\Pi}(\mathds{1}_{|\pi|=1}B'_\pi {+} \nabla \phi^{T}_{\xi,B'_\pi})\Big|^2.
\end{align*}
Taking expectation in the previous display, exploiting stationarity of the linearized homogenization correctors,
adding zero, and applying H\"older's and Poincar\'e's inequalities then yields
\begin{align*}
&\big\langle\big\|\nabla\phi^{2T}_{\xi,B} - \nabla\phi^T_{\xi,B}
\big\|^{2}_{L^2(B_1)}\big\rangle
\\&
\leq C^2\frac{1}{T}\big\langle\big\|\nabla\phi^T_{\xi,B}
\big\|^{2}_{L^2(B_1)}\big\rangle
+ C^2\frac{1}{T}\bigg\langle\bigg|\,\dashint_{B_1}\phi^T_{\xi,B}\bigg|^2\bigg\rangle
\\&~~~
+ C^2\big\langle\big\|\mathds{1}_{L=1}B + \nabla\phi^T_{\xi,B}
\big\|^{2}_{C^\alpha(B_1)}\big\rangle
\big\langle\big\|\nabla\phi^{2T}_{\xi} - \nabla\phi^T_{\xi}
\big\|^{2}_{L^2(B_1)}\big\rangle
\\&~~~
+ C^2\sum_{\Pi}\prod_{\pi\in\Pi}\big\langle\big\|
\mathds{1}_{|\pi|=1}B'_\pi {+} \nabla \phi^{2T}_{\xi,B'_\pi}
\big\|^{2|\Pi|}_{C^\alpha(B_1)}\big\rangle^\frac{1}{|\Pi|}
\big\langle\big\|\nabla\phi^{2T}_{\xi} - \nabla\phi^T_{\xi}
\big\|^{2}_{L^2(B_1)}\big\rangle
\\&~~~
+ C^2\sum_{\Pi}\sup_{\pi\in\Pi}
\big\langle\big\|\nabla\phi^{2T}_{\xi,B'_\pi} - \nabla\phi^T_{\xi,B'_\pi}
\big\|^{2}_{L^2(B_1)}\big\rangle
\\&~~~~~~~~~~~~~~~\times
\sup_{\substack{\pi'\in\Pi \\ \pi'\neq \pi}}
\Big\{1+\big\langle\big\|\nabla \phi^{2T}_{\xi,B'_{\pi'}}
\big\|^{2|\Pi|}_{C^\alpha(B_1)}\big\rangle^\frac{1}{|\Pi|}
+\big\langle\big\|\nabla \phi^{T}_{\xi,B'_{\pi'}}
\big\|^{2|\Pi|}_{C^\alpha(B_1)}\big\rangle^\frac{1}{|\Pi|}\Big\}.
\end{align*}
It is thus a consequence of the induction hypothesis~\eqref{eq:indHypoConv1}
and the corrector 
estimates~\eqref{eq:correctorGradientBoundMassiveApprox}--\eqref{eq:correctorGrowthoundMassiveApprox} that
\begin{align}
\label{eq:indStepConv1}
\big\langle\big\|\nabla\phi^{2T}_{\xi,B} - \nabla\phi^T_{\xi,B}
\big\|^{2}_{L^2(B_1)}\big\rangle
&\leq C^2\frac{\mu_*^2(\sqrt{T})}{T},
\end{align}
which concludes the induction step, and in particular establishes the asserted
estimate~\eqref{eq:convLinearizedCorrectors}.

\textit{Step 3: (Estimates for linearized flux correctors and massive version of homogenized operator)}
The difference $\sigma^{2T}_{\xi,B}{-}\sigma^T_{\xi,B}$ 
of linearized flux correctors is subject to the equation
\begin{align}
\nonumber
&\frac{1}{2T}(\sigma^{2T}_{\xi,B,kl} - \sigma^T_{\xi,B,kl})
- \Delta(\sigma^{2T}_{\xi,B,kl} - \sigma^T_{\xi,B,kl})
\\&\label{eq:PDEdiffConvergenceFlux}
= \frac{1}{2T}\sigma^T_{\xi,B,kl} 
- \nabla\cdot \big((e_l\otimes e_k - e_k\otimes e_l)
(q^{2T}_{\xi,B} - q^T_{\xi,B})\big).
\end{align}
Applying the weighted energy estimate~\eqref{eq:expLocalization}
to equation~\eqref{eq:PDEdiffConvergenceFlux} thus yields
\begin{align*}
&\int \ell_{\gamma,\sqrt{T}} \Big|\Big(\frac{\sigma^{2T}_{\xi,B,kl} - \sigma^T_{\xi,B,kl}}{\sqrt{2T}},
\nabla\sigma^{2T}_{\xi,B,kl} - \nabla\sigma^T_{\xi,B,kl}\Big)\Big|^2
\\&
\leq C^2\int \ell_{\gamma,\sqrt{T}} \frac{1}{2T}\big|\sigma^T_{\xi,B,kl}\big|^2
+ C^2 \int \ell_{\gamma,\sqrt{T}} \big|q^{2T}_{\xi,B} - q^T_{\xi,B}\big|^2.
\end{align*}
Taking expectation in the previous display, exploiting stationarity of the linearized flux correctors,
adding zero, and applying the Poincar\'e inequality entails the estimate
\begin{align*}
&\big\langle\big\|\nabla\sigma^{2T}_{\xi,B,kl} - \nabla\sigma^T_{\xi,B,kl}
\big\|^{2}_{L^2(B_1)}\big\rangle
\\&
\leq C^2\frac{1}{T}\big\langle\big\|\nabla\sigma^T_{\xi,B,kl}
\big\|^{2}_{L^2(B_1)}\big\rangle
+ C^2\frac{1}{T}\bigg\langle\bigg|\,\dashint_{B_1}\sigma^T_{\xi,B,kl}\bigg|^2\bigg\rangle
+ C^2\big\langle\big\|q^{2T}_{\xi,B} - q^T_{\xi,B}\big\|^{2}_{L^2(B_1)}\big\rangle.
\end{align*}
By means of the corrector estimates~\eqref{eq:correctorGradientBoundMassiveApprox}
and~\eqref{eq:correctorGrowthoundMassiveApprox} the previous display updates to
\begin{align}
\label{eq:convLinFluxCorrector}
&\big\langle\big\|\nabla\sigma^{2T}_{\xi,B,kl} - \nabla\sigma^T_{\xi,B,kl}
\big\|^{2}_{L^2(B_1)}\big\rangle
\leq C^2\frac{\mu_*^2(\sqrt{T})}{T}
+ C^2\big\langle\big\|q^{2T}_{\xi,B} - q^T_{\xi,B}\big\|^{2}_{L^2(B_1)}\big\rangle.
\end{align}
It remains to provide an estimate for the difference
$q^{2T}_{\xi,B} - q^T_{\xi,B}$ of linearized fluxes.
The majority of the required work is already done since
we already provided an estimate of required type for
the divergence term on the right hand side of~\eqref{eq:PDEdiffConvergence}.
By definition~\eqref{eq:HigherOrderLinearizedFlux} of the linearized flux, 
it thus suffices to note that by~\eqref{eq:indStepConv1} and~(A2)$_{L}$
from Assumption~\ref{assumption:operators} we obtain
\begin{equation}
\label{eq:convLinFlux}
\begin{aligned}
\big\langle\big\|q^{2T}_{\xi,B} - q^T_{\xi,B}\big\|^{2}_{L^2(B_1)}\big\rangle
&\leq C^2\big\langle\big\|\nabla\phi^{2T}_{\xi,B} - \nabla\phi^T_{\xi,B}
\big\|^{2}_{L^2(B_1)}\big\rangle
+ C^2\frac{\mu_*^2(\sqrt{T})}{T}
\\&
\leq C^2\frac{\mu_*^2(\sqrt{T})}{T}.
\end{aligned}
\end{equation}
Plugging this estimate back into~\eqref{eq:convLinFluxCorrector} therefore yields 
the asserted estimate~\eqref{eq:convLinearizedCorrectors} on the level of linearized flux correctors.

The estimate~\eqref{eq:convHomOperator} is an immediate consequence of~\eqref{eq:convLinFlux}
and the stationarity of the linearized flux.

\textit{Step 4: (Conclusion)} As a consequence of~\eqref{eq:convLinearizedCorrectors},
there exist stationary gradient fields 
\begin{align*}
(\nabla\phi_{\xi,B},\nabla\sigma_{\xi,B})\in 
L^2_{\mathrm{loc}}(\Rd;\Rd) {\times} L^2_{\mathrm{loc}}(\Rd;\mathbb{R}^{d\times d}_{\mathrm{skew}}{\times}\Rd)
\end{align*}
with vanishing expectation, finite second moments and being subject to the anchoring
$\dashint_{B_1}\phi_{\xi,B}=0$ resp.\ $\dashint_{B_1}\sigma_{\xi,B}=0$ such that 
\begin{align}
\label{eq:convStatement}
(\nabla\phi_{\xi,B}^T,\nabla\sigma_{\xi,B,kl}^T) \to
(\nabla\phi_{\xi,B},\nabla\sigma_{\xi,B,kl})
\quad\text{ as } T\to\infty,
\end{align}
strongly in $L^2_{\langle\cdot\rangle}L^2_{\mathrm{loc}}(\Rd;\Rd)$.

For any $x_0\in\Rd$, let 
$f_{x_0}:=\frac{1}{|B_1|}\mathds{1}_{B_1}{-}\frac{1}{|B_1(x_0)|}\mathds{1}_{B_1(x_0)}$.
We then take $v$ to be the solution of the Neumann problem for Poisson's equation 
$\Delta v = f_{x_0}$ in the ball $B_{1{+}|x_0|}$. This in turn enables us to
provide a solution of $\nabla\cdot g_{x_0} = f_{x_0}$ in $\Rd$ by means
of $g_{x_0}:=\mathds{1}_{B_{1+|x_0|}}\nabla v$. Furthermore, it holds
$\int_{B_{1{+}|x_0|}(0)}|g_{x_0}|^2\dx \leq C\mu_*^2(1{+}|x_0|)$.
We may then estimate by adding zero, exploiting stationarity,
and applying Poincar\'e's inequality and Fatou's lemma
\begin{align*}
\bigg\langle\,\dashint_{B_1(x_0)}\big|\phi_{\xi,B}\big|^2\bigg\rangle^\frac{1}{2}
&\leq \bigg\langle\,\dashint_{B_1(x_0)}\big|\nabla\phi_{\xi,B}\big|^2\bigg\rangle^\frac{1}{2}
+ \bigg\langle\bigg|\,\dashint_{B_1(x_0)}\phi_{\xi,B}
-\dashint_{B_1}\phi_{\xi,B}\bigg|^2\bigg\rangle^\frac{1}{2}
\\
&= \bigg\langle\,\dashint_{B_1(x_0)}\big|\nabla\phi_{\xi,B}\big|^2\bigg\rangle^\frac{1}{2} 
+ \bigg\langle\bigg|\int_{B_{1+|x_0|}} g_{x_0}\cdot\nabla\phi_{\xi,B}\bigg|^2\bigg\rangle^\frac{1}{2}
\\
&\leq \liminf_{T\to\infty} \bigg\langle\,\dashint_{B_1(x_0)}\big|\nabla\phi^T_{\xi,B}\big|^2\bigg\rangle^\frac{1}{2} 
+ \bigg\langle\bigg|\int_{B_{1+|x_0|}} g_{x_0}\cdot\nabla\phi^T_{\xi,B}\bigg|^2\bigg\rangle^\frac{1}{2}.
\end{align*}
By the corrector estimates~\eqref{eq:linearFunctionalCorrectorGradientBoundMassiveApprox}
and~\eqref{eq:correctorGradientBoundMassiveApprox}, and an analogous argument for $\sigma_{\xi,B}$, 
it follows that
\begin{align}
\label{eq:growthEstimate}
\bigg\langle\,\dashint_{B_1(x_0)}\big|\big(\phi_{\xi,B},\sigma_{\xi,B})\big|^2\bigg\rangle^\frac{1}{2}
\leq C\mu_*(1{+}|x_0|).
\end{align}
Hence, on one side we infer from~\eqref{eq:growthEstimate} that 
\begin{align*}
(\phi_{\xi,B},\sigma_{\xi,B})\in 
H^1_{\mathrm{loc}}(\Rd;\Rd) {\times} H^1_{\mathrm{loc}}(\Rd;\mathbb{R}^{d\times d}_{\mathrm{skew}}{\times}\Rd)
\quad\text{almost surely}.
\end{align*}
On the other side, we also learn from~\eqref{eq:growthEstimate} (by a covering argument
and definition~\eqref{eq:scalingCorrectorBounds} of the scaling function) that the pair 
$(\phi_{\xi,B},\sigma_{\xi,B})$ features sublinear growth at infinity in
the precise sense of Definition~\ref{def:correctorsHigherOrderLinearization}.

It remains to verify the validity of the associated PDE \eqref{eq:PDEhigherOrderLinearizedCorrector}
for the linearized homogenization correctors resp.\ the associated PDEs for the linearized flux correctors
\eqref{eq:PDEhigherOrderLinearizedFluxCorrector}~and~\eqref{eq:PDEhigherOrderLinearizedHelmholtzDecomp}
(almost surely in a distributional sense). To this end, we first note that
as a consequence of the corrector estimates~\eqref{eq:correctorGradientBoundMassiveApprox}
and~\eqref{eq:correctorGrowthoundMassiveApprox} and
stationarity of the linearized correctors $(\phi^T_{\xi,B},\sigma^T_{\xi,B},\psi^T_{\xi,B})$ that
\begin{align*}
&\bigg\langle\,\dashint_{B_1(x_0)}
\Big|\Big(\phi^T_{\xi,B},\sigma^T_{\xi,B},
\frac{\psi^T_{\xi,B}}{\sqrt{T}}\Big)\Big|^2\bigg\rangle^\frac{1}{2}
\\&
\leq \bigg\langle\,\dashint_{B_1}
\Big|\Big(\nabla\phi^T_{\xi,B},\nabla\sigma^T_{\xi,B},
\frac{\nabla\psi^T_{\xi,B}}{\sqrt{T}}\Big)\Big|^2\bigg\rangle^\frac{1}{2}
+ \bigg\langle\bigg|\,\dashint_{B_1}\Big(\phi^T_{\xi,B},\sigma^T_{\xi,B},
\frac{\psi^T_{\xi,B}}{\sqrt{T}}\Big)\bigg|^2\bigg\rangle^\frac{1}{2}
\\&
\leq C\mu_*(\sqrt{T}).
\end{align*}
In particular, 
\begin{align}
\label{eq:vanishingMassiveTerms}
\Big(\frac{\phi^T_{\xi,B}}{\sqrt{T}},\frac{\sigma^T_{\xi,B}}{\sqrt{T}},\frac{\psi^T_{\xi,B}}{T}\Big)
\to 0 \quad\text{ strongly in } L^2_{\langle\cdot\rangle}L^2_{\mathrm{loc}}(\Rd).
\end{align}
Moreover, due to the strong convergence~\eqref{eq:convStatement} in
$L^2_{\langle\cdot\rangle}L^2_{\mathrm{loc}}(\Rd)$ of the gradients,
the gradients also converge $\Prob\otimes\mathcal{L}^d$ almost everywhere
in the product space $\Omega\times K$, for every compact $K\subset\Rd$.
In particular, by a straightforward inductive argument (with the base case
provided in Appendix~\ref{app:baseCaseInd}) we deduce convergence of the linearized fluxes
from~\eqref{eq:HigherOrderLinearizedFlux} resp.\ \eqref{eq:LinearizedFlux} in the sense of
\begin{align}
\label{eq:pointwiseConvergenceLinearizedFlux}
q^T_{\xi,B} \to q_{\xi,B}
\quad \Prob\otimes\mathcal{L}^d \text{ almost everywhere in } \Omega\times K,
\end{align}
for every compact $K\subset\Rd$. Uniform boundedness of $(q^{T}_{\xi,B})_{T\geq 1}$ 
in $L^q_{\langle\cdot\rangle}L^p_{\mathrm{loc}}(\Rd)$
for every pair of exponents $q\in [1,\infty)$ and $p\in [2,\infty)$ (which is a consequence of the
corrector estimates~\eqref{eq:correctorSchauderMassiveApprox}, the definition~\eqref{eq:HigherOrderLinearizedFlux}
of the linearized flux, and~(A2)$_{L}$ from Assumption~\ref{assumption:operators})
upgrades~\eqref{eq:pointwiseConvergenceLinearizedFlux} to strong convergence 
\begin{align}
\label{eq:strongConvergenceLinearizedFlux}
q^T_{\xi,B} \to q_{\xi,B}
\quad\text{ strongly in } L^2_{\langle\cdot\rangle}L^2_{\mathrm{loc}}(\Rd),
\end{align}
which by stationarity of the linearized fluxes $q_{\xi,B}^T$ and $q_{\xi,B}$
in particular entails~\eqref{eq:convFluxes}.
Validity of the PDEs \eqref{eq:PDEhigherOrderLinearizedCorrector},
\eqref{eq:PDEhigherOrderLinearizedFluxCorrector} and \eqref{eq:PDEhigherOrderLinearizedHelmholtzDecomp}
(almost surely in a distributional sense) thus follows from taking the limit in
the corresponding massive versions \eqref{eq:PDEhigherOrderLinearizedCorrectorLocalized},
\eqref{eq:PDEhigherOrderLinearizedFluxCorrectorLocalized}
and \eqref{eq:PDEhigherOrderLinearizedHelmholtzDecompLocalized}.
This concludes the proof of Lemma~\ref{lem:limitMassiveApprox}. \qed

\subsection{Proof of Theorem~\ref{theo:correctorBounds}
{\normalfont (Corrector estimates for higher-order linearizations)}}
We proceed in three steps.

\textit{Step 1: (Proof of the corrector 
estimates~\emph{\eqref{eq:linearFunctionalCorrectorGradientBound}--\eqref{eq:correctorGrowthBound}})}
The estimates~\eqref{eq:linearFunctionalCorrectorGradientBound}
and~\eqref{eq:correctorGradientBound} are immediate consequences of
the corresponding estimates~\eqref{eq:linearFunctionalCorrectorGradientBoundMassiveApprox}
and~\eqref{eq:correctorGradientBoundMassiveApprox} from Theorem~\ref{theo:correctorBoundsMassiveApprox},
an application of Fatou's inequality based on Lemma~\ref{lem:limitMassiveApprox},
and multilinearity of the map $B\mapsto\phi_{\xi,B}$.
For a proof of~\eqref{eq:correctorGrowthBound}, we estimate by adding zero, 
stationarity of the corrector gradients, and applying Poincar\'e's inequality
\begin{align*}
\bigg\langle\bigg|\,\dashint_{B_1(x_0)}
\big|\big(\phi_{\xi,B},\sigma_{\xi,B}
\big)\big|^2\bigg|^q\bigg\rangle^\frac{1}{q}
&\leq \bigg\langle\bigg|\,\dashint_{B_1}
\big|\nabla(\phi_{\xi,B},\sigma_{\xi,B})
\big|^2\bigg|^q\bigg\rangle^\frac{1}{q}
\\&~~~
+ \bigg\langle\bigg|\,\dashint_{B_1(x_0)}
(\phi_{\xi,B},\sigma_{\xi,B})
{-} \dashint_{B_1}(\phi_{\xi,B},\sigma_{\xi,B})
\bigg|^{2q}\bigg\rangle^\frac{1}{q}.
\end{align*}
For any $x_0\in\Rd$, let 
$f_{x_0}:=\frac{1}{|B_1|}\mathds{1}_{B_1}{-}\frac{1}{|B_1(x_0)|}\mathds{1}_{B_1(x_0)}$.
We then take $v$ to be the solution of the Neumann problem for Poisson's equation 
$\Delta v = f_{x_0}$ in the ball $B_{1{+}|x_0|}$. This in turn enables us to
provide a solution of $\nabla\cdot g_{x_0} = f_{x_0}$ in $\Rd$ by means
of $g_{x_0}:=\mathds{1}_{B_{1+|x_0|}}\nabla v$. Furthermore, it holds
$\int_{B_{1{+}|x_0|}(0)}|g_{x_0}|^2\dx \leq C\mu_*^2(1{+}|x_0|)$.
We then obtain by means of the estimates~\eqref{eq:linearFunctionalCorrectorGradientBound}
and~\eqref{eq:correctorGradientBound}
\begin{align*}
&\bigg\langle\bigg|\,\dashint_{B_1(x_0)}
\big|\big(\phi_{\xi,B},\sigma_{\xi,B}
\big)\big|^2\bigg|^q\bigg\rangle^\frac{1}{q}
\\&
\leq \bigg\langle\bigg|\,\dashint_{B_1}
\big|\nabla(\phi_{\xi,B},\sigma_{\xi,B})
\big|^2\bigg|^q\bigg\rangle^\frac{1}{q}
+ \bigg\langle\bigg|\int g_{x_0}
\cdot\nabla(\phi_{\xi,B},\sigma_{\xi,B})
\bigg|^{2q}\bigg\rangle^\frac{1}{q}
\\&
\leq C^2q^{2C}|B|^2\mu_*^2(1{+}|x_0|).
\end{align*}
This concludes the proof of the corrector estimates.

\textit{Step 2: (Qualitative differentiability of linearized correctors)}
We argue that the maps $\xi\mapsto\phi_{\xi,B}$ and $\xi\mapsto\sigma_{\xi,B}$
are Fr\'echet differentiable with values in the Fr\'echet space $L^q_{\langle\cdot\rangle}H^1_{\mathrm{loc}}(\Rd)$,
and that for every unit vector $e\in\Rd$ we have the following representation of the
directional derivative $\partial_\xi(\phi_{\xi,B},\sigma_{\xi,B})[e]
=(\phi_{\xi,B\odot e},\sigma_{\xi,B\odot e})$.

On the level of the (stationary) corrector gradients, it follows from~\eqref{eq:regGradient} by an application
of Fatou's inequality based on Lemma~\ref{lem:limitMassiveApprox} that
\begin{align*}
&\big\langle\big\|\big(\nabla\phi_{\xi+he,B}{-}\nabla\phi_{\xi,B}{-}\nabla\phi_{\xi,B\odot e}h,
\nabla\sigma_{\xi+he,B}{-}\nabla\sigma_{\xi,B}{-}\nabla\sigma_{\xi,B\odot e}h
\big)\big\|^{2q}_{L^2(B_1)}\big\rangle^\frac{1}{q}
\\&
\leq C^2q^{2C}|B|^2|h|^{4(1-\beta)}
\end{align*}
for all unit vectors $e\in\Rd$ and all $|h|\leq 1$. On the level of the correctors themselves, 
note that the argument for~\eqref{eq:correctorGrowthBound} in \textit{Step 1} of this proof
is linear in the variable $(\phi_{\xi,B},\sigma_{\xi,B})$. Hence, we may run it based on
the first-order Taylor expansion $(\phi_{\xi+he,B}{-}\phi_{\xi,B}{-}\phi_{\xi,B\odot e}h,
\sigma_{\xi+he,B}{-}\sigma_{\xi,B}{-}\sigma_{\xi,B\odot e}h)$ which in light of~\eqref{eq:regLinearFunctionals} 
entails the estimate
\begin{align*}
&\big\langle\big\|\big(\phi_{\xi+he,B}{-}\phi_{\xi,B}{-}\phi_{\xi,B\odot e}h,
\sigma_{\xi+he,B}{-}\sigma_{\xi,B}{-}\sigma_{\xi,B\odot e}h
\big)\big\|^{2q}_{L^2(B_1(x_0))}\big\rangle^\frac{1}{q}
\\&
\leq C^2q^{2C}|B|^2|h|^{4(1-\beta)}\mu_*^2(1{+}|x_0|)
\end{align*}
for all unit vectors $e\in\Rd$, all $|h|\leq 1$ and all $x_0\in\Rd$. The two previous displays
immediately imply the claim.

\textit{Step 3: (Proof of the 
estimates~\eqref{eq:correctorGradientBoundDifferences} and~\eqref{eq:correctorGrowthBoundDifferences})}
As a consequence of the previous step and Taylor's formula, we may represent the (well-defined) Taylor expansions
from~\eqref{eq:TaylorLinearizedHomCorrector} resp.\ \eqref{eq:TaylorLinearizedFluxCorrector}
in terms of the one-parameter family of linearized correctors 
$(\phi_{s\xi{+}(1{-}s)\xi_0,B\odot(\xi{-}\xi_0)^{\odot (K{+}1)}}
\sigma_{s\xi{+}(1{-}s)\xi_0,B\odot(\xi{-}\xi_0)^{\odot (K{+}1)}})$, $s\in [0,1]$.
The estimates~\eqref{eq:correctorGradientBoundDifferences} 
and~\eqref{eq:correctorGrowthBoundDifferences} thus follow
from the corrector estimates~\eqref{eq:linearFunctionalCorrectorGradientBound}
and~\eqref{eq:correctorGrowthBound}.

This concludes the proof of Theorem~\ref{theo:correctorBounds}. \qed

\subsection{Proof of Theorem~\ref{theo:diffHomOperator}
{\normalfont (Higher-order regularity of the homogenized operator)}}
Differentiability of the homogenized operator, the representation
of the directional derivatives~\eqref{eq:repDerivHomOperator}, and
the corresponding bound~\eqref{eq:boundDerivHomOperator} are
in the case of first-order differentiability 
immediate consequences of the 
estimates~\eqref{eq:BaseCaseRegMassiveVersionHomOperator},
\eqref{eq:convFluxes} and~\eqref{eq:BaseCaseConvFluxes},
and in case of higher-order differentiability of
the estimates~\eqref{eq:regMassiveVersionHomOperator}
and~\eqref{eq:convFluxes}. \qed

\appendix
\section{A toolbox from elliptic regularity theory}
\label{app:toolsEllipticRegularity}

The aim of this appendix is to list (and in parts to prove) several
results from (deterministic resp.\ random) elliptic regularity theory.
We start with the probably most basic result concerning the Caccioppoli and 
hole filling estimates.

\begin{lemma}[Caccioppoli inequality and hole filling estimate]
\label{lem:caccioppoliHoleFilling}
Let $a\colon\Rd\to\Rd[d\times d]$ be a uniformly elliptic and bounded coefficient field
with respect to constants $(\lambda,\Lambda)$. For a given $T\in [1,\infty)$
and $f,g\in L^2_{\mathrm{loc}}(\Rd)$, let $u\in H^1_{\mathrm{loc}}(\Rd)$
be a solution of
\begin{align*}
\frac{1}{T}u - \nabla\cdot a\nabla u = \frac{1}{T}f + \nabla\cdot g \quad\text{in } \Rd.
\end{align*} 
Then, we have for all $x_0\in\Rd$ and all $R>0$ the Caccioppoli estimate
\begin{align}
\label{eq:Caccioppoli}
\nonumber
&\Big\|\Big(\frac{u}{\sqrt{T}},\nabla u\Big)\Big\|_{L^2(B_R(x_0))}^2
\\& \tag{T1}
\lesssim_{d,\lambda,\Lambda}  
\inf_{b\in\Rd[]} \Big\{\frac{1}{R^2}\big\|u-b\big\|_{L^2(B_{2R}(x_0))}^2 + \frac{1}{T}|b|^2 \Big\}
+ \Big\|\Big(\frac{f}{\sqrt{T}},g\Big)\Big\|_{L^2(B_{2R}(x_0))}^2.
\end{align}
Moreover, there exists $\delta=\delta(d,\lambda,\Lambda)$ 
such that for all $r<R$ and all $x_0\in\Rd$ we have the hole filling estimate
\begin{align}
\label{eq:holeFillingEstimate}
\nonumber
&\Big\|\Big(\frac{u}{\sqrt{T}},\nabla u\Big)\Big\|_{L^2(B_r(x_0))}^2
\\& \tag{T2}
\lesssim_{d,\lambda,\Lambda}  
\Big(\frac{R}{r}\Big)^{-\delta}\Big\|\Big(\frac{u}{\sqrt{T}},\nabla u\Big)\Big\|_{L^2(B_R(x_0))}^2
+ \int_{B_R(x_0)} \frac{r^\delta}{|x{-}x_0|^\delta}\Big|\Big(\frac{f}{\sqrt{T}},g\Big)\Big|^2 \dx.
\end{align}
\end{lemma}

\begin{proof}
The standard proofs carry over immediately to the setting with an
additional massive term.
\end{proof}

Perturbing a uniformly elliptic PDE by a massive term has the very convenient
consequence of entailing a weighted energy estimate in terms of a suitable
exponential weight. More precisely, we have the following standard result.

\begin{lemma}[Exponential localization]
\label{lem:expLocalization}
Let $a\colon\Rd\to\Rd[d\times d]$ be a uniformly elliptic and bounded coefficient field
with respect to constants $(\lambda,\Lambda)$. For given $T\in [1,\infty)$
and $f,g\in L^2_{\mathrm{loc}}(\Rd)$, let $u\in H^1_{\mathrm{loc}}(\Rd)$
be a solution of
\begin{align*}
\frac{1}{T}u - \nabla\cdot a\nabla u = \frac{1}{T}f + \nabla\cdot g \quad\text{in } \Rd.
\end{align*} 
Assume that $\limsup_{R\to\infty} R^{-2k-d}\|(u,\nabla u, f, g)\|_{L^2(B_R)}^2 = 0$
for some $k\in\N$. For any $\gamma>0$ and $R>0$ define the exponential weight $\ell_{\gamma,R}(x):=
\frac{1}{R^d}\exp(-\frac{\gamma|x|}{R})$. 
Then, there exists $\gamma=\gamma(d,\lambda,\Lambda)$ such that for all $R\geq\sqrt{T}$ we have the 
weighted energy estimate
\begin{align}
\label{eq:expLocalization}
\tag{T3}
\Big\|\Big(\frac{u}{\sqrt{T}},\nabla u\Big)\Big\|_{L^2(\Rd;\ell_{\gamma,R}\dx)}
\lesssim_{d,\lambda,\Lambda} \Big\|\Big(\frac{f}{\sqrt{T}}, g\Big)\Big\|_{L^2(\Rd;\ell_{\gamma,R}\dx)}.
\end{align}
\end{lemma}

\begin{proof}
The idea is to test the equation with $u\ell_{\gamma,R}$,
and to run an absorption argument which gets facilitated by
an appropriate choice of $\gamma\in (0,1)$. Details of this
(standard) argument are provided in, e.g., the proof of \cite[Lemma~36]{Fischer2019}.
\end{proof}

In case of H\"older continuous coefficients, classical
elliptic regularity provides local Schauder estimates and local Calder\'on--Zygmund estimates.
The corresponding version for the massive approximation with an additional
explicit dependence of the constant on the H\"older norm of the coefficient field reads as follows.

\begin{lemma}[Local regularity estimates for H\"older continuous coefficients, cf.\ \cite{Josien2020}]
\label{lem:locRegularity}
Consider $\alpha\in (0,1)$ and $T\geq 1$, and let $a\in C^\alpha(B_2)$ be a uniformly elliptic and bounded
coefficient field with respect to constants $(\lambda,\Lambda)$. Let $p\in [2,\infty)$, and for 
given $f,g\in L^p(B_2)$ let $u\in H^1(B_2)$ be a solution of
\begin{align*}
\frac{1}{T}u - \nabla\cdot a\nabla u = \frac{1}{T}f + \nabla\cdot g \quad\text{in } B_2.
\end{align*} 
Then, the following local Calder\'on--Zygmund estimate holds true
\begin{align}
\label{eq:localCalderonZygmund}
\nonumber
&\Big\|\Big(\frac{u}{\sqrt{T}},\nabla u\Big)\Big\|_{L^p(B_1)} 
\\&\tag{T4}
\lesssim_{d,\lambda,\Lambda,\alpha,p}
\|a\|_{C^\alpha(B_2)}^{\frac{d}{\alpha}(\frac{1}{2}-\frac{1}{p})}
\Big\|\Big(\frac{u}{\sqrt{T}},\nabla u\Big)\Big\|_{L^2(B_2)}
+ \Big\|\Big(\frac{f}{\sqrt{T}},g\Big)\Big\|_{L^p(B_2)}.
\end{align}
Furthermore, in case of $f=0$ and $g\in C^\alpha(B_2)$ the following local Schauder estimate is satisfied
\begin{align}
\label{eq:localSchauder}
\nonumber
&\Big\|\Big(\frac{u}{\sqrt{T}},\nabla u\Big)\Big\|_{C^\alpha(B_1)} 
\\&\tag{T5} 
\lesssim_{d,\lambda,\Lambda,\alpha}
\|a\|_{C^\alpha(B_2)}^{\frac{d}{\alpha}(\frac{1}{2}+\frac{1}{d})}
\Big\|\Big(\frac{u}{\sqrt{T}},\nabla u\Big)\Big\|_{L^2(B_2)}
+ \|a\|_{C^\alpha(B_2)}^{\frac{1}{\alpha}-1}\|g\|_{C^\alpha(B_2)}.
\end{align}
\end{lemma}

\begin{proof}
The claims are standard except for the explicitly spelled-out dependence of the
estimates on the H\"older regularity of the coefficient field and the uniformity
with respect to the parameter $T\in [1,\infty)$. A proof
can be found in~\cite[Lemma~A.3]{Josien2020}. 
\end{proof}

The previous result is typically applied to the random setting on the 
level of the linearized coefficient field $a_\xi^T:=\partial_\xi A(\omega,\xi{+}\nabla\phi^T_\xi)$.
However, for this one first needs to verify that the linearized coefficient field 
is actually H\"older regular to be able to apply the above local regularity theory.
Moreover, the arguments in the main text require stretched exponential moments
for the corresponding H\"older norm. The key step towards these goals is
a proof on the level of the correctors.

\begin{lemma}[Annealed H\"older regularity for the corrector of the nonlinear problem]
\label{lem:annealedRegCorrectorNonlinear}
Let the requirements and notation of (A1), (A2)$_0$, (A3)$_0$ of
Assumption~\ref{assumption:operators}, (P1) and (P2) of
Assumption~\ref{assumption:ensembleParameterFields}, and (R) of 
Assumption~\ref{assumption:smallScaleReg} be in place.
Given $\xi\in\Rd$ and $T\in [1,\infty)$, denote by $\phi_\xi^T\in H^1_{\mathrm{uloc}}(\Rd)$
the unique solution of the corrector equation~\eqref{eq:PDEMonotoneHomCorrectorLocalized}.
Let finally $M>0$ be fixed. 

There exist constants $C=C(d,\lambda,\Lambda,\nu,\rho,\eta,M)$ 
and $\alpha=\alpha(d,\lambda,\Lambda)\in (0,\eta)$ such that 
for all $q\in [1,\infty)$ and all $|\xi|\leq M$
\begin{align}
\label{eq:annealedRegLinCorrectorNonlinear}
\tag{T6}
\big\langle\|\nabla\phi^T_\xi\|^{2q}_{C^\alpha(B_1)}
\big\rangle^\frac{1}{q} \leq C^2q^{2C}.
\end{align}
\end{lemma}

\begin{proof}
Note first that by smuggling in a spatial average over the unit ball,
and subsequently applying the triangle inequality and Jensen's inequality, we obtain
for all $\alpha\in (0,1)$ and all $p\geq 1$
\begin{align*}
\bigg(\,\dashint_{B_1} \big|\nabla\phi^T_\xi\big|^p \bigg)^\frac{1}{p}
\leq \big[\nabla\phi^T_\xi\big]_{C^\alpha(B_1)} +
\bigg(\,\dashint_{B_1} \big|\nabla\phi^T_\xi\big|^2 \bigg)^\frac{1}{2}.
\end{align*}
The left hand side term of the previous display converges to
$\|\nabla\phi^T_\xi\|_{L^\infty(B_1)}$ as we let $p\to\infty$.
Hence, in light of the corrector bounds~\eqref{eq:BaseCaseCorrectorBounds}
it suffices to derive an annealed estimate for the H\"older seminorm.
More precisely, we have to prove that there exist 
constants $C=C(d,\lambda,\Lambda,\nu,\rho,\eta,M)$ and $\alpha=\alpha(d,\lambda,\Lambda)\in (0,\eta)$
such that for all $q\in [1,\infty)$ and all $|\xi|\leq M$ it holds
\begin{align}
\label{eq:annealedHoelderGradientNonlinearCorrector}
\big\langle\big[\nabla\phi_\xi^T\big]^{2q}_{C^\alpha(B_1)}
\big\rangle^\frac{1}{q} \leq C^2q^{2C}.
\end{align}

For a proof of~\eqref{eq:annealedHoelderGradientNonlinearCorrector}, we start
by recalling the equivalence of H\"older seminorms and Campanato seminorms.
In other words, we face the task of finding constants 
$C=C(d,\lambda,\Lambda,\nu,\rho,\eta,M)$ and $\alpha=\alpha(d,\lambda,\Lambda)\in (0,\eta)$
such that for all $q\in [1,\infty)$ and all $|\xi|\leq M$ it holds
\begin{align}
\label{eq:annealedCampanatoGradientNonlinearCorrector}
\bigg\langle\bigg|\sup_{x_0\in B_1}\sup_{0<\kappa\leq 1}
\frac{1}{\kappa^{2\alpha}}\dashint_{B_\kappa(x_0)}
\Big|\nabla\phi^T_\xi{-}\dashint_{B_\kappa(x_0)}\nabla\phi^T_\xi\Big|^2\,
\bigg|^q\bigg\rangle^\frac{1}{q} \leq C^2q^{2C}.
\end{align}
To this end, let us introduce some auxiliary quantities.
First, define the random variable 
\begin{align}
\label{eq:auxiliaryRV}
\mathcal{X}_\eta:=\sup_{x,y\in B_4,\,x\neq y}\frac{|\omega(x)-\omega(y)|}{|x-y|^\eta}.
\end{align}
The non-negative random variable $\mathcal{X}_\eta$
has stretched exponential moments because of condition~(R) from Assumption~\ref{assumption:smallScaleReg}.
In addition, for every $m\in\N_0$ we define 
\begin{align}
\label{eq:auxiliaryScaleEvent}
\text{a scale } r_m:=2^{-m}
\text{ and the event } \mathcal{A}_m^\eta
:=\{r_m^{-\eta}{-}1\leq \mathcal{X}_\eta< r_{m+1}^{-\eta}{-}1\}.
\end{align}
Consider also $\theta=\theta(d,\lambda,\Lambda,\eta)\in (0,1)$. (The precise choice of $\theta$
will be determined further below.) Finally, for every $x_0\in B_1$ we
define a field 
\begin{align}
\label{eq:auxiliaryField}
u^T_{\xi,x_0} := \xi\cdot (x{-}x_0) + \phi^T_\xi.
\end{align}
Note that $\nabla u^T_{\xi,x_0}$ is stationary and that in the sequel we may freely switch between $\nabla\phi^T_\xi$
and $\nabla u^T_{\xi,x_0}$ in~\eqref{eq:annealedCampanatoGradientNonlinearCorrector} for fixed $x_0\in B_1$. 

Decomposing for every $m\in\N_0$, every $x_0\in B_1$, and every $\alpha\in (0,1)$
\begin{align*}
&\sup_{0<\kappa\leq 1} \frac{1}{\kappa^{2\alpha}}\dashint_{B_\kappa(x_0)}
\Big|\nabla \phi^T_\xi{-}\dashint_{B_\kappa(x_0)}\nabla \phi^T_\xi\Big|^2
\\& 
\leq \sup_{0<\kappa\leq \theta r_m} \frac{1}{\kappa^{2\alpha}}\dashint_{B_\kappa(x_0)}
\Big|\nabla \phi^T_\xi{-}\dashint_{B_\kappa(x_0)}\nabla \phi^T_\xi\Big|^2
+ 2C(d,\lambda,\Lambda,\alpha,\eta)\frac{1}{r_m^{2\alpha}}\dashint_{B_{r_m}(x_0)} \big|\nabla \phi^T_\xi\big|^2
\end{align*}
exploiting that $\mathds{1}_\Omega=\sum_{m=0}^\infty\mathds{1}_{\mathcal{A}_m^\eta}$
by definition~\eqref{eq:auxiliaryScaleEvent} of the events $\mathcal{A}_m^\eta$,
and relying on the stationarity of $\nabla \phi^T_{\xi}$, we obtain for all $\alpha\in (0,1)$ the estimate
\begin{align*}
&\bigg\langle\bigg|\sup_{x_0\in B_1}\sup_{0<\kappa\leq 1}
\frac{1}{\kappa^{2\alpha}}\dashint_{B_\kappa(x_0)}
\Big|\nabla\phi^T_\xi{-}\dashint_{B_\kappa(x_0)}\nabla\phi^T_\xi\Big|^2\,
\bigg|^q\bigg\rangle^\frac{1}{q}
\\&
\lesssim_{d,\lambda,\Lambda,\alpha,\eta} \sum_{m=0}^\infty\big\langle\big\|
\nabla \phi^T_\xi\big\|_{L^2(B_1)}^{4q}\big\rangle^\frac{1}{2q}
r_m^{-d-2\alpha}\Prob\big[\mathcal{A}_m^\eta\big]^\frac{1}{2q}
\\&~~~
+ \sum_{m=0}^\infty \bigg\langle\bigg|\sup_{x_0\in B_1}\sup_{0<\kappa\leq \theta r_m}
\mathds{1}_{\mathcal{A}_m^\eta}\frac{1}{\kappa^{2\alpha}}\dashint_{B_\kappa(x_0)}
\Big|\nabla\phi^T_\xi{-}\dashint_{B_\kappa(x_0)}\nabla\phi^T_\xi\Big|^2\,
\bigg|^q\bigg\rangle^\frac{1}{q}.
\end{align*} 
As a consequence of the corrector estimates from Proposition~\ref{prop:estimates1stOrderCorrector}, the
definition~\eqref{eq:auxiliaryScaleEvent} of the scales $r_m$ and the events $\mathcal{A}_m^\eta$,
and the random variable $\mathcal{X}_\eta$ from~\eqref{eq:auxiliaryRV} 
admitting stretched exponential moments, the previous display updates to
\begin{align}
\label{eq:intermediateEstimate2}
&\bigg\langle\bigg|\sup_{x_0\in B_1}\sup_{0<\kappa\leq 1}
\frac{1}{\kappa^{2\alpha}}\dashint_{B_\kappa(x_0)}
\Big|\nabla\phi^T_\xi{-}\dashint_{B_\kappa(x_0)}\nabla\phi^T_\xi\Big|^2\,
\bigg|^q\bigg\rangle^\frac{1}{q}
\lesssim_{d,\lambda,\Lambda,\alpha,\eta} C^2q^{2C}|\xi|^2 
\\&~~~~~~~~~~~~~~~\nonumber
+ \sum_{m=0}^\infty \bigg\langle\bigg|\sup_{x_0\in B_1}\sup_{0<\kappa\leq \theta r_m}
\mathds{1}_{\mathcal{A}_m^\eta}\frac{1}{\kappa^{2\alpha}}\dashint_{B_\kappa(x_0)}
\Big|\nabla\phi^T_\xi{-}\dashint_{B_\kappa(x_0)}\nabla\phi^T_\xi\Big|^2\,
\bigg|^q\bigg\rangle^\frac{1}{q}.
\end{align}
To estimate the second right hand side term in~\eqref{eq:intermediateEstimate2},
we proceed (not surprisingly) by harmonic approximation and split the task into two parts.

\textit{Claim 1:} For all $\mu\in (0,d)$ there exists $\theta=\theta(d,\lambda,\Lambda,\eta,\mu)\in (0,1)$ 
such that for all $x_0\in B_1$, all $m\in\N_0$ and all $\kappa\in (0,\theta r_m]$ it holds
\begin{equation}
\label{eq:claim1}
\begin{aligned}
&\mathds{1}_{\mathcal{A}_m^\eta} \int_{B_\kappa(x_0)} |\nabla u^T_{\xi,x_0}|^2 
\\&
\lesssim_{d,\lambda,\Lambda,\eta,\mu} \mathds{1}_{\mathcal{A}_m^\eta}
\Big(\frac{\kappa}{r_m}\Big)^{\mu} \bigg\{  
r_m^{2+d}|\xi|^2 + \int_{B_{r_m}(x_0)} |\nabla u^T_{\xi,x_0}|^2\bigg\}.
\end{aligned}
\end{equation}
For a proof of~\eqref{eq:claim1}, fix $x_0\in B_1$, $R\in (0,\theta r_m)$ and $\kappa\in (0,\frac{R}{2})$. 
Note that we may write the equation of $u^T_{\xi,x_0}$ defined in~\eqref{eq:auxiliaryField} in form of
\begin{equation}
\label{eq:PDEAux1}
\begin{aligned}
&\frac{1}{T}u^T_{\xi,x_0} - \nabla\cdot A(\omega(x_0),\nabla u^T_{\xi,x_0})
\\&
= \frac{1}{T}\xi\cdot (x{-}x_0) -\nabla\cdot\big\{
A(\omega(x_0),\nabla u^T_{\xi,x_0}) - A(\omega,\nabla u^T_{\xi,x_0})\big\}.
\end{aligned}
\end{equation}
We then consider the harmonic approximation with massive term
\begin{equation}
\label{eq:PDEAux2}
\begin{aligned}
\frac{1}{T} v_R - \nabla\cdot A(\omega(x_0),\nabla v_R) &= 0
&& \text{in } B_R(x_0),
\\
v_R &= u^T_{\xi,x_0} && \text{on } \partial B_R(x_0).
\end{aligned}
\end{equation}
Moser iteration applied to the equation
\begin{align}
\label{eq:PDEAux3}
\frac{1}{T}\partial_i v_R - \nabla\cdot \partial_\xi A(\omega(x_0),\nabla v_R)\nabla \partial_i v_R = 0
\end{align}
entails that
\begin{align}
\label{eq:MoserIterationHarmonicApprox}
\int_{B_\kappa(x_0)} |\nabla v_R|^2 \lesssim_{d} 
\kappa^d \sup_{x\in B_{\frac{R}{2}}(x_0)} |\nabla v_R|^2
\lesssim_{d,\lambda,\Lambda} \Big(\frac{\kappa}{R}\Big)^d
\int_{B_R(x_0)} |\nabla v_R|^2.
\end{align}
Moreover, by a simple energy estimate
and~(A3)$_0$ from Assumption~\ref{assumption:operators} we have
\begin{equation}
\label{eq:EnergyEstimateAux}
\begin{aligned}
\int_{B_\kappa(x_0)} |\nabla u^T_{\xi,x_0} {-} \nabla v_R|^2 
&\leq \int_{B_R(x_0)} |\nabla u^T_{\xi,x_0} {-} \nabla v_R|^2 
\\&
\lesssim_{d,\lambda,\Lambda} \frac{1}{T}R^{2+d}|\xi|^2 
+ R^{2\eta} \mathcal{X}_\eta^2 \int_{B_R(x_0)} |\nabla u^T_{\xi,x_0}|^2.
\end{aligned}
\end{equation}
As $R\leq\theta r_m$ and $T\geq 1$, we obtain from definition~\eqref{eq:auxiliaryRV}
of the random variable $\mathcal{X}$ and definition~\eqref{eq:auxiliaryScaleEvent} 
of the event $\mathcal{A}_m^\eta$ that
\begin{align}
\label{eq:EnergyEstimateAm}
\mathds{1}_{\mathcal{A}_m^\eta} \int_{B_\kappa(x_0)} |\nabla u^T_{\xi,x_0} {-} \nabla v_R|^2
\lesssim_{d,\lambda,\Lambda} \mathds{1}_{\mathcal{A}_m^\eta}\Big\{ 
R^{2+d}|\xi|^2 + \theta^{2\eta} \int_{B_R(x_0)} |\nabla u^T_{\xi,x_0}|^2 \Big\}.
\end{align}
Combining the estimates~\eqref{eq:MoserIterationHarmonicApprox} and~\eqref{eq:EnergyEstimateAm}
thus shows that
\begin{equation}
\label{eq:CampanatoIterationEnergy}
\begin{aligned}
&\mathds{1}_{\mathcal{A}_m^\eta} \int_{B_\kappa(x_0)} |\nabla u^T_{\xi,x_0}|^2
\\&
\lesssim_{d,\lambda,\Lambda} \mathds{1}_{\mathcal{A}_m^\eta}\bigg\{  
R^\mu r_m^{2+d-\mu}|\xi|^2 + \Big(\theta^{2\eta}{+}\Big(\frac{\kappa}{R}\Big)^d\Big)
\int_{B_R(x_0)} |\nabla u^T_{\xi,x_0}|^2\bigg\}.
\end{aligned}
\end{equation}
Note that the estimate from the previous display is trivially fulfilled
in the regime $R\in (0,\theta r_m)$ and $\kappa\in (\frac{R}{2}, R)$.
Choosing $\theta=\theta(d,\lambda,\Lambda,\eta,\mu)$ sufficiently small
so that one can iterate~\eqref{eq:CampanatoIterationEnergy}, we
obtain for all $R\in (0,\theta r_m]$ and all $\kappa\in (0,R]$
\begin{align*}
\mathds{1}_{\mathcal{A}_m^\eta} \int_{B_\kappa(x_0)} |\nabla u^T_{\xi,x_0}|^2
\lesssim_{d,\lambda,\Lambda,\eta,\mu} \mathds{1}_{\mathcal{A}_m^\eta}
\Big(\frac{\kappa}{R}\Big)^\mu \bigg\{  
R^\mu r_m^{2+d-\mu}|\xi|^2 + \int_{B_R(x_0)} |\nabla u^T_{\xi,x_0}|^2\bigg\}.
\end{align*}
This in turn immediately implies the claim~\eqref{eq:claim1}.

\textit{Claim 2:}
There exist $\theta=\theta(d,\lambda,\Lambda,\eta)\in (0,1)$ 
and $\alpha=\alpha(d,\lambda,\Lambda)\in (0,\eta)$ 
such that for all $x_0\in B_1$, all $m\in\N_0$, all $R\in (0,\theta r_m]$
and all $\kappa\in (0,R]$
\begin{equation}
\label{eq:claim2}
\begin{aligned}
&\mathds{1}_{\mathcal{A}_m^\eta}
\int_{B_\kappa(x_0)} \Big|\nabla u^T_{\xi,x_0} {-} 
\dashint_{B_\kappa(x_0)} \nabla u^T_{\xi,x_0} \Big|^2
\\&
\lesssim_{d,\lambda,\Lambda,\eta} \mathds{1}_{\mathcal{A}_m^\eta}
\Big(\frac{\kappa}{R}\Big)^{2\alpha + d} \bigg\{
\int_{B_R(x_0)} \Big|\nabla u^T_{\xi,x_0} {-} 
\dashint_{B_R(x_0)} \nabla u^T_{\xi,x_0} \Big|^2
\\&~~~~~~~~~~~~~~~~~~~~~~~~~~~~~~~
+ R^{2\alpha + d}\big(1{+}\mathcal{X}_\eta^2\big)|\xi|^2
\\&~~~~~~~~~~~~~~~~~~~~~~~~~~~~~~~
+ R^{2\alpha + d} \big(1{+}\mathcal{X}_\eta^2\big) 
\dashint_{B_{r_m}(x_0)} |\nabla u^T_{\xi,x_0}|^2 \bigg\}.
\end{aligned}
\end{equation}
For a proof, fix $x_0\in B_1$, $R\in (0,\theta r_m)$ and $\kappa\in (0,\frac{R}{8})$. 
We first introduce a suitable decomposition for the gradient of the solution $v_R$ of~\eqref{eq:PDEAux2}
because of the massive term appearing in the equation. More precisely, consider
for all $i\in\{1,\ldots,d\}$ the auxiliary Dirichlet problem
\begin{equation}
\label{eq:PDEAux5}
\begin{aligned}
- \nabla\cdot A(\omega(x_0),\nabla v_R)\nabla w_{i,R} &= -\frac{1}{T}\partial_i v_R
&& \text{in } B_\frac{R}{2}(x_0),
\\
w_{i,R} &= 0 && \text{on } \partial B_\frac{R}{2}(x_0).
\end{aligned}
\end{equation}
By the De Giorgi--Nash--Moser estimate in combination with Moser iteration 
applied to the equation
\begin{align*}
- \nabla\cdot A(\omega(x_0),\nabla v_R)(\nabla w_{i,R} - \nabla\partial_i v_R) = 0
&& \text{in } B_\frac{R}{2}(x_0),
\end{align*}
we then find $\gamma=\gamma(d,\lambda,\Lambda)\in (0,1)$
such that for all $i\in\{1,\ldots,d\}$ we have excess decay in form of
\begin{equation}
\label{eq:DeGiorgiNashMoserHarmonicApprox}
\begin{aligned}
&\int_{B_\kappa(x_0)} \Big|(w_{i,R} {-} \partial_i v_R) {-} 
\dashint_{B_\kappa(x_0)} (w_{i,R} {-} \partial_i v_R) \Big|^2 
\\&
\lesssim_{d,\lambda,\Lambda} \Big(\frac{\kappa}{R}\Big)^{2\gamma + d}
\int_{B_\frac{R}{4}(x_0)} \Big|(w_{i,R} {-} \partial_i v_R) {-} 
\dashint_{B_\frac{R}{4}(x_0)} (w_{i,R} {-} \partial_i v_R) \Big|^2 .
\end{aligned}
\end{equation}
Moreover, by a simple energy estimate for~\eqref{eq:PDEAux5}
(rewriting to this end the right hand side of~\eqref{eq:PDEAux5} in form of 
$-\frac{1}{T}\partial_i v_R=-\frac{1}{T}\nabla\cdot 
e_i (v_R-\,\dashint_{B_{\frac{R}{2}}(x_0)} v_R)$)
in combination with Poincar\'e's inequality and $T \geq 1$ we obtain 
for all $i\in\{1,\ldots,d\}$
\begin{align}
\nonumber
\int_{B_\kappa(x_0)} \Big| w_{i,R} {-} \dashint_{B_\kappa(x_0)} w_{i,R} \Big|^2 
&\leq \int_{B_\frac{R}{2}(x_0)} \Big| w_{i,R} {-} \dashint_{B_\frac{R}{2}(x_0)} w_{i,R} \Big|^2 
\\\nonumber
&\lesssim_{d} R^2\int_{B_\frac{R}{2}(x_0)} |\nabla w_{i,R}|^2
\\\label{eq:ernergyEstimateAux5}
&\lesssim_{d,\lambda,\Lambda} R^2
\int_{B_\frac{R}{2}(x_0)} \Big| v_R {-}
\dashint_{B_\frac{R}{2}(x_0)} v_R \Big|^2
\\\nonumber
&\lesssim_{d,\lambda,\Lambda} R^4
\int_{B_R(x_0)} |\nabla v_R {-} \nabla u^T_{\xi,x_0}|^2
+ R^4 \int_{B_R(x_0)} |\nabla u^T_{\xi,x_0}|^2.
\end{align}
Finally, because of~\eqref{eq:EnergyEstimateAux}, $R\leq 1$, $x_0\in B_1$ and $T\geq 1$ we get
\begin{equation}
\label{eq:EnergyEstimateAux2}
\begin{aligned}
\int_{B_\kappa(x_0)} |\nabla u^T_{\xi,x_0} {-} \nabla v_R|^2 
&\leq \int_{B_R(x_0)} |\nabla u^T_{\xi,x_0} {-} \nabla v_R|^2 
\\&
\lesssim_{d,\lambda,\Lambda} R^{2+d} |\xi|^2
+ R^{2\eta} \mathcal{X}^2_\eta
\int_{B_R(x_0)} |\nabla u_\xi^T|^2.
\end{aligned}
\end{equation}
In total, the decomposition $\partial_i u^T_{\xi,x_0} = (\partial_i u^T_{\xi,x_0} {-} \partial_i v_R)
+ (\partial_i v_R {-} w_{i,R}) + w_{i,R}$ together with 
the estimates~\eqref{eq:DeGiorgiNashMoserHarmonicApprox}--\eqref{eq:EnergyEstimateAux2}
implies that for all $i\in\{1,\ldots,d\}$ we have the estimate
\begin{align*}
&\int_{B_\kappa(x_0)} \Big| \partial_i u^T_{\xi,x_0} {-} \dashint_{B_\kappa(x_0)} \partial_i u^T_{\xi,x_0} \Big|^2 
\\&
\lesssim_{d,\lambda,\Lambda} \Big(\frac{\kappa}{R}\Big)^{2\gamma + d}
\int_{B_R(x_0)} \Big| \partial_i u^T_{\xi,x_0} {-} 
\dashint_{B_R(x_0)} \partial_i u^T_{\xi,x_0} \Big|^2 
\\&~~~
+ R^{2+d} |\xi|^2
+ R^{2\eta} \big(1{+}\mathcal{X}^2_\eta\big)
\int_{B_R(x_0)} |\nabla u^T_{\xi,x_0}|^2.
\end{align*}
Now, the estimate~\eqref{eq:claim1} (with $\kappa$ replaced by $R$) entails 
\begin{align*}
\mathds{1}_{\mathcal{A}_m^\eta} \int_{B_R(x_0)} |\nabla u^T_{\xi,x_0}|^2 
\lesssim_{d,\lambda,\Lambda,\eta,\mu} \mathds{1}_{\mathcal{A}_m^\eta}
R^\mu\bigg\{r_m^{2+d-\mu}|\xi|^2 + \dashint_{B_{r_m}(x_0)} |\nabla u^T_{\xi,x_0}|^2\bigg\}.
\end{align*}
The previous two displays in turn show that after 
choosing $\mu=\mu(d,\lambda,\Lambda,\eta)\in (0,d)$ appropriately 
there exists $\alpha=\alpha(d,\lambda,\Lambda)\in (0,\eta\wedge\gamma)$ 
such that
\begin{align*}
&\mathds{1}_{\mathcal{A}_m} \int_{B_\kappa(x_0)} \Big| \partial_i u^T_{\xi,x_0} 
{-} \dashint_{B_\kappa(x_0)} \partial_i u^T_{\xi,x_0} \Big|^2 
\\&
\lesssim_{d,\lambda,\Lambda,\eta} \mathds{1}_{\mathcal{A}_m}
\Big(\frac{\kappa}{R}\Big)^{2\gamma + d}
\int_{B_R(x_0)} \Big| \partial_i u^T_{\xi,x_0} {-} 
\dashint_{B_R(x_0)} \partial_i u^T_{\xi,x_0} \Big|^2 
\\&~~~
+ \mathds{1}_{\mathcal{A}_m} R^{2\alpha + d}
\bigg\{\big(1{+}\mathcal{X}_\eta^2\big)|\xi|^2
+ \big(1{+}\mathcal{X}_\eta^2\big)\,\dashint_{B_{r_m}(x_0)} |\nabla u^T_{\xi,x_0}|^2\bigg\}.
\end{align*}
Iterating the previous display then establishes the asserted estimate~\eqref{eq:claim2}.

\textit{Conclusion:} We make use of~\eqref{eq:claim2} in order to estimate
the second right hand side term of~\eqref{eq:intermediateEstimate2}.
More precisely, the estimate~\eqref{eq:claim2} applied with $R=\theta r_m$
entails that it holds
\begin{align*}
&\sup_{x_0\in B_1}\sup_{0<\kappa\leq \theta r_m}
\mathds{1}_{\mathcal{A}_m^\eta}\frac{1}{\kappa^{2\alpha}} \,\dashint_{B_\kappa(x_0)}
\Big|\nabla u^T_{\xi,x_0}{-}\dashint_{B_\kappa(x_0)}\nabla u^T_{\xi,x_0}\Big|^2
\\&
\lesssim_{d,\lambda,\Lambda,\eta} \mathds{1}_{\mathcal{A}_m^\eta}
\big(1{+}\mathcal{X}_\eta^2\big)
r_m^{-d-2\alpha} \int_{B_2} |\nabla u^T_{\xi,x_0}|^2
+ \mathds{1}_{\mathcal{A}_m^\eta} 
\big(1{+}\mathcal{X}_\eta^2\big)|\xi|^2.
\end{align*}
Plugging this back into the second right hand side term of~\eqref{eq:intermediateEstimate2},
and then making use of H\"older's inequality, the corrector estimates from Proposition~\ref{prop:estimates1stOrderCorrector},
the definition~\eqref{eq:auxiliaryScaleEvent} of the scales $r_m$ and the events $\mathcal{A}_m^\eta$,
and that the random variable $\mathcal{X}_\eta$ from~\eqref{eq:auxiliaryRV} 
admits stretched exponential moments, therefore upgrades~\eqref{eq:intermediateEstimate2} to
\begin{align*}
&\bigg\langle\bigg|\sup_{x_0\in B_1}\sup_{0<\kappa\leq 1}
\frac{1}{\kappa^{2\alpha}}\dashint_{B_\kappa(x_0)}
\Big|\nabla\phi^T_\xi{-}\dashint_{B_\kappa(x_0)}\nabla\phi^T_\xi\Big|^2\,
\bigg|^q\bigg\rangle^\frac{1}{q}
\leq C^2q^{2C}|\xi|^2.
\end{align*}
This concludes the proof of~\eqref{eq:annealedCampanatoGradientNonlinearCorrector},
which in turn implies~\eqref{eq:annealedHoelderGradientNonlinearCorrector}. 
\end{proof}

An immediate consequence of Lemma~\ref{lem:annealedRegCorrectorNonlinear}
is now the following regularity result for the linearized coefficient field
$a_\xi^T:=\partial_\xi A(\omega,\xi{+}\nabla\phi^T_\xi)$.

\begin{lemma}[Annealed H\"older regularity for the linearized coefficient]
\label{lem:annealedHoelderRegLinCoefficient}
Let the requirements and notation of (A1), (A2)$_1$, (A3)$_1$ of
Assumption~\ref{assumption:operators}, (P1) and (P2) of
Assumption~\ref{assumption:ensembleParameterFields}, and (R) of 
Assumption~\ref{assumption:smallScaleReg} be in place.
Given $\xi\in\Rd$ and $T\in [1,\infty)$, denote by $\phi_\xi^T\in H^1_{\mathrm{uloc}}(\Rd)$
the unique solution of the corrector equation~\eqref{eq:PDEMonotoneHomCorrectorLocalized}.
Let finally $M>0$ be fixed. 

Then, there exist constants $C=C(d,\lambda,\Lambda,\nu,\rho,\eta,M)$ 
and $\alpha=\alpha(d,\lambda,\Lambda)\in (0,\eta)$ such that 
for all $q\in [1,\infty)$ and all $|\xi|\leq M$
\begin{align}
\label{eq:annealedHoelderRegLinCoefficient}
\tag{T7}
\big\langle\|\partial_\xi A(\omega,\xi{+}\nabla\phi_\xi^T)\|^{2q}_{C^\alpha(B_1)}
\big\rangle^\frac{1}{q} \leq C^2q^{2C}.
\end{align}
\end{lemma}

\begin{proof}
Because of~(A2)$_{1}$ resp.\ (A3)$_{1}$ from Assumption~\ref{assumption:operators}
as well as~(R) from Assumption~\ref{assumption:smallScaleReg}, this is an
immediate consequence of Lemma~\ref{lem:annealedRegCorrectorNonlinear}.
\end{proof}

In terms of regularity theory in the random setting, the last missing
ingredient is given by an annealed Calder\'on--Zygmund estimate. We emphasize
that for the purpose of the present work, it suffices to consider a perturbative regime.

\begin{lemma}[The annealed Calder\'on--Zygmund estimate: the perturbative regime, cf.\ \cite{Josien2020}]
\label{lem:annealedCZMeyers}
Consider an ensemble of uniformly elliptic and 
bounded coefficient fields $a\colon\Rd\to\Rd[d\times d]$ with respect to constants $(\lambda,\Lambda)$. 
There exists $c=c(d,\lambda,\Lambda)>0$ such that for all $q\in [2,c)$
and all random fields $f,g\in L^2(\Rd)$, the solution of
\begin{align*}
\frac{1}{T}u - \nabla\cdot a\nabla u = \frac{1}{T}f + \nabla\cdot g \quad\text{in } \Rd
\end{align*} 
is subject to the estimate
\begin{align}
\label{eq:annealedCZMeyers}
\tag{T8}
\Big\|\Big\langle\Big|\Big(\frac{u}{\sqrt{T}}, \nabla u\Big)\Big|^q\Big\rangle^\frac{1}{q}\Big\|^2_{L^2(\Rd)}
\lesssim_{d,\lambda,\Lambda}
\Big\|\Big\langle\Big|\Big(\frac{f}{\sqrt{T}}, g\Big)\Big|^q\Big\rangle^\frac{1}{q}\Big\|^2_{L^2(\Rd)}.
\end{align}
\end{lemma}

\begin{proof}
For a proof, we refer the reader to~\cite[Proposition~7.1(i)]{Josien2020}.
\end{proof}

\section{Existence of localized, higher-order linearized homogenization correctors and flux correctors}
\label{app:existenceCorrectors}

\begin{proof}[Proof of Lemma~\ref{eq:lemmaExistenceLinearizedCorrectorsShort}
(Existence of localized correctors)]
We proceed by an induction over the linearization order $L\in\N$. For the 
sake of completeness, we refer the reader to~\cite[Lemma~12]{Fischer2019}
for the existence of localized correctors $\phi^T_\xi$ of the nonlinear
corrector problem~\eqref{eq:PDEMonotoneHomCorrectorLocalized}.

\textit{Step 1: (Base case)} Let the requirements and notation 
of (A1), (A2)$_{0}$ and (A3)$_{0}$ of Assumption~\ref{assumption:operators} be in place, 
and consider an \textit{arbitrary} measurable parameter field $\tilde\omega\colon\Rd\to B^n_1$,
as well as a unit vector $v\in\Rd$. Then, there exists a unique solution 
\begin{align}
\label{eq:existenceBaseCase}
\Big(\frac{\phi^T_{\xi,v}(\cdot,\tilde\omega)}{\sqrt{T}},
\nabla\phi^T_{\xi,v}(\cdot,\tilde\omega)\Big)\in L^2_{\mathrm{uloc}}(\Rd;\Rd[]{\times}\Rd)
\end{align}
of the first-order linearized corrector problem with massive term given by
\begin{align}
\label{eq:firstOrderEquation}
\frac{1}{T}\phi^T_{\xi,v} - 
\nabla\cdot\partial_\xi A(\tilde\omega,\xi{+}\nabla\phi^T_\xi)\nabla\phi^T_{\xi,v}
= \nabla\cdot\partial_\xi A(\tilde\omega,\xi{+}\nabla\phi^T_\xi)v.
\end{align}
Under the stronger set of conditions (A1), (A2)$_{1}$ and (A3)$_{1}$ of Assumption~\ref{assumption:operators},
and under the stronger requirement~\eqref{eq:ergodicityParameterField} on the
parameter field $\tilde\omega\colon\Rd\to B^n_1$, we in addition claim that for all $p\geq 2$
\begin{align}
\label{eq:ergodicityLinearizedCorrectorBaseCase}
\sup_{x_0\in\Rd} \limsup_{R\to\infty}\dashint_{B_R(x_0)}
\Big|\Big(\frac{\phi^T_{\xi,v}(\cdot,\tilde\omega)}{\sqrt{T}},
\nabla\phi^T_{\xi,v}(\cdot,\tilde\omega)\Big)\Big|^p < \infty.
\end{align} 

\textit{Proof of first claim (Existence of solutions to~\eqref{eq:firstOrderEquation} 
in the function space~\eqref{eq:existenceBaseCase}):}
For any $R\geq 1$, there exists a unique Lax-Milgram solution
$\phi^{T,R}_{\xi,v}\in H^1(\Rd)$ of
\begin{align}
\label{eq:firstOrderEquationApprox}
\frac{1}{T}\phi^{T,R}_{\xi,v} - 
\nabla\cdot\partial_\xi A(\tilde\omega,\xi{+}\nabla\phi^T_\xi)\nabla\phi^{T,R}_{\xi,v}
= \nabla\cdot \mathds{1}_{B_R}\partial_\xi A(\tilde\omega,\xi{+}\nabla\phi^T_\xi)v.
\end{align}
The sequence $\big(\frac{\phi^{T,R}_{\xi,v}(\cdot,\tilde\omega)}{\sqrt{T}},
\nabla\phi^{T,R}_{\xi,v}(\cdot,\tilde\omega)\big)_{R\geq 1}$ is Cauchy in 
$L^2_{\mathrm{uloc}}(\Rd;\Rd[]{\times}\Rd)$ as a consequence
of the weighted energy estimate~\eqref{eq:expLocalization}
applied to differences of solutions to~\eqref{eq:firstOrderEquationApprox},
and the limit is easily identified as a distributional solution to~\eqref{eq:firstOrderEquation}.
Uniqueness follows again by an application of the weighted energy estimate~\eqref{eq:expLocalization},
this time with respect to two solutions of~\eqref{eq:firstOrderEquation}
in the function space~\eqref{eq:existenceBaseCase}.

\textit{Proof of second claim (Improved regularity~\eqref{eq:ergodicityLinearizedCorrectorBaseCase}
under stronger assumptions):} Thanks to the assumption~\eqref{eq:ergodicityParameterField}
on the parameter field $\tilde\omega$, the proof of Lemma~\ref{lem:annealedRegCorrectorNonlinear}---in particular
the proof of the annealed estimate~\eqref{eq:annealedHoelderGradientNonlinearCorrector} for
the H\"older seminorm---carries over verbatim to the present setting. Indeed, one simply
needs to replace stochastic moments $\langle|\cdot|^q\rangle$ by 
$\limsup_{R\to\infty}\,\dashint_{B_R(x_0)} |\cdot|^p$, the random variable $\mathcal{X}_\eta$
defined in~\eqref{eq:auxiliaryRV} by the field $\mathcal{X}_\eta(x,\tilde\omega) := \sup_{y,z\in B_1(x),\,y\neq z} 
\frac{|\tilde\omega(y)-\tilde\omega(z)|}{|y-z|^\eta}$, and the condition~(R) of Assumption~\ref{assumption:smallScaleReg}
by the assumption~\eqref{eq:ergodicityParameterField}. The upshot of this is that we obtain 
an ``annealed'' H\"older estimate on the linearized coefficient 
$\partial_\xi A(\tilde\omega,\xi{+}\nabla\phi^T_{\xi}(\cdot,\tilde\omega))$ in form of
\begin{align}
\label{eq:ergodicHoelderRegLinCoefficient}
\sup_{x_0\in\Rd} \limsup_{R\to\infty}\dashint_{B_R(x_0)}
\big\|\partial_\xi A(\tilde\omega,\xi{+}\nabla\phi^T_{\xi}(\cdot,\tilde\omega))\big\|^p_{C^\alpha(B_1(x))} < \infty
\end{align}
for some suitable $\alpha=\alpha(d,\lambda,\Lambda)\in (0,\eta)$
and all $p\geq 2$.

The information provided by~\eqref{eq:ergodicHoelderRegLinCoefficient}
is now leveraged as follows. By means of the local
Calder\'on--Zygmund estimate~\eqref{eq:localCalderonZygmund}
applied to the equation~\eqref{eq:firstOrderEquation}, the estimate~\eqref{eq:ergodicHoelderRegLinCoefficient}
and the regularity $\big(\frac{\phi^{T}_{\xi,v}(\cdot,\tilde\omega)}{\sqrt{T}},
\nabla\phi^{T}_{\xi,v}(\cdot,\tilde\omega)\big) \in L^2_{\mathrm{uloc}}(\Rd;\Rd[]{\times}\Rd)$
we infer that for all $p\geq 2$ it holds
\begin{align*}
&\sup_{x_0\in\Rd} \limsup_{R\to\infty}\dashint_{B_R(x_0)}
\Big|\Big(\frac{\phi^T_{\xi,v}(\cdot,\tilde\omega)}{\sqrt{T}},
\nabla\phi^T_{\xi,v}(\cdot,\tilde\omega)\Big)\Big|^p 
\\&
\lesssim \sup_{x_0\in\Rd}\limsup_{R\to\infty}\dashint_{B_R(x_0)}\bigg(\,\dashint_{B_1(x)}
\Big\|\Big(\frac{\phi^T_{\xi,v}(\cdot,\tilde\omega)}{\sqrt{T}},
\nabla\phi^T_{\xi,v}(\cdot,\tilde\omega)\Big)\Big\|_{L^p(B_1(x))}^p\bigg)\dx < \infty.
\end{align*} 
This concludes the proof of~\eqref{eq:ergodicityLinearizedCorrectorBaseCase}.
 
\textit{Step 2: (Formulation of the induction hypotheses)} Let $L\geq 2$ and $T\in [1,\infty)$
be fixed. Let the requirements and notation 
of (A1), (A2)$_{L{-}1}$ and (A3)$_{L{-}1}$ of Assumption~\ref{assumption:operators} be in place.
Fix also a parameter field $\tilde\omega\colon\Rd\to B^n_1$ subject to the condition~\eqref{eq:ergodicityParameterField}.
For any linearization order $1\leq l\leq L{-}1$, and any collection of unit vectors $v_1',\ldots,v_l'\in\Rd$
we assume that---under the above conditions---the associated localized 
$l$th-order linearized homogenization corrector in direction $B':=v_1'\odot\cdots\odot v_l'$
\begin{align}
\label{eq:existenceIndHypo}
\Big(\frac{\phi^T_{\xi,B'}(\cdot,\tilde\omega)}{\sqrt{T}},
\nabla\phi^T_{\xi,B'}(\cdot,\tilde\omega)\Big)\in L^2_{\mathrm{uloc}}(\Rd;\Rd[]{\times}\Rd)
\end{align}
exists, and is subject to the estimate
\begin{align}
\label{eq:ergodicityLinearizedCorrectorIndHypo}
\sup_{x_0\in\Rd} \limsup_{R\to\infty}\dashint_{B_R(x_0)}
\Big|\Big(\frac{\phi^T_{\xi,B'}(\cdot,\tilde\omega)}{\sqrt{T}},
\nabla\phi^T_{\xi,B'}(\cdot,\tilde\omega)\Big)\Big|^p < \infty
\end{align} 
for all $p\geq 2$.

\textit{Step 3: (Induction step)} Let $L\geq 2$ and $T\in [1,\infty)$
be fixed. Let the requirements and notation 
of (A1), (A2)$_{L{-}1}$ and (A3)$_{L{-}1}$ of Assumption~\ref{assumption:operators} be in place.
Let $\tilde\omega\colon\Rd\to B^n_1$ be a parameter field subject to the condition~\eqref{eq:ergodicityParameterField}.
We finally fix a set of unit vectors $v_1,\ldots,v_L\in\Rd$ and define $B:=v_1\odot\cdots\odot v_L$.

As a consequence of the induction hypothesis~\eqref{eq:ergodicityLinearizedCorrectorIndHypo},
there exists for any $R\geq 1$ a unique Lax-Milgram solution $\phi^{T,R}_{\xi,B}\in H^1(\Rd)$ of
\begin{align*}
&\frac{1}{T}\phi^{T,R}_{\xi,B} 
- \nabla\cdot a_\xi^T\nabla \phi^{T,R}_{\xi,B}
\\&
= \nabla\cdot \mathds{1}_{B_R}\sum_{\substack{\Pi\in\mathrm{Par}\{1,\ldots,L\} \\ \Pi\neq\{\{1,\ldots,L\}\}}}
\partial_\xi^{|\Pi|} A(\tilde\omega,\xi{+}\nabla \phi^T_\xi)
\Big[\bigodot_{\pi\in\Pi}(\mathds{1}_{|\pi|=1}B'_\pi + \nabla \phi^T_{\xi,B'_\pi})\Big].
\end{align*}
An application of the weighted energy estimate~\eqref{eq:expLocalization}
to differences of solutions with respect to the equation from the previous display,
and making use of the induction hypothesis~\eqref{eq:ergodicityLinearizedCorrectorIndHypo} 
shows that the sequence $\big(\frac{\phi^{T,R}_{\xi,B}(\cdot,\tilde\omega)}{\sqrt{T}},
\nabla\phi^{T,R}_{\xi,B}(\cdot,\tilde\omega)\big)_{R\geq 1}$ is Cauchy in the desired function space
$L^2_{\mathrm{uloc}}(\Rd;\Rd[]{\times}\Rd)$. Details are left to the reader. Moreover, the limit constitutes
the unique distributional solution of the linearized corrector problem~\eqref{eq:PDEhigherOrderLinearizedCorrectorLocalized}
in the required function space. The proof of~\eqref{eq:ergodicityLinearizedCorrector}
follows along the same lines as the argument in favor of~\eqref{eq:ergodicityLinearizedCorrectorBaseCase}.
This in turn concludes the proof of the induction step.

\textit{Step 4: (Existence of linearized flux correctors)}
This is a straightforward consequence of standard arguments relying on the form of the
flux corrector equations~\eqref{eq:PDEhigherOrderLinearizedFluxCorrectorLocalized}
resp.\ \eqref{eq:PDEhigherOrderLinearizedHelmholtzCorrectorLocalized}, the
already established existence and regularity results for linearized homogenization correctors $\phi^T_{\xi,B}$,
and the definition~\eqref{eq:HigherOrderLinearizedFlux} of linearized fluxes $q^T_{\xi,B}$.

\textit{Step 5: (Almost sure existence for random parameter fields)}
As a consequence of the small-scale regularity condition~(R) of Assumption~\ref{assumption:smallScaleReg}
and Birkhoff's ergodic theorem (recall to this end Assumption~\ref{assumption:ensembleParameterFields}),
there exists a subset $\Omega'\subset\Omega$ of full $\Prob$-measure such that all random fields $\omega\in\Omega'$
satisfy the condition~\eqref{eq:ergodicityParameterField}. Hence, the claim on almost sure
existence of linearized correctors for random parameter fields follows immediately from the
previous four steps of this proof. This in turn concludes the 
proof of Lemma~\ref{eq:lemmaExistenceLinearizedCorrectorsShort}.
\end{proof}

\begin{lemma}[G\^ateaux differentiability of localized correctors 
with respect to parameter fields]
\label{eq:lemmaExistenceLinearizedCorrectorsExtended}
Let $L\in\N$ and $T\in [1,\infty)$ be fixed.
Let the requirements and notation of (A1), (A2)$_L$ and (A3)$_L$ of
Assumption~\ref{assumption:operators} be in place.
We also fix a parameter field $\tilde\omega\colon\Rd\to B^1_n$
subject to the condition~\eqref{eq:ergodicityParameterField}.
Consider in addition a smooth parameter field $\delta\tilde\omega\colon\Rd\to B^1_n$ 
being compactly supported in the unit ball $B_1$.

Then, for every $\xi\in\Rd$ and every $B:=v_1\odot\cdots\odot v_L$ formed by unit vectors $v_1,\ldots,v_L\in\Rd$,
the associated linearized homogenization corrector $\phi^T_{\xi,B}(\cdot,\tilde\omega)$ from
Lemma~\ref{eq:lemmaExistenceLinearizedCorrectorsShort} is
G\^ateaux differentiable at $\tilde\omega$ in direction of $\delta\tilde\omega$. The corresponding
G\^ateaux derivative $\delta\phi^T_{\xi,B}(\cdot,\tilde\omega)$ satisfies
\begin{align}
\label{eq:regFromEquGateaux}
\Big(\frac{\delta\phi^T_{\xi,B}(\cdot,\tilde\omega)}{\sqrt{T}},
\nabla\delta\phi^T_{\xi,B}(\cdot,\tilde\omega)\Big)
\in L^2_{\mathrm{uloc}}(\Rd;\Rd[]{\times}\Rd)
\end{align}
as well as
\begin{align}
\label{eq:ergodicityLinearizedCorrectorGateaux}
\sup_{x_0\in\Rd} \limsup_{R\to\infty}\dashint_{B_R(x_0)}
\Big|\Big(\frac{\delta\phi^T_{\xi,B}(\cdot,\tilde\omega)}{\sqrt{T}},
\nabla\delta\phi^T_{\xi,B}(\cdot,\tilde\omega)\Big)\Big|^p < \infty.
\end{align} 
Analogous statements hold true for the linearized flux correctors
$\sigma^T_{\xi,B}(\cdot,\tilde\omega)$ resp.\ $\psi^T_{\xi,B}(\cdot,\tilde\omega)$
from Lemma~\ref{eq:lemmaExistenceLinearizedCorrectorsShort}. 

In particular, under the requirements of 
(A1), (A2)$_L$ and (A3)$_L$ of Assumption~\ref{assumption:operators},
(P1) and (P2) of Assumption~\ref{assumption:ensembleParameterFields}, and 
(R) of Assumption~\ref{assumption:smallScaleReg}, there exists
a set $\Omega'\subset\Omega$ of full $\Prob$-measure on which 
the existence of G\^ateaux derivatives for
(higher-order) linearized correctors is guaranteed
in the above sense for all random parameter fields $\omega\in\Omega'$,
with directions given by all smooth $\delta\omega\colon\Rd\to B^n_1$ 
which are compactly supported in the unit ball $B_1$.
\end{lemma}

\begin{proof}
By an induction over the linearization order $0\leq l\leq L$, one 
may provide solutions with the regularity~\eqref{eq:regFromEquGateaux}
and~\eqref{eq:ergodicityLinearizedCorrectorGateaux} to the equations
obtained from the linearized corrector problem~\eqref{eq:PDEhigherOrderLinearizedCorrectorLocalized}
by formally differentiating with respect to the parameter field. Indeed,
by the same arguments as in the proof of Lemma~\ref{eq:lemmaExistenceLinearizedCorrectorsShort},
there exists a unique solution $(\frac{\delta\phi^T_{\xi,B}(\cdot,\tilde\omega)}{\sqrt{T}},
\nabla\delta\phi^T_{\xi,B}(\cdot,\tilde\omega))\in L^2_{\mathrm{uloc}}(\Rd;\Rd[]{\times}\Rd)$ 
with the additional regularity~\eqref{eq:ergodicityLinearizedCorrectorGateaux} to the equation
\begin{align*}
&\frac{1}{T}\delta\phi^T_{\xi,B} - 
\nabla\cdot \partial_\xi A(\tilde\omega,\xi+\nabla\phi^T_\xi)\nabla\delta\phi^T_{\xi,B}
\\&
= \nabla\cdot\partial_\omega\partial_\xi A(\tilde\omega,\xi+\nabla\phi^T_\xi)
\big[\delta\tilde\omega\odot\nabla \phi^T_{\xi,B}\big]
\\&~~~
+ \nabla\cdot\partial_\xi^2 A(\tilde\omega,\xi+\nabla\phi^T_\xi)
\big[\nabla\delta\phi^T_\xi\odot\nabla \phi^T_{\xi,B}\big]
\\&~~~
+\nabla\cdot \sum_{\Pi} \partial_\omega\partial_\xi^{|\Pi|} A(\tilde\omega,\xi{+}\nabla \phi^T_\xi)
\Big[\delta\tilde\omega\odot\bigodot_{\pi\in\Pi}(\mathds{1}_{|\pi|=1}B'_\pi {+} \nabla \phi^T_{\xi,B'_\pi})\Big]
\\&~~~
+\nabla\cdot \sum_{\Pi} \partial_\xi^{1+|\Pi|} A(\tilde\omega,\xi{+}\nabla \phi^T_\xi)
\Big[\nabla\delta\phi_\xi^T\odot\bigodot_{\pi\in\Pi}(\mathds{1}_{|\pi|=1}B'_\pi {+} \nabla \phi^T_{\xi,B'_\pi})\Big]
\\&~~~
+\nabla\cdot \sum_{\Pi} \partial_\xi^{|\Pi|} A(\tilde\omega,\xi{+}\nabla \phi^T_\xi)
\Big[\sum_{\pi\in\Pi}\nabla\delta\phi^T_{\xi,B'_\pi}\odot
\bigodot_{\substack{{\pi'\in\Pi} \\ \pi'\neq \pi}}
(\mathds{1}_{|\pi'|=1}B'_{\pi'} {+} \nabla \phi^T_{\xi,B'_{\pi'}})\Big].
\end{align*}
In order to identify $\delta\phi^T_{\xi,B}(\cdot,\tilde\omega)$ as the
G\^ateaux derivative of $\phi^T_{\xi,B}(\cdot,\tilde\omega)$ in direction
of the compactly supported and smooth perturbation $\delta\tilde\omega$,
one proceeds as follows. For any $|h|\leq 1$, note that $\tilde\omega + h\delta\tilde\omega$
also satisfies~\eqref{eq:ergodicityParameterField}. In particular, for any $|h|\leq 1$
one may construct a linearized homogenization corrector $\phi^T_{\xi,B}(\cdot,\tilde\omega + h\delta\tilde\omega)$
in the precise sense of Lemma~\ref{eq:lemmaExistenceLinearizedCorrectorsShort}.
Based on that observation, the next step consists of studying the equation
satisfied by the ``first-order Taylor expansion'' $$\phi^T_{\xi,B}(\cdot,\tilde\omega + h\delta\tilde\omega)
- \phi^T_{\xi,B}(\cdot,\tilde\omega) - \delta\phi^T_{\xi,B}(\cdot,\tilde\omega)h.$$
As a consequence of~(A2)$_L$ and~(A3)$_L$ from Assumption~\ref{assumption:operators},
the weighted energy estimate~\eqref{eq:expLocalization} applied to the equation satisfied
by the expansion from the previous display, and an induction over the linearization
order $0\leq l\leq L$ we obtain by straightforward computations that
\begin{align*}
\sup_{x_0\in\Rd} \Big\|\frac{\phi^T_{\xi,B}(\cdot,\tilde\omega {+} h\delta\tilde\omega)
{-} \phi^T_{\xi,B}(\cdot,\tilde\omega) {-} \delta\phi^T_{\xi,B}(\cdot,\tilde\omega)h}{\sqrt{T}}\Big\|_{L^2(B_1(x_0))}
&= o(|h|),
\\
\sup_{x_0\in\Rd} \Big\|\nabla\phi^T_{\xi,B}(\cdot,\tilde\omega {+} h\delta\tilde\omega)
{-} \nabla\phi^T_{\xi,B}(\cdot,\tilde\omega) {-} \nabla\delta\phi^T_{\xi,B}(\cdot,\tilde\omega)h\Big\|_{L^2(B_1(x_0))}
&= o(|h|).
\end{align*}
This in turn entails the asserted differentiability result for the
linearized homogenization corrector $\phi^T_{\xi,B}(\cdot,\tilde\omega)$.

Finally, the claims from the statement of Lemma~\ref{eq:lemmaExistenceLinearizedCorrectorsExtended}
concerning linearized flux correctors and random parameter fields now follow as in
the proof of Lemma~\ref{eq:lemmaExistenceLinearizedCorrectorsShort}.
\end{proof}

\section{Corrector bounds for higher-order linearizations: Proofs for the base cases}
\label{app:baseCaseInd}

\subsection{Proof of Proposition~\ref{prop:estimates1stOrderCorrector}
{\normalfont (Estimates for localized homogenization correctors of the nonlinear problem)}}
The corrector estimates~\eqref{eq:BaseCaseCorrectorBounds} hold true by a combination
of \cite[Lemma~17a), Lemma~19, Estimate~(111)]{Fischer2019}. The small-scale
annealed Schauder estimate~\eqref{eq:BaseCaseAnnealedSchauder} was already established
in Lemma~\ref{lem:annealedRegCorrectorNonlinear}.

For a proof of~\eqref{eq:BaseCaseRepresentationMalliavinDerivative}
and~\eqref{eq:BaseCaseSensitivityBound}, let $\delta\omega\colon\Rd\to \Rd[n]$ 
be compactly supported and smooth with $\|\delta\omega\|_{L^\infty}\leq 1$. Differentiating the defining 
equation~\eqref{eq:PDEMonotoneHomCorrectorLocalized} for the localized homogenization
corrector with respect to the parameter field in the direction of $\delta\omega$ yields
\begin{align*}
\frac{1}{T}\delta\phi^T_\xi - \nabla\cdot\partial_\xi A(\omega,\xi{+}\nabla\phi^T_\xi)\nabla\delta\phi^T_\xi
= \nabla\cdot\partial_\omega A(\omega,\xi{+}\nabla\phi^T_\xi)\delta\omega,
\end{align*}
with $(\frac{\delta\phi^T_\xi}{\sqrt{T}},\nabla\delta\phi^T_\xi)\in L^2(\Rd;\Rd[]{\times}\Rd)$.
By the usual duality argument, we compute for the centered random variable
$F_\phi:=\int g\cdot\nabla\phi^T_\xi$ that (with $a^{T,*}_\xi$ denoting the
transpose of the uniformly elliptic and bounded coefficient field 
$\partial_\xi A(\omega,\xi{+}\nabla\phi^T_\xi)$)
\begin{align*}
\delta F_\phi &= \int g\cdot\nabla\delta\phi^T_\xi
= \int \partial_\omega A(\omega,\xi{+}\nabla\phi^T_\xi)\delta\omega \cdot
\nabla\Big(\frac{1}{T}{-}\nabla\cdot a^{T,*}_\xi\nabla\Big)^{-1}(\nabla\cdot g).
\end{align*}
This proves~\eqref{eq:BaseCaseRepresentationMalliavinDerivative}
as by means of~(A3)$_0$ of Assumption~\ref{assumption:operators} it holds
\begin{align}
\label{eq:dualRandomFieldSensitivity}
G^T_\xi := \big(\partial_\omega A(\omega,\xi{+}\nabla\phi^T_\xi)\big)^*
\nabla\Big(\frac{1}{T}{-}\nabla\cdot a^{T,*}_\xi\nabla\Big)^{-1}(\nabla\cdot g)
\in L^1_{\mathrm{uloc}}(\Rd;\Rd[n]).
\end{align}
From the previous display, it follows by duality in $L^q_{\langle\cdot\rangle}$, the stationarity
of the localized homogenization corrector $\phi^T_\xi$, and H\"older's inequality that
\begin{align*}
&\bigg\langle\bigg|\,\int\bigg(
\,\dashint_{B_1(x)}|G^T_{\xi}|\,\bigg)^2\bigg|^{q}\bigg\rangle^\frac{1}{q}
\\&
\leq C^2q^2\big\langle\big\|\xi{+}\nabla\phi^T_\xi\big\|^{2\frac{q}{\kappa}}_{L^2(B_1)}
\big\rangle^\frac{\kappa}{q}
\sup_{\langle F^{2q_*} = 1 \rangle}
\int\Big\langle\Big|\nabla\Big(\frac{1}{T}{-}\nabla\cdot a^{T,*}_\xi\nabla\Big)^{-1}
(\nabla\cdot Fg)\Big|^{2(\frac{q}{\kappa})_*}\Big\rangle^\frac{1}{(\frac{q}{\kappa})_*}.
\end{align*}
For large enough $q\in [1,\infty)$, we may then apply the annealed Calder\'on--Zygmund
estimate from~\eqref{eq:annealedCZMeyers} to infer from the previous display that
\begin{align*}
&\bigg\langle\bigg|\,\int\bigg(
\,\dashint_{B_1(x)}|G^T_{\xi}|\,\bigg)^2\bigg|^{q}\bigg\rangle^\frac{1}{q}
\\&
\leq C^2q^2\big\langle\big\|\xi{+}\nabla\phi^T_\xi\big\|^{2(q/\kappa)}_{L^2(B_1)}
\big\rangle^\frac{1}{q/\kappa}\sup_{\langle F^{2q_*}\rangle=1}
\int\big\langle|Fg|^{2(q/\kappa)_*}\big\rangle^\frac{1}{(q/\kappa)_*}.
\end{align*}
The estimate~\eqref{eq:BaseCaseSensitivityBound} thus follows from
the corrector estimates~\eqref{eq:BaseCaseCorrectorBounds}. 

Finally, denote by $G^{T,r}_\xi$ the random field with values in
$L^1_{\mathrm{uloc}}(\Rd;\Rd[n])$ defined by~\eqref{eq:dualRandomFieldSensitivity}, however with
$g$ replaced by $g_r$. Since
\begin{align*}
\nabla\Big(\frac{1}{T}{-}\nabla\cdot a^{T,*}_\xi\nabla\Big)^{-1}(\nabla\cdot g_r)
\to \nabla\Big(\frac{1}{T}{-}\nabla\cdot a^{T,*}_\xi\nabla\Big)^{-1}(\nabla\cdot g) \in L^2_{\mathrm{uloc}}(\Rd;\Rd)
\end{align*}
it follows that $G^{T,r}_\xi$ converges in $L^1_{\mathrm{uloc}}(\Rd;\Rd[n])$ to
the random field $G^{T}_\xi$ defined by~\eqref{eq:dualRandomFieldSensitivity}.
By Fatou's lemma and the estimate~\eqref{eq:BaseCaseSensitivityBound} being
already established for $G^{T,r}_\xi$, $r\geq 1$, we then infer that
\begin{align*}
\bigg\langle\bigg|\,\int\bigg(
\,\dashint_{B_1(x)}|G^T_{\xi}|\,\bigg)^2\bigg|^{q}\bigg\rangle^\frac{1}{q}
&\leq \liminf_{r\to\infty} \bigg\langle\bigg|\,\int\bigg(
\,\dashint_{B_1(x)}|G^{T,r}_{\xi}|\,\bigg)^2\bigg|^{q}\bigg\rangle^\frac{1}{q}
\\&
\leq C^2q^{2C}|\xi|^2\sup_{\langle F^{2q_*}\rangle=1}
\int\big\langle|Fg|^{2(q/\kappa)_*}\big\rangle^\frac{1}{(q/\kappa)_*}.
\end{align*}  
This in turn concludes the proof of Proposition~\ref{prop:estimates1stOrderCorrector}. \qed

\subsection{Estimates for differences of localized homogenization correctors of the nonlinear problem}
We next turn to a result which provides a proof of the base case 
for the induction in the proof of Lemma~\ref{lem:differencesLinearizedCorrectors}.

\begin{lemma}[Estimates for differences of localized homogenization correctors of the nonlinear problem]
\label{prop:estimatesDiffHomCorrectorNonlinear}
Let the requirements and notation of (A1), (A2)$_1$ and (A3)$_1$ of
Assumption~\ref{assumption:operators}, (P1) and (P2) of
Assumption~\ref{assumption:ensembleParameterFields},
and~(R) of Assumption~\ref{assumption:smallScaleReg} be in place.
Let $T\in [1,\infty)$ and $M>0$ be fixed. For any vector $\xi\in\Rd$ let 
\begin{align*}
\phi^T_{\xi}\in H^1_{\mathrm{uloc}}(\Rd)
\end{align*}
denote the unique solution of the localized corrector problem~\eqref{eq:PDEMonotoneHomCorrectorLocalized}.
For any unit vector $e\in\Rd$ and any $|h|\leq 1$, the
difference of localized homogenization correctors $\phi^T_{\xi+he}-\phi^T_\xi$ 
then satisfies the following estimates:
\begin{itemize}[leftmargin=0.4cm]
\item There exists a constant $C=C(d,\lambda,\Lambda,\nu,\rho,\eta,M)$ such that for
all $|\xi|\leq M$, all $q\in [1,\infty)$, and all compactly supported and square-integrable $f,g$ we have 
\emph{corrector estimates for differences}
\begin{equation}
\label{eq:BaseCaseCorrectorBoundsDiff}
\begin{aligned}
\bigg\langle\bigg|\int g\cdot\big(\nabla\phi^T_{\xi+he}{-}\nabla\phi^T_{\xi}\big)
\bigg|^{2q}\bigg\rangle^\frac{1}{q}
&\leq C^2q^{2C}|h|^{2}\int \big|g\big|^2,
\\ 
\bigg\langle\bigg|\int \frac{1}{T}f \big(\phi^T_{\xi+he}{-}\phi^T_{\xi}\big)\bigg|^{2q}\bigg\rangle^\frac{1}{q}
&\leq C^2q^{2C}|h|^{2}\int \Big|\frac{f}{\sqrt{T}}\Big|^2,
\\
\Big\langle\Big\|\Big(\frac{\phi^T_{\xi+he}{-}\phi^T_{\xi}}{\sqrt{T}},
\nabla\phi^T_{\xi+he}{-}\nabla\phi^T_{\xi}\Big)\Big\|^{2q}_{L^2(B_1)}\Big\rangle^\frac{1}{q}
&\leq C^2q^{2C}|h|^{2}.
\end{aligned}
\end{equation}
\begin{subequations}
\item Fix $p\in (2,\infty)$, and let $g$ be a compactly supported and $p$-integrable random field. 
						Then there exists a random field
						$G_{\xi,h,e}^T\in L^1_{\mathrm{uloc}}(\Rd;\Rd[n])$ being related to $g$ via $\phi^T_{\xi+he}{-}\phi^T_{\xi}$ 
						in the sense that, $\Prob$-almost surely, it holds for all compactly supported and smooth perturbations 
						$\delta\omega\colon\Rd\to \Rd[n]$ with $\|\delta\omega\|_{L^\infty}\leq 1$
						\begin{align}
						\label{eq:BaseCaseRepresentationMalliavinDerivativeDiff}
						\int g\cdot\nabla\big(\delta\phi^T_{\xi+he}{-}\delta\phi^T_{\xi}\big) 
						= \int G^T_{\xi,h,e}\cdot\delta\omega.
						\end{align}
						
						For any $\kappa\in (0,1]$, there then exists 
						a constant $C=C(d,\lambda,\Lambda,\nu,\rho,M,\eta,\kappa)$
						such that for all $|\xi|\leq M$, and all $q\in [1,\infty)$ the random field
						$G^T_{\xi,h,e}$ gives rise to a \emph{sensitivity estimate for differences of correctors}
						\begin{align}
						\label{eq:BaseCaseSensitivityBoundDiff}
						\bigg\langle\bigg|\,\int\bigg(
						\,\dashint_{B_1(x)}|G^T_{\xi,h,e}|\,\bigg)^2\bigg|^{q}\bigg\rangle^\frac{1}{q}
						&\leq C^2q^{2C}|h|^{2}\sup_{\langle F^{2q_*}\rangle=1}
						\int\big\langle|Fg|^{2(\frac{q}{\kappa})_*}\big\rangle^\frac{1}{(\frac{q}{\kappa})_*}.
						\end{align}
						
						If $(g_r)_{r\geq 1}$ is a sequence of compactly supported and $p$-integrable random fields, denote by
						$G_{\xi,h,e}^{T,r}\in L^1_{\mathrm{uloc}}(\Rd;\Rd[n])$, $r\geq 1$, the random field
						associated to $g_r$, $r\geq 1$, in the sense of~\eqref{eq:BaseCaseRepresentationMalliavinDerivativeDiff}.
						Let $g$ be an $L^p_{\mathrm{uloc}}(\Rd;\Rd)$-valued random field, and assume that
						$\Prob$-almost surely it holds $g_r\to g$ in $L^p_{\mathrm{uloc}}(\Rd;\Rd)$.
						Then there exists a random field $G^T_{\xi,h,e}$ such that $\Prob$-almost surely
						\begin{align}
						\label{eq:BaseCaseSensitivityApproxDiff}
						G_{\xi,h,e}^{T,r} \to G_{\xi,h,e}^{T} \text{ as } r\to\infty 
						\text{ in } L^1_{\mathrm{uloc}}(\Rd;\Rd[n]).
						\end{align}
						In the special case of $g_r = \mathds{1}_{B_r}g$, $r\geq 1$, the limit random
						field is in addition subject to	the sensitivity estimate~\eqref{eq:BaseCaseSensitivityBoundDiff}.
\end{subequations}
\item There exists an exponent $\alpha=\alpha(d,\lambda,\Lambda)\in (0,\eta)$ and, for any $\beta\in (0,1)$, 
some constant $C=C(d,\lambda,\Lambda,\nu,\rho,\eta,M,\beta)$
such that for all $|\xi|\leq M$, and all $q\in [1,\infty)$ we have a 
\emph{small-scale annealed Schauder estimate} of the form
\begin{align}
\label{eq:BaseCaseAnnealedSchauderDiff}
\big\langle\big\|\nabla\phi^T_{\xi+he}{-}\nabla\phi^T_{\xi}\big\|^{2q}_{C^\alpha(B_1)}\big\rangle^\frac{1}{q}
\leq C^2q^{2C}|h|^{2}.
\end{align}
\end{itemize}
\end{lemma}

\begin{proof}
The estimates~\eqref{eq:BaseCaseCorrectorBoundsDiff} follow from a combination
of the qualitative differentiability result~\cite[Lemma~20]{Fischer2019}
with~\cite[Lemma~21a), Lemma~23, Proposition~14]{Fischer2019}. Note that
these results are even available under the weaker small-scale regularity condition~(R) from
Assumption~\ref{assumption:smallScaleReg} as we still have
annealed H\"older regularity of the linearized coefficient fields
$\partial_\xi A(\omega,\xi+\nabla\phi^T_\xi)$ at our disposal,
see Lemma~\ref{lem:annealedHoelderRegLinCoefficient}.
For a proof of the remaining assertions, note that the equation for the difference of
localized correctors is given by
\begin{align*}
&\frac{1}{T}(\phi^T_{\xi+he} - \phi^T_\xi)
- \nabla\cdot\big\{A(\omega,\xi{+}\nabla\phi^T_{\xi+he})
- A(\omega,\xi{+}\nabla\phi^T_\xi)\big\}
\\
&= - \nabla\cdot\big\{A(\omega,\xi{+}\nabla\phi^T_{\xi+he})
- A(\omega,\xi{+}\nabla\phi^T_{\xi+he}{+}he)\big\}.
\end{align*}
By means of~(A2)$_{0}$ from Assumption~\ref{assumption:operators}, we may express
the equation for the difference in equivalent form as follows:
\begin{align*}
&\frac{1}{T}(\phi^T_{\xi+he} - \phi^T_\xi)
- \nabla\cdot\bigg(\int \partial_\xi A(\omega,\xi{+}s\nabla\phi^T_{\xi+he}{+}(1{-}s)\nabla\phi^T_\xi)\ds\bigg)
(\nabla\phi^T_{\xi+he}-\nabla\phi^T_\xi)
\\
&= - \nabla\cdot\int \partial_\xi A(\omega,\xi{+}\nabla\phi^T_{\xi+he}{+}she) \ds\, he.
\end{align*}

The coefficient in the equation of the previous display is uniformly elliptic and bounded
with respect to the constants $(\lambda,\Lambda)$ from Assumption~\ref{assumption:operators},
and by means of Lemma~\ref{lem:annealedHoelderRegLinCoefficient} H\"older continuous with
an annealed estimate for the associated H\"older norm of the form~\eqref{eq:annealedHoelderRegLinCoefficient}.
In particular, applying the local Schauder estimate~\eqref{eq:localSchauder}
to the equation from the previous display in combination with the corrector
estimates~\eqref{eq:BaseCaseCorrectorBoundsDiff} implies the small-scale annealed Schauder
estimate~\eqref{eq:BaseCaseAnnealedSchauderDiff} for differences of localized correctors.

The remaining two assertions~\eqref{eq:BaseCaseRepresentationMalliavinDerivativeDiff}
and~\eqref{eq:BaseCaseSensitivityBoundDiff} follow by an argument similar to the
proof of~\eqref{eq:BaseCaseRepresentationMalliavinDerivative}
and~\eqref{eq:BaseCaseSensitivityBound}. Details are left to the interested reader.
\end{proof}

\subsection{Differentiability of localized homogenization correctors of the nonlinear problem}
The base case of the induction in the proof of Lemma~\ref{lem:diffMassiveApprox}
is covered by the following result.

\begin{lemma}[Differentiability of localized homogenization correctors of the nonlinear problem]
\label{prop:estimatesRegHomCorrectorNonlinear}
Let the requirements and notation of (A1), (A2)$_1$ and (A3)$_1$ of
Assumption~\ref{assumption:operators}, (P1) and (P2) of
Assumption~\ref{assumption:ensembleParameterFields},
and~(R) of Assumption~\ref{assumption:smallScaleReg} be in place.
Let $T\in [1,\infty)$ and $M>0$ be fixed. For any $\xi\in\Rd$ let 
$\phi^T_{\xi}\in H^1_{\mathrm{uloc}}(\Rd)$ resp.\
$\phi^T_{\xi,e} \in H^1_{\mathrm{uloc}}(\Rd)$
denote the unique solutions of the problems~\eqref{eq:PDEMonotoneHomCorrectorLocalized}
resp.\ \eqref{eq:PDEhigherOrderLinearizedCorrectorLocalized}, and let
$q^T_\xi$ resp.\ $q^T_{\xi,e}$ denote the associated fluxes 
from~\eqref{eq:fluxNonlinear} resp.\ \eqref{eq:HigherOrderLinearizedFlux}.

For any unit vector $e\in\Rd$ and any $|h|\leq 1$, the first-order Taylor expansion
of localized homogenization correctors $\phi^T_{\xi+he}{-}\phi^T_\xi{-}\phi^T_{\xi,e}h$
then satisfies the following estimate: there exists a constant $C=C(d,\lambda,\Lambda,\nu,\rho,\eta,M)$
such that for all $|\xi|\leq M$ it holds
\begin{align}
\label{eq:BaseCaseTaylor1}
\big\langle\big\|\phi^T_{\xi+he}{-}\phi^T_\xi{-}\phi^T_{\xi,e}h\big\|^2_{L^2(B_1)}
\big\rangle &\leq C^2h^{4}.
\end{align}
In particular, the map $\xi\mapsto\nabla\phi^T_{\xi}$ is Fr\'echet differentiable
with values in the Fr\'echet space $L^2_{\langle\cdot\rangle}L^2_{\mathrm{loc}}(\Rd)$.
Finally, we also have the estimate
\begin{align}
\label{eq:BaseCaseRegMassiveVersionHomOperator}
\big|\big\langle q^T_{\xi+he} \big\rangle
- \big\langle q^T_{\xi} \big\rangle
- \big\langle q^T_{\xi,e}h\big\rangle\big|^{2}
&\leq C^2h^{4}.
\end{align}
\end{lemma}

\begin{proof}
For a proof of~\eqref{eq:BaseCaseTaylor1}, we start by computing 
the equation for the first-order Taylor expansion
of localized homogenization correctors $\phi^T_{\xi+he}{-}\phi^T_\xi{-}\phi^T_{\xi,e}h$.
Abbreviating $a^T_\xi:=\partial_\xi A(\omega,\xi{+}\nabla\phi^T_\xi)$ and adding zero yields
\begin{align*}
&\frac{1}{T}(\phi^T_{\xi+he}{-}\phi^T_\xi{-}\phi^T_{\xi,e}h)
- \nabla\cdot a^T_\xi(\nabla\phi^T_{\xi+he}{-}\nabla\phi^T_\xi{-}\nabla\phi^T_{\xi,e}h)
\\&
= -\nabla\cdot \Big\{\partial_\xi A(\omega,\xi{+}\nabla\phi^T_\xi)\nabla\phi^T_{\xi+he}
- A(\omega,\xi{+}he{+}\nabla\phi^T_{\xi+he})\Big\}
\\&~~~
+ \nabla\cdot \Big\{\partial_\xi A(\omega,\xi{+}\nabla\phi^T_\xi)\nabla\phi^T_\xi
-	A(\omega,\xi{+}\nabla\phi^T_\xi)\Big\}
\\&~~~
- \nabla\cdot \partial_\xi A(\omega,\xi{+}\nabla\phi^T_\xi)eh.
\end{align*}
Adding zero again, we may rewrite the equation from the previous display
in the following equivalent form
\begin{align}
\label{eq:TaylorEquAux}
\frac{1}{T}(\phi^T_{\xi+he}{-}\phi^T_\xi{-}\phi^T_{\xi,e}h)
- \nabla\cdot a^T_\xi(\nabla\phi^T_{\xi+he}{-}\nabla\phi^T_\xi{-}\nabla\phi^T_{\xi,e}h)
= -\nabla\cdot\sum_{i=1}^3 R_i
\end{align}
with the divergence form right hand side terms given by
\begin{align*}
R_1 &:= - \Big\{A(\omega,\xi{+}he{+}\nabla\phi^T_\xi) - A(\omega,\xi{+}\nabla\phi^T_\xi)
- \partial_\xi A(\omega,\xi{+}\nabla\phi^T_\xi)eh\Big\}
\\
R_2 &:= A(\omega,\xi{+}he{+}\nabla\phi^T_\xi) 
- A(\omega,\xi{+}he{+}\nabla\phi^T_{\xi+he})
\\&~~~~
- \partial_\xi A(\omega,\xi{+}he{+}\nabla\phi^T_\xi)(\nabla\phi^T_\xi - \nabla\phi^T_{\xi+he})
\\
R_3 &:= \big(\partial_\xi A(\omega,\xi{+}\nabla\phi^T_\xi)
- \partial_\xi A(\omega,\xi{+}he{+}\nabla\phi^T_\xi)\big)
(\nabla\phi^T_\xi - \nabla\phi^T_{\xi+he}).
\end{align*}
As a consequence of the weighted energy estimate~\eqref{eq:expLocalization}
applied to equation~\eqref{eq:TaylorEquAux} as well as stationarity we then obtain
\begin{align*}
\big\langle\big\|\phi^T_{\xi+he}{-}\phi^T_\xi{-}\phi^T_{\xi,e}h\big\|^2_{L^2(B_1)}\big\rangle
\leq C^2\sum_{i=1}^3\big\langle\big\|R_i\big\|^2_{L^2(B_1)}\big\rangle.
\end{align*}
Observe that by means of~(A2)$_{0}$ from Assumption~\ref{assumption:operators} we 
may express the right hand side terms of~\eqref{eq:TaylorEquAux} as follows:
\begin{align*}
R_1 &= - \int \Big\{\partial_\xi A(\omega,\xi{+}\nabla\phi^T_\xi{+}she)
- \partial_\xi A(\omega,\xi{+}\nabla\phi^T_\xi)\Big\} eh \ds,
\\
R_2 &= \int \Big\{\partial_\xi A(\omega,\xi{+}he{+}s\nabla\phi^T_{\xi}{+}(1{-}s)\nabla\phi^T_{\xi+he})
- \partial_\xi A(\omega,\xi{+}he{+}\nabla\phi^T_\xi)\Big\} 
\\&~~~~~~~~~~~~~~~~~~~~~~~~~~~~~~~~~~~~~~~~~~~~~~~~~~~~~~~~~~~~~~~~\times
(\nabla\phi^T_\xi - \nabla\phi^T_{\xi+he})\ds,
\\
R_3 &= - \int \partial_\xi^2 A(\omega,\xi{+}\nabla\phi^T_\xi{+}she)
\big[eh\odot (\nabla\phi^T_\xi - \nabla\phi^T_{\xi+he})\big].
\end{align*}
The previous two displays in combination with~(A2)$_{1}$ from Assumption~\ref{assumption:operators},
the corrector estimates for differences~\eqref{eq:BaseCaseCorrectorBoundsDiff},
and H\"older's inequality then imply the asserted estimate~\eqref{eq:BaseCaseTaylor1}.

For a proof of~\eqref{eq:BaseCaseRegMassiveVersionHomOperator}, observe first
that because of stationarity and Jensen's inequality we obtain
\begin{align*}
\big|\big\langle q^T_{\xi+he} \big\rangle
- \big\langle q^T_{\xi} \big\rangle
- \big\langle q^T_{\xi,e}h\big\rangle\big|^{2}
\leq \big\langle\big\|q^T_{\xi+he}-q^T_{\xi}-q^T_{\xi,e}h\big\|^2_{L^2(B_1)}\big\rangle.
\end{align*}
Moreover, by definition~\eqref{eq:fluxNonlinear} resp.\ \eqref{eq:HigherOrderLinearizedFlux}
of the fluxes we have
\begin{align*}
q^T_{\xi+he}-q^T_{\xi}-q^T_{\xi,e}h = 
a^T_\xi(\nabla\phi^T_{\xi+he}{-}\nabla\phi^T_\xi{-}\nabla\phi^T_{\xi,e}h)
- \sum_{i=1}^3 R_i.
\end{align*}
Hence, the above reasoning for the proof of~\eqref{eq:BaseCaseTaylor1} together with
the estimate~\eqref{eq:BaseCaseTaylor1} itself then entails the estimate~\eqref{eq:BaseCaseRegMassiveVersionHomOperator}.
This concludes the proof of Lemma~\ref{prop:estimatesRegHomCorrectorNonlinear}.
\end{proof}

\subsection{Limit passage in massive approximation of the nonlinear corrector problem}
We finally formulate the result covering the base case of the induction in the 
proof of Lemma~\ref{lem:limitMassiveApprox}.

\begin{lemma}[Limit passage in massive approximation of the nonlinear corrector problem]
\label{prop:estimatesLimitHomCorrectorNonlinear}
Let the requirements and notation of (A1), (A2)$_0$ and (A3)$_0$ of
Assumption~\ref{assumption:operators}, as well as (P1) and (P2) of
Assumption~\ref{assumption:ensembleParameterFields} be in place. Let $T\in [1,\infty)$ be fixed. 
For any given vector $\xi\in\Rd$ let $\phi^T_{\xi}\in H^1_{\mathrm{uloc}}(\Rd)$ denote 
the unique solution of~\eqref{eq:PDEMonotoneHomCorrectorLocalized}, 
and let $q^T_\xi$ denote the associated flux from~\eqref{eq:fluxNonlinear}.

There exists a constant $C=C(d,\lambda,\Lambda,\nu,\rho,M)$ such that for all $|\xi|\leq M$ it holds
\begin{align}
\label{eq:BaseCaseConv1}
\big\langle\big\|\nabla\phi^{2T}_{\xi} - \nabla\phi^T_{\xi}
\big\|^{2}_{L^2(B_1)}\big\rangle
&\leq C^2\frac{\mu_*^2(\sqrt{T})}{T}.
\end{align}
In particular, the sequence $(\nabla\phi^T_\xi)_{T\in [1,\infty)}$ is Cauchy 
in $L^2_{\langle\cdot\rangle}L^2_{\mathrm{loc}}(\Rd)$ (with respect
to the strong topology). The limit gives rise to the unique homogenization corrector
of the nonlinear PDE in the sense of Definition~\ref{def:correctorsMonotonePDE}.
Finally, it holds
\begin{align}
\label{eq:BaseCaseConvFluxes}
\big\langle\big|q^{T}_{\xi}-q_{\xi}\big|^2\big\rangle \to 0
\end{align}
with the limiting flux $q_\xi$ defined in~\eqref{eq:PDEMonotoneFlux}.
\end{lemma}

\begin{proof}
This follows from~\cite[Lemma~26, Estimate~(68)]{Fischer2019}.
\end{proof}

\bibliographystyle{abbrv}
\bibliography{higher_order_linearized_corrector}

\end{document}